\documentclass[openany, 12pt]{amsbook}
\usepackage{comment}
\usepackage[utf8]{inputenc}
\usepackage[T1]{fontenc}
\DeclareMathAlphabet{\mathpzc}{OT1}{pzc}{m}{it}
\usepackage{geometry}

\usepackage{appendix}
\usepackage{chngcntr}
\usepackage{etoolbox}

\usepackage[dvipsnames]{xcolor}

\usepackage{amsfonts}
\usepackage{amssymb}
\usepackage{amsthm}
\usepackage{amsmath}
\usepackage{amscd}
\usepackage[shortlabels]{enumitem}
\usepackage{mathrsfs}
\usepackage{tikz}
\usetikzlibrary{calc,arrows,decorations.pathreplacing}
\usepackage{nicefrac, xfrac}
\usepackage{mathtools,xparse}

\usepackage{imakeidx}

\makeindex

\usepackage[pagebackref, colorlinks = true, 
linkcolor = red,
urlcolor  = blue,
citecolor = blue,
anchorcolor = blue]{hyperref}
\hypersetup{pdftitle={\@title},pdfauthor={\@author}}

\theoremstyle{plain}

\newtheorem{theorem}{Theorem}[section]
\newtheorem{corollary}[theorem]{Corollary}
\newtheorem{lemma}[theorem]{Lemma}

\newtheorem{proposition}[theorem]{Proposition}

\newtheorem{question}[theorem]{Question} 
\newtheorem{claim}{Claim}[theorem]

\theoremstyle{definition}
\newtheorem{definition}[theorem]{Definition}
\newtheorem{observation}[theorem]{Observation}
\newtheorem{remark}[theorem]{Remark}
\newtheorem{example}[theorem]{Example}

\newtheorem*{theorem*}{Theorem}

\newtheorem{thmx}{Theorem}
\newtheorem{corx}[thmx]{Corollary}

\newenvironment{subproof}[1][\proofname]{%
	\begin{proof}[#1]%
}{%
	\end{proof}%
}

\numberwithin{section}{chapter}
\numberwithin{subsection}{section}

\renewcommand{\epsilon}{\varepsilon}

\newcommand{\vp}{\ensuremath{\varphi}}
\newcommand{\spn}{\ensuremath{\mathrm{span}}}

\newcommand{\sub}{\ensuremath{ \subseteq}} 
\newcommand{\bus}{\ensuremath{ \supseteq}} 
\newcommand{\union}{\ensuremath{\bigcup}} 
\newcommand{\set}[1]{\ensuremath{ \left\{#1\right\} }} 
\newcommand{\norm}[1]{\ensuremath{ \Vert#1\Vert }} 

\newcommand{\vertiii}[1]{{\left\vert\kern-0.25ex\left\vert\kern-0.25ex\left\vert #1
	\right\vert\kern-0.25ex\right\vert\kern-0.25ex\right\vert}}
\newcommand{\trinorm}[1]{\ensuremath{\vertiii{#1}}} 
\newcommand{\absval}[1]{\ensuremath{ \left\lvert#1\right\rvert }} 
\newcommand{\spot}{\ensuremath{ \makebox[1ex]{\textbf{$\cdot$}} }}

\newcommand{\ol}[1]{\overline{#1}}
\newcommand{\Ul}[1]{\underline{#1}}

\newcommand{\bs}{\ensuremath{\backslash}}

\renewcommand\~{\tilde}
 
\newcommand{\cEP}{\ensuremath{\cE P}} 
\newcommand{\pzh}{\ensuremath{\mathpzc{h}}}

\newcommand{\floor}[1]{\ensuremath{\left\lfloor#1 \right\rfloor}}

\newcommand{\lt}{\ensuremath{\left}}
\newcommand{\rt}{\ensuremath{\right}}

\newcommand{\oio}{\ensuremath{\omega\in\Omega}}

\newcommand{\Om}{\ensuremath{\Omega}}
\newcommand{\Lm}{\ensuremath{\Lambda}}
\newcommand{\Gm}{\ensuremath{\Gamma}}

\newcommand{\Sg}{\ensuremath{\Sigma}}

\newcommand{\Dl}{\ensuremath{\Delta}}

\newcommand{\Ta}{\ensuremath{\Theta}}

\newcommand{\om}{\ensuremath{\omega}}
\newcommand{\lm}{\ensuremath{\lambda}}
\newcommand{\gm}{\ensuremath{\gamma}}
\newcommand{\al}{\ensuremath{\alpha}}
\newcommand{\bt}{\ensuremath{\beta}}
\newcommand{\sg}{\ensuremath{\sigma}}
\newcommand{\ep}{\ensuremath{\epsilon}}
\newcommand{\eps}{\ensuremath{\epsilon}}
\newcommand{\dl}{\ensuremath{\delta}}
\newcommand{\kp}{\ensuremath{\kappa}}
\newcommand{\zt}{\ensuremath{\zeta}}
\newcommand{\ta}{\ensuremath{\theta}}

\newcommand{\vta}{\ensuremath{\vartheta}}
\newcommand{\vkp}{\ensuremath{\varkappa}}
\newcommand{\vrho}{\ensuremath{\varrho}}

\DeclareSymbolFont{bbold}{U}{bbold}{m}{n}
\DeclareSymbolFontAlphabet{\mathbbold}{bbold}

\newcommand{\ind}{\ensuremath{\mathbbold{1}}}
\DeclareMathOperator*{\esssup}{ess\sup}
\DeclareMathOperator*{\essinf}{ess\inf}
\DeclareMathOperator{\HD}{HD}
\DeclareMathOperator{\diam}{diam}
\DeclareMathOperator{\supp}{supp}
\DeclareMathOperator{\rank}{rank}
\DeclareMathOperator{\End}{End}

\newcommand{\cA}{\ensuremath{\mathcal{A}}}
\newcommand{\cB}{\ensuremath{\mathcal{B}}}
\newcommand{\cC}{\ensuremath{\mathcal{C}}}
\newcommand{\cD}{\ensuremath{\mathcal{D}}}
\newcommand{\cE}{\ensuremath{\mathcal{E}}}
\newcommand{\cF}{\ensuremath{\mathcal{F}}}
\newcommand{\cG}{\ensuremath{\mathcal{G}}}

\newcommand{\cI}{\ensuremath{\mathcal{I}}}
\newcommand{\cJ}{\ensuremath{\mathcal{J}}}

\newcommand{\cL}{\ensuremath{\mathcal{L}}}

\newcommand{\tcL}{\ensuremath{\tilde{\mathcal{L}}}}

\newcommand{\cO}{\ensuremath{\mathcal{O}}}
\newcommand{\cP}{\ensuremath{\mathcal{P}}}
\newcommand{\cQ}{\ensuremath{\mathcal{Q}}}

\newcommand{\cV}{\ensuremath{\mathcal{V}}}

\newcommand{\cX}{\ensuremath{\mathcal{X}}}

\newcommand{\cZ}{\ensuremath{\mathcal{Z}}}

\newcommand{\sA}{\ensuremath{\mathscr{A}}}
\newcommand{\sB}{\ensuremath{\mathscr{B}}}
\newcommand{\sC}{\ensuremath{\mathscr{C}}}

\newcommand{\sF}{\ensuremath{\mathscr{F}}}

\newcommand{\sL}{\ensuremath{\mathscr{L}}}

\newcommand{\CC}{\ensuremath{\mathbb C}} 

\newcommand{\LL}{\ensuremath{\mathbb L}}

\newcommand{\NN}{\ensuremath{\mathbb N}}

\newcommand{\RR}{\ensuremath{\mathbb R}}

\newcommand{\XX}{\ensuremath{\mathbb X}}
\newcommand{\YY}{\ensuremath{\mathbb Y}}
\newcommand{\ZZ}{\ensuremath{\mathbb Z}} 

\def\lra{\longrightarrow}
\def\var{\text{{\rm var}}}

\def\BV{\text{{\rm BV}}}
\def\Leb{\text{{\rm Leb}}}
\def\Bd{\text{{\rm B}}}
\newcommand{\BVT}{\ensuremath{\BV_{temp}}}

\DeclareMathOperator*{\nuessinf}{\ensuremath{\nu\text{-ess}\inf}}
\DeclareMathOperator*{\Lebessinf}{\ensuremath{\text{Leb-ess}\inf}}
\DeclareMathOperator*{\nuesssup}{\ensuremath{\nu\text{-ess}\sup}}
\DeclareMathOperator*{\Lebesssup}{\ensuremath{\text{Leb-ess}\sup}}

\newcommand{\maeom}{$m$ a.e. $\om\in\Om$}

\def\lt{\left}
\def\rt{\right}

\providecommand{\phantomsection}{}
\AtBeginDocument{\let\textlabel\label}
\makeatletter
\newcommand{\mylabel}[2]{\raisebox{.7\normalbaselineskip}{\phantomsection}(#1)%
	\def\@currentlabel{#1}\textlabel{#2}}
\makeatother

\makeatletter
\newcommand{\myglabel}[2]{\raisebox{.7\normalbaselineskip}{\phantomsection}%
	\def\@currentlabel{#1}\textlabel{#2}}
\makeatother

\makeatletter
\newcommand\xlabel[2][]{\phantomsection\def\@currentlabelname{#1}\label{#2}}
\makeatother

\newcommand{\bcomma}{,\allowbreak}

\ExplSyntaxOn
\NewDocumentCommand{\mathlist}{ O{,} m m }
{
	\egreg_mathlist:nnn { #1 } { #2 } { #3 }
}
\seq_new:N \l__egreg_mathlist_seq
\cs_new_protected:Npn \egreg_mathlist:nnn #1 #2 #3
{
	\seq_set_split:Nnn \l__egreg_mathlist_seq { #1 } { #3 }
	\seq_use:Nnnn \l__egreg_mathlist_seq { #2 } { #2 } { #2 }
}
\ExplSyntaxOff

\AtBeginEnvironment{subappendices}{%
	\chapter*{Chapter 2 Appendices}
	\counterwithin{figure}{section}
	\counterwithin{table}{section}
}

\AtEndEnvironment{subappendices}{%
	\counterwithout{figure}{section}
	\counterwithout{table}{section}
}

\allowdisplaybreaks
\numberwithin{equation}{section}
\title[]{Thermodynamic Formalism and Perturbation Formulae for Quenched Random Open Dynamical Systems}
\date{\today}

\author{Jason Atnip}
\address{School of Mathematics and Statistics, University of New South Wales, Sydney, NSW 2052, Australia}
\curraddr{School of Mathematics and Physics, The University of Queensland, Brisbane, QLD 4072, Australia}
\email{\href{j.atnip@uq.edu.au}{j.atnip@uq.edu.au} }

\author{Gary Froyland}
\address{School of Mathematics and Statistics, University of New South Wales, Sydney, NSW 2052, Australia}
\email{\href{g.froyland@unsw.edu.au}{g.froyland@unsw.edu.au} }

\author{Cecilia Gonz\'alez-Tokman}
\address{School of Mathematics and Physics, The University of Queensland, Brisbane, QLD 4072, Australia}
\email{\href{cecilia.gt@uq.edu.au}{cecilia.gt@uq.edu.au} }

\author{Sandro Vaienti}
\address{Aix Marseille Université, Université de Toulon, CNRS, CPT, 13009 Marseille, France}
\email{\href{vaienti@cpt.univ-mrs.fr}{vaienti@cpt.univ-mrs.fr} }

\usepackage[automake]{glossaries-extra}

\makeglossaries

\newglossaryentry{H1}
{
	name=H1,
	description={For every $\oio$ we have $E_\om\bigoplus F_\om=X_\om$ and 
        \begin{align*}
            \max\{\|\Pi_{F_\om\|E_\om}\|, \|\Pi_{E_\om\|F_\om}\|\}\leq \Ta.
        \end{align*}
	}
}

\newglossaryentry{H2}
{
	name=H2,
	description={For each $\oio$ we have $L_\om F_\om=E_{\sg\om}$. Moreover, for every $n\in\NN$ and $f\in E_\om$ we have 
        \begin{align*}
            \|L_\om^n f\|\geq C_\lm \lm^n\|f\|.
        \end{align*}
	}
}

\newglossaryentry{H3}
{
	name=H3,
	description={For each $\oio$ we have $L_\om F_\om \sub F_{\sg\om}$ and for every $n\in\NN$ we have 
        \begin{align*}
            \|L_\om^n\rvert_{F_\om}\|\leq C_\mu \mu^n.
        \end{align*}
	}
}

\newglossaryentry{M1}
{
	name=M1,
	description={The map $T:\cJ_0\to\cJ_0$ is measurable with respect to $\sF\otimes\sB$
	}
}

\newglossaryentry{M2}
{
	name=M2,
	description={For every measurable function $f \in \cB$, the map $(\om, x) \mapsto (\cL_0 f)_\om(x)$ is measurable
	}
}

\newglossaryentry{CCM}
{
	name=CCM,
	description={There exists a random probability  measure $\nu_0=\set{\nu_{\om,0}}_{\om\in\Om}\in \cP_\Om(\cJ_0)$ and measurable functions $\lm_0:\Om \to\RR\bs\set{0}$ and $\phi_0:\cJ_0\to (0,\infty)$ with $\phi_0\in\cB$ such that 
		\begin{align*}
			\cL_{\om,0}(\phi_{\om,0})=\lm_{\om,0}\phi_{\sg\om,0}
			\quad\text{ and }\quad
			\nu_{\sg\om,0}(\cL_{\om,0}(f))=\lm_{\om,0}\nu_{\om,0}(f)
		\end{align*}
		for all $f\in\cB_\om$ where $\phi_{\om,0}(\spot):=\phi_0(\om,\spot)$. Furthermore, the fiber measures $\nu_{\om,0}$ are non-atomic and that $\lm_{\om,0}:=\nu_{\sg\om,0}(\cL_{\om,0}\ind)$ with $\log\lm_{\om,0}\in L^1(m)$. The $T$-invariant random probability measure	$\mu_0$ on $\cJ_0$ is then given by
		\begin{align*}
			\mu_{\om,0}(f):=\int_{\cJ_{\om,0}} f\phi_{\om,0} \ d\nu_{\om,0},  \qquad f\in L^1(\nu_{\om,0}).
		\end{align*}
	}
}

\newglossaryentry{X}
{
	name=X,
	description={For $m$-a.e. $\om\in\Om$ we have $X_{\om,\infty}\neq\emptyset$
	}
}

\newglossaryentry{D}
{
	name=D,
	description={For $m$-a.e. $\om\in\Om$ we have that 
		\begin{align*}
			D_{\om,\infty}\neq\emptyset
		\end{align*}
	}
}

\newglossaryentry{T1}
{
	name=T1,
	description={For each $\om\in\Om$, the map $T_\om:I\to I$  is surjective
	}
}

\newglossaryentry{T2}
{
	name=T2,
	description={For each $\om\in\Om$ there exists a finite partition $\cZ_\om$ of $I$ such that 
		$T_\om(Z)$ is an interval for each  $Z\in\cZ_\om$
	}
}

\newglossaryentry{T3}
{
	name=T3,
	description={For each $\om\in\Om$ there exists a finite partition $\cZ_\om$ of $I$ such that $T_\om\rvert_Z $ is continuous and strictly monotone for each $Z\in\cZ_\om$
	}
}

\newglossaryentry{LIP}
{
	name=LIP,
	description={$\log\#\cZ_\om\in L^1(m)$
	}
}

\newglossaryentry{GP}
{
	name=GP,
	description={The partitions $\cZ_\om$ are generating, i.e. 
		\begin{align*}
			\bigvee_{n=1}^\infty \cZ_\om^{(n)}=\sB
		\end{align*}
	}
}

\newglossaryentry{M}
{
	name=M,
	description={The map $T:\Om\times I\to\Om\times I$ is measurable
	}
}

\newglossaryentry{C}
{
	name=C,
	description={There exists a random probability measure $\nu_0=\set{\nu_{\om,0}}_{\om\in\Om}\in \cP_\Om(\Om\times I)$ and measurable functions $\lm_0:\Om \to\RR\bs\set{0}$ and $\phi_0:\cJ_0\to (0,\infty)$ with $\phi_0\in\BV_\Om(I)$ such that 
		\begin{align*}
			\cL_{\om,0}(\phi_{\om,0})=\lm_{\om,0}\phi_{\sg\om,0}
			\quad\text{ and }\quad
			\nu_{\sg\om,0}(\cL_{\om,0}(f))=\lm_{\om,0}\nu_{\om,0}(f)
		\end{align*}
		for all $f\in\BV(I)$. Furthermore, we suppose that the fiber measures $\nu_{\om,0}$ are non-atomic and that $\lm_{\om,0}:=\nu_{\sg\om,0}(\cL_{\om,0}\ind)$ with $\log\lm_{\om,0}\in L^1(m)$. The $T$-invariant random probability measure $\mu_0$ on $\Om\times I$ is given by
		\begin{align*}
			\mu_{\om,0}(f):=\int_{I} f\phi_{\om,0} \ d\nu_{\om,0},  \qquad f\in L^1(\nu_{\om,0}).
		\end{align*}
	}
}

\newglossaryentry{Q1}
{
	name=Q1,
	description={
		\begin{align*}
			\lim_{n\to\infty}\frac{1}{n}\log\norm{g_{\om}^{(n)}}_\infty
			+
			\limsup_{n\to\infty}\frac{1}{n}\log\xi_{\om}^{(n)}
			<
			\lim_{n\to\infty}\frac{1}{n}\log\rho_{\om}^n
			=
			\int_\Om \log\rho_{\om}\, dm(\om).
		\end{align*}
	}
}

\newglossaryentry{Q2}
{
	name=Q2,
	description={For each $n\in\NN$ we have $\log\xi_\om^{(n)}\in L^1(m)$.
	}
}

\newglossaryentry{Q3}
{
	name=Q3,
	description={For each $n\in\NN$, $\log\dl_{\om,n}\in L^1(m)$, where
		\begin{align*}
			\dl_{\om,n}:=\min_{Z\in\cZ_{\om,g}^{(n)}}\Lm_\om(\ind_Z).
		\end{align*}
	}
}

\newglossaryentry{Q2'}
{
	name=Q2',
	description={$\log\xi_\om^{(N_*)}\in L^1(m)$.
	}
}

\newglossaryentry{Q3'}
{
	name=Q3',
	description={$\log\dl_{\om,N_*}\in L^1(m)$, where $\dl_{\om,n}$ is defined by \eqref{eq: def of dl_om,n}
	}
}

\newglossaryentry{Z}
{
	name=Z,
	description={There exists $\hat\al\geq 0$ sufficiently large such that $\cZ_\om^{(n)}\in\ol\sA_\om^{(n)}(\hat\al)$ for each $n\in\NN$ and each $\om\in\Om$.  
	}
}

\newglossaryentry{Q0hat}
{
	name={\ensuremath{\widehat{\textrm{Q0}}}},
	description={$\ol\cZ_{\om,F}^{(1)}\neq\emptyset$ for $m$-a.e. $\om\in\Om$
	},
	sort={QZ0hat}
}

\newglossaryentry{Q1hat}
{
	name={\ensuremath{\widehat{\textrm{Q1}}}},
	description={\begin{align*}
			\lim_{n\to\infty}\frac{1}{n}\log\norm{g_{\om}^{(n)}}_\infty
			+
			\limsup_{n\to\infty}\frac{1}{n}\log\zt_{\om}^{(n)}
			<
			\lim_{n\to\infty}\frac{1}{n}\log\rho_{\om}^n
			=
			\int_\Om \log\rho_{\om}\, dm(\om),
		\end{align*}
	},
	sort={QZ1hat}
}

\newglossaryentry{Q2hat}
{
	name={\ensuremath{\widehat{\textrm{Q2}}}},
	description={For each $n\in\NN$ we have $\log^+\zt_\om^{(n)}\in L^1(m)$
	},
	sort={QZ2hat}
}

\newglossaryentry{Q3hat}
{
	name={\ensuremath{\widehat{\textrm{Q3}}}},
	description={For each $n\in\NN$, $\log\hat\dl_{\om,n}\in L^1(m)$, where
		\begin{align*}
			\hat\dl_{\om,n}:=\min_{Z\in\ol\cZ_{\om,F}^{(n)}}\Lm_\om(\ind_Z).
		\end{align*}
	},
	sort={QZ3hat}
}

\newglossaryentry{Z'}
{
	name=Z',
	description={There exists $\hat\al\geq 0$ such that $\cZ_\om^{(N_*)}\in\ol\sA_\om^{(N_*)}(\hat\al)$ for each $\om\in\Om$
	}
}

\newglossaryentry{Q2hat'}
{
	name={\ensuremath{\widehat{\textrm{Q2}}}'},
	description={$\log^+\zt_\om^{(N_*)}\in L^1(m)$
	},
	sort={QZ2hat'}
}

\newglossaryentry{Q3hat'}
{
	name={\ensuremath{\widehat{\textrm{Q3}}}'},
	description={$\log\hat\dl_{\om,N_*}\in L^1(m)$
	},
	sort={QZ3hat'}
}

\newglossaryentry{P1}
{
	name=P1,
	description={There exists a	function $C_1:\Om\to\RR_+$ such that for $f\in\cB_\om$  we have
		\begin{align*}
			\sup_{\ep\geq 0}\norm{\cL_{\om,\ep}(f)}_{\cB_{\sg\om}}\leq C_1(\om)\norm{f}_{\cB_\om}
		\end{align*}
	}
}

\newglossaryentry{P2}
{
	name=P2,
	description={For each $\om\in\Om$ and $\ep\geq 0$ there is a functional $\nu_{\om,\ep}\in\cB_\om^*$, the dual space of $\cB_\om$, $\lm_{\om,\ep}\in\CC\bs\set{0}$, and $\phi_{\om,\ep}\in\cB_\om$ such that
		\begin{align*}
			\cL_{\om,\ep}(\phi_{\om,\ep})=\lm_{\om,\ep}\phi_{\sg\om,\ep}
			\quad\text{ and }\quad
			\nu_{\sg\om,\ep}(\cL_{\om,\ep}(f))=\lm_{\om,\ep}\nu_{\om,\ep}(f)
		\end{align*}
		for all $f\in\cB_\om$. Furthermore,
		$$
		\nu_{\om,0}(\phi_{\om,\ep})=1.
		$$
	}
}

\newglossaryentry{P3}
{
	name=P3,
	description={There is an operator $Q_{\om,\ep}:\cB_\om\to\cB_{\sg\om}$ such that for each $f\in\cB_\om$ we have
		\begin{align*}
			\lm_{\om,\ep}^{-1}\cL_{\om,\ep}(f)=\nu_{\om,\ep}(f)\cdot\phi_{\sg\om,\ep}+Q_{\om,\ep}(f).
		\end{align*}
		Furthermore, we have
		\begin{align*}
			Q_{\om,\ep}(\phi_{\om,\ep})=0
			\quad\text{ and }\quad
			\nu_{\sg\om,\ep}(Q_{\om,\ep}(f))=0
		\end{align*}
	}
}

\newglossaryentry{P4}
{
	name=P4,
	description={There exists a function $C_2:\Om\to\RR_+$ such that
		\begin{align*}
			\sup_{\ep\geq 0}\norm{\phi_{\om,\ep}}_{\cB_{\om}}=C_2(\om)<\infty
		\end{align*}
	}
}

\newglossaryentry{P5}
{
	name=P5,
	description={For each $\om\in\Om$ we have
		\begin{align*}
			\lim_{\ep\to 0}\eta_{\om,\ep}=0
		\end{align*}
	}
}

\newglossaryentry{P6}
{
	name=P6,
	description={For each $\om\in\Om$ such that $\Dl_{\om,\ep}>0$ for every $\ep>0$, there exists a  function $C_3:\Om\to\RR_+$ such that
		\begin{align*}
			\limsup_{\ep\to 0}\frac{\eta_{\om,\ep}}{\Dl_{\om,\ep}}:=C_3(\om) <\infty.
		\end{align*}
		Given $\om\in\Om$, if there is $\ep_0>0$ such that for each $\ep\leq \ep_0$ we have that $\Dl_{\om,\ep}=0$ then we also have that $\eta_{\om,\ep}=0$ for each $\ep\leq \ep_0$
	}
}

\newglossaryentry{P7}
{
	name=P7,
	description={For each $\om\in\Om$ 
		\begin{align*}
			\lim_{\ep \to 0}\nu_{\om,\ep}(\phi_{\om,0})=1
		\end{align*}	
	}
}

\newglossaryentry{P8}
{
	name=P8,
	description={For each $\om\in\Om$ with $\Dl_{\om,\ep}>0$ for all $\ep>0$ we have
		\begin{flalign*}
			\lim_{n\to\infty}\limsup_{\ep \to 0} \Dl_{\om,\ep}^{-1}\nu_{\sg\om,0}\lt(\lt(\cL_{\om,0}-\cL_{\om,\ep}\rt)\lt(Q_{\sg^{-n}\om,\ep}^n\phi_{\sg^{-n}\om,0}\rt)\rt)=0.
		\end{flalign*}
	}
}

\newglossaryentry{P9}
{
	name=P9,
	description={For each $\om\in\Om$ with $\Dl_{\om,\ep}>0$ for all $\ep>0$ the limit
		\begin{align*}
			q_{\om,0}^{(k)}:=\lim_{\ep\to 0} \frac{\nu_{\sg\om,0}\left((\cL_{\om,0}-\cL_{\om,\ep})(\cL_{\sg^{-k}\om,\ep}^k)(\cL_{\sg^{-(k+1)}\om,0}-\cL_{\sg^{-(k+1)}\om,\ep})(\phi_{\sg^{-(k+1)}\om,0})\right)}{\nu_{\sg\om,0}\left((\cL_{\om,0}-\cL_{\om,\ep})(\phi_{\om,0})\right)}
		\end{align*}
		exists for each $k\geq 0$
	}
}

\newglossaryentry{B}
{
	name=B,
	description={For each $\om$ $\cB_\om\sub L^\infty(\nu_{\om,0})$, where $\nu_0=(\nu_{\om,0})_{\om\in\Om}$ is the random probability measure given by \eqref{CCM}, and there exists a measurable $m$-a.e. finite function $K:\Om\to[1,\infty)$ such that $\norm{f}_{\infty,\om}\leq K_\om\norm{f}_{\cB_\om}$ 
		for all $f\in\cB_\om$ and each $\om\in\Om$, where $\|\cdot\|_{\infty,\om}$ denotes the supremum norm with respect to $\nu_{\om,0}$.
	}
}

\newglossaryentry{A}
{
	name=A,
	description={For each $\ep>0$ $H_\ep\sub \cJ_0$ is measurable with respect to the product $\sg$-algebra $\sF\otimes\sB$ on $\cJ_0$ such that $H_\ep'\sub H_\ep$ for each $0<\ep'\leq \ep$
	}
}

\newglossaryentry{A'}
{
	name=A',
	description={$H_{\om,\ep'}\sub H_{\om,\ep}$ for each $\ep'\leq \ep$ and each $\om\in\Om$
	}
}

\newglossaryentry{C1}
{
	name=C1,
	description={There exists a measurable $m$-a.e. finite function $C_1:\Om\to\RR_+$ such that for $f\in\cB_\om$ we have
		\begin{align*}
			\sup_{\ep\geq 0}\norm{\cL_{\om,\ep}(f)}_{\cB_{\sg\om}}\leq C_1(\om)\norm{f}_{\cB_\om}
		\end{align*}
	}
}

\newglossaryentry{C2}
{
	name=C2,
	description={For each $\ep\geq 0$ there is a random measure  $\set{\nu_{\om,\ep}}_{\om\in\Om}$ supported in $\cJ_{0}$
		and measurable functions $\lm_{\ep}:\Om\to(0,\infty)$ 
		with $\log\lm_{\om,\ep}\in L^1(m)$ and $\phi_{\ep}:\cJ_0\to \RR$ such that 
		\begin{align*}
			\cL_{\om,\ep}(\phi_{\om,\ep})=\lm_{\om,\ep}\phi_{\sg\om,\ep}
			\quad\text{ and }\quad
			\nu_{\sg\om,\ep}(\cL_{\om,\ep}(f))=\lm_{\om,\ep}\nu_{\om,\ep}(f)
		\end{align*}
		for all $f\in\cB_\om$. Furthermore, for $m$-a.e. $\om\in\Om$
		$$
		\nu_{\om,0}(\phi_{\om,\ep})=1
		\quad\text{ and } \quad
		\nu_{\om,0}(\ind)=1
		$$
	}
}

\newglossaryentry{C3}
{
	name=C3,
	description={There is an operator $Q_{\om,\ep}:\cB_\om\to\cB_{\sg\om}$ such that for $m$-a.e. $\om\in\Om$ and each $f\in\cB_\om$ we have
		\begin{align*}
			\lm_{\om,\ep}^{-1}\cL_{\om,\ep}(f)=\nu_{\om,\ep}(f)\cdot\phi_{\sg\om,\ep}+Q_{\om,\ep}(f).
		\end{align*}
		Furthermore, for $m$-a.e. $\om\in\Om$ we have
		\begin{align*}
			Q_{\om,\ep}(\phi_{\om,\ep})=0
			\quad\text{ and }\quad
			\nu_{\sg\om,\ep}(Q_{\om,\ep}(f))=0
		\end{align*}
	}
}

\newglossaryentry{C4}
{
	name=C4,
	description={For each $f\in\cB$ there exist measurable functions $C_f:\Om\to(0,\infty)$ and $\al:\Om\times \NN\to(0,\infty)$ with $\al_\om(N)\to 0$ as $N\to\infty$ such that for $m$-a.e. $\om\in\Om$ and all $N\in\NN$
		\begin{align*}
			\sup_{\ep\geq 0}\norm{Q_{\om,\ep}^N f_{\om}}_{\infty,\sg^N\om}
			&\leq 
			C_f(\om)\al_\om(N)\norm{f_{\om}}_{\cB_{\om}}, 
			\\
			\sup_{\ep\geq 0}\norm{Q_{\sg^{-N}\om,\ep}^Nf_{\sg^{-N}\om}}_{\infty,\om}
			&\leq 
			C_f(\om)\al_\om(N)\norm{f_{\sg^{-N}\om}}_{\cB_{\sg^{-N}\om}}, 
		\end{align*} 
		and $\norm{\phi_{\sg^{N}\om,0}}_{\infty,\sg^N\om}\al_\om(N)\to 0$, $\norm{\phi_{\sg^{-N}\om,0}}_{\infty,\sg^{-N}\om}\al_\om(N)\to 0$ as $N\to\infty$
	}
}

\newglossaryentry{C5}
{
	name=C5,
	description={There exists a measurable $m$-a.e. finite function $C_2:\Om\to[1,\infty)$ such that 
		\begin{align*}
			\sup_{\ep\geq 0}\norm{\phi_{\om,\ep}}_{\infty,\om}\leq C_2(\om) 
			\quad\text{ and }\quad
			\norm{\phi_{\om,0}}_{\cB_\om}\leq C_2(\om)
		\end{align*}
	}
}

\newglossaryentry{C6}
{
	name=C6,
	description={For $m$-a.e. $\om\in\Om$ we have
		\begin{align*}
			\lim_{\ep\to 0}\eta_{\om,\ep}=0
		\end{align*}
	}
}

\newglossaryentry{C7}
{
	name=C7,
	description={There exists a measurable $m$-a.e. finite function $C_3:\Om\to[1,\infty)$ such that  for all $\ep>0$ sufficiently small we have
		\begin{align*}
			\inf\phi_{\om,0}\geq C_3^{-1}(\om)>0
			\qquad\text{ and }\qquad
			\essinf_\om\inf\phi_{\om,\ep}\geq 0
		\end{align*}
	}
}

\newglossaryentry{C8}
{
	name=C8,
	description={For $m$-a.e. $\om\in\Om_+$ 
		we have that the limit $\hat{q}_{\om,0}^{(k)}:=\lim_{\ep\to 0} \hat{q}_{\om,\ep}^{(k)}$ exists for each $k\geq 0$, where $\hat{q}_{\om,\ep}^{(k)}$ is as in \eqref{def of hat q}
	}
}

\newglossaryentry{S}
{
	name=S,
	description={For any fixed random scaling function $t\in L^\infty(m)$ with $t>0$, there are sequences of functions $z_{N},\xi_N\in L^\infty(m)$  and a constant $W<\infty$ satisfying 
		$$\mu_{\omega,0}(\{h_\omega(x)-z_{\omega,N}>0\})=(t_\omega+\xi_{\omega,N})/N, \mbox{ for a.e.\ $\omega$ and each $N\ge 1$}
		$$
		where:\\
		(i)
		$\lim_{N\to\infty} \xi_{\omega,N}=0$ for a.e.\ $\omega$ and \\
		(ii) $|\xi_{\omega,N}|\le W$ for a.e.\ $\omega$ and all $N\ge 1$
	}
}

\newglossaryentry{C1'}
{
	name=C1',
	description={There exists $C_1\geq 1$ such that for $m$-e.a. $\om\in\Om$ we have 
		\begin{align*}
			C_1^{-1}\leq \cL_{\om,0}\ind\leq C_1
		\end{align*}
	}
}

\newglossaryentry{C4'}
{
	name=C4',
	description={For each $f\in\cB$ and each $N\in\NN$ there exists $C_f>0$ and $\al(N)>0$ (independent of $\om$) with $\al:=\sum_{N=1}^\infty\al(N)<\infty$ such that for $m$-a.e. $\om\in\Om$, all $N\in\NN$
		\begin{align*}
			\sup_{\ep\geq 0}\norm{Q_{\om,\ep}^N f_{\om}}_{\infty,\sg^N\om}\leq C_f\al(N)\norm{f_{\om}}_{\cB_{\om}}
		\end{align*}
	}
}

\newglossaryentry{C5'}
{
	name=C5',
	description={There exists $C_2\geq 1$ such that
		\begin{align*}
			\sup_{\ep\geq 0}\norm{\phi_{\om,\ep}}_{\infty,\om}\leq C_2
			\quad\text{ and }\quad
			\norm{\phi_{\om,0}}_{\cB_\om}\leq C_2
		\end{align*}
		for $m$-a.e. $\om\in\Om$
	}
}

\newglossaryentry{C7'}
{
	name=C7',
	description={There exists $C_3\geq 1$ such that for all $\ep>0$ sufficiently small we have
		\begin{align*}
			\essinf_\om\inf\phi_{\om,0}\geq C_3^{-1} >0
			\qquad\text{ and }\qquad 
			\essinf_\om\inf\phi_{\om,\ep}\geq 0
		\end{align*}
	}
}

\newglossaryentry{E1}
{
	name=E1,
	description={There exists $C\geq 1$ such that 
		\begin{align*}
			\esssup_\om |T_\om'|\leq C
			\qquad\text{ and }\qquad
			\esssup_\om D(T_\om)
			\leq C,
		\end{align*}
		where $D(T_\om):=\sup_{y\in[0,1]}\# T_\om^{-1}(y)$
	}
}

\newglossaryentry{MC}
{
	name=MC,
	description={The map $\sg:\Om\to\Om$ is a homeomorphism, the skew-product map $T:\Om\times [0,1]\to\Om\times [0,1]$ is measurable, and $\omega\mapsto T_\omega$ has countable range
	}
}

\newglossaryentry{E2}
{
	name=E2,
	description={The weight function $g_{\om,0}$ lies in $\BV$ for each $\om\in\Om$ and satisfies
		\begin{equation*}
			\esssup_\omega \| g_{\omega,0}\|_{\infty,1}<\infty
		\end{equation*}
	}
}

\newglossaryentry{E3}
{
	name=E3,
	description={\begin{equation*}
			\essinf_\omega \inf g_{\omega,0}>0
		\end{equation*}
	}
}

\newglossaryentry{E4}
{
	name=E4,
	description={For every subinterval $J\subset [0,1]$ there is a $k=k(J)$ such that for a.e.\ $\omega$ one has $T^k_\omega(J)=[0,1]$
	}
}

\newglossaryentry{E5}
{
	name=E5,
	description={There is a uniform-in-$\epsilon$ and uniform-in-$\omega$ upper bound on the number of connected components of $H_{\om,\ep}$
	}
}

\newglossaryentry{E6}
{
	name=E6,
	description={\begin{align*}
			\lim_{\ep\to 0}\esssup_\om \Leb(H_{\om,\ep})=0
		\end{align*}
	}
}

\newglossaryentry{EX}
{
	name=EX,
	description={There exists an $\ep>0$ and an open neighborhood $\~H_{\om,\ep}\bus H_{\om,\ep}$ such that  
		$T_\om(U_\om)\bus \~H_{\sg\om,\ep}^c$, where $U_\om:=\cup_{Z\in\cZ_{\om,0}}\ol{Z}_{\ep}$ and $\ol{Z_{\ep}}$ denotes the closure of $Z_\ep\in A_{\om,\ep}:=\{Z\cap \~H_{\om,\ep}^c:Z\in\cZ_{\om,0}\}$ with $m(\set{\om\in\Om: \#A_{\om,\ep}\geq 2})>0$
	}
}

\newglossaryentry{E7}
{
	name=E7,
	description={For $m$-a.e. $\om\in\Om$ and all $\ep>0$ sufficiently small 
		\begin{align*}
			T_\om(\cJ_{\om,\ep})=[0,1]
		\end{align*}
	}
}

\newglossaryentry{E8}
{
	name=E8,
	description={There exists $n'\geq1$ and $\epsilon_0>0$ such that
		\begin{equation*}
			9\cdot\esssup_\om\|g_{\omega,0}^{(n')}\|_{\infty,1}
			<
			\essinf_\om\inf_{0\leq\ep\le \ep_0}\inf\cL_{\om,\ep}^{n'}\ind
		\end{equation*}
	}
}

\newglossaryentry{E9}
{
	name=E9,
	description={There exists $k_o(n')\in\NN$ such that for $m$-a.e. $\om\in\Om$, all $\ep>0$ sufficiently small, and all $Z\in\cZ_{\om,*,\ep}^{(n')}$ we have $T_\om^{k_o(n')}(Z)=[0,1]$, where $n'$ is the number coming from \eqref{E8}
	}
}

\newglossaryentry{E9a}
{
	name=E9a,
	description={There exists $k_o(n')\in\NN$ and $\dl>0$ such that for $m$-a.e. $\om\in\Om$, all $\ep>0$ sufficiently small, we have $\Leb(Z)>\dl$ for all $Z\in\cZ_{\om,*,\ep}^{(n')}(\cA)$
	}
}

\newglossaryentry{E9b}
{
	name=E9b,
	description={There exists $c>0$ such that $\essinf_\om |T_\om'|>c$
	}
}

\newglossaryentry{A1}
{
	name=A1,
	description={For $m$-a.e. $\om\in\Om$ we have $\inf \vp_{\om,0}, \sup\vp_{\om,0}\in L^1(m)$.}
}

\newglossaryentry{A2}
{
	name=A2,
	description={For $m$-a.e. $\om\in\Om$ we have $g_{\om,0}\in\BV(I)$.}
}

\newglossaryentry{CCH}
{
	name=CCH,
	description={The map $\log \pzh_\om\in L^1(m)$, where $\pzh_\om$ denotes the number of connected components of $H_\om$.}
}

\makeatletter
\newglossarystyle{mystyle}{%
	\setglossarystyle{list}%
	{\end{description}}
	\renewcommand*{\glossentry}[2]{%
	\item[\glsentryitem{##1}%
	\glstarget{##1}{\glossentryname{##1}}]%
	\mbox{}\par\nobreak\@afterheading
	\glossentrydesc{##1}\glspostdescription
	{\def\hfill{\hskip 25pt plus 3fill}\dotfill\mbox{\textit{see page} ##2}}%
	}%
}
\makeatother

\setglossarystyle{mystyle}
\glstocfalse 

\begin{document}
\begin{abstract}
We develop a quenched thermodynamic formalism for open random dynamical systems generated by finitely branched, piecewise-monotone mappings of the interval.
The openness refers to the presence of holes in the interval, which terminate trajectories once they enter;  the holes may also be random. Our random driving is generated by an invertible, ergodic, measure-preserving transformation $\sigma$ on a probability space $(\Omega,\mathscr{F},m)$.
For each $\omega\in\Omega$ we associate a piecewise-monotone, surjective map $T_\omega:I\to I$, and a hole $H_\omega\subset [0,1]$;  the map $T_\omega$, the random potential $\vp_\omega$, and the hole $H_\omega$ generate the corresponding open transfer operator $\mathcal{L}_\omega$. The paper is divided into two chapters. In the first chapter we prove, for a contracting potential, that there exists a unique random probability measure $\nu_\omega$ supported on the survivor set ${X}_{\omega,\infty}$ satisfying $\nu_{\sigma(\omega)}(\mathcal{L}_\omega f)=\lambda_\omega\nu_\omega(f)$.
Correspondingly, we also prove the existence of a unique (up to scaling and modulo $\nu$) random family of functions $\phi_\omega$ that satisfy $\mathcal{L}_\omega \phi_\omega=\lambda_\omega \phi_{\sigma(\omega)}$.
Together, these provide an ergodic random invariant measure $\mu=\nu \phi     $ supported on the global survivor set $\mathcal{X}_{\infty}$, while $\phi$ combined with the random closed conformal measure yields a random absolutely continuous conditional invariant measure (RACCIM) $\eta$ supported on $[0,1]$.
Further, we prove quasi-compactness of the transfer operator cocycle generated by $\mathcal{L}_\omega$ and exponential decay of correlations for $\mu$.
The escape rates of the random closed conformal measure and the RACCIM $\eta$ coincide, and are given by the difference of the expected pressures for the closed and open random systems. Finally, we prove that the Hausdorff dimension of the surviving set $X_{\omega,\infty}$ is equal to the unique zero of the expected pressure function for almost every fiber $\omega\in\Omega$. We provide examples, including a large class of random Lasota-Yorke maps with holes, for which the above results apply.  
\\
In the second chapter of the paper we consider quasi-compact linear operator cocycles $\mathcal{L}^{n}_{\omega,0}:=\mathcal{L}_{\sigma^{n-1}\omega,0}\circ\cdots\circ\mathcal{L}_{\sigma\omega,0}\circ \mathcal{L}_{\omega,0}$, and their small perturbations $\mathcal{L}_{\omega,\epsilon}^{n}$. The operators $\cL_{\om,0}$ and $\cL_{\om,\ep}$ need not be transfer operators. 
We prove an abstract $\omega$-wise first-order formula for the leading Lyapunov multipliers  $\lambda_{\omega,\epsilon}=\lambda_{\omega,0}-\theta_{\omega}\Delta_{\omega,\epsilon}+o(\Delta_{\omega,\epsilon})$, where $\Delta_{\omega,\epsilon}$ quantifies the closeness of $\mathcal{L}_{\omega,\epsilon}$ and $\mathcal{L}_{\omega,0}$.
We then consider the situation where $\mathcal{L}_{\omega,0}^{n}$ is a transfer operator cocycle for a closed random map cocycle $T_\omega^{n}$ 
and the perturbed transfer operators $\mathcal{L}_{\omega,\epsilon}$ are defined by the introduction of small random holes $H_{\omega,\epsilon}$ in $[0,1]$, creating a random open dynamical system.
We obtain a first-order perturbation formula in this setting, which reads $\lambda_{\omega,\epsilon}=\lambda_{\omega,0}-\theta_{\omega}\mu_{\omega,0}(H_{\omega,\epsilon})+o(\mu_{\omega,0}(H_{\omega,\epsilon})),$ where $\mu_{\omega,0}$ is the unique equivariant random measure (and equilibrium state) for the original closed random dynamics. 
Our new machinery is then deployed to create a spectral approach for a quenched extreme value theory that considers random dynamics and random observations.
An extreme value law is derived using the first-order terms $\theta_\omega$.
Further, in the setting of random piecewise expanding interval maps, we establish the existence of random equilibrium states and conditionally invariant measures for random open systems with small holes via a random perturbative approach, in contrast to the cone-based arguments of the first chapter.
Finally we prove quenched statistical limit theorems for random equilibrium states arising from contracting potentials.
We illustrate all of the above theory with a variety of examples.

\end{abstract}

\maketitle
\newpage
\glsxtrRevertTocMarks
\tableofcontents
\clearpage

\setcounter{chapter}{-1}

\chapter{Introduction}
In this paper we develop a quenched thermodynamic formalism for random open dynamical systems, and their perturbations. This work may be seen as a successor to \cite{AFGTV20}, which was devoted to the construction of conformal measures and equilibrium states for random potentials associated to countably branched random maps of the unit interval. We now extend such a formalism to the more challenging setting of random \textit{open} dynamical systems. A random dynamical system is qualified to be open if it contains holes which terminate trajectories once they enter. Similarly, when referring to a \textit{closed} dynamical system, we mean a dynamical system (in the usual sense) without a hole in phase space. These notions of openness and closedness have nothing to do with the topological concepts of open and closed sets.
Throughout the text we will refer to open and closed measures and open and closed transfer operators, by which we mean measures and transfer operators associated to the random dynamical system with or without holes, respectively. We note that our closed transfer operator does not have any relation to the notion of closed (unbounded) linear operators.

The current paper is divided into three chapters.
In the first chapter, the random holes are allowed to be large ensuring that the asymptotic dynamics take place on a surviving set and the objective will be to define equilibrium states and their variational principles. A key step in this program will be the construction of (random) conditionally invariant probability measures. Moreover we will give an explicit formula for the escape rate. 
In the second chapter, the introduction of a decreasing family of small random holes with measure tending to zero are seen as a perturbation of a transfer operator. We will first develop a perturbation theory for (nonautonomous) cocycles of transfer operators, 
which we will then use to study the escape rates and recurrence to small random holes. 
We then use our general random perturbation theory to prove extreme value laws. Finally, in the third chapter, we present a short selection of open questions. 

\, 

\begin{center}
\textbf{Summary of Chapter 1. Thermodynamic formalism for random interval maps with holes}
\end{center}

\, 

Deterministic closed transitive dynamics $T:[0,1]\to[0,1]$ with enough expansivity enjoy a ``thermodynamic formalism'':  the transfer operator with a sufficiently regular potential $\vp$ has a unique absolutely continuous invariant measure (ACIM) $\mu_0$, absolutely continuous with respect to the conformal measure $\nu_0$.
Furthermore, $\mu_0$ arises as an equilibrium state, i.e.\ a maximiser of the free energy, which is a sum of the integral of the potential $\vp$ and the metric entropy $h(\mu_0)$.

Continuing with the deterministic setting, if one introduces a hole $H\subset [0,1]$, the situation becomes considerably more complicated.
In the simplest case where the potential is the usual geometric potential $\vp=-\log|T'|$, because of the lack of mass conservation, one expects at best an absolutely continuous \emph{conditionally} invariant measure (ACCIM) $\eta$, conditioned according to survival from the infinite past, and absolute continuous with respect to the closed conformal measure $\nu_0$. Considering the dynamics restricted to the survivor set $X_\infty$, namely the set of points whose infinite forward trajectories remain in $[0,1]$, there exists a unique open ACIM $\mu$, supported on $X_\infty$, absolutely continuous with respect to the open conformal measure $\nu$. 
Early work on the existence of the ACCIM and exponential convergence of non-equilibrium densities, includes \cite{vandenBedemChernov97,Colletetal00}.
The paper \cite{LMD} handles general potentials that are contracting \cite{LSV} for the closed system, demonstrating exponential decay for $\mu$.
There has been further work on the Lorentz gas and billiards \cite{DemersWrightYoung10,Demers13}, intermittent maps \cite{DemersFernandez14,DemersTodd17intermittent}, and multimodal maps \cite{DemersTodd17multimodal, DemersTodd20multimodal}.
In the setting of diffeomorphisms with SRB measures, following the introduction of a hole, relations between escape rates and the pressures have been studied in \cite{DemersWrightYoung12}. 

Looking to the fractal dimension of the surviving set $X_\infty$, the machinery of thermodynamic formalism was first employed by Bowen \cite{bowen_hausdorff_1979} to find the Hausdorff dimension of the limit sets of quasi-Fuchsian groups in terms of the pressure function, and then pioneered in the setting of open dynamical systems in \cite{urbanski_hausdorff_1986, U87}. 

In the random setting, repeated iteration of a single deterministic map is replaced with the composition of maps $T_\omega:X\to X$ drawn from a collection $\{T_\omega\}_{\omega\in\Omega}$.
A driving map $\sigma:\Omega\to\Omega$ on a probability space $(\Omega,\sF,m)$ creates a map cocycle $T_\omega^n:=T_{\sigma^{n-1}\omega}\circ\cdots\circ T_{\sigma\omega}\circ T_\omega$.
The authors recently developed a complete, quenched thermodynamic formalism for random, countably-branched, piecewise monotonic interval maps \cite{AFGTV20}, enabling the treatment of discontinuous, non-Markov $T_\omega$.

There are two approaches to study random dynamical systems, \emph{annealed} and \emph{quenched}. 
In the quenched setting, which is the setting in which we work, one considers general ergodic random driving and uses a fiberwise transfer operator $\cL_\om$ defined for each realization $\om$ of the driving to find an $\omega$-family of invariant measures. Then one seeks to prove results which hold almost surely, or for a ``typical'' driving realization.
On the other hand, in the annealed setting, one considers i.i.d. random dynamics (see Remark \ref{rem: iid iteration} for details) and uses the annealed ($\omega$-averaged) transfer operator to find an annealed invariant measure. In the setting of i.i.d. random driving, the annealed invariant measure is the marginal of the quenched random invariant measure, see \cite[2.1.8 Theorem]{arnold_random_1998}. Thus, the quenched approach covers a broader class of driving mechanisms, and when the annealed approach is also available, the quenched approach provides information about specific random trajectories, as opposed to information about the averaged behaviour from the annealed perspective.

The situation of random open dynamics is relatively untouched.
For a single piecewise expanding map $T:[0,1]\to[0,1]$ with holes $H_\omega$ randomly chosen in an i.i.d.\ fashion, \cite{BahsounVaienti13} consider escape rates for the annealed transfer operator in the small hole limit (the Lebesgue measure of the $H_\omega$ goes to zero).
In a similar setting, now assuming $T$ to be Markov and considering non-vanishing holes, \cite{Bahsounetal15} show existence of equilibrium states, again for the annealed transfer operator.
For the first time in \cite{froyland_metastability_2013} the authors consider escape rates for quenched random open interval maps where they are able to show that the escape rate is bounded above by the Lyapunov exponent of a Perron-Frobenius cocycle. 
In \cite{atnip_critically_2020}, the authors consider random, full-branched interval maps with negative Schwarzian derivative. The maps are allowed to have critical points, but the partition of monotonicity and holes, made up of finitely many open intervals, are fixed and non-random. 
In this setting the existence of a unique invariant random probability measure is proven as well as a formula for the Hausdorff dimension of the surviving set. 
In our current setting, we do not allow the existence of critical points, however our maps may have non-full branches, and our partitions of monotonicity as well as our holes are allowed to vary randomly from fiber to fiber.

Sequential systems with holes have been considered in \cite{GeigerOtt21}, where a cocycle $T_\om$ of open maps is generated by a single $\omega$ orbit.
The maps (which include the hole) must be chosen in a small neighborhood of a fixed map (with hole), in contrast to our setting where our cocycle may include very different maps. 
Moreover in \cite{GeigerOtt21}, Lebesgue is used as a reference measure and the specific potential $-\log|\det DT|$ is used.
The theory is developed for uniformly expanding maps in higher dimensions and the main goal is to establish the ``conditional memory loss'', a concept analogous to exponential decay of correlations for closed dynamics.


Other related work includes \cite{AFGTV-IVC}, which considered non-transitive random interval maps with holes. In \cite{AFGTV-IVC} we proved a complete thermodynamic formalism for random interval maps (i) containing sufficiently many full branches and (ii) random potentials satisfying a strong contracting potential assumption. In the current work we treat maps that contain no full branches and potentials that satisfy a significantly weaker contracting potential assumption at the cost of assuming the closed system satisfies a mild covering condition. This allows us to obtain results for a large class of random Lasota-Yorke maps (Section \ref{sec: p1 ly map example}). Furthermore, the current work and that of \cite{AFGTV-IVC} are complementary, and neither work generalizes the other. 
For example, the current work does not need a full-branch condition in order to obtain a thermodynamic formalism for systems as in Chapter~\ref{part 1}, while \cite{AFGTV-IVC} provides such a theory (Theorems 6.1, 6.4 and 6.6) for non-transitive maps (Example 7.5) which are not covered in this work.

In Chapter \ref{part 1} of the present paper, we establish a full, quenched thermodynamic formalism for piecewise monotonic random dynamics with general potentials and general driving---the random driving $\sigma$ can be any invertible ergodic process on $\Omega$.
We begin with the random closed dynamics dealt with in \cite{AFGTV20}:  piecewise monotonic interval maps satisfying a random covering condition;  we have no Markovian assumptions, our maps may have discontinuities and may lack full branches.
The number of branches of our maps need not be uniformly bounded above in $\omega$ and our potentials $\vp_\omega$ need not be uniformly bounded below or above in $\omega$. 
To this setting we introduce random holes $H_\omega$ and formulate sufficient conditions that guarantee a random conformal measure $\nu_\omega$ and corresponding equivariant measure $\mu_\omega$ supported on the random survivor set $X_{\om,\infty}$, and a random ACCIM $\eta_\omega$ supported on $H_\om^c$.
These augment the notion of a random contracting potential \cite{AFGTV20} with accumulation rates of contiguous ``bad'' intervals (with zero conformal measure), and extend similar constructions of \cite{LMD} to the random situation.

To establish the existence of the family of measures $(\nu_\omega)_{\om\in\Om}$, we follow the limiting functional approach of \cite{LSV, LMD} by defining a random functional $\Lm_\om$ which is a limit of ratios of transfer operators and then showing that $\Lm_\om$ may be identified with the open conformal measure $\nu_\om$. This technique improves on the approach of \cite{AFGTV20}, which uses the Schauder-Tichonov Fixed Point Theorem to prove the existence of $\nu_{\om,0}$, by eliminating the extra steps necessary to show that the family $(\nu_\omega)_{\omega\in\Omega}$ is measurable with respect to $m$. Several steps are needed to achieve the construction of the conformal random measures. We start in  Section~\ref{sec:cones} by giving background material on Birkhoff cone techniques and the construction of our random cones. In Section~\ref{sec: LY ineq} we develop several random Lasota-Yorke type inequalities in terms of the variation and the random functional $\Lm_\om$. Section~\ref{sec:good} sees the construction of a large measure set of ``good'' fibers $\Om_G\sub\Om$ for which we obtain cone invariance at a uniform time step, and in Section~\ref{sec:bad} we show that the remaining ``bad'' fibers occur infrequently and behave sufficiently well. In Section~\ref{sec: props of Lm} we collect further properties of the random functional $\Lm$, which are then used in Section~\ref{sec: fin diam} to construct a large measure set of fibers $\Om_F\sub\Om$ for which we obtain cone contraction with a finite diameter image in a random time step. Using Hilbert metric contraction arguments, Section~\ref{sec: conf and inv meas} collects together the fruits of Sections~\ref{sec:good}-\ref{sec: fin diam} to prove our main technical lemma (Lemma~\ref{lem: exp conv in C+ cone}), which is then used to (i) obtain the existence of a random density $\phi$, (ii) prove the existence of a unique non-atomic random conformal measure $\nu$, and (iii) a random $T$-invariant measure $\mu$ which is absolutely continuous with respect to $\nu$. All these facts are collected in  our first result (detailed statements can be found in 
Corollary \ref{cor: exist of unique dens q},
Proposition \ref{prop: lower bound for density},
Lemma \ref{lem: Lm is a linear functional},
Lemma \ref{lem: Lm is conf meas},
Proposition \ref{prop: mu_om T invar},
Proposition \ref{prop: uniqueness of nu and mu},
Lemma \ref{lem: raccim eta}, and Theorem \ref{thm:eqstates}).
\begin{thmx}\label{main thm: existence}
Given a random open system $(\mathlist{\bcomma}{\Om, m, \sg, \cJ_0, T, \cB, \cL_0, \nu_0, \phi_0, H})$ (see Section \ref{sec: open systems}) satisfying \eqref{T1}-\eqref{T3}, \eqref{LIP}, \eqref{GP}, \eqref{A1}-\eqref{A2}, 
and \eqref{cond Q1}-\eqref{cond Q3} (see Sections \ref{sec: prelim} and \ref{sec: tr op and Lm}), the following hold. 
\begin{enumerate}
\item There exists a unique random probability measure $\nu\in\cP_\Om(\Om\times I)$ supported in $\cX_\infty$ such that 
\begin{align*}
\nu_{\sg(\om)}(\cL_\om f)=\lm_\om\nu_\om(f),
\end{align*}
for each $f\in\BV(I)$, where 
\begin{align*}
\lm_\om:=\nu_{\sg(\om)}(\cL_\om\ind_\om).
\end{align*} 	
Furthermore, we have that $\log\lm_\om\in L^1(m)$.
\item There exists a function $\phi\in\BV_\Om(I)$ such that $\nu(\phi)=1$ and for $m$-a.e. $\om\in\Om$ we have 
\begin{align*}
\cL_\om \phi_\om=\lm_\om \phi_{\sg(\om)}.
\end{align*}
Moreover, $\phi$ is unique modulo $\nu$. 
\item The measure $\mu:=\phi\nu$ is a $T$-invariant and ergodic random probability measure supported in $\cX_{\infty}$ and the unique relative equilibrium state for the potential $\vp$ satisfying the following variational principle: 
\begin{align*}
\cEP(\vp)
= h_\mu(T)+\int_{\Om \times I} \vp \,d\mu
=
\sup_{\eta\in\cP_{T,m}^H(\Om\times I)} \lt(h_\eta(T)+\int_{\Om \times I} \vp \,d\eta\rt).
\end{align*}
Furthermore, for each $\eta\in\cP_{T,m}^H(\Om)$ (the set of all random $T$-invariant Borel probability measures supported on $\Om\times I$) different from $\mu$ we have that 
\begin{align*}
h_\eta(T)+\int_{\Om \times I} \vp \,d\eta
<
h_\mu(T)+\int_{\Om \times I} \vp \,d\mu.
\end{align*}

\item There exists a random conditionally invariant probability measure $\eta$ absolutely continuous with respect to $\nu_0$, which is supported on $\cup_{\om\in\Om}(\{\om\}\times I\bs H_\om)$, and whose disintegrations are given by
\begin{align*}
\eta_\om(f):=\frac{\nu_{\om,0}\lt(\ind_{H_\om^c} \phi_{\om} f\rt)}{\nu_{\om,0}\lt(\ind_{H_\om^c} \phi_\om\rt)}
\end{align*} 
for all $f\in\BV(I)$. 
\end{enumerate}
\end{thmx}
We also show that the operator cocycle is quasi-compact.
\begin{thmx}\label{main thm: quasicompact}
With the same hypotheses as Theorem~\ref{main thm: existence}, there exists $\kp\in(0,1)$ such that 
for each $f\in\BV(I)$ there exists a measurable function $\Om\ni\om\mapsto D_f(\om)\in(0,\infty)$ 
such that for $m$-a.e. $\om\in\Om$ and all $n\in\NN$ we have
\begin{align*}
\norm{\lt(\lm_\om^n\rt)^{-1}\cL_{\om}^n f - \nu_{\om}(f)\phi_{\sg^n(\om)}}_\infty\leq D_f(\om)\norm{f}_\BV\kp^n
\end{align*}
and 
\begin{align*}
\absval{
	\frac{\eta_\om\lt(f \rvert_{X_{\om,n}}\rt)}
	{\eta_\om\lt(X_{\om,n}\rt)}
	- 
	\mu_\om(f)
}\leq D_f(\om)\norm{f}_\BV\kp^n.
\end{align*}
Furthermore, for all sets $A$ that are a finite union of intervals, there exists a measurable function $\Om\ni\om\mapsto D_A(\om)\in(0,\infty)$ such that we have 
\begin{align*}
\absval{
	\nu_{\om,0}\lt(T_\om^{-n}(A)\,\rvert\, X_{\om,n}\rt)
	-
	\eta_{\sg^n(\om)}(A)	
}
\leq 
D_A(\om)\kp^n.
\end{align*}
\end{thmx}

For the proof of Theorem~\ref{main thm: quasicompact}, as well as a more general statement, see Theorem~\ref{thm: exp convergence of tr op} and Corollary~\ref{cor: exp conv of eta}.
From quasi-compactness we easily deduce the exponential decay of correlations for the invariant measure $\mu$. 
\begin{thmx}\label{main thm: exp dec of corr}
With the same hypotheses as Theorem~\ref{main thm: existence}, 
for every $h\in \BV(I)$ there exists a measurable function $\Om\ni\om\mapsto C_h(\om)\in(0,\infty)$ such that for every $f\in L^1(\mu)$, every $n\in\NN$, and for $m$-a.e. $\om\in\Om$ we have 
\begin{align*}
\absval{
	\mu_{\om}
	\lt(\lt(f_{\sg^{n}(\om)}\circ T_{\om}^n\rt)h \rt)
	-
	\mu_{\sg^{n}(\om)}(f_{\sg^{n}(\om)})\mu_{\om}(h)
}
\leq C_h(\om)
\norm{f_{\sg^n(\om)}}_{L^1(\mu_{\sg^n(\om)})}\norm{h}_\infty\vkp^n
\end{align*} 
for any $\vkp\in(\kp,1)$, with $\kp$ as in Theorem~\ref{main thm: quasicompact}.
\end{thmx}

Theorem~\ref{main thm: exp dec of corr} (stated in more detail in Theorem \ref{thm: dec of cor}) is proven in Section~\ref{sec: dec of cor}. The presence of holes leads naturally to introduce the notion of fiberwise escape rate $R(\rho_{\om})$ of the measure $\rho_{\om}$ from the holes; the definition in the random setting in given in \ref{def: escape rate}. We will show that the escape rate is constant $m$-almost everywhere and is given in terms of the closed and open expected pressures, denoted with $\cEP(\vp)$ for a given potential $\vp,$ which are properly  defined in Definition \ref{def: expected pressure}. 

\begin{thmx}\label{main thm: escape rate}
With the same hypotheses as Theorem~\ref{main thm: existence}, for $m$-a.e. $\om\in\Om$ we have that 
\begin{align*}
R(\nu_{\om,0})=R(\eta_\om)=\cEP(\vp_0)-\cEP(\vp)=\int_\Om\log\frac{\lm_{\om,0}}{\lm_\om}\, dm(\om).
\end{align*}
\end{thmx}
Theorem~\ref{main thm: escape rate} is proven in Section~\ref{sec: exp press}.
The expected pressure function is further developed and used to prove a Bowen's formula type result for the Hausdorff dimension of the survivor set $X_{\om,\infty}$ for $m$-a.e. $\om\in\Om$ in Section~\ref{sec: bowen}. This requires us to introduce the bounded distortion property for a given potential, large images of the map $T$, and large images with respect to the hole $H$; see Section~\ref{sec: bowen} for the full definitions of these terms. We therefore have:
\begin{thmx}\label{main thm: Bowens formula}
With the same hypotheses as Theorem~\ref{main thm: existence}, we additionally suppose that 
$$\int_\Om \log\inf|T_\om'|\ dm(\om)>0$$ 
and  $g_0=1/|T'|$ has bounded distortion.
Then there exists a unique $h\in[0,1]$ such that $\cEP(t)>0$ for all $0\leq t<h$ and $\cEP(t)<0$ for all $h<t\leq 1$.

Furthermore, if $T$ has large images and large images with respect to $H$, then for $m$-a.e. $\om\in\Om$
\begin{align*}
\HD(X_{\om,\infty})=h,		
\end{align*}
where $\HD(A)$ denotes the Hausdorff dimension of the set $A$.
\end{thmx}
The proof of Theorem~\ref{main thm: Bowens formula} appears in Section~\ref{sec: bowen}.

In Section \ref{sec: examples} we apply our general theory to a large class of random $\bt$-transformations with random holes as well as a general random Lasota-Yorke maps with random holes.
In fact, our theory applies to all of the finitely-branched examples discussed in \cite{AFGTV20} (this includes maps which are non-uniformly expanding or have contracting branches which appear infrequently enough that we still maintain on-average expansion) when suitable conditions are put on the holes $H_\om$. This includes the case where $H_\om$ is composed of finitely many intervals and the number of connected components of $H_\om$ is $\log$-integrable with respect to $m$. 

\, 

\begin{center}
\textbf{Summary of Chapter 2. Perturbation formulae for quenched random dynamics with applications to open systems and extreme value theory}
\end{center}

\,

The spectral approach to studying deterministic closed dynamical systems $T:X\to X$ on a phase space $X$, centers on the analysis of a transfer operator $\mathcal{L}:\mathcal{B}(X)\to\mathcal{B}(X)$, given by $\mathcal{L}f(y)=\sum_{x\in T^{-1}y}e^{\vp(x)}f(x)$, for $f$ in a suitable Banach space $\mathcal{B}(X)$ and for suitable potential function $\vp:X\to\mathbb{R}$.
If the map $T$ is \emph{covering} and the potential function is \emph{contracting} in the sense of \cite{LSV} (or similarly if $\sup\vp <P(T)$ as in \cite{denker_existence_1991,denker1990}
or if $\sup\vp-\inf\vp <h_{top}(T)$ as in \cite{hofbauer_equilibrium_1982,bruin_equilibrium_2008}, or if $\vp$ is \emph{hyperbolic} as in \cite{InoquioRivera}),
then one obtains the existence of an equilibrium state $\mu$, with associated conformal measure $\nu$, with the topological pressure $P(T)$ and the density of equilibrium state $d\mu/d\nu$ given by the logarithm of the leading (positive) eigenvalue $\lambda$ of $\mathcal{L}$ and the corresponding positive eigenfunction $h$, respectively. The map $T$ exhibits an exponential decay of correlations with respect to $\nu$ and $\mu$. 

Keller and Liverani \cite{KL99} showed that the leading eigenvalue $\lambda$ and eigenfunction $h$ of $\mathcal{L}$ vary continuously with respect to certain small perturbations of $\mathcal{L}$.
One example of such a perturbation is the introduction of a small hole $H\subset X$.
The set of initial conditions of trajectories that never land in $H$ is the survivor set $X_\infty$.
For small holes, specialising to Lasota-Yorke maps of the interval, Liverani and Maume-Dechamps \cite{LMD} apply the perturbation theory of \cite{KL99} to obtain the existence of a unique conformal measure $\nu$ and an absolutely continuous \emph{conditionally} invariant measure $\mu$, with density $h\in\BV([0,1])$. 
The leading eigenvalue $\lambda$ is interpretable as an \emph{escape rate}, and the open system displays an exponential decay of correlations with respect to $\nu$ and $\mu$.

To obtain finer information on the behaviour of $\lambda$ with respect to perturbation size, particularly in the situation where the perturbation is not smooth, such as perturbations arising from the introduction of a hole, one requires some additional control on the perturbation.
Keller and Liverani \cite{keller_rare_2009} develop abstract conditions on $\mathcal{L}$ and its perturbations $\mathcal{L}_\epsilon$ to ensure good first-order behaviour with respect to the perturbation size. 
Following \cite{keller_rare_2009}, several authors \cite{LMD,ferguson_escape_2012,FFT015,pollicott_open_2017,BDT18} have used the Keller-Liverani \cite{KL99} perturbation theory to obtain similar first-order behaviour of the escape rate with respect to the perturbation size for open systems in various settings.

This ``linear response'' of $\lambda$ is exploited in Keller \cite{keller_rare_2012} to develop an elegant spectral approach to deriving an exponential extreme value law to describe likelihoods of observing extreme values from evaluating an observation function $h:X\to\mathbb{R}$ along orbits of $T$.
In particular, the $N\to\infty$ limiting law of
\begin{equation}\label{K12}
\nu\left(\left\{x\in X: h(T^j(x))\le z_{N}, j=0, \dots, N-1\right\}\right),
\end{equation}
where the thresholds $z_N$ are chosen so that $\lim_{N\to\infty} N\mu(\{x:h(x)>z_N\})\to t$ for some $t>0$, is shown to be exponential.
The spectral approach of \cite{keller_rare_2012} also provides a relatively explicit expression for the limit of (\ref{K12}), namely
$$
\lim_{N\to\infty}\nu\left(\left\{x\in X: h(T^j(x))\le z_{N}, j=0, \dots, N-1\right\}\right)=\exp(-t\theta_0),
$$
where $\theta_0$ is the extremal index.

In Chapter \ref{part 2} we begin with sequential composition of linear operators $\mathcal{L}_{\omega,0}^n:=\mathcal{L}_{\sigma^{n-1}\omega,0}\circ\cdots\circ \mathcal{L}_{\sigma\omega,0}\circ \mathcal{L}_{\omega,0}$, 
where $\sigma:\Omega\to\Omega$ is an invertible map on a configuration set $\Omega$.
The driving $\sigma$ could also be an ergodic map on a probability space $(\Omega,\mathcal{F},m)$.
We then consider a family of perturbed cocycles $\mathcal{L}_{\omega,\ep}^n:=\mathcal{L}_{\sigma^{n-1}\omega,\ep}\circ\cdots\circ \mathcal{L}_{\sigma\omega,\ep}\circ \mathcal{L}_{\omega,\ep}$, where the size of the perturbation $\mathcal{L}_{\omega,0}-\mathcal{L}_{\omega,\ep}$ is quantified by the value $\Delta_{\omega,\epsilon}$ (Definition \ref{def: DL_om}).
Our first main result is an abstract quenched formula (Theorem \ref{thm: GRPT}) for the Lyapunov multipliers $\lambda_{\omega,0}$ up to first order in the size of the perturbation $\Delta_{\omega,\epsilon}$ of the operators $\mathcal{L}_{\omega,0}$.
This quenched random formula generalizes the main abstract first-order formula in \cite{keller_rare_2009} stated in the case of a single deterministic operator $\mathcal{L}_0$.
\begin{thmx}\label{urp main thm A}
Suppose that assumptions \eqref{P1}-\eqref{P9} hold (see Section \ref{Sec: Gen Perturb Setup}).
If $\Delta_{\omega,\epsilon}>0$ for all $\epsilon>0$ then for $m$-a.e. $\omega\in\Omega$:
\begin{align*}
\lim_{\ep\to 0}\frac{\lm_{\om,0}-\lm_{\om,\ep}}{\Dl_{\om,\ep}}=1-\sum_{k=0}^{\infty}\hat{q}_{\om,0}^{(k)}
=:\theta_{\omega,0}.
\end{align*}

\end{thmx}


The existence of a random quenched equilibrium state, conformal measure, escape rates, and exponential decay of correlations is established in Chapter \ref{part 1} for relatively large holes, generalizing the large-hole constructions of \cite{LMD} for a single deterministic map $T$ to the random setting with general driving.
In contrast, the focus of Chapter \ref{part 2} is to establish a random quenched analogue of the results of \cite{LMD}, \cite{keller_rare_2009}, and \cite{keller_rare_2012} discussed above.
To this end, 
we let $\mathcal{L}_{\omega,\epsilon}$ be the transfer operator for the open map $T_\omega$ with a hole $H_{\omega,\epsilon}$ introduced in $X$, namely $\mathcal{L}_{\omega,\epsilon}(f)=\mathcal{L}_\omega(\ind_{X\setminus H_{\omega,\epsilon}}f)$.
Our second collection of main results is a quenched formula for the derivative of $\lambda_{\omega,\epsilon}$ with respect to the sample invariant probability measure of the hole $\mu_{\omega,0}(H_{\omega,\epsilon})$ (Theorem \ref{thm: dynamics perturb thm}),
as well as a quenched formula for the derivative of the fiberwise escape rate $R_\ep(\mu_{\om,0})$ with respect to the sample invariant probability measure of the hole $\mu_{\omega,0}(H_{\omega,\epsilon})$ (Corollary \ref{esc rat cor}). 

\begin{corx}\label{urp main cor A}
If $(\mathlist{\bcomma}{\Om, m, \sg, \cJ_0, T, \cB, \cL_0, \nu_0, \phi_0, H_\ep})$ is a  random open  system (see Section \ref{sec: open systems}) with
$\mu_{\om,0}(H_\om)>0$ for all $\ep>0$ and \eqref{C1}--\eqref{C8} hold (see Section \ref{sec:goodrandom}), then for $m$-a.e.\ $\omega\in\Omega$:
\begin{align}\label{cor A eq}
\lim_{\ep\to 0}
\frac{1-\lm_{\om,\ep}/\lm_{\om,0}}{\mu_{\om,0}(H_{\om,\ep})}
=
1-\sum_{k=0}^{\infty}\hat{q}_{\om,0}^{(k)}=:\ta_{\om,0}
\end{align}
and 
\begin{align*}
R_\ep(\mu_{\om,0})=\int_\Om \log \lm_{\om,0}-\log\lm_{\om,\ep}\, dm(\om).
\end{align*}
In addition, if \eqref{DCT cond 1*} holds and the $\mu_{\om,0}$-measures of the random holes scale with $\epsilon$ according to \eqref{EVT style cond}, then for $m$-a.e.\ $\omega\in\Omega$ 
\begin{equation*}
\label{escaperatederiv}
\lim_{\ep\to 0}
\frac{R_\ep(\mu_{\om,0})}{\mu_{\om,0}(H_{\om,\ep})}
=
\int_\Om \ta_{\om,0}\, dm(\om).
\end{equation*}
\end{corx}

To generalize to the random setting the rescaled distribution of the maxima given by (\ref{K12}), we now consider the real-valued random observables $h_{\omega}$ defined on the phase space $X$
and construct a process $h_{\sigma^j\omega}\circ T^j_{\omega}$.
We are interested in determining the limiting law of
\begin{equation}\label{HHU}
\mu_{\omega,0}\left(\left\{x\in X: h_{\sigma^j\omega}(T^j_{\omega}(x))\le z_{\sigma^j\omega, N}, j=0, \dots, N-1\right\}\right),
\end{equation}
where $\{z_{\sigma^j\omega, N}\}_{0\le j\le N-1}$ is a collection of real-valued  thresholds.
The sample probability measure $\mu_{\omega,0}$ enjoys the equivariance property $T^*_{\omega}\mu_{\omega,0}=\mu_{\sigma\omega,0},$ however the process $h_{\sigma^j\omega}\circ T^j_{\omega}$ is not stationary in the probabilistic sense, which makes the theory slightly more difficult.

The first approach to non-stationary extreme value theory (EVT) was given 
under convenient assumptions, by H\"usler in \cite{H83,H86}. He was able to recover the usual extremal behaviour seen for
i.i.d. or stationary sequences under Leadbetter’s conditions \cite{LED}, namely (i)
guaranteed mixing properties for the probability
measure governing the process 
and (ii) 
that the exceedances should appear scattered through the time period under consideration.
H\"usler's results can not be applied in the dynamical systems setting because his uniform bounds on the control of the exceedances are not  satisfied for deterministic, random, or sequential compositions of maps. 
The first contribution dealing explicitly with extreme value theory for random and sequential systems is the paper \cite{FFV017};  see also \cite{FFMV016} for an  application to point processes. 
These works were an adaptation of  Leadbetter's conditions and Hüsler’s approach: let us call them the probabilistic approach to extreme value theory, to distinguish it from the spectral and perturbative approach used in the current paper.

As in the deterministic case, in order to avoid a degenerate limit distribution, one should conveniently choose the thresholds $z_{\sigma^j\omega,N}.$  H\"usler proved convergence to the Gumbel's law if for some $0<t<\infty$ we have convergence of the sum 	\begin{equation}\label{BLH}
\sum_{j=0}^{N-1}\mu_{\omega,0}\left(h_{\sigma^j\omega}(T^j_{\omega}(x))> z_{\sigma^j\omega, N}\right)\rightarrow t
\end{equation} for $m$-a.e.\ $\omega.$
In our current framework we will additionally allow the positive number $t$ to be any positive random variable in $L^\infty(m)$.
The nonstationary theory developed in \cite{FFV017} for quenched random processes, has the further restrictions that the observation function is fixed ($\omega$-independent), and the thresholds $z_N$ (like the scaling $t$) are just real numbers, and requires the obvious restricted equivalent of (\ref{BLH}).
In our framework the observation function $h_\omega$, the scaling $t_\omega$, and the thresholds $z_{\omega,N}$ may all be random (but need not be).
We generalize and simplify the requirement (\ref{BLH}) to
\begin{equation}\label{os}
N\mu_{\omega,0}\left(h_{\omega}(x)> z_{\omega, N}\right)=t_{\omega}+\xi_{\omega, N},
\end{equation}
where the scaling $t$ may be a random variable $t\in L^\infty(m)$ and the ``errors'' $\xi_{\omega,N}$ satisfy (i) $\lim_{N\rightarrow \infty}\xi_{\omega,N}=0$ a.e., and (ii) $|\xi_{\omega,N}|\le W<\infty$ for a.e.\ $\omega$ and all sufficiently large $N$. 
We provide a more detailed discussion of the relationship between the conditions (\ref{BLH}) and (\ref{os}) at the end of Section \ref{EEVV}.

In summary, we derive a spectral approach for a quenched random extreme value law, where the dynamics $T_\omega$ is random, the observation functions $h_\omega$ can be random, the thresholds controlling what is an extreme value can be random, and the scalings of the likelihoods of observing extreme values can be random, all controlled by general invertible ergodic driving.
Moreover, we obtain a formula for the explicit form of Gumbel law for the extreme value distribution.
This leads to our main extreme value theory result (stated in detail later as Theorem \ref{evtthm}):
\begin{thmx}\label{urp main thm B}
For a random open system  $(\mathlist{\bcomma}{\Om, m, \sg, \cJ_0, T, \cB, \cL_0, \nu_0, \phi_0, H_\ep})$ (see Section \ref{sec: open systems}),  assuming \eqref{C1'}, \eqref{C2}, \eqref{C3}, \eqref{C4'}, \eqref{C5'}, \eqref{C7'}, \eqref{C8},  and \eqref{xibound} (see Sections~\ref{sec:goodrandom} and \ref{EEVV}),  for $m$-a.e.\ $\omega\in\Omega$ one has
\begin{align*}
&\lim_{N\to\infty} \nu_{\omega,0}\left(x\in X: h_{\sigma^j\omega}(T^j_{\omega}(x))\le z_{\sigma^j\omega, N}, j=0, \dots, N-1\right)\\
&\qquad\qquad =\lefteqn{\lim_{N\to\infty} \mu_{\omega,0}\left(x\in X: h_{\sigma^j\omega}(T^j_{\omega}(x))\le z_{\sigma^j\omega, N}, j=0, \dots, N-1\right)}\\
&\qquad\qquad =\exp\left(-\int_\Omega t_\omega\theta_{\omega,0}\ dm(\omega)\right),
\end{align*}
where $\nu_{\omega,0}$ and $\mu_{\omega,0}$ are the random conformal measure and the random invariant measure, respectively, for our random dynamics, $t_\omega$ is a random scaling function, and $\theta_{\omega,0}$ is an $\omega$-local extremal index corresponding to the quantity given in Corollary \ref{urp main cor A}.
\end{thmx}
This result generalizes the spectral approach to extreme value theory in \cite{keller_rare_2012}, for a single map $T$, single observation function $h$, single scaling, and single sequence of thresholds.

Given a family of random holes $\mathcal{H}_{\om,N}:=\{H_{\sigma^j\om, \epsilon_N}\}_{j\ge 0},$ one can define
the first (random) hitting time to a hole, starting at initial condition $x$ and random configuration $\omega$:
$$
\tau_{\om, \mathcal{H}_{\om,N}}(x):=\inf\{k\ge 1, T^k_{\om}(x)\in H_{\sigma^k\om, \epsilon_N}\}.$$
When this family of holes shrink with increasing $N$ according to Condition (\ref{xibound}) (see Section \ref{sec:hts}),
Theorem \ref{urp main thm B} provides a description of the statistics of random hitting times, scaled by the measure of the holes (see Theorem \ref{hve}).
\begin{corx}\label{urp main cor B}
For a random open system  $(\mathlist{\bcomma}{\Om, m, \sg, \cJ_0, T, \cB, \cL_0, \nu_0, \phi_0, H_\ep})$ (see Section \ref{sec: open systems}),  assume \eqref{C1'}, \eqref{C2}, \eqref{C3}, \eqref{C4'}, \eqref{C5'}, \eqref{C7'}, \eqref{C8},  and \eqref{xibound} (see Sections~\ref{sec:goodrandom} and \ref{EEVV}).
For $m$-a.e.\ $\omega\in\Omega$ one has
\begin{equation}
\label{eqhit0}
\lim_{N\to\infty}\mu_{\om,0}\left(\tau_{\om, \mathcal{H}_{\om,N}} \mu_{\om,0}(H_{\om, \epsilon_N})>t_{\om}\right) = \exp\left(-{\int_\Om t_{\om} \theta_{\om,0}\ dm(\omega)}\right).
\end{equation}
\end{corx}

By assuming some additional uniformity in $\omega$ on the maps $T_\omega$ we use a recent random perturbative result \cite{C19} to obtain a complete quenched thermodynamic formalism. 
The following existence result extends Theorem C of \cite{LMD}, which concerned inserting a single small hole into the phase space of a single deterministic map $T$, to the situation of random map cocycles with small random holes with the random process controlled by general invertible ergodic driving $\sigma$.

\begin{thmx}\label{urp main thm C}
Suppose 
that (\ref{E1})--(\ref{E9})
(see Section~\ref{sec: existence})
hold for the random open system $(\mathlist{\bcomma}{\Om, m, \sg, [0,1], T, \BV([0,1]), \cL_0, \nu_0, \phi_0, H_\ep})$ (see Section \ref{sec: open systems}).
Then for each $\ep>0$ sufficiently small there exists a unique random $T$-invariant probability measure $\mu_\ep=\set{\mu_{\om,\ep}}_{\om\in\Om}$ with $\supp(\mu_{\om,\ep})\sub X_{\om,\infty,\ep}$. 
Furthermore, $\mu_\ep$ is the unique relative equilibrium state for the random open system 
and satisfies a forward and backward exponential decay of correlations. In addition, there exists a random absolutely continuous (with respect to $\set{\nu_{\om,0}}_{\om\in\Om}$) conditionally invariant probability measure $\varrho_\ep=\set{\varrho_{\om,\ep}}_{\om\in\Om}$ with $\supp(\varrho_{\om,\ep})\sub[0,1]\bs H_{\om,\ep}$ and density function $\psi_{\om,\ep}\in \BV([0,1])$. 

\end{thmx}
For a more explicit statement of Theorem \ref{urp main thm C} as well as the relevant assumptions and definitions see Section~\ref{sec: existence} and Theorem~\ref{EXISTENCE THEOREM}.\\

In Section \ref{sec: limit theorems} we prove some quenched limit theorems for closed random dynamics. 
These limit results are new for more general potentials and their associated equilibrium states. We use two approaches. 
The first is based on the perturbative technique developed in \cite{dragicevic_spectral_2018}, which  generalizes the Nagaev-Guivarc'h method to random cocycles. 
This technique establishes a relation between a suitable twisted operator cocycle and  the distribution of the random Birkhoff sums. 
As a consequence, it is possible to get quenched versions of the large deviation principle, the central limit theorem, and the local central limit theorem. 
The second approach invokes the martingale  techniques previously used in the quenched random setting in \cite{DFGTV18A}. 
We obtain the almost sure invariance principle (ASIP) for the equivariant measure $\mu_{\om,0},$ which also implies the central limit theorem and the law of iterated logarithms, a general bound for large deviations and a dynamical Borel-Cantelli lemma. 
In addition, using the Sprindzuk theorem we are able to obtain a quenched shrinking target result.

We conclude in Section \ref{sec:examples} with several explicit examples of Theorems \ref{urp main thm A}--\ref{urp main thm C}.
We start in Example \ref{example1} with the weight $1/T'_\omega$ for a family of random maps, random scalings, and random observations $h_\omega$ with a common extremum location in phase space, which is a common fixed point of the $T_\omega$.
The special cases of a fixed map $T$ on the one hand, and a fixed scaling $t$ on the other, are also considered.
The same calculations can be extended to observation functions with common extrema on a periodic orbit common to all $T_\omega$.
Next in Example \ref{example2} we consider the more difficult case where orbits are distributed according to equilibrium states of a general geometric weight $|DT_\omega|^{-r}$,
using random $\beta$-maps and random observation functions $h_\omega$ with a common extremum at $x=0$.
Example \ref{example3} investigates a fixed map $T$ with random observation functions $h_\omega$ with extrema in a shrinking neighbourhood of a fixed point of $T$, where the neighbourhood lacks the symmetry of Example \ref{example1}.
In the last example, Example \ref{example4}, we again consider random maps $T_\om$  
with random observations $h_\omega$, but now the maxima of the observations are not related to fixed or periodic points of $T_\om$.

Though we apply our results to the setting of random interval maps our results apply equally well to other random settings including random subshifts 
\cite{Bogenschutz_RuelleTransferOperator_1995a,mayer_countable_2015}, random distance expanding maps \cite{mayer_distance_2011}, random polynomial systems \cite{Bruck_Generalizediteration_2003}, random transcendental maps \cite{mayer_random_2018}. In addition, Theorem \ref{urp main thm A} (as well as \eqref{cor A eq} of Corollary \ref{urp main cor A})
applies to sequential and semi-group settings including \cite{atnip_dynamics_2020,stadlbauer_quenched_2020}. See Remark~\ref{rem seq exist} for a sequential version of Theorem \ref{urp main thm C}.

\section{Preliminaries on random open systems }\label{sec:IntroPrelims}

In this section we introduce the general setup of random open systems.
We begin with a probability space $(\Om,\sF,m)$\index{$\Om$}\index{$\sF$} and an ergodic, invertible map $\sg:\Om\to\Om$ which preserves the measure $m$, i.e.
\begin{align*}
m\circ\sg^{-1}=m.
\end{align*}\index{$m$}\index{$\sg:\Om\to\Om$}
We will refer to the tuple $(\Om,\sF,m,\sg)$ as the \textit{base dynamical system}\index{base dynamical system}.
For each $\om\in \Om$, let $\cJ_{\om,0}$\index{$\cJ_{\om,0}$} be a closed subset of a complete metrisable space $X$ such that the map
\begin{align*}
\Om\ni \om\longmapsto\cJ_{\om,0}
\end{align*}
is a closed random set\index{closed random set}, i.e. $\cJ_{\om,0}\sub X$ is closed for each $\om\in\Om$ and the map $\om\mapsto\cJ_{\om,0}$ is measurable  
(see \cite{crauel_random_2002}), and we consider the maps
\begin{align*}
T_\om:\cJ_{\om,0}\to\cJ_{\sg\om,0}.
\end{align*}
Random iteration is given by $T_\om^n:\cJ_{\om,0}\to\cJ_{\sg^n\om,0}$\index{$T_\om^n:\cJ_{\om,0}\to\cJ_{\sg^n\om,0}$}, by which we mean the $n$-fold composition 
\begin{align*}
T_{\sg^{n-1}\om}\circ\dots\circ T_\om:\cJ_{\om,0}\to\cJ_{\sg^n\om,0}.
\end{align*}
Thus, our random iterates are compositions which are driven by the orbits of the base dynamical system $\sg:\Om\to\Om$.
Given a set $A\sub\cJ_{\sg^n\om,0}$ we let
\begin{align*}
T_\om^{-n}(A):=\set{x\in\cJ_{\om,0}:T_\om^n(x)\in A}
\end{align*}
denote the inverse image of $A$ under the map $T_\om^n$ for each $\om\in\Om$ and $n\geq 1$.
Now let
\begin{align*}
\cJ_0:=\bigcup_{\om\in\Om}\set{\om}\times\cJ_{\om,0}\sub \Om\times X,
\end{align*}\index{$\cJ_0$}
and define the induced skew-product map $T:\cJ_0\to\cJ_0$\index{$T:\cJ_0\to\cJ_0$} by
\begin{align*}
T(\om,x)=(\sg\om,T_\om(x)).
\end{align*}
Let $\sB$\index{$\sB$} denote the Borel $\sg$-algebra of $X$ and let $\sF\otimes\sB$\index{$\sF\otimes\sB$} be the product $\sg$-algebra on $\Om\times X$. Throughout the text we denote Lebesgue measure by $\Leb$.\index{$\Leb$} We suppose the following:

\,

\begin{enumerate}[align=left,leftmargin=*,labelsep=\parindent]
\item[(\Gls*{M1})]\myglabel{M1}{M1} The map $T:\cJ_0\to\cJ_0$ is measurable with respect to $\sF\otimes\sB$.
\end{enumerate} 

\,

\begin{remark}\label{rem: iid iteration}
    Throughout this manuscript we consider general ergodic driving $\sg:\Om\to\Om$. To consider the simpler setting of iid driving let $\Sg=\{1,\dots,k\}$ be a finite set together with an associated probability vector $p=(p_1,\dots, p_k)$. Then take $\Om=\Sg^\ZZ$ to be the space of all bi-infinite sequences $\om=\dots\om_{-2}\om_{-1}.\om_0\om_1\om_2\dots$ with symbols from $\Sg$, let $\sg:\Om\to\Om$ be the left-shift map, and let $m$ be the Bernoulli measure on $\Om$ generated by the vector $p$. Now for each of the $k$ symbols of $\Sg$, we associate a map $T_j:X\to X$ ($1\leq j\leq k$), and for each $\oio$ we denote $T_\om:=T_{\om_0}\rvert_{\cJ_{\om,0}}:\cJ_{\om,0}\to X$.   
\end{remark}

\begin{definition}\label{def: random prob measures}	
A measure $\mu$ on $\Om\times X$ with respect to the product $\sg$-algebra $\sF\otimes\sB$ is said to be \textit{random measure}\index{random measure} relative to $m$ if it has marginal $m$, i.e. if
$$
\mu\circ\pi^{-1}_1=m.
$$ 
The disintegrations $\set{\mu_\om}_{\om\in\Om}$ of $\mu$ with respect to the partition $\lt(\{\om\}\times X\rt)_{\om\in\Om}$ satisfy the following properties:
\begin{enumerate}
\item For every $B\in\sB$, the map $\Om\ni\om\longmapsto\mu_\om(B)\in [0,\infty]$ is measurable, 
\item For $m$-a.e. $\om\in\Om$, the map $\sB\ni B\longmapsto\mu_\om(B)\in [0,\infty]$ is a Borel measure.
\end{enumerate}
We say that the random measure $\mu=\set{\mu_\om}_{\om\in\Om}$ is a \textit{random probability measure}\index{random measure!random probability measure} if for $m$-a.e. $\om\in\Om$ the fiber measure $\mu_\om$ is a probability measure. Given a set $Y=\cup_{\om\in\Om}\set{\om}\times Y_\om\sub\Om\times X$, we say that the random measure $\mu=\set{\mu_\om}_{\om\in\Om}$ is supported in $Y$ if $\supp(\mu)\sub Y$ and consequently $\supp(\mu_\om)\sub Y_\om$ for $m$-a.e. $\om\in\Om$. We let $\cP_\Om(Y)$\index{$\cP_\Om(Y)$} denote the set of all random probability measures supported in $Y$. We will frequently denote a random measure $\mu$ by $\set{\mu_\om}_{\om\in\Om}$.
\end{definition}
The following proposition from Crauel \cite{crauel_random_2002}, shows that a random probability measure $\set{\mu_\om}_{\om\in\Om}$ on $\cJ_0$ uniquely identifies a probability measure on $\cJ_0$.
\begin{proposition}[\cite{crauel_random_2002}, Propositions 3.3]\label{prop: random measure equiv}
If $\set{\mu_\om}_{\om\in\Om}\in\cP_\Om(\cJ_0)$ is a random probability measure on $\cJ_0$, then for every bounded measurable function $f:\cJ_0\to\RR$, the function 
$$
\Om\ni\om\longmapsto \int_{\cJ_{\om,0}} f(\om,x) \, d\mu_\om(x)
$$ 
is measurable and 
$$
\sF\otimes\sB\ni A\longmapsto\int_\Om \int_{\cJ_{\om,0}} \ind_A(\om,x) \, d\mu_\om(x)\, dm(\om)
$$
defines a probability measure on $\cJ_0$.
\end{proposition}

For functions $f:\cJ_0\to\RR$ and $F:\Om\to\RR$ we let 
\begin{align*}
S_{n,T}(f_\om):=\sum_{j=0}^{n-1}f_{\sg^j(\om)}\circ T_\om^j
\quad\text{ and }\quad
S_{n,\sg}(F):=\sum_{j=0}^{n-1}F\circ\sg^j
\end{align*}\index{$S_{n,T}$}\index{$S_{n,\sg}$}
denote the Birkhoff sums of $f$ and $F$ with respect to $T$ and $\sg$ respectively. 
We will consider a potential\index{potential} of the form $\vp_0:\cJ_0\to\RR$\index{$\vp_0:\cJ_0\to\RR$}, and for each $n\geq 1$ we consider the weight $g_0^{(n)}:\cJ_0\to\RR$\index{weight} whose disintegrations are given by
\begin{align}\label{def: formula for g_om^n}
g_{\om,0}^{(n)}:=\exp(S_{n,T}(\vp_{\om,0})) =\prod_{j=0}^{n-1}g_{\sg^j\om,0}^{(1)}\circ T_\om^j
\end{align}\index{$g_{\om,0}$}
for each $\om\in\Om$. We will often denote $g_{\om,0}^{(1)}$ by $g_{\om,0}$. We assume there exists a family of Banach spaces $\set{\cB_\om,\norm{\spot}_{\cB_\om}}_{\om\in\Om}$\index{$\cB_\om$} of real-valued functions on each $\cJ_{\om,0}$ with $g_{\om,0}\in\cB_\om$ such that the fiberwise (Perron-Frobenius) transfer operator $\cL_{\om,0}:\cB_\om\to\cB_{\sg\om}$ given by
\begin{align}\label{glob def: closed tr op}
\cL_{\om,0}(f)(x):=\sum_{y\in T_\om^{-1}(x)}f(y)g_{\om,0}(y), \quad f\in\cB_\om, \, x\in\cJ_{\sg\om,0}
\end{align}\index{$\cL_{\om,0}$}
is well defined.
Using induction we see that iterates $\cL_{\om,0}^n:\cB_{\om}\to\cB_{\sg^n\om}$ of the transfer operator are given by
\begin{align*}
\cL_{\om,0}^n(f)(x):=\sum_{y\in T_\om^{-n}(x)}f(y)g_{\om,0}^{(n)}(y), \quad f\in\cB_\om, \, x\in\cJ_{\sg^n\om,0}.
\end{align*}
We let $\cB$ denote the space of functions $f:\cJ_0\to\RR$ such that $f_\om\in\cB_\om$ for each $\om\in\Om$ and we define the global transfer operator $\cL_0:
\cB\to \cB$ by \index{$\cB$}
$$
(\cL_0 f)_\om(x):=\cL_{\sg^{-1}\om,0}f_{\sg^{-1}\om}(x)
$$ 
for $f\in\cB$ and $x\in\cJ_{\om,0}$.
We assume the following measurability assumption:

\,

\begin{enumerate}[align=left,leftmargin=*,labelsep=\parindent]

\item[(\Gls*{M2})]\myglabel{M2}{M2} For every measurable function $f \in \cB$, the map 
$(\om, x) \mapsto (\cL_0 f)_\om(x)$ is measurable.

\end{enumerate}

\,

We suppose the following condition on the existence of a closed conformal measure. 

\,

\begin{enumerate}[align=left,leftmargin=*,labelsep=\parindent]
\item[(\Gls*{CCM})]\myglabel{CCM}{CCM} There exists a random probability  measure $\nu_0=\set{\nu_{\om,0}}_{\om\in\Om}\in \cP_\Om(\cJ_0)$ and measurable functions $\lm_0:\Om \to\RR\bs\set{0}$ and $\phi_0:\cJ_0\to (0,\infty)$ with $\phi_0\in\cB$ such that 
\begin{align*}
\cL_{\om,0}(\phi_{\om,0})=\lm_{\om,0}\phi_{\sg\om,0}
\quad\text{ and }\quad
\nu_{\sg\om,0}(\cL_{\om,0}(f))=\lm_{\om,0}\nu_{\om,0}(f)
\end{align*}
for all $f\in\cB_\om$ where $\phi_{\om,0}(\spot):=\phi_0(\om,\spot)$. Furthermore, we suppose that the fiber measures $\nu_{\om,0}$ are non-atomic and that $\lm_{\om,0}:=\nu_{\sg\om,0}(\cL_{\om,0}\ind)$ with $\log\lm_{\om,0}\in L^1(m)$. 
We then define the random probability measure 
$\mu_0$ on $\cJ_0$ by
\begin{align}\label{eq: def of mu_om,0}
\mu_{\om,0}(f):=\int_{\cJ_{\om,0}} f\phi_{\om,0} \ d\nu_{\om,0},  \qquad f\in L^1(\nu_{\om,0}).
\end{align}
\end{enumerate}

\,

From the definition, one can easily show that $\mu_0$ is $T$-invariant, that is,  
\begin{align}\label{eq: mu_om,0 T invar}
\int_{\cJ_{\om,0}} f\circ T_\om \ d\mu_{\om,0}
=
\int_{\cJ_{\sg\om,0}} f \ d\mu_{\sg\om,0}, \qquad f\in L^1(\mu_{\sg\om,0}).
\end{align}
\begin{remark}
Our Assumption \eqref{CCM} has been shown to hold in for several large classes of random dynamical systems, for example random interval maps \cite{AFGTV20,AFGTV-IVC}, random subshifts \cite{Bogenschutz_RuelleTransferOperator_1995a,mayer_countable_2015}, 
random distance expanding maps \cite{mayer_distance_2011}, 
random polynomial systems \cite{Bruck_Generalizediteration_2003}, 
and random transcendental maps \cite{mayer_random_2018}. 
\end{remark}

\begin{definition}\label{def CRS}
We will call the collection $(\mathlist{\bcomma}{\Om, m, \sg, \cJ_0, T, \cB, \cL_0, \nu_0, \phi_0})$ a \textit{closed random dynamical system}\index{closed random dynamical system} if the assumptions \eqref{M1}, \eqref{M2}, and \eqref{CCM}  
are satisfied.
\end{definition}
We are now ready to introduce holes into the closed systems.

\subsection{Random Open Systems}\label{sec: open systems}

We let $H\sub \cJ_0$ be measurable with respect to the product $\sg$-algebra $\sF\otimes\sB$ on $\cJ_0$ such that 
\begin{align*}
    0<\nu_0(H)<1.
\end{align*}
For each $\om\in\Om$ the sets $H_\om\sub \cJ_{\om,0}$ are uniquely determined by the condition that 
\begin{align}\label{hole defn}
\set{\om}\times H_\om=H\cap\lt(\set{\om}\times \cJ_{\om,0}\rt).
\end{align}
Equivalently we have 
\begin{align*}
H_\om=\pi_2(H\cap(\set{\om}\times \cJ_{\om,0})),
\end{align*}
where $\pi_2:\cJ_0\to \cJ_{\om,0}$ is the projection onto the second coordinate. By definition we have that the sets $H_\om$ are $\nu_{\om,0}$-measurable.
Now define
\begin{align*}
\cJ_\om:=\cJ_{\om,0}\bs H_\om,
\end{align*}\index{$\cJ_\om$}
and let
\begin{align*}
\cJ:=\bigcup_{\om\in\Om}\set{\om}\times\cJ_\om.
\end{align*}\index{$\cJ$}
Throughout the manuscript, in particular Chapter \ref{part 1}, we denote $\ind_\om:=\ind_{\cJ_\om}$.\index{$\ind_\om$}
For each $\om\in\Om$, $n\geq 0$ we define
\begin{align}\label{def: X_n surv}
X_{\om,n}:&=\set{x\in\cJ_{\om,0}: T_\om^j(x)\notin H_{\sg^j\om} \text{ for all } 0\leq j\leq n }
=\bigcap_{j=0}^n T_\om^{-j}\left(\cJ_{\sg^j\om}\right)
\end{align}\index{$X_{\om,n}$}
to be the set of points in $\cJ_\om$ which survive, i.e. those points whose trajectories do not land in the holes, for $n$ iterates. We then naturally define
\begin{align}\label{def: X_infty surv}
X_{\om,\infty}:=\bigcap_{n=0}^\infty X_{\om,n}=\bigcap_{n=0}^\infty T_\om^{-n}(\cJ_{\sg^n\om})
\end{align}\index{$X_{\om,\infty}$}
to be the set of points which will never land in a hole under iteration of the maps $T_\om^n$ for any $n\geq 0$. We call $X_{\om,\infty}$ the \textit{$\om$-surviving set}.\index{surviving set} 
Note that the sets $X_{\om,n}$ and $X_{\om,\infty}$ are forward invariant satisfying the properties
\begin{align}
T_\om(X_{\om,n})\sub X_{\sg\om,n-1}
\qquad\text{ and }\qquad
T_\om(X_{\om,\infty})\sub X_{\sg\om,\infty}.
\label{surv set forw inv}
\end{align}
Now for any $0\leq \al\leq \infty$ we set
\begin{align*}
\hat{X}_{\om,\al}:=\ind_{X_{\om,\al}}
= 
\prod_{j=0}^{\al}\ind_{\cJ_{\sg^j\om}}\circ T_\om^j.
\end{align*}\index{$\hat{X}_{\om,n}$}\index{$\hat{X}_{\om,\infty}$}
The global surviving set is defined as
\begin{align*}
\cX_{\al}:=\bigcup_{\om\in\Om}\set{\om}\times X_{\om,\al}
\end{align*}\index{$\cX_{n}$}\index{$\cX_{\infty}$}
for each $0\leq \al\leq \infty$.
Then $\cX_{\infty}\sub\cJ$ is precisely the set of points that survive under forward iteration of the skew-product map $T$. We will assume that the fiberwise survivor sets are nonempty:

\,

\begin{enumerate}[align=left,leftmargin=*,labelsep=\parindent]
\item[(\Gls*{X})]\myglabel{X}{cond X}
For $m$-a.e. $\om\in\Om$ we have $X_{\om,\infty}\neq\emptyset$.
\end{enumerate}

\,

As an immediate consequence of \eqref{cond X} we have that $\cX_{\infty}\neq\emptyset$. In fact, \eqref{cond X} together with the forward invariance of the sets $X_{\om,\infty}$ imply that $\cX_\infty$ is infinite. 
The following proposition presents a setting in random piecewise continuous open systems for which \eqref{cond X} holds and the survivor set is nonempty for random piecewise continuous open dynamics. 
\begin{proposition}\label{prop surv nonemp}
Suppose that for $m$-a.e. $\om\in\Om$ 
there exist $V_\om,U_{\om,1}, \dots, U_{\om,k_\om}\sub\cJ_\om$ nonempty compact subsets such that for $m$-a.e. $\om\in\Om$
\begin{enumerate}
\item $T_\om\rvert_{U_\om,j}$ is continuous for each $1\leq j\leq k_\om$,
\item $T_\om(U_\om)\bus V_{\sg\om}$, where $U_\om:=\cup_{j=1}^{k_\om}U_{\om,j}\sub V_\om$ for each $\om$.
\end{enumerate}
Then $X_{\om,\infty}\neq\emptyset$ for $m$-a.e. $\om\in\Om$ and consequently $\cX_{\infty}\neq\emptyset$.
Furthermore, if 
$$
m(\set{\om\in\Om: k_\om>1})>0
$$ 
then for $m$-a.e. $\om\in\Om$ the survivor set $X_{\om,\infty}$ is uncountable. 
\end{proposition}
\begin{proof}
For each $1\leq j\leq k_\om$ let $T_{\om,j}:U_{\om,j}\to\cJ_{\sg\om,0}$ denote the continuous map $T_\om\rvert_{U_{\om,j}}$, and let $T_{\om,U}:U_\om\to\cJ_{\sg\om,0}$ denote the map $T_\om\rvert_{U_\om}$. Since $V_{\sg\om}$ is compact,  for each $1\leq j\leq k_\om$ we have that $T_{\om,j}^{-1}(V_{\sg\om})$ is a nonempty compact subset of $U_{\om,j}$. 
Given $n\geq 1$ let $\gm=\gm_0\gm_1\dots\gm_{n-1}$ be an $n$-length word with $1\leq \gm_j\leq k_{\sg^j\om}$ for each $0\leq j\leq n-1$. Let $\Gm_{\om,n}$ denote the finite collection of all such words of length $n$. Then for each $n\geq 1$ and each $\gm\in\Gm_{\om,n}$
$$
T_{\om,\gm}^{-n}(V_{\sg^n\om}):=
T_{\om,\gm_0}^{-1}\circ\dots\circ T_{\sg^{n-1}\om,\gm_{n-1}}^{-1}(V_{\sg^n\om})
\sub U_{\om,\gm_0}
$$
is compact. 
Furthermore, $T_{\om,\gm}^{-n}(V_{\sg^n\om})$ forms a decreasing sequence in $U_{\om,\gm_0}$.  Hence, we see that 
$$
T_{\om,U}^{-n}(V_{\sg^n\om})=\bigcup_{\gm\in\Gm_{\om,n}}T_{\om,\gm}^{-n}(V_{\sg^n\om})
$$
forms a decreasing sequence of compact subsets of $U_\om$. 
Thus,
$$
X_{\om,\infty}=\bigcap_{n=0}^\infty T_\om^{-n}(V_{\sg^n\om})
\bus \bigcap_{n=0}^\infty T_{\om,U}^{-n}(V_{\sg^n\om})\neq\emptyset
$$ 
as desired. The final claim follows from the ergodicity of $\sg$, which ensures that for a.e. $\om$ there are infinitely many $j\in\NN$ such that $k_{\sg^j\om}>1$, and the usual bijection between a point in $X_{\om,\infty}$ and an infinite word $\gm$ in the fiberwise sequence space $\Sg_\om:=\set{\gm=\gm_1\gm_2\dots: 1\leq \gm_j\leq k_{\sg^j\om}}$.
\end{proof}
Now we define the perturbed operator $\cL_\om:\cB_\om\to\cB_{\sg\om}$ by
\begin{align}\label{glob def: open tr op}
\cL_\om(f):=\cL_{\om,0}\left(f\cdot\ind_{\cJ_\om}\right)
=\sum_{y\in T_\om^{-1}(x)}f(y)\ind_{\cJ_\om}g_{\om,0}(y)
=\sum_{y\in T_\om^{-1}(x)}f(y)g_\om(y)
, \quad f\in\cB_\om,
\end{align}\index{$\cL_\om$}
where for each $\om\in\Om$ we define $g_\om:=g_{\om,0}\ind_{\cJ_\om}$, and, similarly, for each $n\in\NN$, 
$$
g_\om^{(n)}:=\prod_{j=0}^{n-1}g_{\sg^j\om}\circ T_\om^j.
$$
Note that measurability of $H\sub \cJ_0$ and condition \eqref{M2} imply that for every $f\in\cB$ the map $(\om,x)\mapsto (\cL f)_\om (x)$ is also measurable.
Iterates of the perturbed operator $\cL_\om^n:\cB_{\om}\to\cB_{\sg^n\om}$ are given by
\begin{align*}
\cL_\om^n:=\cL_{\sg^{n-1}\om}\circ\dots\circ\cL_\om,
\end{align*}
which, using induction, we may write as
\begin{align*}
\cL_\om^n(f)=\cL_{\om,0}^n\left(f\cdot\hat{X}_{\om,n-1}\right), \qquad f\in\cB_\om.
\end{align*}
We define the sets $D_{\om,n}$\index{$D_{\om,n}$} to be the support of $\cL_{\sg^{-n}(\om)}^n\ind_{\sg^{-n}(\om)}$, that is, we set
\begin{align}\label{defn of D sets}
D_{\om,n}:=\set{x\in \cJ_{\sg^{-n}\om,0}: \cL_{\sg^{-n}(\om)}^n\ind_{\sg^{-n}(\om)}(x)\neq 0}.
\end{align}
Note that, by definition, we have 
\begin{align*}
D_{\om,n+1}\sub D_{\om,n} 
\end{align*}
for each $n\in\NN$, and we similarly define
\begin{align*}
D_{\om,\infty}:=\bigcap_{n=0}^\infty D_{\om,n}.
\end{align*}\index{$D_{\om,\infty}$}
From this moment on we will assume that for $m$-a.e. $\om\in\Om$ we have that 
\begin{align}
&D_{\om,\infty}\neq\emptyset.
\label{cond D}\tag{D}
\end{align}
\begin{remark}
We note that if $T_\om(\cJ_\om)=\cJ_{\sg\om,0}$ for $m$-a.e. $\om$, then  
$$
D_{\om,\infty}=\cJ_{\om,0}
$$
for $m$-a.e. $\om\in\Om$, and in particular we have that \eqref{cond D} holds.
\end{remark}
We let
\begin{align}\label{eq: hat D notation}
\hat{D}_{\om,\al}:=\ind_{D_{\om,\al}}
\end{align}\index{$\hat{D}_{\om,n}$}\index{$\hat{D}_{\om,\infty}$}
for each $0\leq\al\leq \infty$. Since $D_{\om,n}$ is the support of $\cL_{\sg^{-n}(\om)}^n\ind_{\sg^{-n}(\om)}$, using the notation of \eqref{eq: hat D notation} we may write
\begin{align*}
\cL_\om^n(f)=\hat{D}_{\sg^n(\om),n}\cL_{\om}^n(f).
\end{align*}
More generally, we have that, for $k>j$, $D_{\sg^{k}(\om),j}$ is the support of $\cL_{\sg^{k-j}(\om)}^{j}\ind$, i.e.
\begin{align}\label{support of pert tr op}
\cL_{\sg^{k-j}(\om)}^{j}(f)=\hat{D}_{\sg^k(\om),j}\cL_{\sg^{k-j}(\om)}^{j}(f)
\end{align}
for $f\in L^1(\nu_{\sg^k(\om),0})$.
Note that 
\begin{align}\label{eq: D sets in terms of X sets}
D_{\om,n}=T_{\sg^{-n}(\om)}^n(X_{\sg^{-n}(\om),n-1}).
\end{align}
Finally, we note that since $g_\om^{(n)}:= g_{\om,0}^{(n)}\rvert_{X_{\om,n-1}}$, for each $n\in\NN$ we have that
\begin{align}\label{eq: ineq leads to log integ}
\inf g_{\om,0}^{(n)}
\leq 
\inf_{X_{\om,n-1}} g_\om^{(n)}
\leq
\norm{g_\om^{(n)}}_\infty
\leq 
\norm{g_{\om,0}^{(n)}}_\infty
\end{align}
and 
\begin{align}\label{eq: ineq leads to log integ2}
\inf g_{\om,0}^{(n)}
\leq 
\inf_{D_{\sg^n(\om),\infty}} \cL_\om^n\ind_\om
\leq 
\norm{\cL_\om^n\ind_\om}_\infty
\leq
\norm{\cL_{\om,0}\ind}_\infty.
\end{align}

\begin{definition}\label{def ROS prelim}
We will call a closed random dynamical system $(\mathlist{\bcomma}{\Om, m, \sg, \cJ_0, T, \cB, \cL_0, \nu_0, \phi_0})$ (meaning that \eqref{M1}, \eqref{M2}, and \eqref{CCM} are satisfied) a \textit{random open system} if assumptions \eqref{cond D}  and \eqref{cond X} are also satisfied. We let $(\mathlist{\bcomma}{\Om, m, \sg, \cJ_0, T, \cB, \cL_0, \nu_0, \phi_0, H})$ denote the random open system generated by the random maps $T_\om:\cJ_{\om,0}\to\cJ_{\sg\om,0}$ and random holes $H_\om\sub\cJ_{\om,0}$.
\end{definition}

\section{Simple Examples}
We finish this introductory chapter by presenting a simple class of examples that satisfy our main assumptions. Consequently we have that Theorems \ref{main thm: existence} -- \ref{main thm: escape rate}, \ref{urp main thm A}, \ref{urp main thm B}, \ref{urp main thm C} and Corollaries \ref{urp main cor A}, \ref{urp main cor B} hold. In particular, the examples in this section are presented purely for pedagogical purposes and do not represent the full generality of our theory. As such, we will omit technical details and proofs. 
 
In this section we will build on the abstract framework of Sections \ref{sec:IntroPrelims} and \ref{sec: open systems} and consider the i.i.d. iteration of $\bt$-transformations. 
We begin by following the i.i.d. setup of Remark \ref{rem: iid iteration}, that is we let $\Sg=\{1,\dots,k\}$ be a finite set together with a probability vector $p=(p_1,\dots,p_k)$, and then we take $\Om=\Sg^\ZZ$, $\sg:\Om\to\Om$ to be the shift map, and $m$ is the Bernoulli measure generated by $p$. 
For each $\oio$ we take the fibers $\cJ_{\om,0}=[0,1]$ and we consider the measurable function $\Om\ni\om\mapsto\bt_\om\in [a,b]$ where $2<a<b<\infty$.  As in Remark \ref{rem: iid iteration}, we assume that the maps $\bt_\om=\bt_{\om_0}$ are determined by the zeroth coordinate of the bi-infinite sequence $\dots\om_{-1}.\om_0\om_1\dots$ from a finite set of $\bt_j\in[a,b]$ for $1\leq j\leq k$.
For each $\oio$, we take the Banach spaces $\cB_\om=\BV$, the space of bounded variation functions, and let $T_\om:[0,1]\to[0,1]$ be given by $T_\om(x)=\bt_\om x\mod 1$. We consider the (closed) Perron-Frobenius operator $\cL_{\om,0}:\BV\to\BV$ defined by 
\begin{align*}
    \cL_{\om,0} f(x)=\sum_{y\in T_\om^{-1}(x)}\frac{f(y)}{|T_\om'(y)|}.
\end{align*}
In this case we have that the potential $\vp_\om=-\log|T_\om'|$. Furthermore, we have that for almost every $\oio$,  the closed conformal measure is Lebesgue, i.e. $\nu_{\om,0}\equiv \Leb$, and it follows from \cite{buzzi_exponential_1999, AFGTV20} that the assumptions \eqref{M1}, \eqref{M2}, and \eqref{CCM} are satisfied. Let $\cZ_\om$ be the partition of $[0,1]$ into intervals of monotonicity of the map $T_\om$.
\\
Now depending on how one chooses the random holes $H_\om$, either \textit{large} holes or holes which are sufficiently \emph{small} to apply a perturbative approach, we are able to apply a different collection of our results. We will begin with the results that apply for \emph{large} holes described in Chapter \ref{part 1}.
\\
\subsection{Large holes}
For each $\oio$, we take $H_\om$ to be one of the intervals of monotonicity $Z\in\cZ_\om$, after which we can define the open transfer operator $\cL_\om:\BV\to\BV$ and the surviving sets $X_{\om,n}$ and $X_{\om,\infty}$ as in Section \ref{sec: open systems}. Since $T_\om$ contains full branches outside of the hole $H_\om$, we have that the open transfer operator $\cL_\om$ is fully supported, meaning that $D_{\om,\infty}=[0,1]$, which further implies that our assumption \eqref{cond D} holds. Furthermore, applying Proposition \ref{prop surv nonemp} with the sets $U_{\om,j}$ to be these full branches outside of the hole, we see that the surviving sets $X_{\om,\infty}$ are nonempty, and thus that our assumption \eqref{cond X} is satisfied. Finally, assume that 
$$
    \sum_{j=1}^kp_j (\log\lfloor\bt_j\rfloor-1)
    > \log 5. 
$$
Without presenting the details, it can easily be checked that our assumptions \eqref{T1}-\eqref{T3}, \eqref{LIP}, \eqref{GP}, \eqref{A1}-\eqref{A2}, and \eqref{cond Q1}-\eqref{cond Q3} (which are presented in Chapter \ref{part 1}) are satisfied for the i.i.d. random $\beta$-transformations described here. Thus Theorems \ref{main thm: existence} -- \ref{main thm: escape rate} hold. Under some additional minor assumptions (bounded distortion and large images) that will be explained in Section \ref{sec: bowen}, Theorem \ref{main thm: Bowens formula} can be shown to hold in this setting as well.
For details and additional examples in greater generality see Section \ref{sec: examples}. 
 
Now we consider examples for which the holes are chosen sufficiently small that we may take a perturbative approach as described in Chapter \ref{part 2}.
\subsection{Small holes} 
For each $\oio$ fix some $x_\om\in[0,1]$ such that $x_\om$ is not a point of discontinuity for the map $T_\om$. Let $\ep>0$ and for each $\oio$ take $H_{\om,\ep}=B(x_\om,\ep)$ be the ball of radius $\ep$ centered at the point $x_\om$. 
For each $\ep>0$ define the open operators $\cL_{\om,\ep}$ and surviving sets $X_{\om,\infty,\ep}$ as in \ref{sec: open systems}.
As the holes are contained in single branches of monotonicity, following the same reasoning as in the setting of \emph{large} holes, we again see that our assumptions \eqref{cond D} and \eqref{cond X} hold. It can be shown that the assumptions \eqref{E1}--\eqref{E9} of Section \ref{sec: existence} hold in this setting, and thus we have that Theorem \ref{urp main thm C} holds. In particular, we have that for each $\ep>0$ there exists a random conformal measure $\nu_\ep=\{\nu_{\om,\ep}\}_{\oio}$, a random invariant density $\phi_\ep$, and a random invariant measure $\mu_\ep=\{\mu_{\om,\ep}\}_{\oio}$ such that 
\begin{align*}
    \cL_{\om,\ep}(\phi_{\om,\ep})=\lm_{\om,\ep}\phi_{\sg\om,\ep}
    \quad\text{ and }\quad
    \nu_{\sg\om,\ep}(\cL_{\om,\ep}(f))=\lm_{\om,\ep}\nu_{\om,\ep}(f)
\end{align*}
for all $f\in\BV$ and that there exists $C>0$ and $\kp\in(0,1)$ such that for all $f\in\BV$ 
\begin{align*}
\sup_{\ep>0}\|\lm_{\om,\ep}^{-n}\cL_{\om,\ep}^n(f)-\nu_{\om,\ep}(f)\cdot\phi_{\sg^n\om,\ep}\|_\infty\leq C\|f\|_\BV\kp^n. 
\end{align*}
It immediately follows that our assumptions \eqref{C1}--\eqref{C7} of Section \ref{sec:goodrandom} hold as well as the relevant strengthenings \eqref{C1'},  \eqref{C4'}, \eqref{C5'}, and \eqref{C7'} of Section \ref{EEVV}.  See Section \ref{sec: existence} for further details.
 
In order to apply Corollary \ref{urp main cor A} (which is a dynamical version and corollary of the general perturbation result Theorem \ref{urp main thm A}), Theorem \ref{urp main thm B}, and Corollary \ref{urp main cor B}, we must check the assumptions \eqref{C8} and \eqref{xibound} of Sections \ref{sec:goodrandom} and \ref{EEVV} respectively. First, to satisfy \eqref{xibound} we further assume that $x_\om=0$ is non-random and we take a sequence $\ep_N\to 0$ such that 
\begin{align*}
    \mu_{\om,0}(H_{\om,\ep_N})=\frac{t}{N}
\end{align*}
for some $t>0$, which can be done since $\phi_0$ (the invariant density of the random ACIM $\mu_0$) is uniformly (in $\om$) bounded above and away from zero.  
To see that \eqref{C8} holds, we consider the quantity 
\begin{align*}
    q_{\om,N}^{(k)}=
    \frac{\mu_{\sg^{-(k+1)}\om,0}\left(
	T_{\sg^{-(k+1)}\om}^{-(k+1)}(H_{\om,\ep_N})
	\cap\left(\bigcap_{j=1}^k T_{\sg^{-(k+1)}\om}^{-(k+1)+j} (H_{\sg^{-j}\om,\ep_N}^c)\right)
	\cap H_{\sg^{-(k+1)}\om,\ep_N}
	\right)}
    {\mu_{\sg^{-(k+1)}\om,0}\left(
	T_{\sg^{-(k+1)}\om}^{-(k+1)}(H_{\om,\ep_N})\right)},
\end{align*}
which is a simplification of the quantity $\hat q_{\om,\ep}^{(k)}$ defined in \eqref{def of hat q}. A simple calculation shows that  $\lim_{N\to\infty}q_{\om,N}^{(k)}=0$ for all $k>0$. Thus we have only to calculate $\lim_{N\to\infty}q_{\om,N}^{(0)}$ to verify \eqref{C8}: 
\begin{align*}
    \lim_{N\to \infty}q_{\om,N}^{(0)}&=\lim_{N\to\infty}
    \frac{\mu_{\sg^{-1}\om,0}\left(
	T_{\sg^{-1}\om}^{-1}(H_{\om,\ep_N})
	\cap H_{\sg^{-1}\om,\ep_N}
	\right)}
    {\mu_{\om,0}(H_{\om,\ep_N})}
    =\frac{1}{\bt_{\sg^{-1}\om}}.
\end{align*}
Thus, we have that Corollary \ref{urp main cor A}, Theorem \ref{urp main thm B}, and Corollary \ref{urp main cor B} hold with the random extremal index given by 
\begin{align*}
    \ta_\om =1-\frac{1}{\bt_{\sg^{-1}\om}}.
\end{align*}
For further details and more general examples see Section \ref{sec:examples}.

\chapter{Thermodynamic formalism for random interval maps with holes}\label{part 1}
\normalsize
In this first chapter we consider fixed random holes $H_{\om}$ and the main objective will be to  construct a random conformal measure $\nu_{\om}$ and corresponding equivariant  measures $\mu_{\om}$ supported on the random survivor set $X_{\om,\infty}.$ The measures $\mu_\om$ will be shown to be the unique relative equilibrium state for the potential $\vp$. We will also get a random  absolutely continuous conditionally invariant measure $\eta_{\om}$ supported on $H_{\om}^c.$ Successively we define the escape rate of the closed conformal measure and show that it equals the difference of the expected pressures of the closed and open random systems. Finally we establish a Bowen's like  formula for the Hausdorff dimension of the survivor set for a specific potential.

\section{Preliminaries of random interval maps with holes}\label{sec: prelim}
We begin with a base dynamical system $(\Om,\sF,m,\sg)$, i.e. the map $\sg:\Om\to\Om$ is invertible, ergodic, and preserves the measure $m$.
For the remainder of Chapter \ref{part 1} for each $\om\in\Om$ we take $\cJ_{\om,0} = I$ to be a compact interval in $\RR$, and we consider the map
$T_\om:I\to I$ such that there exists a finite partition $\cZ_\om$\index{$\cZ_\om$} of $I$ such that $Z$ is an interval for each $Z\in\cZ_\om$ and 
\begin{enumerate}
\item[(\Gls*{T1})]\myglabel{T1}{T1}
    $T_\om:I\to I \text{ is surjective},$
    
\item[(\Gls*{T2})]\myglabel{T2}{T2}
    $T_\om(Z) \text{ is an interval for each } Z\in\cZ_\om,$
\item[(\Gls*{T3})]\myglabel{T3}{T3}
    $T_\om\rvert_Z \text{ is continuous and strictly monotone for each } Z\in\cZ_\om.$
\end{enumerate}
In addition, we will assume that 
\begin{enumerate}
    \item[(\Gls*{LIP})]\myglabel{LIP}{LIP}
    $\log\#\cZ_\om\in L^1(m).$
\end{enumerate}
The maps $T_\om$ induce the skew product map $T:\Om\times I\to \Om\times I$ given by 
\begin{align*}
T(\om,x)=(\sg(\om),T_\om(x)).
\end{align*}
For each $n\in\NN$ we consider the fiber dynamics of the maps $T_\om^n:I\to I$ given by the compositions
\begin{align*}
T_\om^n(x)=T_{\sg^{n-1}(\om)}\circ\dots\circ T_\om(x).
\end{align*}
We let $\cZ_{\om}^{(n)}$, for $n\geq 2$, denote the monotonicity partition of $T_\om^n$ on $I$ which is given by 
\begin{align*}
\cZ_\om^{(n)}:=\bigvee_{j=0}^{n-1}T_\om^{-j}\lt(\cZ_{\sg^j(\om)}\rt).
\end{align*}\index{$\cZ_\om^{(n)}$}
Given $Z\in\cZ_\om^{(n)}$, we denote by
$$	
T_{\om,Z}^{-n}:T_\om^n(Z)\lra Z
$$ 
the inverse branch of $T_\om^n$ which takes $T_\om^n(x)$ to $x$ for each $x\in Z$.
We will assume that the partitions $\cZ_\om$ are generating, i.e. 
\begin{align}\myglabel{GP}{GP}
\bigvee_{n=1}^\infty \cZ_\om^{(n)}=\sB,
\tag{\Gls*{GP}}
\end{align}
where $\sB=\sB(I)$ denotes the Borel $\sg$-algebra of $I$. 
Let $\Bd(I)$\index{$\Bd(I)$} denote the set of all bounded real-valued functions on $I$ and for each $f\in\Bd(I)$ and each $A\sub I$ let
\begin{align*}
\var_A(f):=\sup\set{\sum_{j=0}^{k-1}\absval{f(x_{j+1})-f(x_j)}: x_0<x_1<\dots x_k, \, x_j\in A \text{ for all } k\in\NN},
\end{align*} \index{$\var_A(f)$}\index{$\var(f)$} 
denote the variation of $f$ over $A$. If $A=I$ we denote $\var(f):=\var_I(f)$. We let 
\begin{align*}
\BV(I):=\set{f\in\Bd(I): \var(f)<\infty}
\end{align*}\index{$\BV(I)$}
denote the set of functions of bounded variation on $I$. For each $\om\in\Om$ we will set the Banach space $\cB_\om=\BV(I)$. Let 
$$
\norm{f}_\infty:=\sup(|f|)
\qquad\text{ and }\qquad
\norm{f}_\BV:=\var(f)+\norm{f}_\infty
$$ \index{$\norm{f}_\infty$}\index{$\norm{f}_\BV$}
be norms on the respective Banach spaces $\Bd(I)$ and $\BV(I)$. Given a function $f:\Om\times I\to\RR$, by $f_\om: I\to I$ we mean 
\begin{align*}
f_\om(\spot):=f(\om,\spot).
\end{align*}
\begin{definition}\label{def: random bounded}
We say that a function $f:\Om\times I\to\RR$ is \textit{random bounded}\index{random bounded} if 
\begin{enumerate}[(i)]
\item $f_\om\in\Bd(I)$ for each $\om\in\Om$, 
\item for each $x\in I$ the function $\Om\ni\om\mapsto f_\om(x)$ is measurable, 
\item the function $\Om\ni\om\mapsto\norm{f_\om}_\infty$ is measurable. 
\end{enumerate}
Let $\Bd_\Om(I)$\index{$\Bd_\Om(I)$} denote the collection of all random bounded functions on $\Om\times I$.
\end{definition}
\begin{definition}\label{def: random BV}
We say that a function $f\in\Bd_\Om(I)$ is of \textit{random bounded variation}\index{random bounded variation} if $f_\om\in\BV(I)$ for each $\om\in\Om$ and the map $\om\mapsto\var(f_\om)$ is measurable.
We let $\BV_\Om(I)$\index{$\BV_\Om(I)$}, which will take the place of the space $\cB$ from Section \ref{sec:IntroPrelims}, denote the set of all random bounded variation functions.
\end{definition}

As in \eqref{glob def: closed tr op}, we define the (closed) transfer operator, $\cL_{\om,0}:\Bd(I)\to \Bd(I)$, with respect to the potential $\vp_0:\Om\times I\to\RR$ by 
\begin{align*}
\cL_{\om,0} (f)(x):=\sum_{y\in T_\om^{-1}(x)}g_{\om,0}(y)f(y); \quad 
f\in \Bd(I), \
x\in I.
\end{align*} 
For each $\om\in\Om$ we let $\Bd^*(I)$ and $\BV^*(I)$ denote the respective dual spaces of $\Bd(I)$ and $\BV(I)$. We let $\cL_{\om,0}^*:\Bd^*(I)\to\Bd^*(I)$ denote the dual transfer operator.
\begin{definition}\label{def: admissible potential}
We will say that a measurable potential $\vp_0:\Om\times I\to\RR$ is \textit{admissible}\index{admissible potential} if for $m$-a.e. $\om\in\Om$ we have 
\begin{enumerate}
    \item[(\Gls*{A1})]\myglabel{A1}{A1} 
        $\inf \vp_{\om,0}, \sup\vp_{\om,0}\in L^1(m),$
    \item[(\Gls*{A2})]\myglabel{A2}{A2} 
        $g_{\om,0}\in\BV(I).$
\end{enumerate}
\end{definition}
\begin{remark}\label{rem: A1 and A2 hold}
Note that if $\vp_{\om,0}\in\BV(I)$ for each $\om\in\Om$ then \eqref{A2} is immediate. Furthermore, as $\vp_{\om,0}\in\Bd(I)$ we also have that $\inf g_{\om,0}^{(n)}>0$ for $m$-a.e. $\om\in\Om$ and each $n\in\NN$.
\end{remark}
As an immediate consequence of \eqref{A1} we have that
\begin{flalign}
\log\inf g_{\om,0}, \log\norm{g_{\om,0}}_\infty\in L^1(m).
\label{eq: cons of A1} 	
\end{flalign} 
Note that since we can write 
\begin{align*}
g_{\om,0}^{(n)}:=\prod_{j=0}^{n-1}g_{\sg^j(\om),0}\circ T_\om^j.
\end{align*}
for each $n\in\NN$ we must have that $g_{\om,0}^{(n)}\in\BV(I)$. 
Clearly we have that the sequence $\norm{g_{\om,0}^{(n)}}_\infty$ is submultiplicative, i.e.
\begin{align*}
\norm{g_{\om,0}^{(n+m)}}_\infty\leq \norm{g_{\om,0}^{(n)}}_\infty\cdot\norm{g_{\sg^n(\om),0}^{(m)}}_\infty.
\end{align*}
Similarly we see that the sequence $\inf g_{\om,0}^{(n)}$ is supermultiplicative. Submultiplicativity and supermultiplicativity of $\norm{g_{\om,0}^{(n)}}_\infty$ and $\inf g_{\om,0}^{(n)}$ together with \eqref{eq: cons of A1} gives that 
\begin{align}\label{eq: log sup inf g integ}
\log\norm{g_{\om,0}^{(n)}}_\infty, \log\inf g_{\om,0}^{(n)}\in L^1(m)
\end{align}
for each $n\in\NN$.
Our assumptions \eqref{T1} and \eqref{LIP} combined with \eqref{eq: log sup inf g integ} implies that 
\begin{align}\label{eq: log sup inf L integ}
\log\norm{\cL_{\om,0}^n\ind}_\infty, \log\inf \cL_{\om,0}^n\ind\in L^1(m)
\end{align}
for each $n\in\NN$. Note that, in view of \eqref{eq: ineq leads to log integ}-\eqref{eq: ineq leads to log integ2}, 
\eqref{eq: log sup inf g integ} and \eqref{eq: log sup inf L integ}, imply that 
\begin{align}\label{eq: log int open weight and tr op}
\log\inf_{X_{\om,n-1}}g_\om^{(n)}, \,
\log\norm{g_\om^{(n)}}_\infty, \,
\log\inf_{D_{\sg^n(\om),n}}\cL_\om^n\ind_\om, \,
\log\norm{\cL_{\om}^n\ind_\om}_\infty\in L^1(m).
\end{align}
The Birkhoff Ergodic Theorem then implies that the quantities in \eqref{eq: ineq leads to log integ} and \eqref{eq: ineq leads to log integ2} are \emph{tempered}\index{tempered}, e.g. 
\begin{align*}
\lim_{|k|\to\infty}\frac{1}{|k|}\log\inf g_{\sg^k(\om)}^{(n)}=0
\end{align*}
for $m$-a.e. $\om\in\Om$ and each $n\in\NN$.

In addition to the assumptions \eqref{T1}-\eqref{T3}, \eqref{A1}, \eqref{A2}, \eqref{GP}, and \eqref{LIP} above, we note that in our current setting assumption \eqref{M1} implies that
\begin{enumerate}
\item[(\Gls*{M})]\myglabel{M}{cond M1} The map $T:\Om\times I\to\Om\times I$ is measurable. 
 
\end{enumerate}	
Furthermore, in our current setting our assumption \eqref{CCM} translates to the following:
\begin{enumerate}
\item[(\Gls*{C})]\myglabel{C}{cond C1}  
There exists a random probability measure $\nu_0=\set{\nu_{\om,0}}_{\om\in\Om}\in \cP_\Om(\Om\times I)$ and measurable functions $\lm_0:\Om \to\RR\bs\set{0}$ and $\phi_0:\cJ_0\to (0,\infty)$ with $\phi_0\in\BV_\Om(I)$ such that 
\begin{align*}
\cL_{\om,0}(\phi_{\om,0})=\lm_{\om,0}\phi_{\sg\om,0}
\quad\text{ and }\quad
\nu_{\sg\om,0}(\cL_{\om,0}(f))=\lm_{\om,0}\nu_{\om,0}(f)
\end{align*}
for all $f\in\BV(I)$. Furthermore, we suppose that the fiber measures $\nu_{\om,0}$ are non-atomic and that $\lm_{\om,0}:=\nu_{\sg\om,0}(\cL_{\om,0}\ind)$ with $\log\lm_{\om,0}\in L^1(m)$. The $T$-invariant random probability measure $\mu_0$ on $\Om\times I$ is given by
\begin{align*}
\mu_{\om,0}(f):=\int_{I} f\phi_{\om,0} \ d\nu_{\om,0},  \qquad f\in L^1(\nu_{\om,0}).
\end{align*}
\end{enumerate}
\begin{remark}
Note that \eqref{cond M1} together with \eqref{A2} implies that for $f\in\BV_\Om(I)$ we have $\cL_0 f\in\BV_\Om(I)$.
\end{remark}
\begin{remark}
Examples which satisfy the conditions \eqref{T1}-\eqref{T3}, \eqref{A1}, \eqref{A2}, \eqref{GP}, \eqref{LIP}, \eqref{cond M1}, and \eqref{cond C1} can be found in \cite{AFGTV20}.
\end{remark}

\subsection{Random Interval Maps with Holes}
We now wish to introduce holes into the class of finite branched random weighted covering systems. 

Let $H\sub \Om\times I$ be measurable with respect to the product $\sg$-algebra $\sF\otimes\sB$ on $\Om\times I$. 
By definition, i.e. \eqref{hole defn}, we have that the sets $H_\om$ are $\nu_{\om,0}$-measurable. 
Suppose that $0<\nu_0(H)<1$, that is we have 
\begin{align*}
0<\int_\Om\nu_{\om,0}(H_\om)\, dm(\om)<1.
\end{align*}
Now define
\begin{align*}
I_{\om}:=I\bs H_{\om}.
\end{align*}\index{$I_\om$}
and recall from Section~\ref{sec:IntroPrelims} that we denote 
\begin{align*}
\ind_\om:=\ind_{I_\om}.
\end{align*}\index{$\ind_\om$}
We then let
\begin{align*}
\cI:=H^c=\bigcup_{\om\in\Om}\set{\om}\times I_{\om}.
\end{align*}\index{$\cI$}
For each $\om\in\Om$ and $n\geq 0$ we define the $\om$-surviving sets $X_{\om,n}$ and $X_{\om,\infty}$ as in \eqref{def: X_n surv}-\eqref{def: X_infty surv}.
By definition, for each $n\in\NN$, we have that 
\begin{align*}
T_\om(X_{\om,n})\sub X_{\sg(\om),n-1}
\quad\text{ and }\quad
T_\om(X_{\om,\infty})\sub X_{\sg(\om),\infty}.
\end{align*}
Note, however, that these survivor sets are, in general, only forward invariant and not backward invariant. 
For notational convenience for any $0\leq \al\leq \infty$ we set
\begin{align*}
\hat{X}_{\om,\al}:=\ind_{X_{\om,\al}}.
\end{align*}

For each $\om\in\Om$ we let $\vp_{\om}=\vp_{\om,0}\rvert_{I_{\om}}$, and thus for each $n\in\NN$ this gives
\begin{align*}
g_{\om}^{(n)}=g_{\om,0}^{(n)}\rvert_{X_{\om,n-1}}=\exp\lt(S_{n,T}(\vp_\om)\rt)=\prod_{j=0}^{n-1}g_{\sg^j(\om)}\circ T_\om^j.
\end{align*}
Now define the open operator $\cL_{\om}:L^1(\nu_{\om,0})\to L^1(\nu_{\sg(\om),0})$, where $\nu_0$ comes from our assumption \eqref{cond C1} on the closed system, by
\begin{align}\label{eq: def of open tr op}
\cL_{\om}(f):=\cL_{\om,0}\left(f\cdot\ind_\om\right), \quad f\in L^1(\nu_{\om,0}), \, x\in I.
\end{align}
As a consequence of \eqref{eq: def of open tr op}, we have that 
\begin{align*}
\cL_\om\ind=\cL_\om\ind_\om.
\end{align*}
Iterates of the open operator $\cL_{\om}^n:L^1(\nu_{\om,0})\to L^1(\nu_{\sg^n(\om),0})$ are given by
\begin{align*}
\cL_{\om}^n:=\cL_{\sg^{n-1}(\om)}\circ\dots\circ\cL_{\om},
\end{align*}
which, using induction, we may write in terms of the closed operator $\cL_{\om,0}$ as
\begin{align*}
\cL_{\om}^n(f)=\cL_{\om,0}^n\left(f\cdot\hat{X}_{\om,n-1}\right), \qquad f\in L^1(\nu_{\om,0}).
\end{align*}

\section{Random conditionally invariant probability measures}\label{sec: RCIM}
In this section we introduce the notion of a random conditionally invariant measure and give suitable conditions for their existence. As in the previous section, we consider random interval maps with holes, however we would like to point out that the definitions, results, and proofs of this section hold in the greater generality of Section \ref{sec:IntroPrelims}.
We begin with the following definition. 
\begin{definition}
We say that a random probability measure $\eta\in \cP_\Om(\Om\times I)$ is a random conditionally invariant probability measure (RCIM)\index{random conditionally invariant probability measure (RCIM)} if
\begin{align}\label{RCIM def}
\eta_\om(T_\om^{-n}(A)\cap X_{\om,n})=\eta_{\sg^n(\om)}(A)\eta_\om(X_{\om,n})
\end{align}
for all $n\geq 0$, $\om\in\Om$, and all Borel sets $A\sub I$. If a RCIM $\eta$ is absolutely continuous with respect to a random probability measure $\zt$ we call $\eta$ a random absolutely continuous conditionally invariant probability measure (RACCIM)\index{random absolutely continuous conditionally invariant probability measure (RACCIM)} with respect to $\zt$. 
\end{definition}
Straight from the definition of a RCIM we make the following observations.
\begin{observation}\label{O1 RCIM}
Note that if we plug $A=I_{\om}=X_{\om,0}$ into \eqref{RCIM def} with $n=0$, we have that
$$
\eta_\om(I_{\om})=\eta_\om^2(I_{\om}),
$$
which immediately implies that $\eta_\om(I_\om)$ id either $0$ or $1$. If $\eta_\om(H_\om)=0$ then we have that $\eta_\om$ is supported in $I_{\om}$.
\end{observation}
\begin{observation}\label{O2 RCIM}
Note that since
\begin{align*}
T_\om^{-n}(X_{\sg^n(\om),m})\cap X_{\om,n}=X_{\om,n+m}
\end{align*}
for each $n,m\in\NN$ and $\om\in\Om$ we have that if $\eta$ is a RCIM then
\begin{align*}
\eta_{\om}(X_{\om,n+m})=\eta_{\om}(X_{\om,n})\eta_{\sg^n(\om)}(X_{\sg^n(\om),m})
\end{align*}
for each $n,m\in\NN$. In particular, we have that
\begin{align*}
\eta_\om(X_{\om,n})=\prod_{j=0}^{n-1}\eta_{\sg^j(\om)}(X_{\sg^j(\om),1}).
\end{align*}
\end{observation}
In light of Observation~\ref{O2 RCIM}, given a RCIM $\eta$, for each $\om\in\Om$ we let
$$
\al_{\om}:=\eta_\om(X_{\om,1}).
$$
Thus we have
\begin{align}\label{eq: def of al_om}
\al_{\om}^n:=\prod_{j=0}^{n-1}\al_{\sg^j(\om)}=\eta_\om(X_{\om,n}).
\end{align}
We now prove a useful identity. 
\begin{lemma}\label{lem: useful identity}
Given any $f,h\in\BV(I)$, any $\om\in\Om$, and $n\in\NN$  we have that 
\begin{align*}
\int_{I_{\sg^n(\om)}} h \cdot \cL_\om^n f \,d\nu_{\sg^n(\om),0}
&=
\lm_{\om,0}^n\int_{X_{\om,n}}f\cdot h\circ T_\om^n\, d\nu_{\om,0}.
\end{align*}
\end{lemma}
\begin{proof}
To prove the identity we calculate the following:
\begin{align*}
\int_{I_{\sg^n(\om)}} h \cdot \cL_\om^n f \,d\nu_{\sg^n(\om),0}
&=
\int_I h \cdot\ind_{\sg^n(\om)}\cdot\cL_{\om,0}^n\lt(f\cdot \hat X_{\om,n-1}\rt) \, d\nu_{\sg^n(\om),0}
\\
&=
\int_I \cL_{\om,0}^n\lt(f (h\circ T_\om^n) (\ind_{\sg^n(\om)}\circ T_\om^n)\cdot \hat X_{\om,n-1}\rt)\, d\nu_{\sg^n(\om),0}
\\
&=
\lm_{\om,0}^n \int_I f (h\circ T_\om^n)\cdot \hat X_{\om,n}\, d\nu_{\om,0}
\\
&=\lm_{\om,0}^n\int_{X_{\om,n}}f\cdot h\circ T_\om^n\, d\nu_{\om,0}.	 	
\end{align*}
\end{proof}
The following lemma gives a useful characterization of RACCIM (with respect to $\nu_0$) in terms of the transfer operators $\cL_{\om}$.
\begin{lemma}\label{LMD lem 1.1.1}
Suppose $\eta=\ind_{\cI}h\nu_0$ is a random probability measure on $\cI$ absolutely continuous with respect to $\nu_0$, whose disintegrations are given by
$$
\eta_{\om}=\ind_\om h_{\om}\nu_{\om,0}.
$$
Then $\eta$ is a RACCIM (with respect to $\nu_0$) if and only if there exists $\al_{\om}>0$ such that 
\begin{align}\label{eq1 lem iff accipm}
\cL_{\om}h_{\om}=\lm_{\om,0}\al_{\om}h_{\sg(\om)}
\end{align}
for each $\om\in\Om$.
\end{lemma}
\begin{proof}
Beginning with the ``reverse'' direction, we first suppose \eqref{eq1 lem iff accipm} holds for all $\om\in\Om$. Let $A\in\sB$ (Borel $\sg$-algebra). Using Lemma~\ref{lem: useful identity} gives 
\begin{align}
\eta_{\om}(T_\om^{-n}(A)\cap X_{\om,n})&=\int_{X_{\om,n}}(\ind_{A}\circ T_\om^n)\cdot h_{\om}\, d\nu_{\om,0}
\nonumber\\
&=(\lm_{\om,0}^n)^{-1}\int_{I_{\sg^n(\om)}}\ind_A\cdot \cL_{\om}^nh_{\om}\, d\nu_{\sg^n(\om),0}
\nonumber\\
&=\int_{I_{\sg^n(\om)}}\ind_A\cdot\al_{\om}^nh_{\sg^n(\om)}\, d\nu_{\sg^n(\om),0}
\nonumber\\
&=\al_{\om}^n\eta_{\sg^n(\om)}(A).\label{LD lem 1.1 eq1}
\end{align}
Inserting $A=I_{\sg^n(\om)}$ into \eqref{LD lem 1.1 eq1} gives
\begin{align*}
\eta_{\om}(T_\om^{-n}(I_{\sg^n(\om)})\cap X_{\om,n})
&=\al_{\om}^n\eta_{\sg^n(\om)}(I_{\sg^n(\om)}).
\end{align*}
Observation \ref{O1 RCIM} implies that $\eta_{\sg^n(\om)}(I_{\sg^n(\om)})=1$, and thus 
\begin{align*}
\al_{\om}^n
&=\eta_{\om}(T_\om^{-n}(I_{\sg^n(\om)})\cap X_{\om,n})
=\eta_\om(X_{\om,n}),
\end{align*}
since $T_\om^{-n}(I_{\sg^n(\om)})\cap X_{\om,n}=X_{\om,n}$.
Thus,  for $A\in\sB$ we have
\begin{align*}
\eta_{\om}(T_\om^{-n}(A)\cap X_{\om,n})=\eta_{\sg^n(\om)}(A)\eta_{\om}(X_{\om,n})
\end{align*}
as desired.

Now to prove the opposite direction, suppose $\eta_{\om}(\ind_\om h_{\om})$ is a RACCIM. Then by the definition of a RCIM there exists $\al_{\om}$ such that for any $A\in\sB$ we have
\begin{align*}
\eta_{\om}(T_\om^{-n}(A)\cap X_{\om,n})=\al_{\om}^n\eta_{\sg^n(\om)}(A).
\end{align*}
So we calculate
\begin{align*}
(\lm_{\om,0}^n)^{-1}\int_{I_{\sg^n(\om)}}\ind_A\cdot\cL_{\om}^nh_{\om}\, d\nu_{\sg^n(\om),0}
&=\int_{X_{\om,n}}(\ind_A\circ T_\om^n)h_{\om}\, d\nu_{\om,0}
=\eta_{\om}(T_\om^{-n}(A)\cap X_{\om,n})
\\
&
=\al_{\om}^n\eta_{\sg^n(\om)}(A)
=\al_{\om}^n\int_{I_{\sg^n(\om)}}\ind_A\cdot h_{\sg^n(\om)}\, d\nu_{\sg^n(\om),0}.	 	
\end{align*}
So we have
\begin{align}\label{eq al_om measure part 2}
\cL_{\om}^nh_{\om}=\lm_{\om,0}^n\al_{\om}^nh_{\sg^n(\om)},
\end{align}
which completes the proof.
\end{proof}

\section{Functionals and partitions}\label{sec: tr op and Lm}
In this section we follow \cite{LSV, LMD} and introduce the random functional $\Lm_\om$ that we will later show is equivalent to the conformal measure for the open system. We also introduce certain refinements of the partition of monotonicity which are used to define ``good'' and ``bad'' intervals and are needed to state our main assumptions on the open system. Following the statement of our main hypotheses, we state our main results.  

We begin by defining the functional $\Lm_{\om}:\BV(I)\to\RR$ by
\begin{align}\label{eq: def of Lm}
\Lm_{\om}(f):=\lim_{n\to\infty}\inf_{x\in D_{\sg^n(\om),n}}\frac{\cL_{\om}^n(f)(x)}{\cL_{\om}^n(\ind_\om)(x)}, \quad f\in\BV(I).
\end{align}\index{$\Lm_{\om}(f)$}
We note that this limit exists as the sequence is bounded and increasing. Indeed, we have that
\begin{align}\label{eq: sup norm bound ratio}
-\norm{f}_{\infty}\leq \inf_{x\in D_{\sg^n(\om),n}}\frac{\cL_{\om}^n(f)(x)}{\cL_{\om}^n(\ind_\om)(x)} \leq \norm{f}_{\infty},
\end{align}
and to see that the ratio is increasing we note that
\begin{align}
&\inf_{x\in D_{\sg^{n+1}(\om),n+1}}
\frac{\cL_{\om}^{n+1}(f)(x)}
{\cL_{\om}^{n+1}(\ind_\om)(x)}
=
\inf_{x\in D_{\sg^{n+1}(\om),n+1}}
\frac{\cL_{\sg^n(\om)}\lt(\hat D_{\sg^n(\om),n}\cdot \cL_{\om}^n(f)\rt)(x)}
{\cL_{\om}^{n+1}(\ind_\om)(x)}
\nonumber\\
&\quad=
\inf_{x\in D_{\sg^{n+1}(\om),n+1}}
\frac{\cL_{\sg^n(\om)}\lt(\hat D_{\sg^n(\om),n}\cdot \cL_{\om}^n(\ind_\om)
	\cdot\frac{\cL_{\om}^n(f)}{\cL_{\om}^n(\ind_\om)}\rt)(x)}
{\cL_{\om}^{n+1}(\ind_\om)(x)}
\nonumber\\
&\quad\geq
\inf_{x\in D_{\sg^{n}(\om),n}}
\frac{\cL_{\om}^{n}(f)(x)}
{\cL_{\om}^{n}(\ind_\om)(x)}
\cdot
\inf_{x\in D_{\sg^{n+1}(\om),n+1}}
\frac{\cL_{\sg^n(\om)}\lt(\hat D_{\sg^n(\om),n}\cdot \cL_{\om}^n\lt(\ind_\om\rt)\rt)(x)}
{\cL_{\om}^{n+1}(\ind_\om)(x)}
\nonumber\\
&\quad=
\inf_{x\in D_{\sg^{n}(\om),n}}
\frac{\cL_{\om}^{n}(f)(x)}
{\cL_{\om}^{n}(\ind_\om)(x)}.	
\label{rho^n increasing}
\end{align}
In particular, \eqref{eq: sup norm bound ratio} of the above argument gives that
\begin{align}\label{eq: Lm bdd in sup norm}
-\norm{f}_{\infty}\leq\inf f\leq \Lm_{\om}(f) \leq \norm{f}_{\infty}.
\end{align}
\begin{observation}\label{obs: properties of Lm}
One can easily check that the functional $\Lm_{\om}$ has the following properties.
\begin{enumerate}
\item[\mylabel{1}{Lm prop1}] $\Lm_\om(\ind)=\Lm_{\om}(\ind_\om)=1$.
\item[\mylabel{2}{Lm prop2}] $\Lm_{\om}$ is continuous with respect to the supremum norm.
\item[\mylabel{3}{Lm prop3}] $f\geq h$ implies that $\Lm_{\om}(f)\geq \Lm_{\om}(h)$.
\item[\mylabel{4}{Lm prop4}] $\Lm_{\om}(cf)=c\Lm_{\om}(f)$.
\item[\mylabel{5}{Lm prop5}] $\Lm_{\om}(f+h)\geq \Lm_{\om}(f)+\Lm_{\om}(h)$.
\item[\mylabel{6}{Lm prop6}] $\Lm_{\om}(f+a)=\Lm_{\om}(f)+a$ for all $a\in\RR$.
\item[\mylabel{7}{Lm prop7}] If $A\cap X_{\om,n}=\emptyset$ for some $n\in\NN$ then $\Lm_{\om}(\ind_A)=0$.
\end{enumerate}
Furthermore, we note that the homogeneity \eqref{Lm prop4} and super-additivity \eqref{Lm prop5} imply that $\Lm_{\om}$ is convex. In the sequel, we will show that $\Lm_{\om}$ is in fact linear, and can thus be associated with a unique probability measure on $I_{\om}$ via the Riesz Representation Theorem.
\end{observation}
\begin{remark}
Let $f\in\BV(I)$, then for all $x,y\in I_{\om}$ we have
\begin{align*}
f(x)\leq f(y)+\var(f).
\end{align*}
Using property \eqref{Lm prop2} of $\Lm_{\om}$, together with \eqref{eq: sup norm bound ratio}, implies
\begin{align}
f(x)
\leq 
\inf f+\var(f)
\leq
\Lm_{\om}(f)+\var(f)
\leq 
\norm{f}_{\infty} +\var(f).
\label{f leq Lm(f)+var(f)}
\end{align}
\end{remark}
We set
\begin{align}\label{rho def}
\rho_{\om}:=\Lm_{\sg(\om)}\left(\cL_{\om}(\ind_\om)\right).
\end{align}\index{$\rho_\om$}
The following propositions concern various estimates of $\rho_{\om}$. We begin by setting
\begin{align}\label{rho^n def}
\rho_{\om}^{(n)}:=\inf_{x\in D_{\sg^{n+1}(\om),n}}\frac{\cL_{\sg(\om)}^n(\cL_{\om}(\ind_\om))(x)}{\cL_{\sg(\om)}^{n}(\ind_{\sg(\om)})(x)}.
\end{align}
Then, by the definition and \eqref{rho^n increasing}, we have that
\begin{align}\label{rho^n inc converge to rho}
\rho_{\om}^{(n)}\nearrow\Lm_{\sg(\om)}(\cL_{\om}(\ind_\om))=\rho_{\om}
\end{align}
as $n\to\infty$.
\begin{remark}\label{rem: rho inequalities}
Note that \eqref{rho def} and the definition of $\cL_\om$ together immediately imply that
\begin{align}\label{rho rough up and low bdd}
\inf_{I_\om} g_\om\leq \rho_{\om}\leq \norm{\cL_{\om}\ind_\om}_{\infty},
\end{align}
and similarly, \eqref{rho^n def} implies that
\begin{align}\label{rho^n rough up and low bdd}
\inf_{I_\om} g_\om
\leq 
\inf_{D_{\sg(\om),1}}\cL_{\om}(\ind_\om)
=
\rho_{\om}^{(0)}
\leq 
\rho_{\om}^{(n)}
\leq 
\norm{\cL_{\om}\ind_\om}_{\infty}
\leq 
\#\cZ_\om\norm{g_{\om,0}}_\infty,
\end{align}
for all $\om\in\Om$ and $n\geq 0$.
Furthermore, \eqref{rho rough up and low bdd}, together with \eqref{eq: log int open weight and tr op}, gives that 
\begin{align}\label{eq: rho log int}
\log\rho_\om\in L^1(m).
\end{align}
The ergodic theorem then implies that 
\begin{align*}
\lim_{n\to\infty}\frac{1}{n}\log\rho_\om^n=\int_\Om\log\rho_\om\,dm(\om),
\end{align*}
where 
\begin{align*}
\rho_\om^n:=\prod_{j=0}^{n-1}\rho_{\sg^j(\om)}.
\end{align*}
\end{remark}	

\begin{proposition}\label{prop: D sets stabilize}
There exists a measurable and finite $m$-a.e. function $N_\infty:\Om\to[1,\infty]$ such that 
\begin{align*}
D_{\om,n}=D_{\om,\infty}
\end{align*}
for all $n\geq N_\infty(\om)$. Furthermore, this implies that
\begin{align}\label{eq: L1 pos on D infty}
\inf_{D_{\om,\infty}}\cL_{\sg^{-n}(\om)}^n\ind_{\sg^{-n}(\om)}>0
\end{align}
for all $n\geq N_\infty(\om)$.
\end{proposition}
\begin{proof}
We proceed via contradiction, assuming that there is a sequence $(n_k)_{k=1}^\infty$ in $\NN$ such that
$$
D_{\sg^{n_k+1}(\om),n_k+1}\subsetneq D_{\sg^{n_k+1}(\om),n_k}.
$$
Let $x_{n_k}\in D_{\sg^{n_k+1}(\om),n_k}\bs D_{\sg^{n_k+1}(\om),n_k+1}$. Then, we have that
$$
\rho_{\om}^{(n_k)}=\inf_{x\in D_{\sg^{n_k+1}(\om),n_k}}\frac{\cL_{\sg(\om)}^{n_k}(\cL_{\om}(\ind_\om))(x)}{\cL_{\sg(\om)}^{n_k}(\ind_{\sg(\om)})(x)}
\leq \frac{\cL_{\sg(\om)}^{n_k}(\cL_{\om}(\ind_\om))(x_{n_k})}{\cL_{\sg(\om)}^{n_k}(\ind_{\sg(\om)})(x_{n_k})}.
$$
By our choice of $x_{n_k}$ and by the definition \eqref{defn of D sets}, we have that the numerator of the quantity on the right is zero, while its denominator is strictly positive. As this holds for each  $k\in\NN$, this implies that $\rho_\om=0$ for $m$-a.e. $\om\in\Om$, which contradicts \eqref{eq: rho log int}. Thus, we are done.
\end{proof}
\begin{remark}\label{rem: check cond D}
Note that our assumption \eqref{cond D}, that $D_{\om,\infty}\neq\emptyset$, is satisfied if $T_\om(I_\om)\bus I_{\sg(\om)}$ for $m$-a.e. $\om\in\Om$. 
Moreover, this also implies that $\rho_\om\geq \inf_{I_{\sg(\om)}}\cL_\om\ind_\om>0$. This occurs, for example if for $m$-a.e. $\om\in\Om$ there exists a full branch, i.e. there exists $Z\in\cZ_\om$ with $T_\om(Z)=I$, outside of the hole $H_\om$, in which case we would have that $D_{\om,\infty}=I$ for $m$-a.e. $\om\in\Om$. 
\end{remark}
We now describe various partitions, which depend on the functional $\Lm_\om$, that we will used to obtain a Lasota-Yorke inequality in Section~\ref{sec: LY ineq}. Recall that $\cZ_\om^{(n)}$ denotes the partition of monotonicity of $T_\om^n$. Now, for each $n\in\NN$ and $\om\in\Om$ we let $\sA_\om^{(n)}$\index{$\sA_\om^{(n)}$} be the collection of all finite partitions of $I$ such that
\begin{align}\label{eq: def A partitiona}
\var_{A_i}(g_{\om}^{(n)})\leq 2\norm{g_{\om}^{(n)}}_{\infty}
\end{align}
for each $\cA=\set{A_i}\in\sA_{\om}^{(n)}$.\index{$\cA$}

Given $\cA\in\sA_\om^{(n)}$, let $\widehat\cZ_\om^{(n)}(\cA)$ be the coarsest partition amongst all those finer than $\cA$ and $\cZ_\om^{(n)}$ such that all elements of $\widehat\cZ_\om^{(n)}(\cA)$\index{$\widehat\cZ_\om^{(n)}(\cA)$} are either disjoint from $X_{\om,n-1}$ or contained in $X_{\om,n-1}$. Now, define the following subcollections:
\begin{align}
\cZ_{\om,*}^{(n)}&:=\set{Z\in \widehat\cZ_\om^{(n)}(\cA): Z\sub X_{\om,n-1} },
\label{eq: Z_*}
\\
\cZ_{\om,b}^{(n)}&:=\set{Z\in \widehat\cZ_\om^{(n)}(\cA): Z\sub X_{\om,n-1} \text{ and } \Lm_{\om}(\ind_Z)=0 },
\label{eq: Z_b}
\\
\cZ_{\om,g}^{(n)}&:=\set{Z\in \widehat\cZ_\om^{(n)}(\cA): Z\sub X_{\om,n-1} \text{ and } \Lm_{\om}(\ind_Z)>0}.
\label{eq: Z_g}
\end{align}\index{$\cZ_{\om,*}^{(n)}$}\index{$\cZ_{\om,b}^{(n)}$}\index{$\cZ_{\om,g}^{(n)}$}
\begin{remark}
Note that in light of \eqref{eq: def of Lm} and \eqref{rho^n increasing}, for every $Z\in\cZ_{\om,g}^{(n)}$ we may define the open covering time\index{open covering time} $M_\om(Z)\in\NN$\index{$M_\om(Z)$} to be the least integer such that 
\begin{align}\label{eq: def of open covering}
\inf_{x\in D_{\sg^{M_\om(Z)}(\om),M_\om(Z)}}\frac{\cL_{\om}^{M_\om(Z)}\ind_Z(x)}{\cL_{\om}^{M_\om(Z)}(\ind_\om)(x)} > 0
\end{align}
which is finite since the ratio in \eqref{eq: def of open covering} increases to $\Lm_\om(\ind_Z)>0$.
Conversely, given that the ratio in \eqref{eq: def of open covering} is increasing by \eqref{rho^n increasing} we see that, for $Z\in\cZ_{\om,*}^{(n)}$, if there exists any $N\in\NN$ such that 
\begin{align*}
\inf_{x\in D_{\sg^{N}(\om),N}}\frac{\cL_{\om}^{N}\ind_Z(x)}{\cL_{\om}^{N}(\ind_\om)(x)} > 0,
\end{align*}
then we must have that $\Lm_\om(\ind_Z)>0$ and equivalently $Z\in\cZ_{\om,g}^{(n)}$.
\end{remark}
We adapt the following definition from \cite{LMD}.
\begin{definition}\label{def: contiguous}
We say that two elements $W, Z\in\cZ_{\om,*}^{(n)}$ are \textit{contiguous}\index{contiguous} if either $W$ and $Z$ are contiguous in the usual sense, i.e. they share a boundary point, or if they are separated by a connected component of $\cup_{j=0}^{n-1} T_\om^{-j}(H_{\sg^j(\om)})$.
\end{definition}

We will consider random open systems that satisfy the following conditions.
\begin{enumerate}
\item[(\Gls*{Q1})]\myglabel{Q1}{cond Q1}
For each $\om\in\Om$ and $n\in\NN$ we let $\xi_{\om}^{(n)}$\index{$\xi_{\om}^{(n)}$} denote the maximum number of contiguous elements of $\cZ_{\om,b}^{(n)}$.
We assume 
\begin{align*}
\lim_{n\to\infty}\frac{1}{n}\log\norm{g_{\om}^{(n)}}_\infty
+
\limsup_{n\to\infty}\frac{1}{n}\log\xi_{\om}^{(n)}
<
\lim_{n\to\infty}\frac{1}{n}\log\rho_{\om}^n
=
\int_\Om \log\rho_{\om}\, dm(\om).
\end{align*}

\item[(\Gls*{Q2})]\myglabel{Q2}{cond Q2} We assume that for each $n\in\NN$ we have $\log\xi_\om^{(n)}\in L^1(m)$.

\item[(\Gls*{Q3})]\myglabel{Q3}{cond Q3} Let 
\begin{align}\label{eq: def of dl_om,n}
\dl_{\om,n}:=\min_{Z\in\cZ_{\om,g}^{(n)}}\Lm_\om(\ind_Z).
\end{align}\index{$\dl_{\om,n}$}
We assume that, for each $n\in\NN$, $\log\dl_{\om,n}\in L^1(m)$. 
\end{enumerate}
\begin{remark}
In Section~\ref{sec: examples} we give several alternate hypotheses to our assumptions \eqref{cond Q1}-\eqref{cond Q3} that are slightly more restrictive, but much simpler to check. 
\end{remark}
\begin{remark}\label{rem: further props of Q assumps}
\noindent
\begin{enumerate}[(1)]
\item Note that since $\lim_{n\to\infty}\frac{1}{n}\log\xi_{\om}^{(n)}\geq 0$, assumption \eqref{cond Q1} implies that 
\begin{align*}
\lim_{n\to\infty}\frac{1}{n}\log\norm{g_{\om}^{(n)}}_\infty
<
\lim_{n\to\infty}\frac{1}{n}\log\rho_{\om}^n.
\end{align*}

\item Since $\norm{g_{\om}}_\infty\leq\norm{g_{\om,0}}_\infty$ and $\inf_{D_{\sg(\om), \infty}}\cL_{\om}\ind_\om\leq \rho_\om$ to check \eqref{cond Q1} it suffices to have 
\begin{align*}
\lim_{n\to\infty}\frac{1}{n}\log\norm{g_{\om,0}^{(n)}}_\infty 
+
\lim_{n\to\infty}\frac{1}{n}\log\xi_{\om}^{(n)}
< 
\lim_{n\to\infty}\frac{1}{n}\log\inf_{D_{\sg^n(\om),\infty}}\cL_{\om}^n\ind_\om .
\end{align*} 
\item One can use the open covering times defined in \eqref{eq: def of open covering} to check \eqref{cond Q3}. Indeed, note that if 
\begin{align*}
N\geq M_{\om,n}:=\max\set{M_\om(Z): Z\in\cZ_{\om,g}^{(n)}}
\end{align*}
then we have that 
\begin{align*}
\dl_{\om,n}
&\geq 
\min_{Z\in\cZ_{\om,g}^{(n)}}\inf_{x\in D_{\sg^N(\om),N}}\frac{\cL_\om^N\ind_Z(x)}{\cL_\om^N\ind_\om(x)}
\\
&\geq
\min_{Z\in\cZ_{\om,g}^{(n)}}
\frac{\inf_{x\in D_{\sg^N(\om),N}}\cL_\om^N\ind_Z(x)}{\norm{\cL_\om^N\ind_\om}_\infty}
\\
&\geq 
\frac{\inf_{X_{\om,N-1}} g_\om^{(N)}}{\norm{\cL_\om^N\ind_\om}_\infty}
\\
&\geq
\frac{\inf g_{\om,0}^{(N)}}{\norm{\cL_{\om,0}^N\ind}_\infty}>0.
\end{align*}
Thus \eqref{cond Q3} holds if $\log \inf g_{\om,0}^{(M_{\om,n})}, \log\norm{\cL_{\om,0}^{M_{\om,n}}\ind_\om}_\infty\in L^1(m)$ for each $n\in\NN$.
\end{enumerate}
\end{remark}

\section{Random Birkhoff cones and Hilbert metrics}\label{sec:cones}
In this section we first recall the theory of convex cones first used by Birkhoff in \cite{birkhoff_lattice_1940}, and then present the random cones on which our operator $\cL_{\om}$ will act as a contraction. We begin with a definition. 
\begin{definition}
Given a vector space $\cV$, we call a subset $\cC\sub\cV$ a \textit{convex cone}\index{convex cone} if $\cC$ satisfies the following:
\begin{enumerate}
\item $\cC\cap-\cC=\emptyset$,
\item for all $\al>0$, $\al\cC=\cC$, 
\item $\cC$ is convex,
\item for all $f, h\in\cC$ and all $\al_n\in\RR$ with $\al_n\to\al$ as $n\to\infty$, if $h-\al_nf\in\cC$ for each $n\in\NN$, then $h-\al f\in\cC\cup\set{0}$.
\end{enumerate}
\end{definition}
\begin{lemma}[Lemma 2.1 \cite{LSV}]
The relation $\leq $ defined on $\cV$ by 
$$
f\leq h \text{ if and only if } h-f\in\cC\cup\set{0}
$$
is a partial order satisfying the following:
\begin{flalign*}
& f\leq 0\leq f  \implies f=0,
\tag{i}
&\\
& \lm>0 \text{ and } f\geq 0 \iff \lm f\geq 0,
\tag{ii} \\
& f\leq h \iff 0\leq h-f,
\tag{iii}\\
& \text{for all } \alpha_n\in\RR \text{ with } \al_n\to\al, \, \al_nf\leq h \implies \al f\leq h,
\tag{iv} \\
& f\geq 0 \text{ and } h\geq 0 \implies f+h\geq 0.
\tag{v}
\end{flalign*}
\end{lemma}
The Hilbert metric on $\cC$ is given by the following definition.
\begin{definition}
Define a distance $\Ta(f,h)$ by 
\begin{align*}
\Ta(f,h):=\log\frac{\bt(f,h)}{\al(f,h)},
\end{align*}
where 
\begin{align*}
\al(f,h):=\sup\set{a>0: af\leq h} 
\quad\text{ and }\quad
\bt(f,h):=\inf\set{b>0: bf\geq h}.
\end{align*}
\end{definition}
Note that $\Ta$ is a pseudo-metric as two elements in the cone may be at an infinite distance from each other. Furthermore, $\Ta$ is a projective metric because any two proportional elements must be zero distance from each other. The next theorem, which is due to Birkhoff \cite{birkhoff_lattice_1940}, shows that every positive linear operator that preserves the cone is a contraction provided that the diameter of the image is finite. 
\begin{theorem}[\cite{birkhoff_lattice_1940}]\label{thm: cone distance contraction}
Let $\cV_1$ and $\cV_2$ be vector spaces with convex cones $\cC_1\sub\cV_1$ and $\cC_2\sub\cV_2$ and a positive linear operator $\cL:\cV_1\to\cV_2$ such that $\cL(\cC_1)\sub \cC_2$. If $\Ta_i$ denotes the Hilbert metric on the cone $\cC_i$ and if 
$$
\Dl=\sup_{f,h\in\cC_1}\Ta_2(\cL f,\cL h),
$$
then 
$$
\Ta_2(\cL f,\cL h)\leq \tanh\lt(\frac{\Dl}{4}\rt)\Ta_1(f,h)
$$
for all $f,h\in\cC_1$.
\end{theorem}
Note that it is not clear whether $(\cC,\Ta)$ is complete.
The following lemma of \cite{LSV} addresses this problem by linking the metric $\Ta$ with a suitable norm $\norm{\spot}$ on $\cV$.
\begin{lemma}[\cite{LSV}, Lemma 2.2]\label{lem: birkhoff cone contraction}
Let $\norm{\spot}$ be a norm on $\cV$ such that for all $f,h\in\cV$ if $-f\leq h\leq f$, then $\norm{h}\leq \norm{f}$, and let  $\vrho:\cC\to (0,\infty)$ be a homogeneous and order-preserving function, which means that for all $f,h\in\cC$ with $f\leq h$ and all $\lm>0$ we have 
$$
\vrho(\lm f)=\lm\vrho(f) \qquad\text{ and }\qquad \vrho(f)\leq \vrho(h).
$$
Then, for all $f,h\in\cC$ $\vrho(f)=\vrho(h)>0$ implies that 
$$
\norm{f-h}\leq \lt(e^{\Ta(f,h)}-1\rt)\min\set{\norm{f},\norm{h} }.
$$
\end{lemma}
\begin{remark}
Note that the choice $\vrho(\spot)=\norm{\spot}$ satisfies the hypothesis, however from this moment on we shall make the choice of $\vrho=\Lm_{\om}$.
\end{remark}
\begin{definition}\label{def: + and a cones}

For each $a>0$ and $\om\in\Om$ let 
\begin{align}\label{eq: def of a cones}
\sC_{\om,a}:=\set{f\in\BV(I): f\geq 0,\, \var(f)\leq a\Lm_{\om}(f)}.
\end{align}\index{$\sC_{\om,a}$}
To see that this cone is non-empty, we note that the function $f+c\in\sC_{\om,a}$ for $f\in\BV(I)$ and $c\geq a^{-1}\var(f)-\inf_{X_\om}f$. We also define the cone 
\begin{align*}
\sC_{\om,+}:=\set{f\in\BV(I): f\geq 0}.
\end{align*}

\end{definition} 
Let $\Ta_{\om,a}$ and $\Ta_{\om,+}$\index{$\Ta_{\om,a}$}\index{$\Ta_{\om,+}$} denote the Hilbert metrics induced on the respective cones $\sC_{\om,a}$ and $\sC_{\om,+}$\index{$\sC_{\om,+}$}. 
For each $\om\in\Om$, $a>0$, and any set $Y\sub \sC_{\om,a}$ we let
\begin{align*}
\diam_{\om,a}(Y):=\sup_{x,y\in Y} \Ta_{\om,a}(x,y)
\end{align*}\index{$\diam_{\om,a}(Y)$}
and 
\begin{align*}
\diam_{\om,+}(Y):=\sup_{x,y\in Y} \Ta_{\om,+}(x,y)
\end{align*}\index{$\diam_{\om,+}(Y)$}
denote the diameter of $Y$ in the respective cones $\sC_{\om,a}$ and $\sC_{\om,+}$ with respect to the respective metrics $\Ta_{\om,a}$ and $\Ta_{\om,+}$.
The following lemma collects together the main properties of these metrics. 
\begin{lemma}[\cite{LSV}, Lemmas 4.2, 4.3, 4.5]\label{lem: summary of cone dist prop}
For $f, h\in\sC_{\om,+}$ the $\Ta_{\om,+}$ distance between $f,h$ is given by 
\begin{align*}
\Ta_{\om,+}(f,h)=\log\sup_{x,y\in X_\om}\frac{f(y)h(x)}{f(x)h(y)}.
\end{align*}
If $f, h\in\sC_{\om,a}$, then
\begin{align}\label{eq: Ta+ leq Ta}
\Ta_{\om,+}(f,h)\leq \Ta_{\om,a}(f,h), 
\end{align}
and if $f\in\sC_{\om,\eta a}$, for $\eta\in (0,1)$, we then have
\begin{align*}
\Ta_{\om,a}(\ind,f)\leq \log\frac{\norm{f}_\infty+\eta\Lm_\om(f)}{\min\set{\inf_{X_\om} f, (1-\eta)\Lm_\om(f)}}.
\end{align*}
\end{lemma}


\section{Lasota-Yorke inequalities}\label{sec: LY ineq}
The main goal of this section is to prove a Lasota-Yorke type inequality. We adopt the strategy of \cite{AFGTV20}, where we first prove a less-refined Lasota-Yorke inequality with (random) coefficients that behave in a difficult manner, and then, using the first inequality, prove a second inequality with measurable random coefficients and uniform decay on the variation as in \cite{buzzi_exponential_1999}. 

We now prove a Lasota-Yorke type inequality following the approach of \cite{LMD} utilizing the ``good'' and ``bad'' interval partitions defined in \eqref{eq: Z_*}-\eqref{eq: Z_g}. 
\begin{lemma}\label{ly ineq}	
For all $\om\in\Om$, all $f\in\BV(I)$, and all $n\in\NN$ there exist positive, measurable constants $A_{\om}^{(n)}$ and $B_{\om}^{(n)}$ such that 
\begin{align*}
\var(\cL_{\om}^nf)\leq A_{\om}^{(n)}\var(f)+B_{\om}^{(n)}\Lm_{\om}(|f|), 
\end{align*}
where
\begin{align*}
A_{\om}^{(n)}:=(9+16\xi_{\om}^{(n)})\norm{g_{\om}^{(n)}}_{\infty}
\end{align*}\index{$A_{\om}^{(n)}$}
and 
\begin{align*}
B_{\om}^{(n)}:=8(2\xi_{\om}^{(n)}+1)\norm{g_{\om}^{(n)}}_{\infty}
\dl_{\om,n}^{-1}.
\end{align*}\index{$B_{\om}^{(n)}$}
\end{lemma}
\begin{proof}
Since $\cL_{\om}^n(f)=\cL_{\om,0}^n(f\cdot\hat X_{\om,n-1})$, if $Z\in\widehat\cZ_\om^{(n)}(\cA)\bs\cZ_{\om,*}^{(n)}$, then  $Z\cap X_{\om,n-1}=\emptyset$, and thus, we have $\cL_{\om}^n(f\ind_Z)=0$ for each $f\in\BV(I)$. Thus, considering only intervals $Z$ in $\cZ_{\om,*}^{(n)}$, we are able to write 
\begin{align}\label{eq: ly ineq 1a}
\cL_{\om}^nf=\sum_{Z\in\cZ_{\om,*}^{(n)}}(\ind_Z f g_{\om}^{(n)})\circ T_{\om,Z}^{-n}
\end{align} 
where 
$$	
T_{\om,Z}^{-n}:T_\om^n(I_{\om})\to Z
$$ 
is the inverse branch which takes $T_\om^n(x)$ to $x$ for each $x\in Z$. Now, since 
$$
\ind_Z\circ T_{\om,Z}^{-n}=\ind_{T_\om^n(Z)},
$$
we can rewrite \eqref{eq: ly ineq 1a} as 
\begin{align}\label{eq: ly ineq 2}
\cL_{\om}^nf=\sum_{Z\in\cZ_{\om,*}^{(n)}}\ind_{T_\om^n(Z)} \lt((f g_{\om}^{(n)})\circ T_{\om,Z}^{-n}\rt).
\end{align}
So,
\begin{align}\label{var tr op sum}
\var(\cL_{\om}^nf)\leq \sum_{Z\in\cZ_{\om,*}^{(n)}}\var\lt(\ind_{T_\om^n(Z)} \lt((f g_{\om}^{(n)})\circ T_{\om,Z}^{-n}\rt)\rt).
\end{align}
Now for each $Z\in\cZ_{\om,*}^{(n)}$, using \eqref{eq: def A partitiona}, we have 
\begin{align}
&\var\lt(\ind_{T_\om^n(Z)} \lt((f g_{\om}^{(n)})\circ T_{\om,Z}^{-n}\rt)\rt)
\leq \var_Z(f g_{\om}^{(n)})+2\sup_Z\absval{f g_{\om}^{(n)}}
\nonumber\\
&\qquad\qquad\leq 3\var_Z(f g_{\om}^{(n)})+2\inf_Z\absval{f g_{\om}^{(n)}}
\nonumber\\
&\qquad\qquad\leq 3\norm{g_{\om}^{(n)}}_{\infty}\var_Z(f)+3\sup_Z|f|\var_Z(g_{\om}^{(n)})+2\norm{g_{\om}^{(n)}}_{\infty}\inf_Z|f|
\nonumber\\
&\qquad\qquad\leq 
3\norm{g_{\om}^{(n)}}_{\infty}\var_Z(f)+6\norm{g_{\om}^{(n)}}_{\infty}\sup_Z|f|+2\norm{g_{\om}^{(n)}}_{\infty}\inf_Z|f|
\nonumber\\
&\qquad\qquad\leq 
9\norm{g_{\om}^{(n)}}_{\infty}\var_Z(f)+8\norm{g_{\om}^{(n)}}_{\infty}\inf_Z|f|.
\label{var ineq over partition}
\end{align}
Now, using \eqref{var ineq over partition}, we may further estimate \eqref{var tr op sum} as
\begin{align}
\var(\cL_{\om}^nf)
&\leq 
\sum_{Z\in\cZ_{\om,*}^{(n)}} \lt(9\norm{g_{\om}^{(n)}}_{\infty}\var_Z(f)+8\norm{g_{\om}^{(n)}}_{\infty}\inf_Z|f|\rt)
\nonumber\\
&\leq 
9\norm{g_{\om}^{(n)}}_{\infty}\var(f)+8\norm{g_{\om}^{(n)}}_{\infty}
\lt(\sum_{Z\in\cZ_{\om,g}^{(n)}}\inf_Z|f|+\sum_{Z\in\cZ_{\om,b}^{(n)}}\inf_Z|f|\rt).
\label{sum over *partition}
\end{align}
In order to investigate each of the two sums in the line above, we first note that as $\cZ_{\om,g}^{(n)}$ is finite then, by definition, there exists a constant $\dl_{\om,n}>0$ (defined by \eqref{eq: def of dl_om,n}) such that 
\begin{align*}
\inf_{Z\in\cZ_{\om,g}^{(n)}}\Lm_{\om}(\ind_Z)\geq 2\dl_{\om,n}>0.
\end{align*}
So, we may choose $N_{\om,n}\in\NN$ such that for $x\in D_{\sg^{N_{\om,n}}(\om),N_{\om,n}}$ we have 
\begin{align}\label{eq: ratio bigger than bt_om,n}
\inf_{Z\in\cZ_{\om,g}^{(n)}}\frac{\cL_{\om}^{N_{\om,n}}(\ind_Z)(x)}{\cL_{\om}^{N_{\om,n}}(\ind_\om)(x)}\geq \dl_{\om,n}.
\end{align}
Note that since this ratio is increasing we have that \eqref{eq: ratio bigger than bt_om,n} holds for all $\~N\geq N_{\om,n}$.
Then for each $x\in D_{\sg^{N_{\om,n}}(\om),N_{\om,n}}$ and $Z\in\cZ_{\om,g}^{(n)}$ we have
\begin{align*}
\cL_{\om}^{N_{\om,n}}(|f|\ind_Z)(x)&\geq \inf_Z|f|\cL_{\om}^{N_{\om,n}}(\ind_Z)(x)
\geq \inf_Z|f|\dl_{\om,n}\cL_{\om}^{N_{\om,n}}(\ind_\om)(x).
\end{align*}
In particular, for each $x\in D_{\sg^{N_{\om,n}}(\om),N_{\om,n}}$, we see that 
\begin{align}
\sum_{Z\in\cZ_{\om,g}^{(n)}}\inf_Z|f|
&\leq \dl_{\om,n}^{-1}\sum_{Z\in\cZ_{\om,g}^{(n)}}\frac{\cL_{\om}^{N_{\om,n}}(|f|\ind_Z)(x)}{\cL_{\om}^{N_{\om,n}}(\ind_\om)(x)}
\leq 
\dl_{\om,n}^{-1}\frac{\cL_{\om}^{N_{\om,n}}(|f|)(x)}{\cL_{\om}^{N_{\om,n}}(\ind_\om)(x)}.\label{est of sum over good}
\end{align}
We are now interested in finding appropriate upper bounds for the sum 
\begin{align*}
\sum_{Z\in\cZ_{\om,b}^{(n)}}\inf_Z|f|.
\end{align*}
However, we must first be able to associate each of the elements of $\cZ_{\om,b}^{(n)}$ with one of the elements of $\cZ_{\om,g}^{(n)}$. To that end, let $Z_*$ and $Z^*$ denote the elements of $\cZ_{\om,*}^{(n)}$ that are the furthest to the left and the right respectively.  
Now, enumerate each of the elements of $\cZ_{\om,g}^{(n)}$, $Z_1,\dots, Z_k$ (clearly $k$ depends on $\om$, $n$, and $\cA$), such that $Z_{j+1}$ is to the right of $Z_j$ for $j=1,\dots, k-1$.
Given $Z_j\in\cZ_{\om,g}^{(n)}$ ($1\leq j\leq k$), with $Z_j\neq Z^*$  let $J_{\om,+}(Z_j)$ be the union of all contiguous elements $Z\in\cZ_{\om,b}^{(n)}$ which are to the right of $Z_j$ and also to the left of $Z_{j+1}$. In other words, $J_{\om,+}(Z_j)$\index{$J_{\om,+}(Z)$} is the union of all elements of $\cZ_{\om,b}^{(n)}$ between $Z_j$ and $Z_{j+1}$. Similarly, for $Z_j\neq Z_*$, we define $J_{\om,-}(Z_j)$\index{$J_{\om,-}(Z)$} be the union of all contiguous elements $Z\in\cZ_{\om,b}^{(n)}$ which are to the left of $Z_j$ and also to the right of $Z_{j-1}$. Now, we note that our assumption \eqref{cond Q1} implies that each $J_{\om,-}(Z)$ and $J_{\om,+}(Z)$ ($Z\in \cZ_{\om,g}^{(n)}$) is the union of at most $\xi_{\om}^n$ contiguous elements of $\cZ_{\om,b}^{(n)}$.  
For $Z\in\cZ_{\om,g}^{(n)}$ let  
$$
J_{\om,-}^*(Z)=Z\cup J_{\om,-}(Z) 
\qquad\text{ and }\qquad
J_{\om,+}^*(Z)=Z\cup J_{\om,+}(Z). 
$$
Then for $W\sub J_{\om,-}^*(Z)$ we have 
\begin{align}\label{inf over union BV ineq}
\inf_W |f|\leq \inf_{Z}|f|+\var_{J_{\om,-}^*(Z)}(f).
\end{align}
We obtain a similar inequality for $W\sub J_{\om,+}^*(Z)$. 
We now consider the following two cases. 
\begin{enumerate}
\item[\mylabel{Case 1:}{Case 1}] At least one of the intervals $Z_*$ and $Z^*$ is an element of $\cZ_{\om,g}^{(n)}$. 
\item[\mylabel{Case 2:}{Case 2}] Neither of the intervals $Z_*$, $Z^*$ is an element of $\cZ_{\om,g}^{(n)}$.
\end{enumerate}
If we are in the first case, we assume without loss of generality that $Z_1=Z_*$, and thus every element $Z\in\cZ_{\om,b}^{(n)}$ is contained in exactly one union $J_{\om,+}(Z_j)$ for some $Z_j\in\cZ_{\om,g}^{(n)}$ for some $1\leq j\leq k$. If $Z_1\neq Z_*$ and instead we have that $Z_k=Z^*$ we could simply replace $J_{\om,+}(Z_j)$ with $J_{\om,-}(Z_j)$ in the previous statement. In view of \eqref{inf over union BV ineq}, Case 1 leads to the conclusion that  
\begin{align*}
\sum_{Z\in\cZ_{\om,b}^{(n)}}\inf_Z|f| \leq 	\xi_{\om}^{(n)}\lt(\sum_{Z\in\cZ_{\om,g}^{(n)}}\inf_Z|f|+\var_{J_{\om,-}^*(Z)}(f)\rt).
\end{align*}	
If we are instead in the second case, then for each $Z\in\cZ_{\om,b}^{(n)}$ to the left of $Z_k$ there is exactly one $Z_j$, $1\leq j\leq k$, such that $Z\sub J_{\om,-}(Z_j)$. This leaves each of the elements $Z\in\cZ_{\om,b}^{(n)}$ to the right of $Z_k$ uniquely contained in the union $J_{\om,+}(Z_k)$. Thus, Case 2 yields  
\begin{align*}
\sum_{Z\in\cZ_{\om,b}^{(n)}}\inf_Z|f| &\leq 	\xi_{\om}^{(n)}\lt(\inf_{Z_k}|f|+\var_{J_{\om,+}^*(Z_k)}(f)+\sum_{Z\in\cZ_{\om,g}^{(n)}}\inf_Z|f|+\var_{J_{\om,-}^*(Z)}(f)\rt)\\
&\leq 2\xi_{\om}^{(n)}\lt(\var(f)+\sum_{Z\in\cZ_{\om,g}^{(n)}}\inf_Z|f|\rt).
\end{align*}
Hence, either case gives that 
\begin{align}\label{est of sum over bad}
\sum_{Z\in\cZ_{\om,b}^{(n)}}\inf_Z|f| \leq 2\xi_{\om}^{(n)}\lt(\var(f)+\sum_{Z\in\cZ_{\om,g}^{(n)}}\inf_Z|f|\rt).
\end{align}
Inserting \eqref{est of sum over good} and \eqref{est of sum over bad} into \eqref{sum over *partition} gives 
\begin{align*}
\var(\cL_{\om}^nf)
&\leq 9\norm{g_{\om}^{(n)}}_{\infty}\var(f)\\
&\qquad\quad
+8\norm{g_{\om}^{(n)}}_{\infty}\lt(2\xi_{\om}^{(n)}\lt(\var(f)+\sum_{Z\in\cZ_{\om,g}^{(n)}}\inf_Z|f|\rt)+\dl_{\om,n}^{-1}\frac{\cL_{\om}^{N_{\om,n}}(|f|)(x)}{\cL_{\om}^{N_{\om,n}}(\ind_\om)(x)}\rt)
\\
&\leq 
(9+16\xi_{\om}^{(n)})\norm{g_{\om}^{(n)}}_{\infty}\var(f)
+
8(2\xi_{\om}^{(n)}+1)\norm{g_{\om}^{(n)}}_{\infty}
\dl_{\om,n}^{-1}\frac{\cL_{\om}^{N_{\om,n}}(|f|)(x)}{\cL_{\om}^{N_{\om,n}}(\ind_\om)(x)}.
\end{align*}
In view of \eqref{rho^n increasing}, taking the infimum over $x\in D_{\sg^{N_{\om,n}}(\om),N_{\om,n}}$ allows us to replace the ratio $\frac{\cL_{\om}^{N_{\om,n}}(|f|)(x)}{\cL_{\om}^{N_{\om,n}}(\ind_\om)(x)}$ with $\Lm_{\om}(|f|)$, that is, we have
\begin{align*}
&\var(\cL_{\om}^nf)
\leq 
(9+16\xi_{\om}^{(n)})\norm{g_{\om}^{(n)}}_{\infty}\var(f)
+
8(2\xi_{\om}^{(n)}+1)\norm{g_{\om}^{(n)}}_{\infty}
\dl_{\om,n}^{-1}\Lm_{\om}(|f|).
\end{align*}
Setting 
\begin{align}\label{eq: def of A and B in ly ineq}
A_{\om}^{(n)}:=(9+16\xi_{\om}^{(n)})\norm{g_{\om}^{(n)}}_{\infty}
\qquad\text{ and }\qquad
B_{\om}^{(n)}:=8(2\xi_{\om}^{(n)}+1)\norm{g_{\om}^{(n)}}_{\infty}
\dl_{\om,n}^{-1}
\end{align}
finishes the proof.
\end{proof}
\begin{remark}
As a consequence of Lemma~\ref{ly ineq} we have that 
\begin{align}\label{eq: L_om is a weak contraction on C_+}
\cL_\om\lt(\sC_{\om,+}\rt)\sub \sC_{\sg(\om),+},
\end{align}
and thus $\cL_{\om}$ is a weak contraction on $\sC_{\om,+}$. 
\end{remark}
Define the random constants
\begin{align}\label{eq: def of Q and K}
Q_{\om}^{(n)}:=\frac{A_\om^{(n)}}{\rho_{\om}^n}
\quad \text{ and }\quad
K_{\om}^{(n)}:=\frac{B_\om^{(n)}}{\rho_{\om}^n}. 
\end{align}\index{$Q_{\om}^{(n)}$}\index{$K_{\om}^{(n)}$}
In light of our assumption \eqref{cond Q1} on the potential and number of contiguous bad intervals, we see that $Q_\om^{(n)}\to 0$ exponentially quickly for each $\om\in\Om$.

The following proposition now follows from \eqref{eq: log int open weight and tr op}, \eqref{eq: rho log int}, and assumptions \eqref{cond Q2}-\eqref{cond Q3}. 
\begin{proposition}\label{prop: log integr of Q and K}
For each $n\in\NN$, $\log^+ Q_\om^{(n)}, \log K_\om^{(n)}\in L^1_m(\Om)$. 
\end{proposition}

\begin{lemma}\label{LMD l3.6}
For each $f\in\BV(I)$ and each $n,k\in\NN$ we have
\begin{align}
\Lm_{\sg^k(\om)}\lt(\cL_{\om}^kf\rt)&\geq \Lm_{\sg^k(\om)}\lt(\cL_{\om}^k\ind_\om\rt)\cdot\Lm_{\om}(f).
\label{eq: Lm^k1 ineq}
\end{align}
Furthermore, we have that 
\begin{align}\label{LMD l3.6 key ineq}
\rho_{\om}^n\cdot\Lm_{\om}(f)\leq \Lm_{\sg^n(\om)}(\cL_{\om}^n f).
\end{align}
In particular, this yields
\begin{align*}
\rho_{\om}^n\leq \Lm_{\sg^n(\om)}(\cL_{\om}^n\ind_\om).
\end{align*}
\end{lemma}
\begin{proof}
For each $f\in\BV(I)$ with $f\geq 0$, $k\in\NN$, and $x\in D_{\sg^{n+k}(\om),n}$ we have
\begin{align*}
\frac{\cL_{\sg^k(\om)}^n\lt(\cL_{\om}^kf\rt)(x)}
{\cL_{\sg^k(\om)}^n\lt(\ind_{\sg^k(\om)}\rt)(x)}
&=
\frac{\cL_{\sg^n(\om)}^k\lt(\cL_{\om}^nf\rt)(x)}
{\cL_{\sg^k(\om)}^n\lt(\ind_{\sg^k(\om)}\rt)(x)}
\\
&=
\frac{\cL_{\sg^n(\om)}^k\lt(\hat D_{\sg^n(\om),n}\cdot\frac{\cL_{\om}^nf}{\cL_{\om}^n\ind_\om}\cdot \cL_{\om}^n\ind_\om\rt)(x)}
{\cL_{\sg^k(\om)}^n\lt(\ind_{\sg^k(\om)}\rt)(x)}
\\
&\geq
\frac{\cL_{\sg^k(\om)}^n\lt(\cL_{\om}^k\ind_\om\rt)(x)}
{\cL_{\sg^k(\om)}^n\lt(\ind_{\sg^k(\om)}\rt)(x)}
\cdot\inf_{D_{\sg^n(\om),n}}
\frac{\cL_{\om}^n\lt(f\rt)}
{\cL_{\om}^n\lt(\ind_\om\rt)}.
\end{align*}
Taking the infimum over $x\in D_{\sg^{n+k}(\om),n}$ and letting $n\to\infty$ gives
\begin{align}
\Lm_{\sg^k(\om)}\lt(\cL_{\om}^kf\rt)
&\geq 
\Lm_{\sg^k(\om)}\lt(\cL_{\om}^k\ind_\om\rt)\cdot\Lm_{\om}(f),
\label{eq1 LMD l3.6}
\end{align}
proving the first claim. Now to see the second claim we note that as \eqref{eq1 LMD l3.6} holds for all $\om\in\Om$ with $k=1$, we must also have
\begin{align}
\Lm_{\sg^{n+1}(\om)}\lt(\cL_{\sg^n(\om)}f\rt)&\geq \Lm_{\sg^{n+1}(\om)}\lt(\cL_{\sg^n(\om)}\ind_{\sg^n(\om)}\rt)\cdot\Lm_{\sg^n(\om)}(f)
=\rho_{\sg^n(\om)}\cdot\Lm_{\sg^n(\om)}(f)
\label{eq1 LMD l3.6 all fibers}
\end{align}
for any $f\in\BV(I)$ and each $n\in\NN$. Proceeding via induction, using \eqref{eq1 LMD l3.6} as the base case, we now suppose that
\begin{align}\label{eq: induction step}
\Lm_{\sg^n(\om)}\lt(\cL_{\om}^nf\rt)\geq \rho_{\om}^{n}\cdot\Lm_{\om}(f)
\end{align}
holds for $n\geq 1$.
Using \eqref{eq1 LMD l3.6 all fibers} and \eqref{eq: induction step}, we see
\begin{align*}
\Lm_{\sg^{n+1}(\om)}\lt(\cL_{\sg^n(\om)}(\cL_{\om}^nf)\rt)
&\geq
\rho_{\sg^n(\om)}\cdot\Lm_{\sg^n(\om)}(\cL_{\om}^nf)
\\
&\geq
\rho_{\om}^{n+1}\cdot\Lm_{\om}(f).
\end{align*}	
Considering $f=\ind_\om$ proves the final claim, and thus we are done. 
\end{proof}

Define the normalized operator $\cL_\om:L^1(\nu_{\om,0})\to L^1(\nu_{\sg(\om),0})$ by 
\begin{align}\label{eq: def of tilde cL norm op}
\~\cL_\om f:=\rho_\om^{-1}\cL_\om f; 
\qquad f\in L^1(\nu_{\om,0}).
\end{align}\index{$\~\cL_\om$}
In light of Lemma~\ref{LMD l3.6}, for each $\om\in\Om$, $n\in\NN$, and $f\in\BV(I)$ we have that 
\begin{align}\label{eq: equivariance prop of Lm}
\Lm_\om(f)\leq \Lm_{\sg^n(\om)}\lt(\~\cL_\om^n f\rt). 
\end{align}
Now, considering the normalized operator, we arrive at the following immediate corollary.
\begin{corollary}\label{cor: normalized LY ineq 1}
For all $\om\in\Om$, all $f\in\BV(I)$, and all $n\in\NN$ we have 
\begin{align*}
\var(\~\cL_{\om}^nf)\leq Q_{\om}^{(n)}\var(f)+K_{\om}^{(n)}\Lm_{\om}(|f|).
\end{align*}
\end{corollary}

\begin{definition}
Since $Q_\om^{(n)}\to 0$ exponentially fast by our assumption \eqref{cond Q1}, we let $N_*\in\NN$\index{$N_*$} be the minimum integer $n\geq 1$ such that
\begin{align}\label{eq: def of N}
-\infty< \int_\Om \log Q_\om^{(n)} dm(\om) <0, 
\end{align}
and we define the number
\begin{align}\label{eq: def of ta}
\ta:=-\frac{1}{N_*}\int_\Om \log Q_\om^{(N_*)} dm(\om).
\end{align}\index{$\ta$}
\end{definition}

\begin{remark}\label{rem: alternate hypoth}
As we are primarily interested in pushing forward in blocks of length $N_*$ we are able to weaken two or our main hypotheses. In particular, we may replace \eqref{cond Q2} and \eqref{cond Q3} with the following: 
\begin{enumerate}
\item[(\Gls*{Q2'})]\myglabel{Q2'}{cond Q2'} We have $\log\xi_\om^{(N_*)}\in L^1(m)$.

\item[(\Gls*{Q3'})]\myglabel{Q3'}{cond Q3'} We have $\log\dl_{\om,N_*}\in L^1(m)$, where $\dl_{\om,n}$ is defined by \eqref{eq: def of dl_om,n}. 
\end{enumerate}
\end{remark}

In light of Corollary~\ref{cor: normalized LY ineq 1} we may now find an appropriate upper bound for the BV norm of the normalized transfer operator. 
\begin{lemma}\label{lem: buzzi LY1}
There exists a measurable function $\om\mapsto L_{\om}\in(0,\infty)$\index{$L_\om$} with $\log L_{\om}\in L^1_m(\Om)$ such that for all $f\in\BV(I)$ and each $1\leq n\leq N_*$ we have 
\begin{align}\label{eq: BV norm bound using L}
\norm{\~\cL_{\om}^n f}_\BV \leq L_{\om}^n\lt(\var(f)+\Lm_{\sg^n(\om)}\lt(\~\cL_{\om}^n f\rt)\rt).
\end{align}
where 
\begin{align*}
L_\om^n=L_\om L_{\sg(\om)}\cdots L_{\sg^{n-1}(\om)}\geq 6^n.
\end{align*}
\end{lemma}

\begin{proof}
Corollary~\ref{cor: normalized LY ineq 1} and \eqref{eq: equivariance prop of Lm} give 
\begin{align*}
\norm{\~\cL_{\om}^n f}_\BV
&=
\var(\~\cL_{\om}^n f)+\norm{\~\cL_{\om}^n f}_\infty
\leq 
2\var(\~\cL_{\om}^n f)+\Lm_{\sg^n(\om)}\lt(\~\cL_{\om}^n f\rt)
\\
&\leq
2\lt(Q_{\om}^{(n)}\var(f)+K_{\om}^{(n)}\Lm_{\om}(|f|)\rt)+\Lm_{\sg^n(\om)}\lt(\~\cL_{\om}^n f\rt)
\\
&\leq 
2Q_{\om}^{(n)}\var(f)+\lt(2K_{\om}^{(n)}+1\rt)\Lm_{\sg^n(\om)}\lt(\~\cL_{\om}^n f\rt).
\end{align*}
Now, set
\begin{align*}
\~L_{\om}^{(n)}:=\max\set{6, 2Q_{\om}^{(n)}, 2K_{\om}^{(n)}+1}.
\end{align*}
Finally, setting 
\begin{align}\label{eq: defn of L_om^n}
L_\om:=\max\set{\~L_{\om}^{(j)}: 1\leq j\leq N_*}
\end{align}
and 
\begin{align*}
L_\om^n:=\prod_{j=0}^{n-1}L_{\sg^j(\om)}
\end{align*}
for all $n\geq 1$ suffices. The $\log$-integrability of $L_\om^n$ follows from Proposition~\ref{prop: log integr of Q and K}.
\end{proof}
We now define the number $\zt>0$ by 
\begin{align}\label{eq: def of zt}
\zt:=\frac{1}{N_*}\int_\Om \log L_\om^{N_*} dm(\om). 
\end{align}\index{$\zt$}

The constants $B_\om^{(n)}$ and $K_\om^{(n)}$ in the Lasota-Yorke inequalities from Lemma~\ref{ly ineq} and  Corollary~\ref{cor: normalized LY ineq 1} grow to infinity with $n$, making them difficult to use. Furthermore, the rate of decay of the $Q_\om^{(n)}$ in Corollary~\ref{cor: normalized LY ineq 1} may depend on $\om$. To remedy these difficulties we prove another, more useful, Lasota-Yorke inequality in the style of Buzzi \cite{buzzi_exponential_1999}.

\begin{proposition}\label{prop: LY ineq 2}
For each $\ep>0$ there exists a measurable, $m$-a.e. finite function $C_\ep(\om)>0$\index{$C_\ep(\om)$} such that for $m$-a.e. $\om\in\Om$, each $f\in\BV(I)$, and all $n\in\NN$ we have 
\begin{align*}
\var(\~\cL_{\sg^{-n}(\om)}^nf)\leq C_\ep(\om)e^{-(\ta-\ep)n}\var(f)+C_\ep(\om)\Lm_\om(\~\cL_{\sg^{-n}(\om)}^nf).
\end{align*}
\end{proposition} 
As the proof of Proposition~\ref{prop: LY ineq 2} follows similarly to that of Proposition~4.9 of \cite{AFGTV20}, using \eqref{eq: equivariance prop of Lm} to obtain $\Lm_\om(\~\cL_{\sg^{-n}(\om)}^nf)$ rather than $\Lm_{\sg^{-n}(\om)}(f)$, we leave it to the dedicated reader.

\section{Cone invariance on good fibers}\label{Sec: good fibers}\label{sec:good}

In this section we follow Buzzi's approach \cite{buzzi_exponential_1999}, and describe the good behavior across a large measure set of fibers. In particular, we will show that, for sufficiently many iterates $R_*$, the normalized transfer operator $\~\cL_\om^{R_*}$ uniformly contracts the cone $\sC_{\om,a}$ on ``good'' fibers $\om$ for cone parameters $a>0$ sufficiently large. 
Recall that the numbers $\ta$ and $\zt$ are given by 
\begin{align*}
\ta:=-\frac{1}{N_*}\int_\Om \log Q_\om^{(N_*)} dm(\om)>0
\quad \text{ and } \quad
\zt:=\frac{1}{N_*}\int_\Om \log L_\om^{N_*} dm(\om)>0. 
\end{align*}
Note that Lemma~\ref{lem: buzzi LY1} and the ergodic theorem imply that 
\begin{align}\label{eq: zt geq log 6}
\log 6\leq \zt = \lim_{n\to\infty}\frac{1}{nN_*}\sum_{k=0}^{n-1}\log L_{\sg^{kN_*}(\om)}^{N_*}.
\end{align}
The following definition is adapted from \cite[Definition~2.4]{buzzi_exponential_1999}.
\begin{definition}
We will say that $\omega$ is a \textit{good fiber} with respect to the numbers $\ep$, $a$, $B_*$, and $R_a=q_aN_*$ if the following hold:\index{good fiber}\index{$R_a$}
\begin{flalign} 
& B_*q_ae^{-\frac{\ta}{2}R_a}\leq \frac{1}{3},
\tag{G1}\label{G1} 
&\\
& \frac{1}{R_a}\sum_{k=0}^{\sfrac{R_a}{N_*}-1}\log L_{\sg^{kN_*}(\om)}^{N_*} 
\in[\zt-\ep,\zt+\ep].
\tag{G2}\label{G2} 
\end{flalign}
\end{definition}
Now, we denote
\begin{align}\label{eq: def of ep_0}
\ep_0:=\min\set{1, \frac{\ta}{2}}.
\end{align}
The following lemma describes the prevalence of the good fibers as well as how to find them. 
\begin{lemma}\label{lem: constr of Om_G}
Given $\ep<\ep_0$ and $a>0$, there exist parameters $B_*$\index{$B_*$} and  $R_a$ (both of which depend on $\ep$) such that there is a set $\Om_G\sub \Om$\index{$\Om_G$} of good fibers $\om$ with $m(\Om_G)\geq 1-\sfrac{\ep}{4}$. 
\end{lemma}
\begin{proof}
We begin by letting 
\begin{align}\label{eq: def of Om_1}
\Om_1=\Om_1(B_*):=\set{\om\in\Om: C_\ep(\om)\leq B_*},
\end{align}\index{$B_*$}
where $C_\ep(\om)>0$ is the $m$-a.e. finite measurable constant coming from Proposition~\ref{prop: LY ineq 2}.
Choose $B_*$ sufficiently large such that $m(\Om_1)\geq 1-\sfrac{\ep}{8}$. 
Noting that $\ep<\ta/2$ by \eqref{eq: def of ep_0}, we set $R_0=q_0N_*$ and choose $q_0$ sufficiently large such that 
\begin{align*}
B_*q_0e^{-(\ta-\ep)R_0}\leq B_*q_0e^{-\frac{\ta}{2}R_0}\leq\frac{1}{3}.
\end{align*}
Now let $q_1\geq q_0$ and define the set 
\begin{align*}
\Om_2=\Om_2(q_1):=\set{\om\in\Om: \eqref{G2} \text{ holds for the value } R_1=q_1N_* }.
\end{align*}
Now choose $q_a\geq q_1$ such that $m(\Om_2(q_a))\geq 1-\sfrac{\ep}{8}$. Set $R_a:=q_aN_*$. Set 
\begin{align}\label{eq: def of good set}
\Om_G:=\Om_2\cap \sg^{-R_a}(\Om_1).
\end{align}
Then $\Om_G$ is the set of all $\om\in\Om$ which are good with respect to the numbers $B_*$ and $R_a$, and 
$m(\Om_G)\geq 1-\sfrac{\ep}{4}$.  
\end{proof}

In what follows, given a value $B_*$, we will consider cone parameters 
\begin{align}\label{eq cone param a}
a\geq a_0:=6B_* 
\end{align}\index{$a_0$}
and we set 
\begin{align}\label{eq: def of R*}
q_*=q_{a_0} 
\quad\text{ and }\quad 
R_*:=R_{a_0}=q_*N_*.
\end{align}\index{$R_*$}\index{$q_*$}

Note that \eqref{G1} together with Proposition~\ref{prop: LY ineq 2} implies that, for $\ep<\ep_0$ and $\om\in\Om_G$, we have 
\begin{align}
\var(\~\cL_\om^{R_*}f)
&\leq 
B_*e^{-(\ta-\ep)R_*}\var(f)+B_*\Lm_{\sg^{R_*}(\om)}(\~\cL_\om^{R_*} f)
\nonumber\\
&\leq 
B_*q_*e^{-\frac{\ta}{2}R_*}\var(f)+B_*\Lm_{\sg^{R_*}(\om)}(\~\cL_\om^{R_*} f)
\nonumber\\
&\leq 
\frac{1}{3}\var(f)+B_*\Lm_{\sg^{R_*}(\om)}(\~\cL_\om^{R_*} f).
\label{G1 cons}
\end{align}

The next lemma shows that the normalized operator is a contraction on the fiber cones $\sC_{\om,a}$ and that the image has finite diameter.
\begin{lemma}\label{lem: cone invariance for good om}
If $\om$ is good with respect to the numbers $\ep$, $a_0$, $B_*$, and $R_*$, then for each $a\geq a_0$ we have 
\begin{align*}
\~\cL_{\om}^{R_*}(\sC_{\om,a})\sub \sC_{\sg^{R_*}(\om), \sfrac{a}{2}}\sub \sC_{\sg^{R_*}(\om),a}.
\end{align*}
\end{lemma}
\begin{proof}
For $\om$ good and $f\in\sC_{\om,a}$, \eqref{G1 cons} and \eqref{eq cone param a} give
\begin{align*}
\var(\~\cL_{\om}^{R_*} f)
&\leq 
\frac{1}{3}\var(f) + B_*\Lm_{\sg^{R_*}(\om)}(\~\cL_\om^{R_*}f) 
\\&
\leq 
\frac{a}{3}\Lm_{\om}(f) + \frac{a}{6}\Lm_{\sg^{R_*}(\om)}(\~\cL_\om^{R_*}f) 
\\
&
\leq\frac{a}{2}\Lm_{\sg^{R_*}(\om)}(\~\cL_\om^{R_*} f).
\end{align*}
Hence we have 
\begin{align*}
\~\cL_{\om}^{R_*}(\sC_{\om,a})\sub \sC_{\sg^{R_*}(\om), \sfrac{a}{2}}\sub \sC_{\sg^{R_*}(\om),a}
\end{align*}
as desired.
\end{proof}

\section{Density estimates and cone invariance on bad fibers}\label{sec:bad}
In this section we recall the notion of ``bad'' fibers from \cite{AFGTV20, buzzi_exponential_1999}. We show that for fibers in the small measure set, $$\Om_B:=\Om\bs\Om_G,$$\index{$\Om_B$}
the cone $\sC_{\om,a}$ of positive functions is invariant after sufficiently many iterations for sufficiently large parameters $a>0$. We accomplish this by introducing the concept of bad blocks (coating intervals), which we then show make up a relatively small portion of an orbit. 
As the content of this section is adapted from the closed dynamical setting of Section 7 of \cite{AFGTV20}, we do not provide proofs.

Recall that $R_*$ is given by \eqref{eq: def of R*}. Following Section~7 of \cite{AFGTV20}, and using the same justifications therein, we define the measurable function $y_*:\Om\to\NN$ so that 
\begin{align}\label{eq: def of y_*}
0\leq y_*(\om)<R_*
\end{align}
is the smallest integer such that for either choice of sign $+$ or $-$ we have 
\begin{flalign} 
& \lim_{n\to\infty} \frac{1}{n}\#\set{0\leq k< n: \sg^{\pm kR_*+y_*(\om)}(\om)\in\Om_G} >1-\ep,
\label{def y*1} 
&\\
& \lim_{n\to\infty} \frac{1}{n}\#\set{0\leq k< n: C_\ep\lt(\sg^{\pm kR_*+y_*(\om)}(\om)\rt)\leq B_*} >1-\ep.
\label{def y*2} 
\end{flalign}
Clearly, $y_*:\Om\to\NN$\index{$y_*:\Om\to\NN$} is a measurable function such that 
\begin{flalign} 
& y_*(\sg^{y_*(\om)}(\om))=0,
\label{prop j*1} 
&\\
& y_*(\sg^{R_*}(\om))=y_*(\om).
\label{prop j*2} 
\end{flalign}
In particular, \eqref{prop j*1} and \eqref{prop j*2} together imply that 
\begin{align}\label{prop j*3}
y_*(\sg^{y_*(\om)+kR_*}(\om))=0
\end{align}
for all $k\in\NN$.
Let 
\begin{align}\label{def Gm}
\Gm(\om):=\prod_{k=0}^{q_*-1} L_{\sg^{kN_*}(\om)}^{N_*},
\end{align}\index{$\Gm(\om)$}
where $q_*$ is given by \eqref{eq: def of R*},
and for each $\om\in\Om$, given $\ep>0$, we define the \textit{coating length} $\ell(\om)=\ell_\ep(\om)$ as follows:
\begin{itemize}
\item if $\om\in\Om_G$, then set $\ell(\om):=1$, 
\item if $\om\in\Om_B$, then 
\begin{align}\label{def: coating length}
\ell(\om):=\min\set{n\in\NN: \frac{1}{n}\sum_{0\leq k< n} \lt(\ind_{\Om_B}\log \Gm\rt)(\sg^{kR_*}(\om)) \leq \zt R_*\sqrt{\ep}},
\end{align}\index{$\ell(\om)$}
where $\zt$ is as in \eqref{eq: def of zt}.
If the minimum is not attained we set $\ell(\om)=\infty$.
\end{itemize}
Since $L_\om^{N_*}\geq 6^{N_*}$ by Lemma~\ref{lem: buzzi LY1}, we must have that 
\begin{align}\label{eq: Bm geq 6^R}
\Gm(\om)\geq 6^{R_*}
\end{align}
for all $\om\in\Om$. 
It follows from Lemma~\ref{lem: buzzi LY1} that for all $\om\in\Om$ we have 
\begin{align}\label{eq: LY ineq for bad fibers}
\var(\~\cL_\om^{R_*}f)
&\leq \Gm(\om)(\var(f)+\Lm_{\sg^{R_*}(\om)}(f)).  
\end{align}
Furthermore, if $\om\in\Om_G$ it follows from \eqref{G2} that 
\begin{align}\label{eq: up and low bds for Gamma on good om}
R_*(\zt-\ep)\leq \log\Gm(\om)\leq R_*(\zt+\ep).
\end{align}
The following proposition collects together some of the key properties of the coating length $\ell(\om)$. 
\begin{proposition}\label{prop: ell(om) props}
For all $\ep>0$ sufficiently small the number $\ell(\om)$ satisfies the following.
\begin{flalign*}
& \text{For }m\text{-a.e. } \om\in\Om \text{ such that } y_*(\om)=0 \text{ we have } \ell(\om)<\infty,
\tag{i}\label{prop: ell(om) props item i}
&\\
& \text{If }\om\in\Om_B \text{ then }\ell(\om)\geq 2.
\tag{ii} \label{prop: ell(om) props item ii}
\end{flalign*}
\end{proposition}
\begin{remark}\label{rem: y*=0 implies ell finite}
Given $\om_0\in\Om$, for each $j\geq 0$ let $\om_{j+1}=\sg^{\ell(\om_j)R_*}(\om_j)$. 
As a consequence of Proposition~\ref{prop: ell(om) props} \eqref{prop: ell(om) props item i} and \eqref{prop j*2}, we see that for $m$-a.e. $\om_0\in\Om$ with $y_*(\om_0)=0$, we must have that $\ell(\om_j)<\infty$ for all $j\geq 0$.
\end{remark}
\begin{definition}\label{def: coating intervals}	
We will call a (finite) sequence $\om, \sg(\om), \dots, \sg^{\ell(\om)R_*-1}(\om)$ of $\ell(\om)R_*$ fibers a \textit{good block}\index{good block} (originating at $\om$) if $\om\in\Om_G$ (which implies that $\ell(\om)=1$). If, on the other hand, $\om\in\Om_B$ we call such a sequence a \textit{bad block}\index{bad block}, or coating interval, originating at $\om$.
\end{definition}
For $\ep>0$ sufficiently small 
we have that $\sfrac{\zt \sqrt{\ep}}{\log 6}<1$, and so 
we let $\gm<1$ such that 
\begin{align}\label{eq: def of gm constant}
\frac{\zt \sqrt{\ep}}{\log 6}<\gm<1.
\end{align}

We now wish to show that the normalized operator $\~\cL_\om$ is weakly contracting (i.e. non-expanding) on the fiber cones $\sC_{\om,a}$ for sufficiently large values of $a>a_0$. We obtain this cone invariance on blocks of length $\ell(\om)R_*$, however in order to obtain cone contraction with a finite diameter image we will have to travel along several such blocks. For this reason we introduce the following notation. 

Given $\om\in\Om$ with $y_*(\om)=0$ for each $k\geq 1$ we define the length
\begin{align*}
\Sg_\om^{(k)}:=\sum_{j=0}^{k-1} \ell(\om_j)R_*
\end{align*}\index{$\Sg_\om^{(k)}$}
where $\om_0:=\om$ and for each $j\geq 1$ we set $\om_j:=\sg^{\Sg_\om^{(j-1)}}(\om)$.
This construction is justified as we recall from Proposition~\ref{prop: ell(om) props} that for $m$-a.e. $\om\in\Om$ with $y_*(\om)=0$ we have that $\ell(\om)<\infty$. 
The next lemma was adapted from Lemma~7.5 of \cite{AFGTV20}.
\begin{lemma}\label{lem: cone cont for coating blocks}
For $\ep>0$ sufficiently small, each $N\in\NN$, and $m$-a.e. $\om\in\Om$ with $y_*(\om)=0$ we have that 
\begin{align}\label{eq: important coating block ineq}	
\var\lt(\~\cL_\om^{\Sg_\om^{(N)}}f\rt)
&\leq
\lt(\frac{1}{3}\rt)^{\Sg_\om^{(N)}/R_*}\var(f)
+
\frac{a_*}{6}\Lm_{\sg^{\Sg_\om^{(N)}}(\om)}(\~\cL_\om^{\Sg_\om^{(N)}} f).
\end{align}
Moreover, we have that 
\begin{align}\label{eq: cone inv for large bad blocks}
\~\cL_\om^{\Sg_\om^{(N)}}(\sC_{\om,a_*})\sub \sC_{\sg^{\Sg_\om^{(N)}}(\om), \sfrac{a_*}{2}},
\end{align}
where 
\begin{align}\label{eq: def of a_*}
a_*=a_*(\ep):=2a_0e^{\zt R_*\sqrt{\ep}}=12B_*e^{\zt R_*\sqrt{\ep}}.
\end{align}
\end{lemma}

\begin{proof}
Throughout the proof we will denote $\ell_i=\ell(\om_i)$ and $L_i=\sum_{k=0}^{i-1}\ell_k$ for each $0\leq i<N$. Then $\Sg_\om^{(N)}=L_NR_*$.
Using \eqref{G1 cons} on good fibers and \eqref{eq: LY ineq for bad fibers} on bad fibers, for any $p\geq 1$ and $f\in\sC_{\om,+}$ we have 
\begin{align}\label{eq: var coating length}
\var(\~\cL_\om^{pR_*} f)
&\leq
\lt(\prod_{j=0}^{p-1} \Phi_{\sg^{jR_*}(\om)}^{(R_*)}\rt)\var(f)
+ 
\sum_{j=0}^{p-1}\lt(D_{\sg^{jR_*}(\om)}^{(R_*)}\cdot \prod_{k=j+1}^{p-1}\Phi_{\sg^{kR_*}(\om)}^{(R_*)}\rt)\Lm_{\sg^{pR_*}(\om)}(\~\cL_\om^{pR_*} f),
\end{align}
where 
\begin{equation}\label{eq: def of Phi_tau^R}
\Phi_\tau^{(R_*)}=
\begin{cases}
B_* e^{-(\ta-\ep)R_*} &\text{for } \tau\in\Om_G\\
\Gm(\tau) &\text{for } \tau\in\Om_B
\end{cases}
\end{equation}
and 
\begin{equation}\label{eq: def of D_tau^R}
D_\tau^{(R_*)}=
\begin{cases}
B_* &\text{for } \tau\in\Om_G\\
\Gm(\tau) &\text{for } \tau\in\Om_B.
\end{cases}
\end{equation}
For any $0\leq i<N$ and  $0\leq j<\ell_i$ we can write 
\begin{align*}
\sum_{0\leq k< \ell_i}\lt(\ind_{\Om_B}\log \Gm\rt)(\sg^{kR_*}(\om_i))
=
\sum_{0\leq k< j }\lt(\ind_{\Om_B}\log \Gm\rt)(\sg^{kR_*}(\om_i))
+
\sum_{j\leq k< \ell_i }\lt(\ind_{\Om_B}\log \Gm\rt)(\sg^{kR_*}(\om_i)).
\end{align*}
The definition of $\ell(\om_i)$, \eqref{def: coating length}, then implies that 
\begin{align*}
\frac{1}{j}\sum_{0\leq k< j }\lt(\ind_{\Om_B}\log \Gm\rt)(\sg^{kR_*}(\om_i))
> 
\zt R_*\sqrt{\ep},
\end{align*}
and consequently that 
\begin{align}\label{eq: avg sum log Gm for bad om}
\frac{1}{\ell_i-j}\sum_{j\leq k< \ell_i}\lt(\ind_{\Om_B}\log \Gm\rt)(\sg^{kR_*}(\om_i))
\leq
\zt R_*\sqrt{\ep}. 
\end{align}
Now, using \eqref{eq: Bm geq 6^R}, \eqref{eq: avg sum log Gm for bad om}, and \eqref{eq: def of gm constant} we see that for $\ep$ sufficiently small, the proportion of bad blocks is given by 
\begin{align}
\frac{1}{\ell_i-j}\#\set{j\leq k<\ell_i: \sg^{kR_*}(\om_i)\in\Om_B}
&=
\frac{1}{\ell_i-j}\sum_{j\leq k<\ell_i} \lt(\ind_{\Om_B}\rt)(\sg^{kR_*}(\om_i))
\nonumber\\
&\leq
\frac{1}{(\ell_i-j)R_*\log 6}\sum_{j\leq k<\ell_i} \lt(\ind_{\Om_B}\log \Gm\rt)(\sg^{kR_*}(\om_i))
\leq \gm.
\label{eq: proportion bad blocks}
\end{align}
In view of \eqref{eq: def of Phi_tau^R}, using \eqref{eq: avg sum log Gm for bad om}, \eqref{eq: proportion bad blocks}, for any $0\leq i<N$ and $0\leq j<\ell_i$ we have 
\begin{align}
\prod_{k=j}^{\ell_i-1} \Phi_{\sg^{kR_*}(\om_i)}^{(R_*)}
&=
\prod_{\substack{j\leq k<\ell_i \\ \sg^{kR_*}(\om_i)\in\Om_G}}
B_*e^{-(\ta-\ep)R_*}
\cdot 
\prod_{\substack{j\leq k<\ell_i \\ \sg^{kR_*}(\om_i)\in\Om_B}}
\Gm(\sg^{kR_*}(\om_i))
\nonumber\\
&\leq 
\lt(B_*e^{-(\ta-\ep)R_*}\rt)^{(1-\gm)(\ell_i-j)}\cdot \exp\lt(\lt(\zt R_*\sqrt{\ep}\rt)(\ell_i-j)\rt)
\nonumber\\
&=
\lt(B_*^{1-\gm}\exp\lt(\lt(\zt \sqrt{\ep}-(\ta-\ep)(1-\gm)\rt)R_*\rt)\rt)^{\ell_i-j}
\nonumber\\
&< 
\lt(B_*\exp\lt(\lt(\zt \sqrt{\ep}-(\ta-\ep)(1-\gm)\rt)R_*\rt)\rt)^{\ell_i-j}.
\label{eq: est for coating lengths 1}
\end{align}
Now for any $0\leq j<L_N$ there must exist some $0\leq i_0<N$ and some $0\leq j_0<\ell_{i_0+1}$
such that $L_{i_0-1}+j_0= j<L_{i_0}$. Thus, using \eqref{eq: est for coating lengths 1} we can write
\begin{align}
\prod_{k=j}^{L_N-1} \Phi_{\sg^{kR_*}(\om)}^{(R_*)}
&=
\prod_{k=j_0}^{\ell_{i_0+1}-1} \Phi_{\sg^{kR_*}(\om_{i_0})}^{(R_*)}
\cdot
\prod_{i=i_0+1}^{N-1} \prod_{k=0}^{\ell_i-1} \Phi_{\sg^{kR_*}(\om_i)}^{(R_*)}
\nonumber\\
&< 
\lt(B_*\exp\lt(\lt(\zt \sqrt{\ep}-(\ta-\ep)(1-\gm)\rt)R_*\rt)\rt)^{\ell_{i_0+1}-j_0}
\cdot
\nonumber\\
&\qquad\cdot
\prod_{i=i_0+1}^{N-1} \lt(B_*\exp\lt(\lt(\zt \sqrt{\ep}-(\ta-\ep)(1-\gm)\rt)R_*\rt)\rt)^{\ell_i}
\nonumber\\
&=
\lt(B_*\exp\lt(\lt(\zt \sqrt{\ep}-(\ta-\ep)(1-\gm)\rt)R_*\rt)\rt)^{L_N-j}.
\label{eq: est for coating lengths 1 full}
\end{align}

Now, since $B_*,\Gm(\om)\geq 1$ for all $\om\in\Om$, using \eqref{eq: def of D_tau^R} and \eqref{eq: avg sum log Gm for bad om}, we have that for $0\leq i<N$ and $0\leq j<\ell_i$ 
\begin{align}\label{eq: est for coating lengths 2}
D_{\sg^{jR_*}(\om_i)}^{(R_*)}
\leq
B_*\Gm(\sg^{jR_*}(\om_i))
\leq 
B_*\cdot\prod_{\substack{j\leq k< \ell_i \\ \sg^{kR_*}(\om_i)\in\Om_B}} \Gm(\sg^{kR_*}(\om_i))
\leq
B_*\lt(e^{\zt R_*\sqrt{\ep}}\rt)^{(\ell_i-j)}.		
\end{align}
Similarly to the reasoning used to obtain \eqref{eq: est for coating lengths 1 full}, for any $0\leq j<L_N$ we see that we can improve \eqref{eq: est for coating lengths 2} so that we have 
\begin{align}\label{eq: est for coating lengths 2 full}
D_{\sg^{jR_*}(\om_i)}^{(R_*)}
\leq
B_*\lt(e^{\zt R_*\sqrt{\ep}}\rt)^{L_N-j}.		
\end{align}
Thus, inserting \eqref{eq: est for coating lengths 1 full} and \eqref{eq: est for coating lengths 2 full} into \eqref{eq: var coating length} (with $p=L_N$) we see that
\begin{align*}
&\var\lt(\~\cL_\om^{\Sg_\om^{(N)}} f\rt)
\leq
\lt(\prod_{j=0}^{L_N-1} \Phi_{\sg^{jR_*}(\om)}^{(R_*)}\rt)\var(f)
+ 
\sum_{j=0}^{L_N-1}\lt(D_{\sg^{jR_*}(\om)}^{(R_*)}\cdot \prod_{k=j+1}^{L_N-1}\Phi_{\sg^{kR_*}(\om)}^{(R_*)}\rt)\Lm_{\sg^{\Sg_\om^{(N)}}(\om)}\lt(\~\cL_\om^{\Sg_\om^{(N)}} f\rt),
\\
&\leq 
\lt(B_*\exp\lt(\lt(\zt \sqrt{\ep}-(\ta-\ep)(1-\gm)\rt)R_*\rt)\rt)^{L_N}\var(f)
\\
&\qquad
+
\Lm_{\sg^{\Sg_\om^{(N)}}(\om)}\lt(\~\cL_\om^{\Sg_\om^{(N)}} f\rt)\sum_{j=0}^{L_N-1}
B_*\lt(e^{\zt R_*\sqrt{\ep}}\rt)^{(L_N-j)}
\cdot 
\lt(B_*\exp\lt(\lt(\zt \sqrt{\ep}-(\ta-\ep)(1-\gm)\rt)R_*\rt)\rt)^{L_N-j-1}
\\
&
=
\lt(B_*\exp\lt(\lt(\zt \sqrt{\ep}-(\ta-\ep)(1-\gm)\rt)R_*\rt)\rt)^{L_N}\var(f)
\\
&\qquad
+
B_*e^{\zt R_*\sqrt{\ep}}\cdot \Lm_{\sg^{\Sg_\om^{(N)}}(\om)}\lt(\~\cL_\om^{\Sg_\om^{(N)}} f\rt)\sum_{j=0}^{L_N-1}		
\lt(B_*\exp\lt(\lt(2\zt \sqrt{\ep}-(\ta-\ep)(1-\gm)\rt)R_*\rt)\rt)^{L_N-j-1}.
\end{align*}
Therefore, taking $\ep>0$ sufficiently small\footnote{Any $\ep<\min\set{\lt(\frac{\log 55}{4\zt}\rt)^2, \lt(\frac{\ta}{8\zt}\rt)^2}$ such that $\frac{\sqrt{\ep}\zt}{2\log 55}\leq \gm$, which implies $-\frac{\ta}{2}> 2\sqrt{\ep}\zt-(\ta-\ep)(1-\gm) > \sqrt{\ep}\zt-(\ta-\ep)(1-\gm)$, will suffice; see Observation~7.4 of \cite{AFGTV20} for details.} 
in conjunction with \eqref{G1}, we have that
\begin{align*}
\var\lt(\~\cL_\om^{\Sg_\om^{(N)}}f\rt)
&\leq
\lt(B_*e^{-\frac{\ta}{2}R_*}\rt)^{L_N}\var(f)
+
B_*e^{\zt R_*\sqrt{\ep}}\cdot \Lm_{\sg^{\Sg_\om^{(N)}}(\om)}\lt(\~\cL_\om^{\Sg_\om^{(N)}} f\rt)\sum_{j=0}^{L_N-1}		
\lt(B_*e^{-\frac{\ta}{2}R_*}\rt)^{L_N-j-1} 
\\
&\leq 
\lt(\frac{1}{3}\rt)^{L_N}\var(f)
+
B_*e^{\zt R_*\sqrt{\ep}}\cdot \Lm_{\sg^{\Sg_\om^{(N)}}(\om)}\lt(\~\cL_\om^{\Sg_\om^{(N)}} f\rt)\sum_{j=0}^{L_N-1}		
\lt(\frac{1}{3}\rt)^{L_N-j-1},
\end{align*} 
and so we must have that 
\begin{align*}
\var\lt(\~\cL_\om^{\Sg_\om^{(N)}} f\rt)
&\leq 
\lt(\frac{1}{3}\rt)^{L_N}\var(f)+2B_*e^{\zt R_*\sqrt{\ep}}\Lm_{\sg^{\Sg_\om^{(N)}}(\om)}\lt(\~\cL_\om^{\Sg_\om^{(N)}} f\rt)
\\
&=
\lt(\frac{1}{3}\rt)^{L_N}\var(f)+\frac{a_*}{6}\Lm_{\sg^{\Sg_\om^{(N)}}(\om)}\lt(\~\cL_\om^{\Sg_\om^{(N)}} f\rt),
\end{align*}
which proves the first claim.
Thus, for any $f\in\sC_{\om,a_*}$ we have that 
\begin{align*}
\var(\~\cL_\om^{\Sg_\om^{(N)}} f)
&\leq 
\lt(\frac{1}{3}\rt)^{L_N}\Lm_{\om}(f)+\frac{a_*}{6}\Lm_{\sg^{\Sg_\om^{(N)}}(\om)}\lt(\~\cL_\om^{\Sg_\om^{(N)}} f\rt)
\\
&\leq
\frac{a_*}{3}\Lm_{\sg^{\Sg_\om^{(N)}}(\om)}\lt(\~\cL_\om^{\Sg_\om^{(N)}} f\rt)+\frac{a_*}{6}\Lm_{\sg^{\Sg_\om^{(N)}}(\om)}\lt(\~\cL_\om^{\Sg_\om^{(N)}} f\rt), 
\end{align*}
where we have used the fact that $a_*> 1$, and consequently we have 
\begin{align*}
\~\cL_\om^{\Sg_\om^{(N)}}(\sC_{\om,a_*})\sub \sC_{\sg^{\Sg_\om^{(N)}}(\om),\sfrac{a_*}{2}}\sub \sC_{\sg^{\Sg_\om^{(N)}}(\om),a_*}
\end{align*}
as desired.
\end{proof}

The next lemma shows that the total length of the bad blocks take up only a small proportion of an orbit, however before stating the result we establish the following notation. For each  $n\in\NN$ we let $K_n\geq 0$ be the integer such that 
\begin{align}\label{eq: n KR+zt decomp}
n=K_nR_*+h(n)
\end{align}
where $0\leq h(n)<R_*$ is a remainder term. Given $\om_0\in\Om$, let 
\begin{align}\label{eq: om_j notation}
\om_j=\sg^{\ell(\om_{j-1})R_*}(\om_{j-1})
\end{align} 
for each $j\geq 1$. Then for each $n\in\NN$ we can break the $n$-length $\sg$-orbit of $\om_0$ in $\Om$ into $k_{\om_0}(n)+1$ blocks of length $\ell(\om_j)R_*$ (for $0\leq j\leq k_{\om_0}(n)$) plus some remaining block of length $r_{\om_0}(n)R_*$ where $0\leq r_{\om_0}(n)<\ell(\om_{k_{\om_0}(n)+1})$ plus a remainder segment of length $h(n)$, i.e. we can write 
\begin{align}\label{eq: n length orbit in om_j}
n=\sum_{0\leq j\leq k_{\om_0}(n)}\ell(\om_j)R_* +r_{\om_0}(n)R_* + h(n); 
\end{align}
see Figure~\ref{figure: om_j}. We also note that \eqref{eq: n KR+zt decomp} and \eqref{eq: n length orbit in om_j} imply that 
\begin{align}\label{eq: Kn equality}
K_n=\sum_{0\leq j\leq k_{\om_0}(n)}\ell(\om_j) +r_{\om_0}(n).
\end{align}
\begin{figure}[h]
\centering	
\begin{tikzpicture}[y=1cm, x=1.5cm, thick, font=\footnotesize]    
\draw[line width=1.2pt,  >=latex'](0,0) -- coordinate (x axis) (10,0);       

\draw (0,-4pt)   -- (0,4pt) {};
\node at (0, -.5) {$\om_0$};

\node at (.5, .5) {$\ell(\om_0)R_*$};

\draw (1,-4pt)   -- (1,4pt) {};
\node at (1,-.5) {$\om_1$};

\node at (1.625,.5) {$\ell(\om_1)R_*$};

\draw (2.25,-4pt)  -- (2.25,4pt)  {};
\node at (2.25,-.5) {$\om_2$};

\node at (2.625,.3) {$\cdots$};
\node at (2.625,-.3) {$\cdots$};

\draw (3,-4pt)  -- (3,4pt)  {};
\node at (3,-.5) {$\om_{k-1}$};

\node at (3.825,.5) {$\ell(\om_{k-1})R_*$};

\draw (4.75,-4pt)  -- (4.75,4pt)  {};
\node at (4.75,-.5) {$\om_{k_{\om_0}(n)}$};

\node at (5.75,.5) {$\ell(\om_{k_{\om_0}(n)})R_*$};	

\draw (6.75,-4pt)  -- (6.75,4pt)  {};
\node at (6.75,-.5) {$\om_{k+1}$};

\node at (7.5,.5) {$r_{\om_0}(n)R_*$};

\draw (8.25,-4pt)  -- (8.25,4pt)  {};
\node at (8.25,-.5) {$\sg^{K_nR_*}(\om_0)$};

\node at (8.75,.5) {$h(n)$};

\draw (9.25,-4pt)  -- (9.25,4pt)  {};
\node at (9.25,-.5) {$\sg^{n}(\om_0)$};

\draw (10,-4pt)  -- (10,4pt)  {};
\node at (10,-.5) {$\om_{k+2}$};

\draw[decorate,decoration={brace,amplitude=8pt,mirror}] 
(6.75,-1)  -- (10,-1) ; 
\node at (8.125,-1.6){$\ell(\om_{k+1})R_*$};

\draw[decorate,decoration={brace,amplitude=8pt}] 
(0,1)  -- (9.25,1) ; 
\node at (4.625,1.6){$n$};

\end{tikzpicture}
\caption{The decomposition of $n=\sum_{0\leq j\leq k_{\om_0}(n)}\ell(\om_j)R_*+r_{\om_0}(n)R_* +h(n)$ and the fibers $\om_j$.}
\label{figure: om_j}
\end{figure}
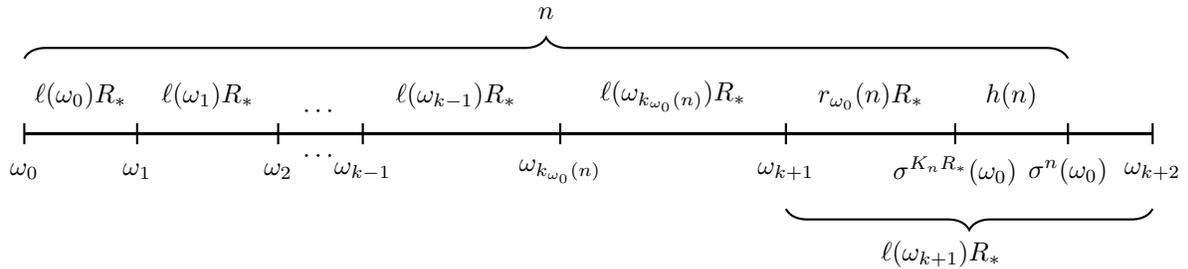
The proof of the following lemma is nearly identical to that of Lemma~7.6 of \cite{AFGTV20}, and therefore it shall be omitted. 
\begin{lemma}\label{lem: proportion of bad blocks}
There exists a measurable function $N_0:\Om\to\NN$ such that for all $n\geq N_0(\om_0)$ and for $m$-a.e. $\om_0\in\Om$ with $y_*(\om_0)=0$  we have 
\begin{align*}
&E_{\om_0}(n):=
\sum_{\substack{0\leq j \leq k_{\om_0}(n) \\ \om_j\in\Om_B}}\ell(\om_j) +r_{\om_0}(n)
<
Y\cdot \ep K_n
\leq 
\frac{Y}{R_*}\ep n
\end{align*}
where 
\begin{align}\label{eq: def of Y}
Y=Y_\ep:=\frac{2(2+\zt) R_*}{\zt R_*\sqrt{\ep}},
\end{align}
and where $K_n$ is as in \eqref{eq: n KR+zt decomp}, 
$\om_j$ is as in \eqref{eq: om_j notation}, and $k_{\om_0}(n)$ and $r_{\om_0}(n)$ are as in \eqref{eq: n length orbit in om_j}.
\end{lemma}	
To end this section we note that 
\begin{align}\label{eq: def ep 4}
\ep \cdot Y_\ep\to 0
\end{align}
as $\ep\to 0$. For the remainder of the document we will assume that $\ep>0$ is always taken sufficiently small such that the results of Section~\ref{sec:bad} apply.

\section{Further properties of $\Lm_\om$}\label{sec: props of Lm}
In this section we prove some additional properties of the functional $\Lm_\om$ that will be necessary in Section~\ref{sec: fin diam} to obtain cone contraction with finite diameter. In particular, in the main result of this section, which is a version of Lemma~3.11 of \cite{LMD} and dates back to \cite[Lemma 3.2]{liverani_decay_1995-1}, we show that for a function $f\in\sC_{\om,a_*}$ there exists a partition element on which the function $f$ takes values at least as large as $\Lm_\om(f)/4$.

Now we prove the following upper and lower bounds for $\Lm_{\sg^{\Sg_\om^{(k)}}(\om)}\lt(\~\cL_{\om}^{\Sg_\om^{(k)}} f\rt)$.
\begin{lemma}\label{lem: LMD 3.8 temporary}
For $m$-a.e. $\om\in\Om$ such that $y_*(\om)=0$, and each $k\in\NN$ we have that 
\begin{align*}
\Lm_{\sg^{\Sg_\om^{(k)}}(\om)}\lt(\~\cL_{\om}^{\Sg_\om^{(k)}}\ind_\om\rt)\Lm_{\om}(f)
\leq
\Lm_{\sg^{\Sg_\om^{(k)}}(\om)}\lt(\~\cL_{\om}^{\Sg_\om^{(k)}} f\rt)
\leq
a_*\Lm_{\sg^{\Sg_\om^{(k)}}(\om)}\lt(\~\cL_{\om}^{\Sg_\om^{(k)}}\ind_\om\rt)\Lm_{\om}(f).
\end{align*}
\end{lemma}
\begin{proof}
From Lemma~\ref{LMD l3.6} we already see that the first inequality holds.
Now, fix $\om\in\Om$ (with $y_*(\om)=0$) and $k\in\NN$. 
To see the other inequality we first let $n, N\in\NN$ and $x\in D_{\sg^N(\om),N+n}$, then
\begin{align*}
\frac{\~\cL_{\om}^{N+n}(f)(x)}
{\~\cL_{\sg^{N}(\om)}^{n}(\ind_{\sg^{N}(\om)})(x)}
&=
\frac{\~\cL_{\om}^{N+n}(f)(x)}
{\~\cL_{\om}^{N+n}(\ind_\om)(x)}
\cdot\frac{\~\cL_{\om}^{N+n}(\ind_\om)(x)}{\~\cL_{\sg^{N}(\om)}^{n}(\ind_{\sg^{N}(\om)})(x)}
\\
&=
\frac{\~\cL_{\om}^{N+n}(f)(x)}
{\~\cL_{\om}^{N+n}(\ind_\om)(x)}
\cdot
\frac{\~\cL_{\sg^{N}(\om)}^{n}\lt(\cL_{\om}^{N}(\ind_\om)\cdot \ind_{\sg^{N}(\om)}\rt)(x)}
{\~\cL_{\sg^{N}(\om)}^{n}(\ind_{\sg^{N}(\om)})(x)}
\\
&\leq
\frac{\~\cL_{\om}^{N+n}(f)(x)}
{\~\cL_{\om}^{N+n}(\ind_\om)(x)}
\cdot
\norm{\~\cL_{\om}^{N}\ind_\om}_\infty.
\end{align*}
Now taking the infimum over $x\in D_{\sg^N(\om),N+n}$ and letting $n\to\infty$ gives
\begin{align}
\Lm_{\sg^{N}(\om)}(\~\cL_{\om}^{N}f) \leq \norm{\~\cL_{\om}^{N}\ind_\om}_\infty \Lm_{\om}(f).
\label{eq: a* bound 1}
\end{align}
Now, set $N=\Sg_\om^{(k)}$.
Since $\ind_\om\in\sC_{\om,a_*}$, \eqref{eq: cone inv for large bad blocks} from Lemma~\ref{lem: cone cont for coating blocks}  implies that
\begin{align}
\norm{\~\cL_{\om}^{\Sg_\om^{(k)}}\ind_\om}_\infty
&\leq 
\var(\~\cL_{\om}^{\Sg_\om^{(k)}}(\ind_\om))
+
\Lm_{\sg^{\Sg_\om^{(k)}}(\om)}(\~\cL_{\om}^{\Sg_\om^{(k)}}(\ind_\om))
\nonumber\\
&\leq 
\lt(\frac{a_*}{2}+1\rt)\Lm_{\sg^{\Sg_\om^{(k)}}(\om)}(\~\cL_{\om}^{\Sg_\om^{(k)}}(\ind_\om))
\nonumber\\
&\leq
a_*\Lm_{\sg^{\Sg_\om^{(k)}}(\om)}(\~\cL_{\om}^{\Sg_\om^{(k)}}(\ind_\om)),
\label{eq: a* bound 2}
\end{align}
where we have used the fact that $a_*>2$ which follows from \eqref{eq: def of a_*}.
Combining \eqref{eq: a* bound 2} with \eqref{eq: a* bound 1}, we see that 
\begin{align*}
\Lm_{\sg^{\Sg_\om^{(k)}}(\om)}(\~\cL_{\om}^{\Sg_\om^{(k)}} f)
\leq
a_*\Lm_{\sg^{\Sg_\om^{(k)}}(\om)}(\~\cL_{\om}^{\Sg_\om^{(k)}}\ind_\om)\Lm_{\om}(f)
\end{align*}
completing the proof.
\end{proof}

\begin{lemma}\label{lem: LMD 3.10}
For each $\dl>0$ and each $\om\in\Om$ there exists $N_{\om,\dl}$ such that for each $n\geq N_{\om,\dl}$, $\cZ_{\om}^{(n)}$ has the property that
\begin{align*}
\sup_{Z\in\cZ_{\om}^{(n)}}\Lm_{\om}(\ind_Z)\leq \dl.
\end{align*}
\end{lemma}
\begin{proof}
Choose $N_{\om,\dl}\in\NN$ such that
$$
\frac{\norm{g_{\om}^{(n)}}_\infty}{\rho_{\om}^n}\leq\dl
$$
for each $n\geq N_{\om,\dl}$. Now, fix some $n\geq N_{\om,\dl}$ and let $m\in\NN$.
Then, for $Z\in\cZ_{\om}^{(n)}$ we have
\begin{align*}
\cL_{\om}^n\ind_Z(x)\leq \norm{g_{\om}^{(n)}}_\infty \leq \dl\rho_{\om}^n.
\end{align*}
For each $x\in D_{\sg^{n+m}(\om),n+m}\sub D_{\sg^{n+m}(\om),m}$ we have
\begin{align*}
\frac
{\cL_{\om}^{n+m}\ind_Z(x)}
{\cL_{\om}^{n+m}\ind_\om(x)}
&\leq
\frac
{\norm{\cL_{\om}^n\ind_Z}_\infty\cL_{\sg^n(\om)}^{m}\ind_{\sg^n(\om)}(x)}
{\cL_{\om}^{n+m}\ind_\om(x)}
\leq
\frac
{\dl\rho_{\om}^n\cL_{\sg^n(\om)}^{m}\ind_{\sg^n(\om)}(x)}
{\cL_{\sg^n(\om)}^m\lt(\cL_{\om}^{n}\ind_\om\rt)(x)}
\\
&
\leq
\dl\rho_{\om}^n\frac{1}{
	\inf_{y\in D_{\sg^{n+m}(\om),m}}
	\frac
	{\cL_{\sg^n(\om)}^m\lt(\cL_{\om}^{n}\ind_\om\rt)(y)}
	{\cL_{\sg^n(\om)}^{m}\ind_{\sg^n(\om)}(y)}
}.
\end{align*}
In view of Lemma~\ref{LMD l3.6}, taking the infimum over $x\in D_{\sg^{n+m}(\om),n+m}$ and letting $m\to \infty$ gives
\begin{align*}
\Lm_{\om}(\ind_Z)\leq \dl\rho_{\om}^n\cdot\frac{1}{\Lm_{\sg^n(\om)}(\cL_{\om}^n\ind_\om)}
\leq
\dl\rho_{\om}^n(\rho_{\om}^n)^{-1}= \dl.
\end{align*}
\end{proof}

We are now ready to prove the main result of this section, a random version of Lemma~3.11 in \cite{LMD}. Let
\begin{align}\label{def: delta0}
\dl_0:=\frac{1}{8a_*^3}.
\end{align}
\begin{lemma}\label{lem: bound on partition ele}
For $m$-a.e. $\om\in\Om$ with $y_*(\om)=0$, for all $\dl<\dl_0$, all $n\geq N_{\om,\dl}$ (where $N_{\om,\dl}$ is as in Lemma~\ref{lem: LMD 3.10}), and all $f\in\sC_{\om,a_*}$ there exists $Z_f\in\cZ_{\om,g}^{(n)}$ such that 
\begin{align*}
\inf_{Z_f}f \geq \frac{1}{4}\Lm_{\om}(f).
\end{align*}
\end{lemma}
\begin{proof}
We shall the prove the lemma via contradiction. To that end suppose that the conclusion is false, that is we suppose that 
\begin{align}\label{eq: contr on part ele}
\inf_Z f< \frac{\Lm_\om(f)}{4}  
\end{align}
for all $Z\in\cZ_{\om,g}^{(n)}$. 
Then, for each $n\geq N_{\om,\dl}$ and each $k\in\NN$ such that $n<\Sg_\om^{(k)}$, using \eqref{eq: contr on part ele} we can write 
\begin{align}
\~\cL_{\om}^{\Sg_\om^{(k)}} f
&=
\sum_{Z\in\cZ_{\om}^{(n)}} \~\cL_{\om}^{\Sg_\om^{(k)}} (f\ind_Z)
=
\sum_{Z\in\cZ_{\om,*}^{(n)}} \~\cL_{\om}^{\Sg_\om^{(k)}} (f\ind_Z)
\nonumber\\
&=
\sum_{Z\in\cZ_{\om,g}^{(n)}} \~\cL_{\om}^{\Sg_\om^{(k)}} (f\ind_Z)
+
\sum_{Z\in\cZ_{\om,b}^{(n)}} \~\cL_{\om}^{\Sg_\om^{(k)}} (f\ind_Z)
\nonumber\\
&\leq 
\frac{\Lm_{\om}(f)}{4}
\sum_{Z\in\cZ_{\om,g}^{(n)}} \~\cL_{\om}^{\Sg_\om^{(k)}} (\ind_Z)
+
\sum_{Z\in\cZ_{\om,g}^{(n)}} \~\cL_{\om}^{\Sg_\om^{(k)}} (\ind_Z)\var_Z(f)
+
\norm{f}_\infty
\sum_{Z\in\cZ_{\om,b}^{(n)}} \~\cL_{\om}^{\Sg_\om^{(k)}} (\ind_Z).
\label{eq: spec part ele 1}
\end{align}	
Now for $Z\in\cZ_{\om,b}^{(n)}$, Lemma~\ref{lem: LMD 3.8 temporary} implies that  
\begin{align}
\Lm_{\sg^{\Sg_\om^{(k)}}(\om)}\lt(\~\cL_{\om}^{\Sg_\om^{(k)}} (\ind_Z)\rt)
\leq
a_*\Lm_{\sg^{\Sg_\om^{(k)}}(\om)}\lt(\~\cL_{\om}^{\Sg_\om^{(k)}}\ind_\om\rt)\Lm_{\om}(\ind_Z)
=0.
\label{eq: functional over bad int eq 0}
\end{align}
Thus, for $Z\in\cZ_{\om,b}^{(n)}$, using \eqref{eq: functional over bad int eq 0} and \eqref{eq: important coating block ineq}, applied along the blocks $\Sg_\om^{(k)}$, we have that 
\begin{align}
\~\cL_{\om}^{\Sg_\om^{(k)}} (\ind_Z)
&\leq 
\Lm_{\sg^{\Sg_\om^{(k)}}(\om)}\lt(\~\cL_{\om}^{\Sg_\om^{(k)}} (\ind_Z)\rt)+\var\lt(\~\cL_{\om}^{\Sg_\om^{(k)}} (\ind_Z)\rt)
\nonumber\\
&=\var\lt(\~\cL_{\om}^{\Sg_\om^{(k)}} (\ind_Z)\rt)
\nonumber\\
&\leq 2\lt(\frac{1}{3}\rt)^{\Sg_\om^{(k)}/R_*}
+
a_*\Lm_{\sg^{\Sg_\om^{(k)}}(\om)}\lt(\~\cL_{\om}^{\Sg_\om^{(k)}} (\ind_Z)\rt)
=2\lt(\frac{1}{3}\rt)^{\Sg_\om^{(k)}/R_*}.		
\label{eq: spec part ele 2}
\end{align}
Note that the right-hand side above goes to zero as $k\to\infty$.
On the other hand, for $Z\in \cZ_{\om,g}^{(n)}$ we again use \eqref{eq: important coating block ineq} in conjunction with Lemma~\ref{lem: LMD 3.8 temporary} to get that 
\begin{align}
\~\cL_{\om}^{\Sg_\om^{(k)}} (\ind_Z)
&\leq 
\Lm_{\sg^{\Sg_\om^{(k)}}(\om)}\lt(\~\cL_{\om}^{\Sg_\om^{(k)}} (\ind_Z)\rt)+\var\lt(\~\cL_{\om}^{\Sg_\om^{(k)}} (\ind_Z)\rt)
\nonumber\\
&\leq 
2\lt(\frac{1}{3}\rt)^{\Sg_\om^{(k)}/R_*}
+
a_*\Lm_{\sg^{\Sg_\om^{(k)}}(\om)}\lt(\~\cL_{\om}^{\Sg_\om^{(k)}} (\ind_Z)\rt)
\nonumber\\
&\leq  
2\lt(\frac{1}{3}\rt)^{\Sg_\om^{(k)}/R_*}+a_*^2\Lm_{\sg^{\Sg_\om^{(k)}}(\om)}\lt(\~\cL_{\om}^{\Sg_\om^{(k)}}\ind_\om\rt)\Lm_{\om}(\ind_Z).
\label{eq: spec part ele 3}
\end{align}
Substituting \eqref{eq: spec part ele 3} and \eqref{eq: spec part ele 2} into \eqref{eq: spec part ele 1}, applying the functional $\Lm_{\sg^{\Sg_\om^{(k)}}}$ to both sides yields
\begin{align}
&\Lm_{\sg^{\Sg_\om^{(k)}}(\om)}\lt(\~\cL_{\om}^{\Sg_\om^{(k)}}f\rt)
\leq 
\Lm_{\sg^{\Sg_\om^{(k)}}(\om)}\lt(\~\cL_{\om}^{\Sg_\om^{(k)}}\ind_\om\rt)
\cdot 
\frac{\Lm_{\om}(f)}{4}
\nonumber\\
&\qquad
+
\sum_{Z\in\cZ_{\om,g}^{(n)}} \lt(\lt(a_*\Lm_{\om}(\ind_Z) +2\lt(\frac{1}{3}\rt)^{\Sg_\om^{(k)}/R_*}+a_*^2\Lm_{\om}(\ind_Z)\rt)\var_Z(f)\rt)\Lm_{\sg^{\Sg_\om^{(k)}}(\om)}\lt(\~\cL_{\om}^{\Sg_\om^{(k)}}\ind_\om\rt)
\nonumber\\
&\quad\qquad
+
\sum_{Z\in\cZ_{\om,b}^{(n)}} 2\lt(\frac{1}{3}\rt)^{\Sg_\om^{(k)}/R_*}\norm{f}_\infty\Lm_{\sg^{\Sg_\om^{(k)}}(\om)}\lt(\~\cL_{\om}^{\Sg_\om^{(k)}}\ind_\om\rt).
\label{eq: spec part ele 4}
\end{align}
Dividing \eqref{eq: spec part ele 4} on both sides by $\Lm_{\sg^{\Sg_\om^{(k)}}(\om)}\lt(\~\cL_{\om}^{\Sg_\om^{(k)}}\ind_\om\rt)$,
letting $k\to\infty$, and using Lemmas~\ref{lem: LMD 3.8 temporary}, \ref{lem: LMD 3.10}, and \ref{lem: cone cont for coating blocks} gives us 
\begin{align*}
\Lm_{\om}(f)
&\leq
\frac{\Lm_{\om}(f)}{4}
+
\sum_{Z\in\cZ_{\om,g}^{(n)}} (a_*+a_*^2)\Lm_{\om}(\ind_Z)\var_Z(f)
\nonumber\\
&\leq
\frac{\Lm_{\om}(f)}{4}
+
(a_*+a_*^2)\var(f)\sup_{Z\in\cZ_{\om,g}^{(n)}}\Lm_{\om}(\ind_Z)
\nonumber\\
&\leq
\lt(\frac{1}{4}
+
2a_*^3\dl\rt)\Lm_{\om}(f).
\end{align*}
Given our choice \eqref{def: delta0} of $\dl<\dl_0$ we arrive at the contradiction
\begin{align*}
\Lm_{\om}(f)\leq \frac{1}{2}\Lm_{\om}(f), 
\end{align*}
and thus we are done.
\end{proof}
\section{Finding finite diameter images}\label{sec: fin diam}
We now find a large measure set of fibers $\om\in \Om_F\sub\Om$ for which the image of the cone $\sC_{\om,a_*}$ has a finite diameter image after sufficiently many iterates of the normalized operator $\~\cL_\om$. Towards accomplishing this task we first recall that for each $\om\in\Om$ with $y_*(\om)=0$ and each $k\geq 0$ 
\begin{align*}
\Sg_\om^{(k)}:=\sum_{j=0}^{k-1} \ell(\om_j)R_*
\end{align*}
where $\om_0:=\om$ and for each $j\geq 1$ we set $\om_j:=\sg^{\Sg_\om^{(j-1)}}(\om)$. For each $\om\in\Om$ with $y_*(\om)=0$, we define the number
\begin{align}\label{eq: def of Sg_om}
\Sg_\om:=\min\set{\Sg_\om^{(k)}: 
	\inf_{x\in D_{\sg^{\Sg_\om^{(k)}}(\om), \Sg_\om^{(k)}}}\frac{\cL_{\om}^{\Sg_\om^{(k)}}\ind_Z(x)}{\cL_{\om}^{\Sg_\om^{(k)}}\ind_\om(x)}
	\geq
	\frac{\Lm_{\om}(\ind_Z)}{2} \text{ for all } Z\in\cZ_{\om,g}^{(N_{\om,\dl_0})}
}.
\end{align}\index{$\Sg_\om$}
Note that by definition we must have that $\Sg_\om\geq N_{\om,\dl_0}$.
Recall from the proof of Lemma~\ref{lem: constr of Om_G} that the set $\Om_1=\Om_1(B_*)$ is given by 
\begin{align}\label{eq: def of Om_1 part 2}
\Om_1:=\set{\om\in\Om: C_\ep(\om)\leq B_*},
\end{align}
where $C_\ep(\om)$ comes from Proposition~\ref{prop: LY ineq 2} and $B_*$ was chosen sufficiently large such that $m(\Om_1)\geq 1-\sfrac{\ep}{8}$. 
For $\al_*>0$ and $C_*\geq 1$ we consider the following
\begin{flalign} 
& \Lm_{\sg^{\Sg_\om}(\om)}\lt(\~\cL_\om^{\Sg_\om}\ind_Z\rt)\geq \al_* \text{ for all } Z\in\cZ_{\om,g}^{(N_{\om,\dl_0})},
\tag{F1}\label{F1} 
&\\
& C_*^{-1}\leq \inf_{D_{\sg^{\Sg_\om}(\om), \Sg_\om}}\~\cL_{\om}^{\Sg_\om}\ind_\om\leq \norm{\~\cL_{\om}^{\Sg_\om}\ind_\om}_\infty\leq  C_*.
\tag{F2}\label{F2} 
\end{flalign}
Now, we define the set $\Om_3$, depending on parameters $S_*=kR_*$ for some $k\in\NN$, $\al_*>0$, and $C_*\geq 1$, by
\begin{align}\label{eq: def of Om_3}
\Om_3=\Om_3(S_*, \al_*, C_*):=\set{\om\in\Om: \Sg_\om\leq S_*, \text{ and } \eqref{F1}-\eqref{F2} \text{ hold }}, 
\end{align} 
and choose $S_*=kR_*$,\index{$S_*$} $\al_*>0$, and $C_*\geq 1$ such that $m(\Om_3)\geq 1-\sfrac{\ep}{8}$. Finally, we define
\begin{align}\label{eq: def of Om_F}
\Om_F:=\Om_1\cap \Om_3, 
\end{align}\index{$\Om_F$}
which must of course have measure $m(\Om_F)\geq 1-\ep$. Furthermore, in light of the definition of $\Om_G$ from \eqref{eq: def of good set}, we have that 
\begin{align*}
\sg^{-R_*}(\Om_F)\sub\Om_G.
\end{align*}

\begin{lemma}\label{lem: cone contraction for good om}
For all $\om\in\Om_F$ such that $y_*(\om)=0$ we have that 
\begin{align*}
\~\cL_{\om}^{\Sg_\om}\sC_{\om,a_*}\sub \sC_{\sg^{\Sg_\om}(\om), \sfrac{a_*}{2}} \sub \sC_{\sg^{\Sg_\om}(\om), a_*}
\end{align*}
with 
\begin{align}\label{eq: def of Dl}
\diam_{\sg^{\Sg_\om}(\om),a_*}\lt(\~\cL_{\om}^{\Sg_\om}\sC_{\om,a_*}\rt)
\leq
\Dl
:=
2\log \frac{8C_*^2a_*(3+a_*)}{\al_*}<\infty.
\end{align}
\end{lemma}
\begin{proof}
The invariance follows from Lemma~\ref{lem: cone cont for coating blocks}. 
To show that the diameter is finite we first note that for $0\not\equiv f\in\sC_{\om,a_*}$ we must have that $\Lm_\om(f)>0$ by definition. Now, 
Lemma~\ref{lem: summary of cone dist prop} implies that for $f\in\sC_{\om,a_*}$ we have 
\begin{align}
\Ta_{\sg^{\Sg_\om}(\om),a_*}(\~\cL_{\om}^{\Sg_\om}f, \ind_{\sg^{\Sg_\om}(\om)})
\leq 
\log \frac{\norm{\~\cL_{\om}^{\Sg_\om}f}_\infty + \frac{1}{2}\Lm_{\sg^{\Sg_\om}(\om)}\lt(\~\cL_{\om}^{\Sg_\om}f\rt)}
{\min\set{\inf_{D_{\sg^{\Sg_\om}(\om), \Sg_\om}} \~\cL_{\om}^{\Sg_\om}f, \frac{1}{2}\Lm_{\sg^{\Sg_\om}(\om)}\lt(\~\cL_{\om}^{\Sg_\om}f\rt)}}.
\label{eq: diam est 1}
\end{align}
Using Lemmas~\ref{lem: cone cont for coating blocks} and \ref{lem: LMD 3.8 temporary} and \eqref{F2} we bound the numerator by  
\begin{align}
\norm{\~\cL_{\om}^{\Sg_\om}f}_\infty + \frac{1}{2}\Lm_{\sg^{\Sg_\om}(\om)}\lt(\~\cL_{\om}^{\Sg_\om}f\rt)
&\leq 
\var\lt(\~\cL_{\om}^{\Sg_\om}f\rt)
+
\frac{3}{2}\Lm_{\sg^{\Sg_\om}(\om)}\lt(\~\cL_{\om}^{\Sg_\om}f\rt)
\nonumber\\
&\leq 
\frac{3a_*+a_*^2}{2}\Lm_{\sg^{\Sg_\om}(\om)}\lt(\~\cL_{\om}^{\Sg_\om}\ind_\om\rt)\Lm_{\om}(f)\nonumber\\
&\leq 
\frac{C_*a_*(3+a_*)}{2}\Lm_{\om}(f).
\label{eq: diam est 2}
\end{align}
To find a lower bound for the denominator we first note that for each $f\in\sC_{\om,a_*}$, by Lemma~\ref{lem: bound on partition ele} there exists $Z_f\in\cZ_{\om,g}^{(N_{\om,\dl_0})}$ such that 
\begin{align}\label{eq: lower bound for f on Z_f}
\inf f\rvert_{Z_f}\geq \frac{\Lm_{\om}(f)}{4}.
\end{align}
Thus, using \eqref{eq: lower bound for f on Z_f}, for each $x\in D_{\sg^{\Sg_\om}(\om), \Sg_\om}$ we have that 	
\begin{align*}
\inf_{x\in D_{\sg^{\Sg_\om}(\om), \Sg_\om}}\~\cL_{\om}^{\Sg_\om} f(x)
&\geq 
\inf_{x\in D_{\sg^{\Sg_\om}(\om), \Sg_\om}}\~\cL_{\om}^{\Sg_\om}( f\ind_{Z_f})(x)
\\
&\geq
\inf_{Z_f} f \cdot		
\inf_{x\in D_{\sg^{\Sg_\om}(\om), \Sg_\om}}\~\cL_{\om}^{\Sg_\om}\ind_{Z_f}(x)
\\
&\geq
\frac{\Lm_{\om}(f)}{4}
\inf_{x\in D_{\sg^{\Sg_\om}(\om), \Sg_\om}}\~\cL_{\om}^{\Sg_\om}\ind_{Z_f}(x)
\\
&\geq 
\frac{\Lm_{\om}(f)}{4}
\inf_{y\in D_{\sg^{\Sg_\om}(\om), \Sg_\om}}\frac{\~\cL_{\om}^{\Sg_\om}\ind_{Z_f}(y)}{\~\cL_{\om}^{\Sg_\om}\ind_\om(y)}
\inf_{z\in D_{\sg^{\Sg_\om}(\om), \Sg_\om}} \~\cL_{\om}^{\Sg_\om}\ind_\om(z).
\end{align*}
In light of conditions \eqref{F1}-\eqref{F2} we in fact have that 
\begin{align}
\inf_{D_{\sg^{\Sg_\om}(\om), \Sg_\om}}\~\cL_{\om}^{\Sg_\om} f
\geq 
\frac{\al_*\Lm_{\om}(f)}{8C_*}>0.
\label{eq: diam est 3}
\end{align}
Combining the estimates \eqref{eq: diam est 2} and \eqref{eq: diam est 3} with \eqref{eq: diam est 1} gives 
\begin{align*}
\Ta_{\sg^{\Sg_\om}(\om),a_*}(\~\cL_{\om}^{\Sg_\om}f, \ind_{\sg^{\Sg_\om}(\om)})
\leq 
\log\frac{8C_*^2a_*(3+a_*)}{\al_*}<\infty.
\end{align*}
Taking the supremum over all functions $f\in\sC_{\om,a_*}$, and applying the triangle inequality finishes the proof. 
\end{proof}

To end this section we recall from Section~\ref{sec:bad} that $0\leq y_*(\om)< R_*$ is chosen to be the smallest integer such that for either choice of sign $+$ or $-$ we have 
\begin{flalign} 
& \lim_{n\to\infty} \frac{1}{n}\#\set{0\leq k< n: \sg^{\pm kR_*+y_*(\om)}(\om)\in\Om_G} >1-\ep,
\label{def y*1 part 2} 
&\\
& \lim_{n\to\infty} \frac{1}{n}\#\set{0\leq k< n: C_\ep\lt(\sg^{\pm kR_*+y_*(\om)}(\om)\rt)\leq B_*} >1-\ep.
\label{def y*2 part 2} 
\end{flalign}
In light of the definition of $\Om_F$ \eqref{eq: def of Om_F} and using the same reasoning as in Section~\ref{sec:bad} for the existence of $y_*$ (see Section~7 of \cite{AFGTV20}), for each $\om\in\Om$, we now let $0\leq v_*(\om)<R_*$ be the least integer such that for either choice of sign $+$ or $-$ we have that the following hold:
\begin{flalign}
& \lim_{n\to\infty} \frac{1}{n}\#\set{0\leq k< n: \sg^{\pm kR_*+v_*(\om)}(\om)\in\Om_G} >1-\ep,
\label{def v*1} 
&\\
&\lim_{n\to\infty} \frac{1}{n}\#\set{0\leq k< n: \sg^{\pm kR_*+v_*(\om)}(\om)\in\Om_F} >1-\ep.
\label{def v*3}
\end{flalign}
Two significant properties of $v_*$ are the following:
\begin{flalign}
& v_*(\sg^{v_*(\om)}(\om))=0,
\label{eq: v* prop1} 
&\\
& \text{if }v_*(\om)=0, \text{ then } y_*(\om)=0.
\label{eq: v* prop2} 
\end{flalign}
\section{Conformal and invariant measures}\label{sec: conf and inv meas}
We are now ready to bring together all of the results from Sections~\ref{sec: LY ineq}-\ref{sec: fin diam} to establish the existence of conformal and invariant measures supported in the survivor set $X_{\om,\infty}$. We follow the methods of \cite{AFGTV20} and  \cite{LMD}, and we begin with the following technical lemma from which the rest of our results will follow. 
\begin{lemma}\label{lem: exp conv in C+ cone}
Let $f,h\in\BV_\Om(I)$, let $\ep>0$ sufficiently small such that the results of Section~\ref{sec:bad} apply, and let $V:\Om\to(0,\infty)$ be a measurable function. 
Suppose that for each $n\in\NN$, each $|p|\leq n$, each $l\geq 0$, and for $m$-a.e. $\om\in\Om$ we have
$f_{\sg^p(\om)}\in\sC_{\sg^p(\om), +}$ with $\var(f_{\sg^p(\om)})\leq e^{\ep n}V(\om)$ and  $h_{\sg^{p-l}(\om)}\in\sC_{\sg^{p-l}(\om), +}$ with $\var(h_{\sg^{p-l}(\om)})\leq e^{\ep (n+l)}V(\om)$.
Then there exists $\vta\in(0,1)$ and a measurable function $N_3:\Om\to\NN$ such that for all $n\geq N_3(\om)$, all $l\geq 0$, and all $|p|\leq n$ we have 
\begin{align}\label{eq: lem 8.1 ineq}
\Ta_{\sg^{n+p}(\om), +}\lt(\~\cL_{\sg^p(\om)}^n f_{\sg^p(\om)}, \~\cL_{\sg^{p-l}(\om)}^{n+l} h_{\sg^{p-l}(\om)}\rt) \leq \Dl\vta^n.
\end{align} 
Furthermore, $\Dl$, defined in \eqref{eq: def of Dl}, and $\vta$ do not depend on $V$.
\end{lemma}
\begin{proof}
We begin by noting that by \eqref{eq: L_om is a weak contraction on C_+} for each $l\geq 0$ we have that $\~\cL_{\sg^{p-l}(\om)}^l h_{\sg^{p-l}(\om)}\in\sC_{\sg^{p}(\om),+}$ for each $h_{\sg^{p-l}(\om)}\in\sC_{\sg^{p-l}(\om),+}$, and let 
\begin{align*}
h_l=\~\cL_{\sg^{p-l}(\om)}^l h_{\sg^{p-l}(\om)}.
\end{align*}
Set $v_*=v_*(\sg^p(\om))$ (defined at the end of Section~\ref{sec: fin diam}) and let $d_*=d_*(\sg^p(\om))\geq 0$ be the smallest integer that satisfies 
\begin{flalign} 
& v_*+d_*R_*\geq \frac{\ep n+\log V(\om)}{\ta-\ep}, 
\label{eq: d prop1} 
&\\
& \sg^{p+v_*+d_*R_*}(\om)\in\Om_F.
\label{eq: d prop2}
\end{flalign}
where $\ta$ was defined in \eqref{eq: def of ta}.
Choose 
\begin{align}\label{eq: coose N_1 ge log V}
N_1(\om)\ge \frac{\log V(\om)}{\ep}
\end{align} 
and let $n\geq N_1(\om)$. Now using \eqref{eq: coose N_1 ge log V} to write
\begin{align*}
\frac{4\ep n}{\ta}=\frac{\ep n+\ep n}{\ta/2}\geq \frac{\ep n+\log V(\om)}{\ta/2}
\end{align*}
and then using \eqref{eq: def of ep_0} and \eqref{eq: def of y_*}, we see that \eqref{eq: d prop1} is satisfied for any $d_*R_*\geq 4\epsilon n/\ta$.
Using \eqref{def v*3}, the construction of $v_*$, and the ergodic decomposition of $\sigma^{R_*}$ following \eqref{def v*3}, we have for $m$-a.e. $\omega\in\Om$ there is an infinite, increasing sequence of integers $d_j\ge 0$ satisfying \eqref{eq: d prop2}. Furthermore, \eqref{def v*3} implies that 
\begin{align*}
\lim_{n\to\infty}\frac{1}{n/R_*}\#\set{0\leq k< \frac{n}{R_*}: \sg^{\pm kR_*+v_*(\om)}(\om)\notin\Om_F} <\ep,
\end{align*}
and thus for $n\in\NN$ sufficiently large (depending measurably on $\om$), say $n\geq N_2(\om)\geq N_1(\om)$, we have that 
\begin{align*}
\#\set{0\leq k< \frac{n}{R_*}: \sg^{\pm kR_*+v_*(\om)}(\om)\notin\Om_F} < \frac{\ep n}{R_*}.
\end{align*}
Thus  the smallest integer $d_*$ satisfying \eqref{eq: d prop1} and \eqref{eq: d prop2} also satisfies 
\begin{align}\label{eq: dR is Oepn}
d_*R_*\leq\frac{4\ep n}{\ta}+\ep n=\lt(\frac{4+\ta}{\ta }\rt)\ep n.
\end{align}
Let 
\begin{align}\label{eq: def of hat y_*}
\hat v_*=v_*+d_*R_*.
\end{align}
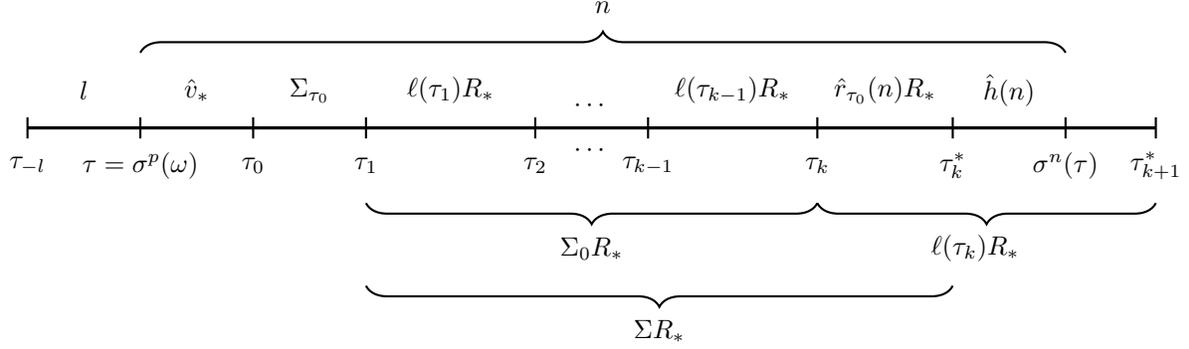
\begin{figure}[h]
\centering	
\begin{tikzpicture}[y=1cm, x=1.5cm, thick, font=\footnotesize]    
\draw[line width=1.2pt,  >=latex'](0,0) -- coordinate (x axis) (10,0);       

\draw (0,-4pt)   -- (0,4pt) {};
\node at (0, -.5) {$\tau_{-l}$};

\node at (.5, .5) {$l$};

\draw (1,-4pt)   -- (1,4pt) {};
\node at (1,-.5) {$\tau=\sg^p(\om)$};

\node at (1.5,.5) {$\hat v_*$};

\draw (2,-4pt) -- (2,4pt) {};
\node at (2,-.5) {$\tau_0$};

\node at (2.5,.5) {$\Sg_{\tau_0}$};

\draw (3,-4pt)  -- (3,4pt)  {};
\node at (3,-.5) {$\tau_1$};

\node at (3.75,.5) {$\ell(\tau_{1})R_*$};

\draw (4.5,-4pt)  -- (4.5,4pt)  {};
\node at (4.5,-.5) {$\tau_{2}$};

\node at (5,.3) {$\cdots$};
\node at (5,-.3) {$\cdots$};

\draw (5.5,-4pt)  -- (5.5,4pt)  {};
\node at (5.5,-.5) {$\tau_{k-1}$};

\node at (6.25,.5) {$\ell(\tau_{k-1})R_*$};	

\draw (7,-4pt)  -- (7,4pt)  {};
\node at (7,-.5) {$\tau_{k}$};

\node at (7.6,.5) {$\hat r_{\tau_0}(n)R_*$};

\draw (8.2,-4pt)  -- (8.2,4pt)  {};
\node at (8.2,-.5) {$\tau_{k}^*$};
\node at (8.7,.5) {$\hat h(n)$};

\draw (9.2,-4pt)  -- (9.2,4pt)  {};
\node at (9.2,-.5) {$\sg^n(\tau)$};

\draw (10,-4pt)  -- (10,4pt)  {};
\node at (10,-.5) {$\tau_{k+1}^*$};


\draw[decorate,decoration={brace,amplitude=8pt,mirror}] 
(3,-2.1)  -- (8.2,-2.1) ; 
\node at (5.6,-2.7){$\Sg R_*$};		

\draw[decorate,decoration={brace,amplitude=8pt,mirror}] 
(3,-1) -- (7,-1) ; 
\node at (5,-1.6){$\Sg_0 R_*$};

\draw[decorate,decoration={brace,amplitude=8pt,mirror}] 
(7,-1)  -- (10,-1) ; 
\node at (8.4,-1.6){$\ell(\tau_{k})R_*$};

\draw[decorate,decoration={brace,amplitude=8pt}] 
(1,1)  -- (9.2,1) ; 
\node at (5.1,1.6){$n$};

\end{tikzpicture}
\caption{The fibers $\tau_j$ and the decomposition of $n=\hat v_*+\Sg_{\tau_0}+\Sg R_*+\hat h(n)$.}
\label{figure: om_k}
\end{figure}
Now, we wish to examine the iteration of our operator cocycle along a collection $\Sg R_*$ of blocks, each of length $\ell(\om)R_*$, so that the images of $\~\cL_\om^{\ell(\om)R_*}$ are contained in $\sC_{\sg^{\ell(\om)R_*}(\om),\sfrac{a_*}{2}}$ as in Lemma~\ref{lem: cone cont for coating blocks}; see Figure~\ref{figure: om_k}.

We begin by establishing some simplifying notation. To that end, set $\tau=\sg^{p}(\om)$, $\tau_{-l}=\sg^{p-l}(\om)$, and $\tau_0=\sg^{p+\hat v_*}(\om)$; see Figure~\ref{figure: om_k}. Note that in light of \eqref{eq: v* prop1}, \eqref{eq: v* prop2}, and \eqref{eq: def of hat y_*} we have that 
\begin{align}\label{eq: y_* tau_0}
v_*(\tau_0)=y_*(\tau_0)=0.
\end{align}
Now, by our choice of $d_*$, we have that if $f_\tau\in\sC_{\tau,+}$ with $\var(f_\tau)\leq e^{\ep n}V(\om)$, then 
\begin{align}\label{eq: f in a cone}
\~\cL_\tau^{\hat v_*} f_\tau\in\sC_{\tau_0,a_*}.
\end{align}
Indeed, applying Proposition~\ref{prop: LY ineq 2}, \eqref{eq: d prop1}, \eqref{eq: d prop2}, and the definition of $\Om_F$ \eqref{eq: def of Om_F}, we have 
\begin{align*}
\var\lt(\~\cL_{\tau}^{\hat v_*}f_\tau\rt)
&\leq 
C_{\ep}(\sg^{\hat v_*}(\tau))e^{-(\ta-\ep)\hat v_*}\var(f_\tau)
+C_{\ep}(\sg^{\hat v_*}(\tau))\Lm_{\sg^{\hat v_*}(\tau)}\lt(\~\cL_{\tau}^{\hat v_*}f_\tau\rt)
\\
&\leq 
B_*e^{-(\ta-\ep)\hat v_*}\var(f_\tau)+B_*\Lm_{\sg^{\hat v_*}(\tau)}\lt(\~\cL_{\tau}^{\hat v_*}f_\tau\rt)
\\&
\leq 
B_*\frac{\var(f_\tau)}{e^{\ep n}V(\om)}+B_*\Lm_{\sg^{\hat v_*}(\tau)}\lt(\~\cL_{\tau}^{\hat v_*}f_\tau\rt) 
\\
&\leq
B_*+B_*\Lm_{\sg^{\hat v_*}(\tau)}\lt(\~\cL_{\tau}^{\hat v_*}f_\tau\rt)  
\\
&
\leq 2B_*\Lm_{\sg^{\hat v_*}(\tau)}\lt(\~\cL_{\tau}^{\hat v_*}f_\tau\rt)\leq \frac{a_*}{6}\Lm_{\sg^{\hat v_*}(\tau)}\lt(\~\cL_{\tau}^{\hat v_*}f_\tau\rt), 
\end{align*}
where we recall that $a_*> 12B_*$ is defined in \eqref{eq: def of a_*}. 
A similar calculation yields that if $h_{\tau_{-l}}\in\sC_{\tau_{-l},+}$ with $\var(h_{\tau_{-l}})\leq e^{\ep (n+l)}V(\om)$, then $\~\cL_{\tau_{-l}}^{l+\hat v_*} h_{\tau_{-l}}\in\sC_{\tau_0,a_*}$.

We now set $\tau_1=\sg^{\Sg_{\tau_0}}(\tau_0)$ and for each $j\geq 2$ let $\tau_j=\sg^{\ell(\tau_{j-1})R_*}(\tau_{j-1})$.
Note that since $\tau_0\in\Om_F$, we have that $\Sg_{\tau_0}\leq S_*$. 

As there are only finitely many blocks (good and bad) that will occur within an orbit of length $n$, let $k\geq 1$ be the integer such that 
\begin{align*}
\hat v_*+\Sg_{\tau_0}+\sum_{j=1}^{k-1} \ell(\tau_j)R_* \leq n < \hat v_*+\Sg_{\tau_0}+\sum_{j=1}^{k} \ell(\tau_j)R_*,
\end{align*}
and let 
\begin{equation*}
\Sg_0:=\sum_{j=1}^{k-1} \ell(\tau_j)
\qquad \text{ and } \qquad
\hat r_{\tau_0}(n):=r_{\tau_0}(n-\hat v_*)
\end{equation*} 
where $r_{\tau_0}(n-\hat v_*)$ is the number defined in \eqref{eq: n length orbit in om_j}.
Finally setting 
$$
\Sg=\Sg_0+\hat r_{\tau_0}(n),
\qquad 
\hat h(n):=n-\hat v_*-\Sg_{\tau_0}-\Sg R_*,
\qquad \text{ and } \qquad
\tau_k^*:=\sg^{\hat r_{\tau_0}(n)}(\tau_k),
$$
we have the right decomposition of our orbit length $n$ into blocks which do not expand distances in the fiber cones $\sC_{\om,a_*}$ and $\sC_{\om,+}$.
Now let
\begin{align}\label{eq: def of N_3}
n\geq N_3(\om):=\max\set{N_2(\om), \frac{S_*}{\ep}}.
\end{align}
Since $v_*, \hat h(n)\leq R_*$, by \eqref{eq: def of N_3}, \eqref{eq: dR is Oepn}, and for 
\begin{align}\label{key}
\ep<\frac{\ta}{8(1+\ta)}
\end{align}
sufficiently small, we must have that
\begin{align}\label{eq: SgR is big O ep n}
\Sg R_*
&=
n-\hat v_*-\Sg_{\tau_0}-\hat h(n)
=
n-v_*-d_*R_*-\Sg_{\tau_0}-\hat h(n)
\nonumber\\
&\geq 
n-\lt(\frac{4+\ta}{\ta}\rt)\ep n-2R_*-S_*
\geq 
n-\lt(\frac{4+\ta}{\ta}\rt)\ep n-3\ep n
\nonumber\\
&\geq 
n\lt(1-4\ep\lt(\frac{1+\ta}{\ta}\rt)\rt) > \frac{n}{2}.
\end{align}
Now we note that since $\~\cL_{\tau_k^*}^{\hat h(n)}(\sC_{\tau_k^*,+})\sub\sC_{\sg^{n+p}(\om), +}$ we have that $\~\cL_{\tau_k^*}^{\hat h(n)}$ is a weak contraction, and hence, we have 
\begin{align}\label{eq: L^r a weak contraction on C_+}
\Ta_{\sg^{n+p}(\om), +}\lt(\~\cL_{\tau_k^*}^{\hat h(n)} f' , \~\cL_{\tau_k^*}^{\hat h(n)} h' \rt)\leq \Ta_{\tau_k^*,+}(f',h'), 
\qquad f',h'\in\Ta_{\tau_k^*,+}.	
\end{align}
Recall that $E_{\tau_1}(n-\hat v_*-\Sg_{\tau_0})$, defined in Lemma~~\ref{lem: proportion of bad blocks}, is the total length of the bad blocks of the $n-\hat v_*$ length orbit starting at $\tau_0$, i.e. 
\begin{align*}
E_{\tau_1}(n-\hat v_*-\Sg_{\tau_0})&=\sum_{\substack{1\leq j< k \\ \tau_j\in\Om_B}} \ell(\tau_j) +r_{\tau_0}(n-\hat v_*).
\end{align*}
Lemma~\ref{lem: proportion of bad blocks} then gives that 
\begin{align}\label{eq: est of S}
E_{\tau_1}(n-\hat v_*-\Sg_{\tau_0})<Y\ep\Sg. 
\end{align}
We are now poised to calculate \eqref{eq: lem 8.1 ineq}, but first we note that we can write
\begin{align}
n&=\hat v_* + \Sg_{\tau_0} + \Sg R_* +\hat h(n)
\nonumber\\
&=\hat v_* + \Sg_{\tau_0} +\Sg_0 R_*+ \hat r_{\tau_0}(n)+ \hat h(n)
\label{eq: lem 8.1 decomp of n final ineq}	
\end{align}
and that the number of good blocks contained in the orbit of length $n-\hat v_*-\Sg_{\tau_0}$ is given by 
\begin{align}\label{eq: number of good blocks}
\Sg_G:=\#\set{1\leq j\leq k: \tau_j\in\Om_G}
=\Sg-E_{\tau_1}(n-\hat v_*-\Sg_{\tau_0})
\leq \Sg_0.
\end{align} 
Now, using \eqref{eq: lem 8.1 decomp of n final ineq} we combine (in order) \eqref{eq: L^r a weak contraction on C_+}, \eqref{eq: Ta+ leq Ta}, and Theorem~\ref{thm: cone distance contraction} (repeatedly) in conjunction with the fact that $\tau_0\in\Om_F$ to see that
\begin{align}
&\Ta_{\sg^{n+p}(\om), +}
\lt(\~\cL_{\tau}^{n} (f_\tau), \~\cL_{\tau_{-l}}^{n+l} (h_{\tau_{-l}}) \rt)
\nonumber\\
&\quad
=
\Ta_{\sg^{n+p}(\om), +}
\lt(
\~\cL_{\tau_k}^{\hat h(n)} \circ \~\cL_{\tau_1}^{\Sg R_*} \circ \~\cL_{\tau_0}^{\Sg_{\tau_0}} \circ \~\cL_{\tau}^{\hat v_*} (f_\tau)
,
\~\cL_{\tau_k}^{\hat h(n)} \circ \~\cL_{\tau_1}^{\Sg R_*} \circ \~\cL_{\tau_0}^{\Sg_{\tau_0}} \circ \~\cL_{\tau}^{\hat v_*}\circ \~\cL_{\tau_{-l}}^{l} (h_{\tau_{-l}}) 
\rt)
\nonumber\\
&\quad
\leq
\Ta_{\tau_k^*, +}
\lt(
\~\cL_{\tau_1}^{\Sg R_*} \circ \~\cL_{\tau_0}^{\Sg_{\tau_0}} \circ \~\cL_{\tau}^{\hat v_*} (f_\tau)
,
\~\cL_{\tau_1}^{\Sg R_*} \circ \~\cL_{\tau_0}^{\Sg_{\tau_0}} \circ \~\cL_{\tau}^{\hat v_*} (h_l) 
\rt)	
\nonumber\\
&\quad
\leq
\Ta_{\tau_k^*, a_*}
\lt(
\~\cL_{\tau_k}^{\hat r_{\tau_0}(n)} \circ \~\cL_{\tau_1}^{\Sg_0 R_*} \circ \~\cL_{\tau_0}^{\Sg_{\tau_0}} \circ \~\cL_{\tau}^{\hat v_*} (f_\tau)
,
\~\cL_{\tau_k}^{\hat r_{\tau_0}(n)} \circ \~\cL_{\tau_1}^{\Sg_0 R_*} \circ \~\cL_{\tau_0}^{\Sg_{\tau_0}} \circ \~\cL_{\tau}^{\hat v_*} (h_l) 
\rt)
\nonumber\\
&\quad
\leq 
\Ta_{\tau_k, a_*}
\lt(
\~\cL_{\tau_1}^{\Sg_0 R_*} \circ \~\cL_{\tau_0}^{\Sg_{\tau_0}} \circ \~\cL_{\tau}^{\hat v_*} (f_\tau)
,
\~\cL_{\tau_1}^{\Sg_0 R_*} \circ \~\cL_{\tau_0}^{\Sg_{\tau_0}} \circ \~\cL_{\tau}^{\hat v_*} (h_l) 
\rt)
\nonumber\\	
&\quad
\leq
\lt(\tanh\lt(\frac{\Dl}{4}\rt)\rt)^{\Sg_G}
\Ta_{\tau_1, a_*}
\lt(
\~\cL_{\tau_0}^{\Sg_{\tau_0}} \circ \~\cL_{\tau}^{\hat v_*} (f_\tau)
,
\~\cL_{\tau_0}^{\Sg_{\tau_0}} \circ \~\cL_{\tau}^{\hat v_*} (h_l) 
\rt).
\label{eq: cone Cauchy 1}
\end{align} 
Now since $\tau_0\in \Om_F$ and in light of \eqref{eq: f in a cone}, applying Lemma~\ref{lem: cone contraction for good om} allows us to estimate the $\Ta_{\tau_1,a_*}$ term in the right hand side of \eqref{eq: cone Cauchy 1} to give 
\begin{align}\label{eq: cone Cauchy Ta leq Dl}
\Ta_{\sg^{n+p}(\om), +}
\lt(\~\cL_{\tau}^{n} (f_\tau), \~\cL_{\tau_{-l}}^{n+l} (h_{\tau_{-l}}) \rt)
\leq \lt(\tanh\lt(\frac{\Dl}{4}\rt)\rt)^{\Sg_G}\Dl.
\end{align}
Using \eqref{eq: number of good blocks}, the fact that $E_{\tau_1}(n-\hat v_*-\Sg_{\tau_0})\geq 1$, \eqref{eq: est of S}, and \eqref{eq: SgR is big O ep n}, we see that
\begin{align}
\Sg_G
&= 
\Sg-E_{\tau_1}(n-\hat v_*-\Sg_{\tau_0})
\nonumber\\
&\geq
\Sg-Y\ep\Sg
\nonumber\\
&=
\Sg\lt(1-Y\ep\rt)
\nonumber\\
&\geq
\frac{\lt(1-Y\ep\rt)n}{2R_*}.
\label{eq: cone Cauchy exponent}
\end{align}
In light of \eqref{eq: def ep 4}, for  all $\ep>0$ sufficiently small we have that $1-Y\ep>0$. 
Finally, inserting \eqref{eq: cone Cauchy Ta leq Dl} and \eqref{eq: cone Cauchy exponent} into \eqref{eq: cone Cauchy 1} gives
\begin{align*}
&\Ta_{\sg^{n+p}(\om), +}
\lt(\~\cL_{\tau}^{n} (f_\tau), \~\cL_{\tau_{-l}}^{n+l} (h_{\tau_{-l}}) \rt)
\leq 
\Dl \vta^n,
\end{align*}
where 
$$
\vta:=\lt(\tanh\lt(\frac{\Dl}{4}\rt)\rt)^{\frac{\lt(1-Y\ep\rt)}{2R_*}}<1,
$$
which completes the proof.
\end{proof}
Combining Lemma~\ref{lem: exp conv in C+ cone} together with Lemma~\ref{lem: birkhoff cone contraction} gives the following immediate corollary.
\begin{corollary}\label{cor: exp conv in sup norm}
Suppose $\ep>0$, $V:\Om\to(0,\infty)$, $f_{\sg^p(\om)}\in\sC_{\sg^p(\om),+}$, and $h_{\sg^{p-l}(\om)}\in\sC_{\sg^{p-l}(\om),+}$ all satisfy the hypotheses of Lemma~\ref{lem: exp conv in C+ cone}. Then 
for $m$-a.e. $\om\in\Om$, all $n\geq N_3(\om)$, all $l\geq 0$, and all $|p|\leq n$ we have 
\begin{align*}
\norm{\~\cL_{\sg^p(\om)}^n f_{\sg^p(\om)} -\~\cL_{\sg^{p-l}(\om)}^{n+l} h_{\sg^{p-l}(\om)}}_\infty
&\leq 
\norm{\~\cL_{\sg^p(\om)}^n f_{\sg^p(\om)}}_\infty \lt(e^{\Dl\vta^n}-1\rt).
\end{align*} 
\end{corollary}

Notice that if we wish to apply Lemma~\ref{lem: exp conv in C+ cone} (or Corollary~\ref{cor: exp conv in sup norm}) repeatedly iterating in the forward direction, i.e. taking $p=0$ so that we push forward starting from the $\om$ fiber, then we only need that $f\in\sC_{\om,+}$ and do not need to be concerned with the assumption on the variation. Indeed, as $p=0$ is fixed, then we will have $\var(f)\leq \var(f)\cdot e^{\ep n}$ for any $n\geq 1$. However, if we wish to apply Lemma~\ref{lem: exp conv in C+ cone} repeatedly with $p=-n$ for $n$ increasing to $\infty$, then we will need to consider special functions $f$. 
\begin{definition}\label{def: set D of tempered BV func}
We let the set $\BVT$\index{$\BVT$} denote the set of functions $f\in\BV_\Om(I)$ such that the function $\Om\ni\om\mapsto\var(f_\om)$ is tempered, that is for $m$-a.e. $\oio$ 
\begin{align*}
    \lim_{|n|\to\infty}\frac{1}{|n|}\log\var(f_{\sg^n\om}) =0.
\end{align*}
In particular, this implies that for each $\ep>0$ there exists a tempered function $V_{f,\ep}:\Om\to (0,\infty)$ such that for all $n\in\ZZ$ \index{$V_{f,\ep}$}
\begin{align*}
    \var(f_{\sg^n(\om)})\leq V_{f,\ep}(\om)e^{\ep|n|}.
\end{align*}
Let $\BVT^+\sub\BVT$\index{$\BVT^+$} denote the collection of all functions $f\in\BVT$ such that $f_\om\geq 0$ for each $\om\in\Om$.
\end{definition}
Note that we may apply Lemma~\ref{lem: exp conv in C+ cone} and Corollary~\ref{cor: exp conv in sup norm} to functions $f,h\in\BVT^+$ by taking $V(\om)=\max\{V_{f,\ep}, V_{h,\ep}\}$ for a given $\ep>0$.
\begin{remark}
Note that the space $\BVT$ is nonempty. In particular, $\BVT$ contains any function $f:\Om\times I\to\RR$ such that $f_\om$ is equal to some fixed function $f\in\BV(I)$ for $m$-a.e. $\oio$, in which case we may take $V_{f,\ep}(\om)=\var(f)$ for each $\ep>0$. More generally, $\BVT$ contains any function $f\in\BV_\Om(I)$ such that $\log\var(f_\om)\in L^1(m)$. 
\end{remark}
\begin{remark}\label{rem: combining D1 and D2}
Note that if $f\in \BVT$ and additionally satisfies 
\begin{align}
    \Lm_{\sg^n(\om)}(|f_{\sg^n(\om)}|)\geq V_{f,\ep}^{-1}(\om)e^{-\ep |n|}
\label{cond cD2}
\end{align}
for each $\ep>0$,
then taking $V_f(\om)=V_{f,\ep}^2(\om)$ measurable and $\ep'=\ep/2$ for a fixed $\ep>0$
we have that
\begin{align*}
\frac{\var(f_{\sg^{-n}(\om)})}{\Lm_{\sg^{-n}(\om)}(f_{\sg^{-n}(\om)})}
\leq
V_{f}(\om)\frac{\var(f_\om)}{\Lm_{\om}(f_\om)}e^{\ep' n}.
\end{align*}
\end{remark}
In the following corollary we establish the existence of an invariant density.
\begin{corollary}\label{cor: exist of unique dens q}
There exists a function $\phi\in\BV_\Om(I)$ and a measurable function $\lm:\Om\to\RR^+$ such that for $m$-a.e. $\om\in\Om$
\begin{align}\label{eq: q invariant density}
\cL_\om \phi_\om = \lm_\om \phi_{\sg(\om)}
\qquad\text{ and }\qquad 
\Lm_\om(\phi_\om)=1.
\end{align}
Furthermore, we have that $\log\lm_\om\in L^1(m)$ and for $m$-a.e. $\om\in\Om$, $\lm_\om\geq \rho_\om$. 
\end{corollary}
\begin{proof}
First fix $\ep>0$ sufficiently small such that Lemma~\ref{lem: exp conv in C+ cone} applies. We note that for any $f\in\BVT^+$ which satisfies \eqref{cond cD2}, Lemma~\ref{LMD l3.6} and Remark~\ref{rem: combining D1 and D2} give that 
\begin{align*}
\var\lt(f_{\om,n}\rt)
&=
\frac{\rho_{\sg^{-n}(\om)}^n}{\Lm_{\om}\lt(\cL_{\sg^{-n}(\om)}^n f_{\sg^{-n}(\om)}\rt)}\var\lt(f_{\sg^{-n}(\om)}
\rt)
\leq 
\frac{\var(f_{\sg^{-n}(\om)})}{\Lm_{\sg^{-n}(\om)}(f_{\sg^{-n}(\om)})}
\leq
V_{f}(\om)\frac{\var(f_\om)}{\Lm_{\om}(f_\om)}e^{\ep n}
\end{align*}	
for all $n\in\NN$ sufficiently large, say for $n\geq N_4(\om)$, and some measurable $V_f:\Om\to(0,\infty)$, where
\begin{align*}
f_{\om,n}:=\frac{f_{\sg^{-n}(\om)}\rho_{\sg^{-n}(\om)}^n}{\Lm_{\om}\lt(\cL_{\sg^{-n}(\om)}^n f_{\sg^{-n}(\om)}\rt)}\in\sC_{\sg^{-n}(\om),+}.
\end{align*}
Thus, Corollary~\ref{cor: exp conv in sup norm} (with $p=-n$ and $V(\om)=V_f(\om)\var(f_\om)/\Lm_\om(f_\om)$) gives that 
$$
(\~\cL_{\sg^{-n}(\om)}^n f_{\om,n})_{n\in\NN}
=
\lt(\frac{\cL_{\sg^{-n}(\om)}^n f_{\sg^{-n}(\om)}}{\Lm_{\om}\lt(\cL_{\sg^{-n}(\om)}^n f\rt)}\rt)_{n\in\NN}
$$ 
forms a Cauchy sequence in $\sC_{\om,+}$, and therefore there must exist some $\phi_{\om,f}\in\sC_{\om,+}$ with 
\begin{align}\label{eq: constr of q}
\phi_{\om,f}:=\lim_{n\to\infty}\frac{\cL_{\sg^{-n}(\om)}^n f_{\sg^{-n}(\om)}}{\Lm_\om\lt(\cL_{\sg^{-n}(\om)}^n f_{\sg^{-n}(\om)}\rt)}.
\end{align}
By construction we have that $\Lm_\om(\phi_{\om,f})=1$.
Now, in view of calculating $\cL_\om \phi_{\om,f}$, we note that \eqref{eq: a* bound 1} (with $N=1$ and $f=\cL_{\sg^{-n}(\om)}^{n} f_{\sg^{-n}(\om)}$) gives that 
\begin{align}\label{eq: 11.21 up bd}
\frac{\Lm_{\sg(\om)}\lt(\cL_{\sg^{-n}(\om)}^{n+1} f_{\sg^{-n}(\om)}\rt)}{\Lm_{\om}\lt(\cL_{\sg^{-n}(\om)}^{n} f_{\sg^{-n}(\om)}\rt)}
\leq 
\frac{\norm{\cL_\om\ind_\om}_\infty\Lm_{\om}\lt(\cL_{\sg^{-n}(\om)}^{n} f_{\sg^{-n}(\om)}\rt)}{\Lm_{\om}\lt(\cL_{\sg^{-n}(\om)}^{n} f_{\sg^{-n}(\om)}\rt)}
=
\norm{\cL_\om\ind_\om}_\infty. 
\end{align}
Lemma~\ref{LMD l3.6} (with $k=1$ and $f=\cL_{\sg^{-n}(\om)}^{n} f_{\sg^{-n}(\om)}$) implies that 
\begin{align*}
\frac{\Lm_{\sg(\om)}\lt(\cL_{\sg^{-n}(\om)}^{n+1} f_{\sg^{-n}(\om)}\rt)}{\Lm_{\om}\lt(\cL_{\sg^{-n}(\om)}^{n} f_{\sg^{-n}(\om)}\rt)}
\geq
\frac{\rho_\om\Lm_{\om}\lt(\cL_{\sg^{-n}(\om)}^{n} f_{\sg^{-n}(\om)}\rt)}{\Lm_{\om}\lt(\cL_{\sg^{-n}(\om)}^{n} f_{\sg^{-n}(\om)}\rt)}
=
\rho_\om,
\end{align*}
and thus, together with \eqref{eq: 11.21 up bd}, we have  
\begin{align}\label{eq: lm geq rho}
\frac{\Lm_{\sg(\om)}\lt(\cL_{\sg^{-n}(\om)}^{n+1} f_{\sg^{-n}(\om)}\rt)}{\Lm_{\om}\lt(\cL_{\sg^{-n}(\om)}^{n} f_{\sg^{-n}(\om)}\rt)}
\in [\rho_\om,\norm{\cL_\om\ind_\om}_\infty].
\end{align}
Thus there must exist a sequence $(n_k)_{k\in\NN}$ along which this ratio converges to some value $\lm_{\om,f}$, that is
\begin{align*}
\lm_{\om,f}:=\lim_{k\to\infty} \frac{\Lm_{\sg(\om)}\lt(\cL_{\sg^{-n_k}(\om)}^{n_k+1} f\rt)}{\Lm_{\om}\lt(\cL_{\sg^{-n_k}(\om)}^{n_k} f\rt)}.
\end{align*}
Hence we have 
\begin{align}\label{eq: Lq=lm q depends on f*}
\cL_\om \phi_{\om,f}
&=
\lim_{k\to\infty} 
\frac{\cL_{\sg^{-n_k}(\om)}^{n_k+1} f}{\Lm_\om\lt(\cL_{\sg^{-n_k}(\om)}^{n_k} f\rt)}
\nonumber\\
&=
\lim_{k\to\infty} 
\frac{\cL_{\sg^{-n_k}(\om)}^{n_k+1} f}{\Lm_{\sg(\om)}\lt(\cL_{\sg^{-n_k}(\om)}^{n_k+1} f\rt)}
\cdot
\frac{\Lm_{\sg(\om)}\lt(\cL_{\sg^{-n_k}(\om)}^{n_k+1} f\rt)}{\Lm_{\om}\lt(\cL_{\sg^{-n_k}(\om)}^{n_k} f\rt)}
=
\lm_{\om,f}\phi_{\sg(\om),f}.		
\end{align}
From \eqref{eq: Lq=lm q depends on f*} it follows that $\lm_{\om,f}$ does not depend on the sequence $(n_k)_{k\in\NN}$, and in fact we have
\begin{align*}
\lm_{\om,f}=\lim_{n\to\infty} \frac{\Lm_{\sg(\om)}\lt(\cL_{\sg^{-n}(\om)}^{n+1} f\rt)}{\Lm_{\om}\lt(\cL_{\sg^{-n}(\om)}^{n} f\rt)},
\end{align*}
and thus,
\begin{align}\label{eq: Lq=lm q depends on f}
\cL_\om \phi_{\om,f}=\lm_{\om,f}\phi_{\sg(\om),f}.
\end{align}
To see that $\phi_{\om,f}$ and $\lm_{\om,f}$ do not depend on $f$, we apply Lemma~\ref{lem: exp conv in C+ cone} (with $p=-n$,  $l=0$, and $V(\om)=\max\set{V_f(\om)\var(f_\om)/\Lm_\om(f_\om), V_h(\om) \var(h_\om)/\Lm_\om(h_\om)}$) to functions $f,h\in\BVT^+$ satisfying \eqref{cond cD2} to get that 
\begin{align}
\Ta_{\om,+}\lt(\phi_{\om,f}, \phi_{\om,h}\rt)
\leq 
\Ta_{\om,+}\lt(\phi_{\om,f}, f_{\om,n}\rt)
+
\Ta_{\om,+}\lt(f_{\om,n}, h_{\om,n}\rt)
+
\Ta_{\om,+}\lt(\phi_{\om,h}, h_{\om,n}\rt)
\leq 
3\Dl\vta^n
\label{eq: Ta estimate for q_om,f}
\end{align}
for each $n\geq N_3(\om)$. Thus, inserting \eqref{eq: Ta estimate for q_om,f} into Lemma~\ref{lem: birkhoff cone contraction} yields
\begin{align*}
\norm{\phi_{\om,f}-\phi_{\om,h}}_\infty
&\leq 
\norm{\phi_{\om,f}}_\infty \lt(e^{\lt(\Ta_{\om,+}(\phi_{\om,f}, \phi_{\om,h})\rt)}-1\rt)
\leq \norm{\phi_{\om,f}}_\infty \lt(e^{3\Dl\vta^n}-1\rt),
\end{align*}
which converges to zero exponentially fast as $n$ tends towards infinity. Thus we must in fact have that $\phi_{\om,f}=\phi_{\om,h}$ for all $f,h$. Moreover, in light of \eqref{eq: Lq=lm q depends on f}, this implies that $\lm_{\om,f}=\lm_{\om,h}$. We denote the common values by $\phi_\om$ and $\lm_\om$ respectively. It follows from \eqref{eq: rho log int} and \eqref{eq: lm geq rho} that 
\begin{align}\label{eq: rho leq lm leq a rho}
0<\rho_\om\leq\lm_\om\leq \norm{\cL_\om\ind_\om}_\infty.
\end{align} 
Measurability of the map $\om\mapsto\lm_\om$ 
follows from the measurability of the sequence 
$$\lt(
\frac{\Lm_{\sg(\om)}\lt(\cL_{\sg^{-n}(\om)}^{n+1}\ind_{\sg^{-n}(\om)}\rt)}
{\Lm_{\om}\lt(\cL_{\sg^{-n}(\om)}^{n}\ind_{\sg^{-n}(\om)}\rt)}
\rt)_{n\in\NN}.
$$ 
The $\log$-integrability of $\lm_\om$ follows from the $\log$-integrability of $\rho_\om$ and \eqref{eq: rho leq lm leq a rho}.
Finally, measurability of the maps $\om\mapsto\inf \phi_\om$ and $\om\mapsto\norm{\phi_\om}_\infty$ follows from the fact that we have 
\begin{align*}
\phi_\om=\lim_{n\to\infty}\frac{\cL_{\sg^{-n}(\om)}^n\ind_{\sg^{-n}(\om)}}{\Lm_\om(\cL_{\sg^{-n}(\om)}^n\ind_{\sg^{-n}(\om)})},
\end{align*}
which is a limit of measurable functions, and thus finishes the proof. 

\end{proof}
\begin{remark}\label{rem: def of lm^k}
For each $k\in\NN$, inducting on \eqref{eq: Lq=lm q depends on f} for any $f\in\BVT^+$ satisfying \eqref{cond cD2}  yields 
\begin{align}
\cL_\om^k(\phi_\om)
&=
\lim_{n\to\infty}
\frac{\cL_{\sg^{-n}(\om)}^{n+k} f_{\sg^{-n}(\om)}}
{\Lm_\om(\cL_{\sg^{-n}(\om)}^n f_{\sg^{-n}(\om)})}
\nonumber\\
&=
\lim_{n\to\infty}
\frac{\cL_{\sg^{-n}(\om)}^{n+k}f_{\sg^{-n}(\om)}}
{\Lm_{\sg^k(\om)}(\cL_{\sg^{-n}(\om)}^{n+k} f_{\sg^{-n}(\om)})}
\cdot 
\frac{\Lm_{\sg^k(\om)}(\cL_{\sg^{-n}(\om)}^{n+k} f_{\sg^{-n}(\om)})}
{\Lm_\om(\cL_{\sg^{-n}(\om)}^n f_{\sg^{-n}(\om)})}
\nonumber\\
&=
\phi_{\sg^k(\om)}\cdot 
\lim_{n\to\infty} 
\frac{\Lm_{\sg^k(\om)}(\cL_{\sg^{-n}(\om)}^{n+k} f_{\sg^{-n}(\om)})}
{\Lm_\om(\cL_{\sg^{-n}(\om)}^n f_{\sg^{-n}(\om)})}.
\label{eq: def of lm_f^k limits}
\end{align}
The final limit in \eqref{eq: def of lm_f^k limits} telescopes to give us 
\begin{align*}
\lim_{n\to\infty} 
\frac{\Lm_{\sg^k(\om)}(\cL_{\sg^{-n}(\om)}^{n+k} f_{\sg^{-n}(\om)})}
{\Lm_\om(\cL_{\sg^{-n}(\om)}^n f_{\sg^{-n}(\om)})}
&=
\lim_{n\to\infty} 
\frac{\Lm_{\sg(\om)}(\cL_{\sg^{-n}(\om)}^{n+1} f_{\sg^{-n}(\om)})}
{\Lm_\om(\cL_{\sg^{-n}(\om)}^n f_{\sg^{-n}(\om)})}
\cdots
\frac{\Lm_{\sg^k(\om)}(\cL_{\sg^{-n}(\om)}^{n+k} f_{\sg^{-n}(\om)})}
{\Lm_{\sg^{k-1}(\om)}(\cL_{\sg^{-n}(\om)}^{n+k-1} f_{\sg^{-n}(\om)})}
\nonumber\\
&=
\lm_\om\lm_{\sg(\om)}\cdots\lm_{\sg^{k-1}(\om)},
\end{align*}
For each $k\geq 1$ we denote 
\begin{align}\label{eq: def of lm_f^k}
\lm_\om^k:=\lm_\om\lm_{\sg(\om)}\cdots\lm_{\sg^{k-1}(\om)}.
\end{align}
Rewriting \eqref{eq: def of lm_f^k limits} gives
\begin{align*}
\cL_\om^k\phi_\om=\lm_\om^k \phi_{\sg^k(\om)}.
\end{align*}
\end{remark}
The following proposition shows that the density $\phi_\om$ coming from Corollary~\ref{cor: exist of unique dens q} is in fact supported on the set $D_{\om,\infty}$.
\begin{proposition}\label{prop: lower bound for density}
For $m$-a.e. $\om\in\Om$ we have that
\begin{align*}
\inf_{D_{\om,\infty}} \phi_\om >0. 
\end{align*} 
\end{proposition}
\begin{proof}

First we note that since $\Lm_\om(\phi_\om)=1>0$ for $m$-a.e. $\om\in\Om$, using the definition of $\Lm_\om$ \eqref{eq: def of Lm}, we must in fact have that 
$$
\inf_{D_{\sg^n(\om),n}} \cL_\om^n(\phi_\om)>0
$$ 
for $n\in\NN$ sufficiently large, which, in turn implies that 
\begin{align}\label{eq: q pos on X_n}
\inf_{X_{\om,n-1}}\phi_\om>0
\end{align} 
for all $n\in\NN$ sufficiently large.
Next, for $m$-a.e. $\om\in\Om$ and all $n\in\NN$ we use \eqref{eq: q invariant density} to see that 
\begin{align}
\inf_{D_{\om,\infty}}\phi_\om
&=
\lt(\lm_{\sg^{-n}(\om)}^n\rt)^{-1}\inf_{D_{\om,\infty}}\cL_{\sg^{-n}(\om)}^n\phi_{\sg^{-n}(\om)}
\nonumber\\
&\geq 
\inf_{X_{\sg^{-n}(\om),n-1}} \phi_{\sg^{-n}(\om)} \lt(\lm_{\sg^{-n}(\om)}^n\rt)^{-1}\inf_{D_{\om,\infty}}\cL_{\sg^{-n}(\om)}^n\ind_{\sg^{-n}(\om)}.
\end{align}	
As the right hand side is strictly positive for all $n\in\NN$ sufficiently large by \eqref{eq: q pos on X_n}, \eqref{eq: rho leq lm leq a rho}, and \eqref{eq: L1 pos on D infty}, we are finished.
\end{proof}
\begin{lemma}\label{lem: Lm is a linear functional}
For each $\om\in\Om$ the functional $\Lm_\om$ is linear, positive, and enjoys the property that 
\begin{align}\label{eq: Lam equivariance}
\Lm_{\sg(\om)}(\cL_\om f)=\lm_\om \Lm_\om(f)
\end{align}
for each $f\in\BV(I)$. Furthermore, for each $\om\in\Om$ we have that 
\begin{align}\label{eq: lam functional integral def}
\lm_\om=\rho_\om=\Lm_{\sg(\om)}(\cL_\om\ind_\om).
\end{align}
\end{lemma}
\begin{proof}
Positivity of $\Lm_\om$ follows from the initial properties of $\Lm_\om$ shown in Observation~\ref{obs: properties of Lm}. To prove the remaining claims we first prove a more robust limit characterization of $\Lm_\om$ than the one given by its definition, \eqref{eq: def of Lm}. 
Now, for any two sequences of points $(x_n)_{n\geq 0}$ and $(y_n)_{n\geq 0}$ with $x_n,y_n\in D_{\sg^n(\om), n}$ we have 
\begin{align}
\lim_{n\to\infty}
\absval{
	\frac{\cL_\om^n f}{\cL_\om^n \ind_\om}(x_n)
	-
	\frac{\cL_\om^n f}{\cL_\om^n \ind_\om}(y_n)
}
&=
\lim_{n\to\infty}
\absval{\frac{\cL_\om^n f}{\cL_\om^n \ind_\om}(y_n)}
\cdot
\absval{
	\frac
	{\cL_\om^n f(x_n) }
	{\cL_\om^n \ind_\om(x_n)}
	\cdot 
	\frac
	{\cL_\om^n \ind_\om(y_n)}
	{\cL_\om^n f(y_n)}
	-
	1
}
\nonumber\\
&\leq 
\norm{f}_\infty
\limsup_{n\to\infty}
\absval{\exp\lt(\Ta_{\sg^n(\om),+}\lt(\~\cL_\om^n f, \~\cL_\om^n\ind_\om\rt)\rt)-1}
= 0.\label{eq: calc to show Lm does not need inf}
\end{align} 
Thus, we have shown that we may remove the infimum from \eqref{eq: def of Lm}, which defines the functional $\Lm_\om$, that is now we may write   
\begin{align}\label{eq: lim characterization of Lm}
\Lm_\om(f)= \lim_{n\to\infty}\frac{\cL_\om^n f}{\cL_\om^n\ind_\om}(x_n)
\end{align}
for all $f\in\sC_{\om,+}$ and all $x_n \in D_{\sg^n(\om),n}$. Moreover, this identity also shows that the functional $\Lm_\om$ is linear. 
To extend \eqref{eq: lim characterization of Lm} to all of $\BV(I)$, we simply write $f=f_+ -f_-$ so that $f_+, f_-\in\sC_{\om,+}$ for each $f\in\BV(I)$ so that we have 
\begin{align}\label{eq: lim char for all BV}
\Lm_\om(f)
=
\Lm_\om(f_+)-\Lm_\om(f_-)
=
\lim_{n\to\infty}\frac{\cL_\om^n f_+}{\cL_\om^n\ind_\om} - \lim_{n\to\infty}\frac{\cL_\om^n f_-}{\cL_\om^n\ind_\om}
=
\lim_{n\to\infty}\frac{\cL_\om^n f}{\cL_\om^n\ind_\om}.
\end{align} 
To prove \eqref{eq: Lam equivariance} and \eqref{eq: lam functional integral def} we use \eqref{eq: lim char for all BV} to note that 
\begin{align}\label{eq: Lam L f limit equality}
\Lm_{\sg(\om)}(\cL_\om f) 
&=\lim_{n\to \infty}\frac{\cL_\om^{n+1} f}{\cL_{\sg(\om)}^{n}\ind_{\sg(\om)}}(x_{n+1})
\nonumber\\
&=\lim_{n\to \infty}\frac{\cL_\om^{n+1} f}{\cL_\om^{n+1}\ind_\om}(x_{n+1})
\cdot \frac{\cL_\om^{n+1} \ind_\om}{\cL_{\sg(\om)}^{n}\ind_{\sg(\om)}}(x_{n+1})
\nonumber\\
&=\Lm_\om(f)\cdot\Lm_{\sg(\om)}(\cL_\om\ind_\om).		
\end{align}
Considering the case where $f=\phi_\om$ in \eqref{eq: Lam L f limit equality} in conjunction with the fact that $\Lm_\om(\phi_\om)=1$ and $\cL_\om \phi_\om=\lm_\om \phi_{\sg(\om)}$ gives
\begin{align}\label{eq: lam functional integral def in proof}
\rho_\om:=
\Lm_{\sg(\om)}(\cL_\om \ind_\om)
=
\Lm_\om(\phi_\om)\Lm_{\sg(\om)}(\cL_\om \ind_\om)
=
\Lm_{\sg(\om)}(\cL_\om \phi_\om)
=
\Lm_{\sg(\om)}(\lm_\om \phi_{\sg(\om)})
=
\lm_\om,
\end{align}
which finishes the proof.
\end{proof}
\begin{remark}\label{rem: log integ lm and rho}
In light of the fact that $\log\rho_\om\in L^1(m)$ by \eqref{eq: rho log int}, Lemma~\ref{lem: Lm is a linear functional} implies that 
\begin{align}\label{eq: lm log int}
\log\lm_\om\in L^1(m).
\end{align}
\end{remark}
In the next lemma we are finally able to show that the functional $\Lm_\om$ can be thought of as Borel probability measure for the random open system. 
\begin{lemma}\label{lem: Lm is conf meas}
There exists a non-atomic Borel probability measure $\nu_\om$ on $I_\om$ such that 
$$
\Lm_\om(f)=\int_{I_\om}f \, d\nu_\om
$$
for all $f\in\BV(I)$. Consequently, we have that  
\begin{align}\label{eq: conformal measure property}
\nu_{\sg(\om)}(\cL_\om f)=\lm_\om\nu_\om(f)
\end{align}
for all $f\in\BV(I)$.
Furthermore, we have that $\supp(\nu_\om)\sub X_{\om,\infty}$.
\end{lemma}
\begin{proof}
The proof that the functional $\Lm_\om$ can be equated to a non-atomic Borel probability measure $\nu_\om$ goes exactly like the proof of Lemma~4.3 in \cite{LMD}. Thus, we have only to prove that $\supp(\nu_\om)\sub X_{\om,\infty}$.
To that end, suppose $f\in L^1(\nu_{\om,0})$ with $f\equiv0$ on $X_{\om,n-1}$. Then
\begin{align*}
\int_{I}f \,d\nu_{\om}
&=\lt(\lm_\om^n\rt)^{-1}\int_{I}\cL_{\om}^n(f)\,d\nu_{\sg^n(\om)}
=\lt(\lm_\om^n\rt)^{-1}\int_{I}\cL_{\om,0}^n(\hat X_{\om,n-1}\cdot f)\,d\nu_{\sg^n(\om)}
=0.
\end{align*}
As $0<\lm_\om^n<\infty$ for each $n\in\NN$, we must have that $\supp(\nu_{\om})\sub X_{\om,\infty}$.
\end{proof}
\begin{remark}\label{rem: conformal meas prop}
We can immediately see, cf. \cite{denker_existence_1991, atnip_critically_2020}, that the conformality of the family $(\nu_\om)_{\om\in\Om}$ produced in Lemma~\ref{lem: Lm is conf meas} enjoys the property that for each $n\geq 1$ and each set $A$ on which $T_\om^n\rvert_A$ is one-to-one we have 
\begin{align*}
\nu_{\sg^n(\om)}(T_\om^n(A))=\lm_\om^n\int_A e^{-S_{n,T}(\vp_\om)} \, d\nu_\om.
\end{align*}
In particular, this gives that for each $n\geq 1$ and each $Z\in\cZ_\om^{(n)}$ we have 
\begin{align*}
\nu_{\sg^n(\om)}(T_\om^n(Z))=\lm_\om^n\int_Z e^{-S_{n,T}(\vp_\om)} \, d\nu_\om.
\end{align*} 	
\end{remark}

\begin{remark}\label{rem: props of norm op}
In light of Lemmas~\ref{lem: Lm is a linear functional} and \ref{lem: Lm is conf meas}, the normalized operator $\~\cL_\om$ is given by $\~\cL_\om(\spot):=\rho_\om^{-1}\cL_\om(\spot)=\lm_\om^{-1}\cL_\om(\spot)$. Furthermore, $\~\cL_\om$ enjoys the properties 
\begin{align*}
\~\cL_\om \phi_\om = \phi_{\sg(\om)}
\qquad\text{ and }\qquad
\nu_{\sg(\om)}\lt(\~\cL_\om(f)\rt)=\nu_\om(f)
\end{align*}
for all $f\in\BV(I)$. 
\end{remark}

For each $\om\in\Om$ we may now define the measure $\mu_\om\in\cP(I)$ by 
\begin{align}\label{eq: def of mu_om}
\mu_\om(f):=\int_{X_{\om,\infty}} f\phi_\om \,d\nu_\om,  \qquad f\in L^1(\nu_\om).
\end{align}
Lemma~\ref{lem: Lm is conf meas} and Proposition~\ref{prop: lower bound for density} together show that, for $m$-a.e. $\om\in\Om$, $\mu_\om$ is a non-atomic Borel probability measure with $\supp(\mu_\om)\sub X_{\om,\infty}$, which is absolutely continuous with respect to $\nu_\om$. Furthermore, in view of Proposition~\ref{prop: lower bound for density}, for $m$-a.e. $\om\in\Om$, we may now define the fully normalized transfer operator $\sL_\om:\BV(I)\to\BV(I)$ by 
\begin{align}\label{eq: def fully norm tr op}
\sL_\om f:= \frac{1}{\phi_{\sg(\om)}}\~\cL_\om(f\phi_\om) = \frac{1}{\lm_\om \phi_{\sg(\om)}}\cL_\om (f\phi_\om), 
\qquad f\in \BV(I).
\end{align}\index{$\sL_\om$}
As an immediate consequence of Remark~\ref{rem: props of norm op} and \eqref{eq: def fully norm tr op}, we get that 
\begin{align}\label{eq: fully norm op fix ind}
\sL_\om\ind_\om =\ind_{\sg(\om)}.
\end{align}
We end this section with the following proposition which shows that the family $(\mu_\om)_{\om\in\Om}$ of measures is $T$-invariant.
\begin{proposition}\label{prop: mu_om T invar}
The family $(\mu_\om)_{\om\in\Om}$ defined by \eqref{eq: def of mu_om} is $T$-invariant in the sense that 
\begin{align}\label{eq: mu_om T invar}
\int_{X_{\om,\infty}} f\circ T_\om \, d\mu_\om 
=
\int_{X_{\sg(\om),\infty}} f \, d\mu_{\sg(\om)}
\end{align}
for $f\in L^1(\mu_{\sg(\om)})=L^1(\nu_{\sg(\om)})$.	
\end{proposition} 
The proof of Proposition~\ref{prop: mu_om T invar} goes just like the proof of Proposition~8.11 of \cite{AFGTV20}, and has thus been omitted.

\section{Decay of correlations}\label{sec: dec of cor}
We are now ready to show that images under the normalized transfer operator $\~\cL_\om$ converge exponentially fast to the invariant density as well as the fact that the invariant measure $\mu_\om$ established in Section~\ref{sec: conf and inv meas} satisfies an exponential decay of correlations. Furthermore, we show that the families $(\nu_\om)_{\om\in\Om}$ and $(\mu_\om)_{\om\in\Om}$ are in fact random measures and then introduce the RACCIM $\eta$ supported on $\cI$. 

To begin this section we state a lemma which shows that the $\BV$ norm of the invariant density $\phi_\om$ does not grow too much along a $\sg$-orbit of fibers by providing a measurable upper bound. In fact, we show that the $\BV$ norm of $\phi_\om$ is tempered. As the proof of the following lemma is the same as the proof of Lemma 8.5 in the closed dynamical setting of \cite{AFGTV20}, its proof is omitted. 
\begin{lemma}\label{lem: BV norm q om growth bounds}
For all $\dl>0$ there exists a tempered measurable random constant $C(\om,\dl)>0$\index{$C(\om,\dl)$} such that for all $k\in\ZZ$ and $m$-a.e. $\om\in\Om$ we have 
\begin{align*}
\norm{\phi_{\sg^k(\om)}}_{\BV}
=
\norm{\phi_{\sg^k(\om)}}_{\infty}+\var(\phi_{\sg^k(\om)})
\leq
C(\om,\dl)e^{\dl|k|}.
\end{align*} 
Consequently, we have that $\phi\in\BVT$.
\end{lemma}
We are now able to prove the following theorem which completes the proof of Theorem~\ref{main thm: quasicompact}. 
\begin{theorem}\label{thm: exp convergence of tr op}
There exist measurable, $m$-a.e. finite functions $C, D:\Om\to\RR$ and $\kp<1$ such that for each $f\in\BVT$, each $n\in\NN$, each $|p|\leq n$, and $m$-a.e. $\oio$ we have
\begin{align}\label{eq: exp conv norm op}
\norm{\~\cL_{\sg^p(\om)}^n f_{\sg^p(\om)} - \nu_{\sg^p(\om)}(f_{\sg^p(\om)})\phi_{\sg^{p+n}(\om)}}_\infty
\leq 
D(\om) (\var f_\om)^{C(\om)}\lt\|f_{\sg^p\om}\rt\|_\infty\kp^n
\end{align}
and 
\begin{align}\label{eq: exp conv full norm op}
\norm{\sL_{\sg^p(\om)}^n f_{\sg^p(\om)} - \mu_{\sg^p(\om)}(f_{\sg^p(\om)})\ind_{\sg^{p+n}(\om)}}_\infty
\leq 
D(\om) (\var f_\om)^{C(\om)}\lt\|f_{\sg^p\om}\rt\|_\infty\kp^n.
\end{align}
\end{theorem}
\begin{proof}
First suppose that $f\in\BVT^+$. 
    In light of Lemma \ref{lem: BV norm q om growth bounds}, we see that $\phi\in\BVT$, and thus taking $V(\om)=\max\{V_f(\om)\var(f_\om), V_\phi(\om)\}$. 
    We recall that the measurable function $N_3(\om)$ from the proof of Lemma \ref{lem: exp conv in C+ cone}, specifically \eqref{eq: coose N_1 ge log V} and \eqref{eq: def of N_3}, is given by
    \begin{align*}
        N_3(\om)=\max\lt\{\ep^{-1}\log V(\om), N_2(\om), S_*\ep^{-1} \rt\},
    \end{align*} 
    and so we can find a measurable function $C_3:\Om\to(0,\infty)$ independent\footnote{Note that from the definition of $N_3$, only $V(\om)$ depends upon $f$.} of $f$ such that 
    \begin{align}\label{eq new ineq 4}
        N_3(\om)\leq C_3(\om)\cdot\log \var(f_\om).
    \end{align}
It follows from Corollary \ref{cor: exp conv in sup norm} (with $h_\om=\phi_\om$ and $l=0$) and the definition of $\phi$ that for all $n\geq N_3(\om)$ we have 
\begin{align}
\norm{\~\cL_{\sg^p(\om)}^nf_{\sg^p(\om)}-\nu_{\sg^p(\om)}(f_{\sg^p(\om)})\phi_{\sg^{p+n}(\om)}}_\infty
&\leq 
|\nu_{\sg^p(\om)}(f_{\sg^p(\om)})|\norm{\phi_{\sg^{p+n}(\om)}}_\infty\lt(e^{\Dl\vta^n}-1\rt)
\nonumber\\
&
\leq 
\norm{f_{\sg^p(\om)}}_\infty\norm{\phi_{\sg^{p+n}(\om)}}_\infty\~\kp^n
\label{eq: exp conv eq 2}
\end{align}
for some\footnote{Any $\~\kp>\vta$ will work for $n$ sufficiently large.}  $\~\kp\in(0,1)$.
\\
Now to deal with $n\leq N_3(\om)$, we use \eqref{rho^n rough up and low bdd} to see that 
\begin{align}
    \norm{\~\cL_{\sg^p(\om)}^nf_{\sg^p(\om)}-\nu_{\sg^p(\om)}(f_{\sg^p(\om)})\phi_{\sg^{p+n}(\om)}}_\infty
    &=
    (\lm_{\sg^p\om}^n)^{-1}\lt\|\cL_{\sg^p\om}^n\lt(f_{\sg^p\om}-\nu_{\sg^p\om}(f_{\sg^p\om})\phi_{\sg^p\om}\rt)\rt\|_\infty
    \nonumber\\
    &\leq 
    (\lm_{\sg^p\om}^n)^{-1}\lt\|\cL_{\sg^p\om}^n \ind\rt\|_\infty\lt\|f_{\sg^p\om}-\nu_{\sg^p\om}(f_{\sg^p\om})\phi_{\sg^p\om}\rt\|_\infty
    \nonumber\\
    &\leq 
    \frac{\|g_{\sg^p\om}^{(n)}\|_\infty}{\lm_{\sg^p\om}^n}\#\cZ_{\sg^p\om}^{(n)}\lt\|f_{\sg^p\om}\rt\|_\infty\lt(1+\lt\|\phi_{\sg^p\om}\rt\|_\infty\rt).
    \label{eq new ineq 0}
\end{align}
In light of Lemma \ref{lem: Lm is a linear functional} we have that $\lm_\om=\rho_\om$ and so it follows from Remark \ref{rem: further props of Q assumps} (as a consequence of our assumption \eqref{cond Q1}) that
\begin{align*}
    \lim_{n\to\infty}\frac{1}{n}\log\frac{\|g_{\sg^p\om}^{(n)}\|_\infty}{\lm_{\sg^p\om}^n}<0.
\end{align*}
Thus, for each $\dl>0$ (sufficiently small) there exists a measurable function $C_{g,\dl}:\Om\to\RR$ such that for each $n\in\NN$ 
\begin{align}\label{eq new ineq 1}
     \frac{\|g_{\sg^p\om}^{(n)}\|_\infty}{\lm_{\sg^p\om}^n}
    \leq 
    C_{g,\dl}(\om)e^{-n\dl}.
\end{align}
Furthermore, it follows from \eqref{LIP} and the Birkhoff Ergodic Theorem 
that for each $\dl>0$ there exists a measurable function $C_{\cZ,\dl}:\Om\to\RR$ such that for each $n\in\NN$ 
\begin{align}\label{eq new ineq 2}
    \#\cZ_{\sg^p\om}^{(n)}
    \leq 
    C_{\cZ,\dl}(\om)e^{n(\int_\Om\log\#\cZ_\om\, dm(\om)+\dl)}.
\end{align}
Combining \eqref{eq new ineq 1} and \eqref{eq new ineq 2} 
we have that 
\begin{align}\label{eq new ineq 3}
    \frac{\|g_{\sg^p\om}^{(n)}\|_\infty}{\lm_{\sg^p\om}^n}\#\cZ_{\sg^p\om}^{(n)}
    &\leq 
    C_{g,\dl}(\om)C_{\cZ,\dl}(\om) e^{n\int_\Om\log\#\cZ_\om\, dm(\om)}
    \nonumber\\
    &\leq  
    C_{g,\dl}(\om)C_{\cZ,\dl}(\om) e^{N_3(\om)\int_\Om\log\#\cZ_\om\, dm(\om)}
\end{align}
    Thus, using \eqref{eq new ineq 4} we may upper bound \eqref{eq new ineq 3} by writing 
    \begin{align}\label{eq new ineq 5}
    \frac{\|g_{\sg^p\om,0}^{(n)}\|_\infty}{\lm_{\sg^p\om}^n}\#\cZ_{\sg^p\om}^{(n)}
    &\leq  
    C_{g,\dl}(\om)C_{\cZ,\dl}(\om) e^{N_3(\om)\int_\Om\log\#\cZ_\om\, dm(\om)}
    \~\kp^{-N_3(\om)}\~\kp^n
    \nonumber\\
    &\leq C_{g,\dl}(\om)C_{\cZ,\dl}(\om) e^{C_3(\om)\log\var f_\om\int_\Om\log\#\cZ_\om\, dm(\om)}
    \~\kp^{-C_3(\om)\log\var f_\om}\~\kp^n
    \nonumber\\
    &= C_{g,\dl}(\om)C_{\cZ,\dl}(\om) (\var f_\om)^{C_3(\om)(\int_\Om\log\#\cZ_\om\, dm(\om)-\log\~\kp)}
    \~\kp^n.
\end{align}
    Inserting \eqref{eq new ineq 5} into \eqref{eq new ineq 0} and using the temperedness of $\norm{\phi_\om}_\infty$, a consequence of Lemma~\ref{lem: BV norm q om growth bounds}, gives that 
    \begin{align}\label{eq new ineq 6}
    &\norm{\~\cL_{\sg^p(\om)}^nf_{\sg^p(\om)}-\nu_{\sg^p(\om)}(f_{\sg^p(\om)})\phi_{\sg^{p+n}(\om)}}_\infty
    \nonumber\\
    &\quad
    \leq 
    C_{g,\dl}(\om)C_{\cZ,\dl}(\om) (\var f_\om)^{C_3(\om)(\int_\Om\log\#\cZ_\om\, dm(\om)-\log\~\kp)}
    \lt\|f_{\sg^p\om}\rt\|_\infty\lt(1+\lt\|\phi_{\sg^p\om}\rt\|_\infty\rt)\~\kp^n
    \nonumber\\
    &\quad
    \leq C(\om,\dl)C_{g,\dl}(\om)C_{\cZ,\dl}(\om) (\var f_\om)^{C_3(\om)(\int_\Om\log\#\cZ_\om\, dm(\om)-\log\~\kp)}\lt\|f_{\sg^p\om}\rt\|_\infty e^{2n\dl}\norm{\phi_\om}_\infty\~\kp^n
    \nonumber\\
    &\quad
    \leq C(\om,\dl)C_{g,\dl}(\om)C_{\cZ,\dl}(\om) (\var f_\om)^{C_3(\om)(\int_\Om\log\#\cZ_\om\, dm(\om)-\log\~\kp)}\lt\|f_{\sg^p\om}\rt\|_\infty\norm{\phi_\om}_\infty\kp^n
    \nonumber\\
    &\quad
    \leq B(\om) (\var f_\om)^{C(\om)}\lt\|f_{\sg^p\om}\rt\|_\infty\kp^n,
\end{align}
where here we have fixed $\dl>0$ sufficiently small such that 
\begin{align*}
e^{2\dl}\~\kp=:\kp<1,
\end{align*}
and we have set 
$$
    B(\om)=C(\om,\dl)C_{g,\dl}(\om)C_{\cZ,\dl}(\om)\norm{\phi_\om}_\infty
$$ 
    and 
$$
    C(\om)=C_3(\om)\lt(\int_\Om\log\#\cZ_\om\, dm(\om)-\log\~\kp\rt).
$$
Now, to extend \eqref{eq new ineq 6} to all of $\BVT$ we write a function $f\in\BVT^+$ as $f=f_+-f_-$, where $f_+,f_-\in\BVT^+$. Applying the triangle inequality and using \eqref{eq new ineq 6} twice gives
\begin{align*}
&\norm{\~\cL_{\sg^p(\om)}^n f_{\sg^p(\om)} - \nu_{\sg^p(\om)}(f_{\sg^p(\om)})\phi_{\sg^{p+n}(\om)}}_\infty
\leq
4B(\om) (\var f_\om)^{C(\om)}\lt\|f_{\sg^p\om}\rt\|_\infty\kp^n.
\end{align*}
Setting $D(\om):=4B(\om)$ finishes the proof of \eqref{eq: exp conv norm op}. To prove the second claim concerning $\sL$ follows easily from the first claim in a similar fashion as in the proof of Theorem 10.4 of \cite{AFGTV20}.
\end{proof}

From the previous result we easily deduce that the invariant measure $\mu$ satisfies an exponential decay of correlations. The following theorem, whose proof is exactly the same as Theorem~11.1 of \cite{AFGTV20}, completes the proof of Theorem~\ref{main thm: exp dec of corr}.
\begin{theorem}\label{thm: dec of cor}
For $m$-a.e. every $\om\in\Om$, every $n\in\NN$, every $|p|\leq n$, every $f\in L^1(\mu)$, and every $h\in \BVT$ we have 
\begin{align*}
\absval{
	\mu_{\tau}
	\lt(\lt(f_{\sg^{n}(\tau)}\circ T_{\tau}^n\rt)h_{\tau} \rt)
	-
	\mu_{\sg^{n}(\tau)}(f_{\sg^{n}(\tau)})\mu_{\tau}(h_{\tau})
}
\leq 
D(\om) \norm{f_{\sg^n(\tau)}}_{L^1(\mu_{\sg^n(\tau)})}(\var h_\tau)^{C(\om)}\lt\|h_{\tau}\rt\|_\infty\kp^n,
\end{align*} 
where $\tau=\sg^p(\om)$. 
\end{theorem}

\begin{remark}\label{rem: strong lim char of nu}
Note that Theorem~\ref{thm: exp convergence of tr op} implies a stronger limit characterization of the measure $\nu_\om$ than what is concluded in \eqref{eq: lim characterization of Lm}. Indeed, Theorem~\ref{thm: exp convergence of tr op} implies that 
\begin{align*}
\nu_\om(f)=\lim_{n\to\infty}\frac{\cL_\om^n f(x_n)}{\cL_\om^n\ind_\om (y_n)}
\end{align*}
for any pair of sequences $(x_n)_{n\in\NN}$ and $(y_n)_{n\in\NN}$ with $x_n,y_n\in D_{\sg^n(\om),\infty}$, which further implies that 
\begin{align}\label{eq: improved limit characterization}
\nu_\om(f)
=\lim_{n\to\infty}\frac{\norm{\cL_\om^n f}_\infty}{\inf\cL_\om^n\ind_\om}
=\lim_{n\to\infty}\frac{\inf\cL_\om^n f}{\norm{\cL_\om^n\ind_\om}_\infty}.
\end{align}
Furthermore, the same holds for $\nu_{\om,0}$; see Lemma 9.2 and Proposition 10.4 of \cite{AFGTV20}.
\end{remark}
We now address the uniqueness of the families of measures $\nu=(\nu_\om)_{\om\in\Om}$ and $\mu=(\mu_\om)_{\om\in\Om}$ as well as the invariant density $\phi$.
\begin{proposition}\label{prop: uniqueness of nu and mu}
\,\newline
\begin{enumerate}
\item 
The family $\nu=(\nu_\om)_{\om\in\Om}$ is a random probability measure which is uniquely determined by \eqref{eq: conformal measure property}. 

\item The global invariant density $\phi\in\BVT^+$ produced in Corollary~\ref{cor: exist of unique dens q} is the unique element of $L^1(\nu)$ (modulo $\nu$) such that 
\begin{align*}
\~\cL_\om \phi_\om=\phi_{\sg(\om)}.
\end{align*} 	
\item The family $\mu=(\mu_\om)_{\om\in\Om}$ is a unique random $T$-invariant probability measure which is absolutely continuous with respect to $\nu$.
\end{enumerate}
\end{proposition}
\begin{proof}
The fact that the family $(\nu_\om)_{\om\in\Om}$ is a random measure as in Definition~\ref{def: random prob measures} follows from the limit characterization given in \eqref{eq: improved limit characterization}, as we have that $\nu_\om$ is a limit of measurable functions. 
Indeed, for every interval $J\subset I$, the measurability of the function $\om \mapsto \nu_\om(J)$ follows from the fact that it is given by the limit of measurable functions by \eqref{eq: improved limit characterization} applied to the characteristic function $f_\om=\ind_J$. 
Since $\sB$ is generated by intervals, $\om \mapsto \nu_\om(B)$ is measurable for every $B\in \sB$. Furthermore, $\nu_\om$ is a Borel probability measure for $m$-a.e. $\om \in \Om$ from Proposition~\ref{lem: Lm is conf meas}.

The remainder of the proof Proposition~\ref{prop: uniqueness of nu and mu} follows along exactly like the proofs of Propositions~9.4 and 10.4 of \cite{AFGTV20}, and is therefore left to the reader.
\end{proof}
The proof of the  following proposition is the same as the proof of Proposition 4.7 of \cite{mayer_distance_2011}, and so it is omitted.
\begin{proposition}
The random $T$-invariant probability measure $\mu$ defined in \eqref{eq: def of mu_om} is ergodic.
\end{proposition}
In the following lemma we establish the existence of the unique random absolutely continuous conditionally invariant probability measure $\eta$ with density in $\BV_\Om$.
\begin{lemma}\label{lem: raccim eta}
The random measure $\eta\in\cP_\Om(\Om\times I)$, whose disintegrations are given by 
\begin{align*}
\eta_\om(f):=\frac{\nu_{\om,0}(f\cdot\ind_\om\cdot \phi_\om)}{\nu_{\om,0}(\ind_\om\cdot \phi_\om)},
\end{align*} 
is the unique random absolutely continuous (with respect to $\nu_0$) conditionally invariant probability measure supported on $\cI$ with fiberwise density of bounded variation. 
\end{lemma}
\begin{proof}
The fact that $\eta$ is an RACCIM follows from Lemma~\ref{LMD lem 1.1.1} and uniqueness follows from Proposition~\ref{prop: uniqueness of nu and mu}. 
\end{proof}

\begin{remark}
We note that \eqref{eq: def of al_om} and \eqref{eq al_om measure part 2} together with Lemma \ref{lem: Lm is a linear functional} implies the following useful property of an RACCIM $\eta$: for each $n\geq 1$ and $m$-a.e. $\oio$ we have 
\begin{align*}
    \eta_{\om}(X_{\om,n})
    =
    \frac{\lm_\om^n}{\lm_{\om,0}^n}.
\end{align*}
\end{remark}

As a corollary of Theorem~\ref{thm: exp convergence of tr op}, the following results gives the exponential convergence of the closed conformal measure $\nu_{\om,0}$ conditioned on the survivor set to the RACCIM $\eta_\om$. 
\begin{corollary}\label{cor: exp conv of eta}
There exist measurable, $m$-a.e. finite functions $C, D:\Om\to\RR$ and $\kp<1$ such that for $m$-a.e. every $\om\in\Om$, every $n\in\NN$, every $|p|\leq n$, every $f\in\BVT$, and every $A\sub[0,1]$ which is the union of finitely many intervals, we have 
\begin{align*}
\absval{
	\nu_{\sg^p(\om),0}\lt(T_{\sg^p(\om)}^{-n}(A) \, \rvert \, X_{\sg^p(\om),n}\rt)
	- 
	\eta_{\sg^{p+n}(\om)}(A)
}\leq \frac{D(\om)\var(\ind_A)^{C(\om)}}{\nu_{\sg^{p+n}(\om),0}\lt(\ind_{\sg^{p+n}(\om)}\phi_{\sg^{p+n}(\om)}\rt)}\kp^n
\end{align*}
and 
\begin{align*}
\absval{
	\frac{\eta_{\sg^p(\om)}\lt(f_{\sg^p(\om)} \rvert_{X_{\sg^p(\om),n}}\rt)}
	{\eta_{\sg^p(\om)}\lt(X_{\sg^p(\om),n}\rt)}
	- 
	\mu_{\sg^p(\om)}(f_{\sg^p(\om)})
}\leq D(\om)(\var f_\om)^{C(\om)}\lt\|f_{\sg^p\om}\rt\|_\infty\kp^n
\end{align*}
for $m$-a.e. $\om\in\Om$.
\end{corollary}

\begin{proof}
Since $A$ is the union of finitely many intervals, $\ind_A\in\BV(I)$, and thus Lemma~\ref{lem: useful identity} allows us to write 
\begin{align*}
\nu_{\sg^{p}(\om),0}\lt(T_{\sg^p(\om)}^{-n}(A)\cap X_{\sg^p(\om), n}\rt)
&=
\int_{X_{\sg^p(\om),n}}\ind_{\sg^p(\om)}\ind_A\circ T_{\sg^p(\om)}^n \, d\nu_{\sg^p(\om),0}
\\
&=\lt(\lm_{\sg^p(\om),0}^n\rt)^{-1}\int_{I_{\sg^{p+n}(\om)}}\ind_A\cL_{\sg^p(\om)}^n\ind_{\sg^p(\om)}\, d\nu_{\sg^{p+n}(\om),0}
\\
&=
\nu_{\sg^{p+n}(\om),0}\lt(\ind_A\ind_{\sg^{p+n}(\om)} \~\cL_{\sg^p(\om)}^n(\ind_{\sg^p(\om)})\rt).
\end{align*}
So, if $A=I$, then we have 
\begin{align*}
\nu_{\sg^{p}(\om),0}\lt(X_{\sg^p(\om),n}\rt)
=
\nu_{\sg^{p+n}(\om),0}\lt(\ind_{\sg^{p+n}(\om)}\~\cL_{\sg^p(\om)}^n(\ind_{\sg^p(\om)})\rt).
\end{align*}
Thus, we apply \eqref{eq: exp conv norm op} of Theorem~\ref{thm: exp convergence of tr op} together with elementary calculation to get that 
\begin{align*}
&\absval{
	\nu_{\sg^{p}(\om),0}\lt(T_{\sg^p(\om)}^{-n}(A)\, \rvert\, X_{\sg^p(\om),n}\rt)
	-
	\eta_{\sg^{p+n}(\om)}(A)	
}
\\
&\qquad
=
\absval{
	\frac{\nu_{\sg^{p+n}(\om),0}\lt(\ind_A\ind_{\sg^{p+n}(\om)}\~\cL_{\sg^p(\om)}^n(\ind_{\sg^p(\om)})\rt)}
	{\nu_{\sg^{p+n}(\om),0}\lt(\ind_{\sg^{p+n}(\om)}\~\cL_{\sg^p(\om)}^n(\ind_{\sg^p(\om)})\rt)}
	-
	\frac{\nu_{\sg^{p+n}(\om),0}\lt(\ind_A\ind_{\sg^{p+n}(\om)}\phi_{\sg^{p+n}(\om)}\rt)}{\nu_{\sg^{p+n}(\om),0}\lt(\ind_{\sg^{p+n}(\om)}\phi_{\sg^{p+n}(\om)}\rt)}
}
\\
&\qquad
\leq \frac{D(\om)\var(\ind_A)^{C(\om)}}{\nu_{\sg^{p+n}(\om),0}\lt(\ind_{\sg^{p+n}(\om)}\phi_{\sg^{p+n}(\om)}\rt)}\kp^n.
\end{align*}
To see the second claim we note that for $f\in\BVT$
\begin{align*}
\frac{\eta_{\sg^p(\om)}\lt(f_{\sg^p(\om)}\rvert_{X_{\sg^p(\om),n}}\rt)}{\eta_{\sg^p(\om)}\lt(X_{\sg^p(\om),n}\rt)}
&=
\frac
{\nu_{\sg^p(\om),0}\lt(f_{\sg^p(\om)} \hat X_{\sg^p(\om),n}\ind_{\sg^p(\om)}\phi_{\sg^p(\om)}\rt)}
{\nu_{\sg^p(\om),0}\lt(\hat X_{\sg^p(\om),n}\ind_{\sg^p(\om)}\phi_{\sg^p(\om)}\rt)}
\\
&=
\frac
{\nu_{\sg^{p+n}(\om),0}\lt(\ind_{\sg^{p+n}(\om)}\~\cL_{\sg^p(\om)}^n(f_{\sg^p(\om)} \phi_{\sg^p(\om)})\rt)}
{\nu_{\sg^{p+n}(\om),0}\lt(\ind_{\sg^{p+n}(\om)}\~\cL_{\sg^p(\om)}^n(\phi_{\sg^p(\om)})\rt)}
\\
&=
\frac
{\nu_{\sg^{p+n}(\om),0}\lt(\ind_{\sg^{p+n}(\om)}\sL_{\sg^p(\om)}^n(f_{\sg^p(\om)}) \phi_{\sg^{p+n}(\om)}\rt)}
{\nu_{\sg^{p+n}(\om),0}\lt(\ind_{\sg^{p+n}(\om)}\phi_{\sg^{p+n}(\om)}\rt)}
\\
&=
\eta_{\sg^{p+n}(\om)}\lt(\sL_{\sg^p(\om)}^n(f_{\sg^p(\om)})\rt).
\end{align*} 
Thus, applying \eqref{eq: exp conv full norm op} of Theorem~\ref{thm: exp convergence of tr op}, we have 
\begin{align*}
\absval{
	\frac{\eta_{\sg^p(\om)}\lt(f_{\sg^p(\om)}\rvert_{X_{\sg^p(\om),n}}\rt)}{\eta_{\sg^p(\om)}\lt(X_{\sg^p(\om),n}\rt)}
	- 
	\mu_{\sg^p(\om)}(f_{\sg^p(\om)})
}
&=
\absval{
	\eta_{\sg^{p+n}(\om)}\lt(\sL_{\sg^p(\om)}^n(f_{\sg^p(\om)})\rt)
	-
	\mu_{\sg^p(\om)}(f_{\sg^p(\om)})
}
\\
&\leq 
D(\om)(\var f_\om)^{C(\om)}\lt\|f_{\sg^p\om}\rt\|_\infty\kp^n,
\end{align*}
which finishes the proof.
\end{proof}
\begin{remark}
Note that the third claim of Corollary \ref{cor: exp conv of eta} differs from the third claim of Theorem A of \cite{LMD} where they substitute the function $f\in\BV$ with the function $\ind_A$ for $A$ a Borel set. As stated this is not true as this result can not hold for general $A\in\sB$; taking $A=X_{\infty}$ produces a counterexample. 
However, if $A$ is taken as the union of finitely many intervals (as stated in Corollary \ref{cor: exp conv of eta}), then $\ind_A\in\BV(I)$ and the third claim of Theorem A of \cite{LMD} holds. 
\end{remark}

We are now able to complete the proof of Theorem~\ref{main thm: quasicompact}. 
\begin{proof}[Proof of Theorem \ref{main thm: quasicompact}]
    Note that taking $f_\om=\~f$ for each $\oio$ for some $\~f\in\BV(I)$, we must have that $f\in\BVT$ since $\var(f_\om)$ is constant. Also, as noted in the proof of Corollary \ref{cor: exp conv of eta}, if $A\in\sB$ is the union of finitely many intervals, then $\ind_A\in\BV(I)$.  Thus, Theorem \ref{thm: exp convergence of tr op} and Corollary \ref{cor: exp conv of eta} complete the proof of Theorem~\ref{main thm: quasicompact} where we take 
\begin{align*}    
    D_f(\om):=D(\om)\var(f)^{C(\om)}
    \qquad\text{ and }\qquad
    D_A(\om):=\frac{D(\om)\var(\ind_A)^{C(\om)}}{\nu_{\sg^{p+n}(\om),0}\lt(\ind_{\sg^{p+n}(\om)}\phi_{\sg^{p+n}(\om)}\rt)}.
\end{align*}  
\end{proof}

\section{Expected pressures and escape rates}\label{sec: exp press}
We now establish the rate at which mass escapes through the hole with respect to the closed conformal measure $\nu_{\om,0}$ and the RACCIM $\eta_\om$ in terms of the open and closed expected pressures.
\begin{definition}\label{def: expected pressure}\index{expected pressure}
Given a potential $\vp_0$ on the closed system we define the expected pressure of the closed and open systems respectively by 
\begin{align*}
\cEP(\vp_0):=\int_\Om\log\lm_{\om,0}\, dm(\om)
\quad\text{ and }\quad
\cEP(\vp):=\int_\Om\log\lm_\om\, dm(\om). 
\end{align*}\index{$\cEP(\vp_0)$}\index{$\cEP(\vp)$}
\end{definition}
In light of \eqref{eq: lm log int} (and \eqref{cond C1}), we see that the definition of the expected pressures $\cEP(\vp_0)$, $\cEP(\vp)\in\RR$, are well defined.
Since $\log\lm_\om, \log\lm_{\om,0}\in L^1(m)$, Birkhoff's Ergodic Theorem gives that 
\begin{align}\label{eq: BET exp pres hole}
\cEP(\vp)
=\lim_{n\to\infty}\frac{1}{n}\log\lm_\om^n
=\lim_{n\to\infty}\frac{1}{n}\log\lm_{\sg^{-n}(\om)}^n
\end{align} 
and
\begin{align}\label{eq: BET exp pres closed}
\cEP(\vp_0)
=\lim_{n\to\infty}\frac{1}{n}\log\lm_{\om,0}^n
=\lim_{n\to\infty}\frac{1}{n}\log\lm_{\sg^{-n}(\om),0}^n.
\end{align}
The following lemma, which is the open analogue of Lemma 10.1 of \cite{AFGTV20}, gives an alternate method for calculating the expected pressure.
\begin{lemma}\label{lem: conv of pressure limits}
For $m$-a.e. $\om\in\Om$ we have that 
\begin{align}\label{eq: backward pressure limit}
\lim_{n\to\infty}\norm{\frac{1}{n}\log\cL_{\sg^{-n}(\om)}^n\ind_{\sg^{-n}(\om)} -\frac{1}{n}\log\lm_{\sg^{-n}(\om)}^n}_\infty=0
\end{align}
and 
\begin{align}\label{eq: forward pressure limit}
\lim_{n\to\infty}\norm{\frac{1}{n}\log\cL_{\om}^n\ind_{\om} -\frac{1}{n}\log\lm_{\om}^n}_\infty=0.
\end{align}	
Furthermore, for $m$-a.e. $\om\in\Om$ we have that 
\begin{align*}
\lim_{n\to\infty}\frac{1}{n}\log\inf_{D_{\sg^n(\om),\infty}} \phi_{\sg^n(\om)}
=
\lim_{n\to\infty}\frac{1}{n}\log\norm{\phi_{\sg^n(\om)}}_\infty
=
0.
\end{align*}
\end{lemma}
As the proof of the previous lemma is exactly the same as the proof of Lemma 10.1 of \cite{AFGTV20} where infima are taken over the appropriate $D_{\om,\infty}$ sets, the proof is left to the reader. Now, in view of the fact that 
\begin{align*}
\cL_\om\ind_\om \leq \cL_{\om,0}\ind,
\end{align*}
Lemma~\ref{lem: conv of pressure limits} and Lemma 10.1 of \cite{AFGTV20}, together with \eqref{eq: BET exp pres hole} and \eqref{eq: BET exp pres closed} imply that 
\begin{align}\label{eq: open pres leq closed pres}
\cEP(\vp)\leq \cEP(\vp_0).
\end{align}
We now define the fiberwise escape rates of a random measure.
\begin{definition}\label{def: escape rate}
Given a random probability measure $\vrho$ on $\cI$, for each $\om\in\Om$, we define the lower and upper fiberwise escape rates respectively by the following:
\begin{align*}
\Ul R(\vrho_\om):=-\limsup_{n\to\infty}\frac{1}{n}\log \vrho_\om(X_{\om,n})
\quad \text{ and } \quad
\ol R(\vrho_\om):=-\liminf_{n\to\infty}\frac{1}{n}\log \vrho_\om(X_{\om,n}).
\end{align*}
If $\Ul R(\vrho_\om)=\ol R(\vrho_\om)$, we say the escape rate exists and denote the common value by $R(\vrho_\om)$. \index{fiberwise escape rate}\index{$\Ul R(\vrho_\om)$}\index{$\ol R(\vrho_\om)$}\index{$R(\vrho_\om)$}
\end{definition}
The previous results allow us to calculate the following escape rates, thus proving Theorem~\ref{main thm: escape rate}.
\begin{proposition}\label{prop: escape rate}
For $m$-a.e. $\om\in\Om$ we have that 
\begin{align*}
R(\nu_{\om,0})=R(\eta_\om)=\cEP(\vp_0)-\cEP(\vp).
\end{align*}
\end{proposition}
\begin{proof}
We begin by noting that 
\begin{align*}
\nu_{\om,0}(X_{\om,n-1})
&=
\lt(\lm_{\om,0}^n\rt)^{-1}\nu_{\sg^n(\om),0}\lt(\cL_{\om,0}^n(\hat X_{\om,n-1})\rt)
=\frac{\lm_\om^n}{\lm_{\om,0}^n}\nu_{\sg^n(\om),0}\lt(\~\cL_\om^n\ind_\om\rt)
\\
&=\frac{\lm_\om^n}{\lm_{\om,0}^n}\lt(\nu_{\sg^n(\om),0}(\phi_{\sg^n(\om)})-\nu_{\sg^n(\om),0}\lt(\~\cL_\om^n\ind_\om -\phi_{\sg^n(\om)}\rt)\rt).
\end{align*}
Using Theorem~\ref{thm: exp convergence of tr op} gives that
\begin{align*}
-R(\nu_{\om,0})
&=
\lim_{n\to\infty}\frac{1}{n}\log\frac{\lm_\om^n}{\lm_{\om,0}^n}
+
\lim_{n\to\infty}\frac{1}{n}\log\nu_{\sg^n(\om),0}(\phi_{\sg^n(\om)}).
\end{align*}
Thus, the temperedness of $\inf_{D_{\om,\infty}} \phi_\om$ and $\norm{\phi_\om}_\infty$, coming from Lemma~\ref{lem: conv of pressure limits}, imply that 
\begin{align*}
0=
\lim_{n\to\infty}\frac{1}{n}\log\inf_{D_{\sg^n(\om),\infty}} \phi_{\sg^n(\om)}
\leq 
\lim_{n\to\infty}\frac{1}{n}\log\nu_{\sg^n(\om),0}(\phi_{\sg^n(\om)})
\leq 
\lim_{n\to\infty}\frac{1}{n}\log\norm{\phi_{\sg^n(\om)}}_\infty
=0,
\end{align*}
which, when combined with \eqref{eq: BET exp pres hole} and \eqref{eq: BET exp pres closed}, completes the proof of the first claim. As the second equality follows similarly to the first, we are done. 	
\end{proof}
\begin{remark}
If there exists a $T$-invariant measure $\mu_0$ on the closed system which is absolutely continuous with respect to $\nu_0$ then the proof of Proposition~\ref{prop: escape rate}, with minor adjustments, also shows that for $m$-a.e. $\om\in\Om$ we have that 
\begin{align*}
R(\nu_{\om,0})=R(\mu_{\om,0})=\cEP(\vp_0)-\cEP(\vp).
\end{align*}
\end{remark}

\section{Equilibrium states}\label{sec: eq states}
In this short section we show that the invariant measures $\mu=\set{\mu_\om}_{\om\in\Om}$ constructed in Section \ref{sec: conf and inv meas}
are in fact the unique relative equilibrium states for the random open system.
\begin{definition}
Let $\cP_{T,m}^H(I)$\index{$\cP_{T,m}^H(I)$} denote the collection of $T$-invariant random probability measures $\zt$ on $\Om\times I$, 
such that the disintegration $\{\zt_\om\}$ satisfy $\zt_\om(H_\om)=0$ for $m$ a.e. $\om\in\Om$. 
We say that a measure $\zt\in\cP_{T,m}^H(\Om)$ is a \textit{relative equilibrium state}\index{relative equilibrium state} for the random open system (open) potential $\vp$ if
$$\cEP(\vp)=h_\zt(T)+\int_{\Om\times I} \vp \,d\zt,$$
where $h_\zt(T)$ denotes the entropy of $T$ with respect to $\zt$.
\end{definition}

The proof of the next result follows similarly to the proof of Theorem 2.23 in \cite{AFGTV20} (see also Remark 2.24, Lemma 12.2 and Lemma 12.3).
\begin{theorem}\label{thm:eqstates}
The random measure $\mu\in\cP_{T,m}^H(\Om)$ with disintegration $\{\mu_\om\}_{\om\in\Om}$ produced in \eqref{eq: def of mu_om} is the unique relative equilibrium state for the potential $\vp$. It satisfies the following variational principle:
\begin{align*}
\cEP(\vp)
= h_\mu(T)+\int_{\Om \times I} \vp \,d\mu
=
\sup_{\zt\in\cP_{T,m}^H(\Om\times I)} \lt(h_\zt(T)+\int_{\Om \times I} \vp \,d\zt\rt).
\end{align*}
Furthermore, for each $\zt\in\cP_{T,m}^H(\Om)$ different from $\mu$ we have that 
\begin{align*}
h_\zt(T)+\int_{\Om \times I} \vp \,d\zt
<
h_\mu(T)+\int_{\Om \times I} \vp \,d\mu.
\end{align*}
\end{theorem}

\section{Bowen's formula}\label{sec: bowen}
This section is devoted to proving a formula for the Hausdorff dimension of the survivor set in terms of the expected pressure function, which was first proven by Bowen in \cite{bowen_hausdorff_1979} in the setting quasi-Fuchsian groups. 
In this section we will consider geometric potentials of the form $\vp_{0,t}(\om,x)=-t\log |T_\om'|(x)$ ($t\in[0,1]$) for maps $T_\om$ which are expanding on average.
We denote the expected pressure of $\vp_{0,t}$ by $\cEP_0(t)$ and the expected pressure of the open potential $\vp_t$ by $\cEP(t)$. In the case that $t=1$, the fiberwise closed conformal measures $\nu_{\om,0,1}$ are equal to Lebesgue measure and $\lm_{\om,0,1}=1$. Furthermore, we note that for any $t\geq0$ we have that 
\begin{align}\label{eq: g_t in terms of g_1}
g_{\om,0,t}^{(n)}=\lt(g_{\om,0,1}^{(n)}\rt)^t=\frac{1}{|(T_\om^n)'|^t}. 
\end{align} 

\begin{definition}\label{def: bdd dist}
We will say that the weight function $g_{\om,0}$ has the \textit{Bounded Distortion Property} if for $m$-a.e. $\om\in\Om$ there exists $K_\om\geq 1$ such that for all $n\in\NN$, all $Z\in\cZ_{\om}^{(n)}$, and all $x,y\in Z$ we have that 
\begin{align*}
\frac{g_{\om,0}^{(n)}(x)}{g_{\om,0}^{(n)}(y)}\leq K_\om.
\end{align*}
\end{definition}
We now adapt the following definitions from \cite{LMD} to the random setting.
\begin{definition}\label{def: large images}
We will say that the map $T$ has \textit{large images} if for $m$-a.e. $\om\in\Om$ we have 
\begin{align*}
\inf_{n\in\NN}\inf_{Z\in\cZ_\om^{(n)}}\nu_{\sg^n(\om),0}\lt(T_\om^n(Z)\rt)>0.
\end{align*}
$T$ is said to have \textit{large images with respect to $H$} if for $m$-a.e. $\om\in\Om$, each $n\in\NN$, and each $Z\in\cZ_\om^{(n)}$ with $Z\cap X_{\om,\infty}\neq\emptyset$ we have 
\begin{align*}
T_\om^n(Z\cap X_{\om,n-1})\bus X_{\sg^n(\om),\infty}.
\end{align*}
\end{definition}
\begin{remark}\label{rem: full support}
If $T$ has large images with respect to $H$ then it follows from Remark~\ref{rem: conformal meas prop} that $\supp(\nu_{\om,t})=X_{\om,\infty}$ for any $t\in[0,1]$.
\end{remark}
We now prove a formula for the Hausdorff dimension of the surviving set, \`a la Bowen, proving Theorem~\ref{main thm: Bowens formula}.
\begin{theorem}\label{thm: Bowen's Formula}
Suppose that $\int_\Om\log\inf|T_\om'|\ dm(\om)>0$ and that $g_0=1/|T'|$ satisfies the bounded distortion property, then 
there exists a unique $h\in [0,1]$ such that $\cEP(t)>0$ for all $0\leq t<h$ and $\cEP(t)<0$ for all $h<t\leq 1$.
Furthermore, if $T$ has large images and large images with respect to $H$, then for $m$-a.e. $\om\in\Om$
\begin{align*}
\HD(X_{\om,\infty})=h.		
\end{align*}
\end{theorem}
\begin{proof}
We will prove this theorem in a series of lemmas.
\begin{lemma}\label{lem: bf lem 1}
The function $\cEP(t)$ is strictly decreasing and there exists $h\in[0,1]$ such that $\cEP(t)>0$ for all $0\leq t<h$ and $\cEP(t)<0$ for all $h<t\leq 1$.
\end{lemma}
\begin{proof}
We first note that, using \eqref{eq: g_t in terms of g_1}, for any $n\in\NN$ and $s<t\in[0,1]$ we can write  
\begin{align*}
\cL_{\om,t}^n\ind_\om
\leq 
\norm{g_{\om,1}^{(n)}}_\infty^{t-s}\cL_{\om,s}^n\ind_\om.  
\end{align*}
This immediately implies that 
\begin{align*}
\cEP(t)< \cEP(s)
\end{align*}
since for $m$-a.e. $\om\in\Om$ we have 
\begin{align*}
\lim_{n\to \infty}\frac{1}{n}\log\norm{g_{\om,1}^{(n)}}_\infty<0.
\end{align*}
Now since $\cEP(0)\geq 0$ and $\cEP(1)\leq \cEP_0(1)=0$, there must exist some $h\in[0,1]$ such that for all $s<h<t$ we have 
\begin{align*}
\cEP(t)<0<\cEP(s).
\end{align*} 
\end{proof}
To prove the remaining claim of Theorem~\ref{thm: Bowen's Formula}, we now suppose that $T$ has large images and large images with respect to $H$.
\begin{lemma}\label{lem: bf lem 2}
If $T$ has large images and large images with respect to $H$, and $\nu_{\om,t}(Z)>0$
for all $t\in[0,1]$, all $n\in\NN$, and all $Z\in\cZ_\om^{(n)}$, then for all $x\in Z$ we have 
\begin{align*}
K_\om^{-1}\leq\frac{\lt(g_{\om,0,1}^{(n)}\rt)^t(x)}{\lm_{\om,t}^n\nu_{\om,t}(Z)} \leq K_\om
\quad \text{ and } \quad
K_\om^{-1}\leq\frac{g_{\om,0,1}^{(n)}(x)}{\Leb(Z)} \leq K_\om.
\end{align*}
\end{lemma}
\begin{proof}
In light of Remark~\ref{rem: full support}, $\supp(\nu_{\om,t})=X_{\om,\infty}$, and thus for any $Z\in\cZ_\om^{(n)}$ for any $n\geq 1$, $\nu_{\om,t}(Z)>0$ if and only if $Z\cap X_{\om,\infty}\neq\emptyset$. 
Furthermore, since $T$ has large images with respect to $H$, we have that 
\begin{align}\label{eq: measur cons of LI wrt H}
\nu_{\sg^n(\om),t}(T_\om^n(Z))=1
\end{align}
for any $Z\in\cZ_\om^{(n)}$ with $Z\cap X_{\om,\infty}\neq\emptyset$.
Thus, we may write 
\begin{align}
\nu_{\om,t}(Z)
&=
\int_{X_{\om,\infty}} \ind_Z \,d\nu_{\om,t}
=
\lt(\lm_{\om,t}^n\rt)^{-1}\int_{X_{\sg^n(\om),\infty}}\cL_{\om,t}^n\ind_Z\, d\nu_{\sg^n(\om),t}
\nonumber\\
&=
\lt(\lm_{\om,t}^n\rt)^{-1}\int_{X_{\sg^n(\om),\infty}}\cL_{\om,0,t}^n\left(\ind_Z \hat X_{\om,n-1}\right)\, d\nu_{\sg^n(\om),t}
\nonumber\\
&=
\lt(\lm_{\om,t}^n\rt)^{-1}\int_{T_\om^n(Z)} \lt(\lt(g_{\om,0,1}^{(n)}\rt)^t \hat X_{\om,n-1}\rt)\circ T_{\om,Z}^{-n}\, d\nu_{\sg^n(\om),t}
\nonumber\\
&=
\lt(\lm_{\om,t}^n\rt)^{-1}\int_{T_\om^n(Z)} \lt(g_{\om,0,1}^{(n)}\rt)^t \circ T_{\om,Z}^{-n}\, d\nu_{\sg^n(\om),t}.
\label{eq: LMD 5.3}
\end{align}
The Bounded Distortion Property implies that for $x\in Z$ we have 
\begin{align*}
K_\om^{-1}\nu_{\sg^n(\om),t}(T_\om^n(Z))\lt(g_{\om,0,1}^{(n)}\rt)^t(x)
&\leq 
\int_{T_\om^n(Z)} \lt(g_{\om,0,1}^{(n)}\rt)^t \circ T_{\om,Z}^{-n}\, d\nu_{\sg^n(\om),t}
\\
&\leq 
K_\om\nu_{\sg^n(\om),t}(T_\om^n(Z))\lt(g_{\om,0,1}^{(n)}\rt)^t(x).
\end{align*} 
Thus 
\begin{align*}
K_\om^{-1}\frac{\lt(g_{\om,0,1}^{(n)}\rt)^t(x)}{\lm_{\om,t}^n\nu_{\om,t}(Z)}
\leq
\frac{1}{\nu_{\sg^n(\om),t}(T_\om^n(Z))}
\leq
K_\om\frac{\lt(g_{\om,0,1}^{(n)}\rt)^t(x)}{\lm_{\om,t}^n\nu_{\om,t}(Z)},
\end{align*}
which then implies that 
\begin{align*}
K_\om^{-1}\frac{1}{\nu_{\sg^n(\om),t}(T_\om^n(Z))}
\leq 
\frac{\lt(g_{\om,0,1}^{(n)}\rt)^t(x)}{\lm_{\om,t}^n\nu_{\om,t}(Z)}
\leq
K_\om\frac{1}{\nu_{\sg^n(\om),t}(T_\om^n(Z))}.
\end{align*}
The first claim follows from \eqref{eq: measur cons of LI wrt H}.
The proof of the second claim involving the Lebesgue measure follows similarly noting that  $\nu_{\om,0,1}=\Leb$ and $\lm_{\om,0,1}=1$.
\end{proof}

Let $\ep>0$ and $n\in\NN$ such that $\diam(Z)<\ep$ for all $Z\in\cZ_\om^{(n)}$. Denote 
\begin{align*}
\cF_\om:=\set{Z\in\cZ_\om^{(n)}: Z\cap X_{\om,\infty}\neq\emptyset},
\end{align*}
which is a cover of $X_{\om,\infty}$ by sets of diameter less than $\ep$. Then, letting $x_Z$ be any element of $Z$ and using Lemma~\ref{lem: bf lem 2} twice (first with respect to $\Leb$ and then with respect to $\nu_{\om,t}$), we have 
\begin{align}
\sum_{Z\in\cF_\om}\diam^t(Z)
&\leq 
\sum_{Z\in\cF_\om}\Leb^t(Z)
\leq 
K_\om^t \sum_{Z\in\cF_\om}\lt(g_{\om,0,1}^{(n)}\rt)^t(x_Z)
\nonumber\\
&\leq 
K_\om^{2t}\sum_{Z\in\cF_\om} \lm_{\om,t}^n\nu_{\om,t}(Z)
=
K_\om^{2t} \lm_{\om,t}^n\nu_{\om,t}(X_{\om,\infty})
=
K_\om^{2t} \lm_{\om,t}^n.
\label{eq: HD up bd 1}	
\end{align}
Now, if $t>h$ we have 
\begin{align*}
\lim_{n\to\infty}\frac{1}{n}\log\lm_{\om,t}^n
=
\cEP(t)<0,
\end{align*}
and thus, for $\dl>0$ sufficiently small and all $n\in\NN$ sufficiently large, 
\begin{align*}
\lm_{\om,t}^n\leq e^{-n\dl}.
\end{align*}
Consequently, we see that the right-hand side of \eqref{eq: HD up bd 1} must go to zero, and thus we must have that $\HD(X_{\om,\infty})\leq h$.

For the lower bound we turn to the following result of Young. 
\begin{proposition}[Proposition, \cite{young_dimension_1982}]\label{prop: Young prop}
Let $X$ be a metric space and $Z\sub X$. Assume there exists a probability measure $\mu$ such that $\mu(Z)>0$. For any $x\in Z$ we define 
\begin{align*}
\Ul{d_\mu}(x):=\liminf_{\ep\to 0}\frac{\log \mu(B(x,\ep))}{\log\ep}.
\end{align*}
If $\Ul{d_\mu}(x)\geq d$ for each $x\in Z$, then $\HD(Z)\geq d$.
\end{proposition}
We will use this result to prove a lower bound for the dimension, thus completing the proof of Theorem~\ref{thm: Bowen's Formula}. 
Let $x\in X_{\om,\infty}$, let $\ep>0$, and in light of Lemma~\ref{lem: bf lem 2}, let $n_{\om,0}+1$ be the least positive integer such that there exists $y_0\in B(x,\ep)$ such that 
\begin{align*}
g_{\om,0}^{(n_{\om,0}+1)}(y_0)\leq 2\ep K_{\om}.
\end{align*}
Note that as $\ep\to 0$ we must have that $n_{\om,0}\to\infty$.
So we must have
\begin{align}\label{eq: bf lb eq 1}
g_{\om,0}^{(n_{\om,0})}(y_0)g_{\sg^{n_{\om,0}}(\om),0}(T_\om^{n_{\om,0}}(y_0))=g_{\om,0}^{(n_{\om,0}+1)}(y_0)
\leq
2\ep K_\om.
\end{align}
Thus, using \eqref{eq: bf lb eq 1} and the definition of $n_{\om,0}$ we have that 
\begin{align*}
2\ep K_\om
<
g_{\om,0}^{(n_{\om,0})}(y_0)
\leq 
\frac{2\ep K_\om}{\inf g_{\sg^{n_{\om,0}}(\om),0}}.
\end{align*}
Now let $Z_{0}\in\cZ_\om^{(n_{\om,0})}$ be the partition element containing $y_0$. Then Lemma~\ref{lem: bf lem 2} gives that
\begin{align}\label{eq: bf lb eq 2}
\diam(Z_0)
\leq
K_\om g_{\om,0}^{(n_{\om,0})}(y_0)
\leq 
\frac{2\ep K_\om^2}{\inf g_{\sg^{n_{\om,0}}(\om),0}},	
\end{align}
and 
\begin{align}\label{eq: bf lb eq 3}
\diam(Z_0)
\geq 
K_\om^{-1} g_{\om,0}^{(n_{\om,0})}(y_0)
>2\ep.
\end{align}
Combining \eqref{eq: bf lb eq 2} and \eqref{eq: bf lb eq 3} gives 
\begin{align}\label{eq: bf lb eq 4}
2\ep < \diam(Z_0) \leq \frac{2\ep K_\om^2}{\inf g_{\sg^{n_{\om,0}}(\om),0}}. 
\end{align}
Now, we define 
\begin{align*}
B_{\om,1}:=B(x,\ep)\bs Z_0,
\end{align*}
which may be empty. If $B_{\om,1}\neq\emptyset$, then we let $n_{\om,1}+1$ be the least positive integer such that there exists $y_1\in B_{\om,1}$ such that 
\begin{align*}
g_{\om,0}^{(n_{\om,1}+1)}(y_1)\leq 2\ep K_{\om}.
\end{align*}
Following the same line of reasoning to derive \eqref{eq: bf lb eq 4}, we see that  
\begin{align}\label{eq: bf lb eq 5}
2\ep
<
\diam(Z_1)
\leq
\frac{2\ep K_\om^2}{\inf g_{\sg^{n_{\om,1}}(\om),0}},
\end{align}  
where $Z_1\in\cZ_\om^{(n_{\om,1})}$ is the partition element containing $y_1$. Note that by definition we have that $n_{\om,1}\geq n_{\om,0}$ and $Z_0\cap Z_1=\emptyset$. This immediately implies that 
$$
B(x,\ep)\sub Z_0\cup Z_1,
$$
as otherwise using the same construction we could find some $y_2\in B_{\om,1}\bs Z_1$, some $n_{\om,2}\geq n_{\om,1}$ and a partition element $Z_2\in\cZ_\om^{(n_{\om,2})}$ containing $y_2$ with diameter greater than $2\ep$. But this would produce three disjoint intervals each with diameter greater than $2\ep$ all of which intersect $B(x,\ep)$, which would obviously be a contradiction.  

Now, using \eqref{eq: LMD 5.3} and Lemma~\ref{lem: bf lem 2} gives that 
\begin{align*}
\nu_{\om,t}(Z_j)
=
\lt(\lm_{\om,t}^{n_{\om,j}}\rt)^{-1}\int_{T_\om^{n_{\om,j}}(Z)} \lt(g_{\om,0,1}^{(n_{\om,j})}\rt)^t \circ T_{\om,Z}^{-n_{\om,j}}\, d\nu_{\sg^{n_{\om,j}}(\om),t}
\leq 
K_\om^t\lt(\lm_{\om,t}^{n_{\om,j}}\rt)^{-1}\diam^t(Z_j)
\end{align*}
for $j\in\set{0,1}$.
Using this we see that 
\begin{align}
&\log\nu_{\om,t}(B(x,\ep))
\leq 
\log\lt(\nu_{\om,t}(Z_0)+\nu_{\om,t}(Z_1)\rt)
\nonumber\\
&\,\,
\leq 
t\log K_\om
+
\log\lt(\lt(\lm_{\om,t}^{n_{\om,0}}\rt)^{-1}\diam^t(Z_0)
+
\lt(\lm_{\om,t}^{n_{\om,1}}\rt)^{-1}\diam^t(Z_1)\rt)
\nonumber\\
&\,\,
\leq
t\log K_\om
+
\log\lt(\lt(\lm_{\om,t}^{n_{\om,0}}\rt)^{-1}\left(\frac{2\ep K_\om^2}{\inf g_{\sg^{n_{\om,0}}(\om),0}}\right)^t
+
\lt(\lm_{\om,t}^{n_{\om,1}}\rt)^{-1}\left(\frac{2\ep K_\om^2}{\inf g_{\sg^{n_{\om,1}}(\om),0}}\right)^t\rt)
\nonumber\\
&\,\,
=
t\log K_\om
+
t\log 2\ep K_\om^2
+
\log\lt(\lt(\lm_{\om,t}^{n_{\om,0}}\rt)^{-1}\left(\inf g_{\sg^{n_{\om,0}}(\om),0}\right)^{-t}
+
\lt(\lm_{\om,t}^{n_{\om,1}}\rt)^{-1}\left(\inf g_{\sg^{n_{\om,1}}(\om),0}\right)^{-t}\rt).		
\label{eq: bowen lb eq 0}
\end{align}
Now since $\log\inf g_{\om,0}\in L^1(m)$, $\inf g_{\om,0}$ is tempered and thus for each $\dl>0$ and all $n\in\NN$ sufficiently large we have that 
\begin{align}\label{eq: bowen lb eq 1}
e^{-nt\dl}\leq \inf \lt(g_{\sg^n(\om),0}\rt)^t.
\end{align}
From \eqref{eq: BET exp pres hole} we get that there for all $n\in\NN$ sufficiently large 
\begin{align}\label{eq: bowen lb eq 2}
\lm_{\om,t}^n\geq e^{n(\cEP(t)-\dl)}.
\end{align}
Thus combining \eqref{eq: bowen lb eq 1} and \eqref{eq: bowen lb eq 2} with \eqref{eq: bowen lb eq 0} gives 
\begin{align}
\log\nu_{\om,t}(B(x,\ep))
&\leq 
t\log K_\om
+
t\log 2\ep K_\om^2
+
\log\lt(e^{n_{\om,0}\dl(t+1)-n_{\om,0}\cEP(t)}+e^{n_{\om,1}\dl(t+1)-n_{\om,1}\cEP(t)}\rt)
\nonumber\\
&\leq 
t\log K_\om
+
t\log 2\ep K_\om^2
+
\log\lt(e^{2n_{\om,0}\dl-n_{\om,0}\cEP(t)}+e^{2n_{\om,1}\dl-n_{\om,1}\cEP(t)}\rt),
\label{eq: bowen lb eq 3}
\end{align}
where we have used the fact that $t\in[0,1]$.
Then for $\dl>0$ sufficiently small and $n_{\om,0}$ and $n_{\om,1}$ sufficiently large (which requires $\ep>0$ to be sufficiently small) we have that 
\begin{align}\label{eq: bowen lb eq 3*}
\log\lt(e^{2n_{\om,0}\dl-n_{\om,0}\cEP(t)}+e^{2n_{\om,1}\dl-n_{\om,1}\cEP(t)}\rt)<0,
\end{align} 
since for all $t<h$ we have that $\cEP(t)>0$.
Dividing both sides of \eqref{eq: bowen lb eq 3} by $\log\ep<0$ and using \eqref{eq: bowen lb eq 3*} yields
\begin{align}
\frac{\log\nu_{\om,t}(B(x,\ep))}{\log\ep}
&\geq 
t\frac{\log K_\om}{\log\ep}
+
t\frac{\log 2\ep K_\om^2}{\log\ep}
+
\frac{\log\lt(e^{2n_{\om,0}\dl-n_{\om,0}\cEP(t)}+e^{2n_{\om,1}\dl-n_{\om,1}\cEP(t)}\rt)}{\log\ep}
\nonumber\\
&\geq
t\frac{\log K_\om}{\log\ep}
+
t\frac{\log 2 K_\om^2}{\log\ep}
+t.	
\label{eq: bowen lb eq 4}
\end{align}
Taking a liminf of \eqref{eq: bowen lb eq 4} as $\ep$ goes to $0$ gives that $\Ul{d_{\nu_{\om,t}}}(x)\geq t$ for all $x\in X_{\om,\infty}$. As this holds for all $t<h$, we must in fact have that $\Ul{d_{\nu_{\om,t}}}(x)\geq h$. In light of Proposition~\ref{prop: Young prop}, we have proven Theorem~\ref{thm: Bowen's Formula}.
\end{proof}

\section{Examples}\label{sec: examples}
In this section we present several examples of our general theory of open random systems.
In particular, we show that our results hold for a large class of random $\bt$-transformations with random holes which have uniform covering times as well as a large class of random Lasota-Yorke maps with random holes. However, we note that in principle any of the finitely branched classes of maps treated in \cite{AFGTV20} will satisfy our assumptions given a suitable choice of hole. This includes random systems where we allow non-uniformly expanding maps, or even maps with contracting branches to appear with positive probability. We also note that the examples we present allow for both random maps and random hole, which, to the authors' knowledge, has not appeared in literature until now. 
Before presenting our examples, we first give alternate hypotheses (to our assumptions \eqref{cond Q1}-\eqref{cond Q3}) that are more restrictive but simpler to check. 

We begin by recalling the definitions of the various partitions constructed in Section~\ref{sec: tr op and Lm} which are used in producing our main Lasota-Yorke inequality (Lemma~\ref{ly ineq}) and are implicitly a part of our main assumptions \eqref{cond Q1}-\eqref{cond Q3}.
Recall that $\cZ_\om^{(n)}$ denotes the partition of monotonicity of $T_\om^n$, and $\sA_\om^{(n)}$ denotes the collection of all finite partitions of $I$ such that
\begin{align}\label{eq: def A partition 2}
\var_{A_i}(g_{\om}^{(n)})\leq 2\norm{g_{\om}^{(n)}}_{\infty}
\end{align}
for each $\cA=\set{A_i}\in\sA_{\om}^{(n)}$.
Given $\cA\in\sA_\om^{(n)}$, $\widehat\cZ_\om^{(n)}(\cA)$ denotes the coarsest partition amongst all those finer than $\cA$ and $\cZ_\om^{(n)}$ such that all elements of $\widehat\cZ_\om^{(n)}(\cA)$ are either disjoint from $X_{\om,n-1}$ or contained in $X_{\om,n-1}$. From $\widehat\cZ_\om^{(n)}$ we recall the subcollections $\cZ_{\om,*}^{(n)}$, $\cZ_{\om,b}^{(n)}$, and $\cZ_{\om,g}^{(n)}$ defined in \eqref{eq: Z_*}-\eqref{eq: Z_g}.

For the purposes of showing that examples easily satisfy our conditions, we take the more general approach to partitions found in Section~2.2 of \cite{AFGTV20}, and instead now set, for $\hat\al\geq 0$,
$\ol\sA_\om^{(n)}(\hat\al)$ to be the collection of all finite partitions of $I$ such that
\begin{align}\label{eq: def A partition 3}
\var_{A_i}(g_{\om}^{(n)})\leq \hat\al\norm{g_{\om}^{(n)}}_{\infty}
\end{align}
for each $\cA=\set{A_i}\in\ol\sA_{\om}^{(n)}(\hat\al)$. Note that for some $\hat\al\leq 1$ the collection $\ol\sA_\om^{(n)}(\hat\al)$ may be empty, but such partitions always exist for any $\hat\al>1$, and may exist even with $\hat\al=0$ if the weight function $g_\om$ is constant; see \cite{rychlik_bounded_1983} Lemma~6. 
We now suppose that we can find $\hat\al\geq 0$ sufficiently large such that 
\begin{enumerate}
\item[(\Gls*{Z})]\myglabel{Z}{cond Z} $\cZ_\om^{(n)}\in\ol\sA_\om^{(n)}(\hat\al)$ for each $n\in\NN$ and each $\om\in\Om$.  
\end{enumerate}
Now we set $\ol\cZ_\om^{(n)}$ be the coarsest partition such that all elements of $\cZ_\om^{(n)}$ are either disjoint from $X_{\om,n-1}$ or contained in $X_{\om,n-1}$. Note that $\ol\cZ_\om^{(n)}=\widehat\cZ_\om^{(n)}(\cZ_\om^{(n)})$.\index{$\ol\cZ_\om^{(n)}$} 
Now, define the following subcollections:
\begin{align*}
\ol\cZ_{\om,*}^{(n)}&:=\set{Z\in \ol\cZ_\om^{(n)}: Z\sub X_{\om,n-1} },
\\
\ol\cZ_{\om,b}^{(n)}&:=\set{Z\in \ol\cZ_\om^{(n)}: Z\sub X_{\om,n-1} \text{ and } \Lm_{\om}(\ind_Z)=0 },
\\
\ol\cZ_{\om,g}^{(n)}&:=\set{Z\in \ol\cZ_\om^{(n)}: Z\sub X_{\om,n-1} \text{ and } \Lm_{\om}(\ind_Z)>0}.
\end{align*}\index{$\ol\cZ_{\om,*}^{(n)}$}\index{$\ol\cZ_{\om,b}^{(n)}$}\index{$\ol\cZ_{\om,g}^{(n)}$}

Consider the collection $\ol\cZ_{\om,F}^{(n)}\sub\ol\cZ_{\om,*}^{(n)}$ such that for $Z\in\ol\cZ_{\om,F}^{(n)}$ we have $T_\om^n(Z)=I$. We will elements of $\ol\cZ_{\om,F}^{(n)}$ ``full intervals''. We let $\ol\cZ_{\om,U}^{(n)}:=\ol\cZ_{\om,*}^{(n)}\bs\ol\cZ_{\om,F}^{(n)}$. Since for any $Z\in\ol\cZ_{\om,F}^{(n)}$ we have that $T_\om^n(Z)=I$, and hence 
\begin{align}\label{eq: check tilde Q3}
\Lm_\om(\ind_Z)
\geq 
\frac{\inf_{D_{\sg^n(\om),n}}\cL_\om^n\ind_Z}{\sup_{D_{\sg^n(\om),n}}\cL_\om^n\ind_\om}	
\geq 
\frac{\inf g_{\om,0}^{(n)}}{\norm{\cL_{\om,0}^n\ind}_\infty}>0.
\end{align}
Consequently, we have that $\ol\cZ_{\om,F}^{(n)}\sub\ol\cZ_{\om,g}^{(n)}$, and thus 
\begin{align}\label{eq: bad interv cont in non-full interv}
\ol\cZ_{\om,b}^{(n)}\sub\ol\cZ_{\om,U}^{(n)}.
\end{align} 
We let $\zt_\om^{(n)}\geq 0$ denote the maximum number of contiguous non-full intervals for $T_\om^n$ in $\ol\cZ_{\om,*}^{(n)}$ for each $\om\in\Om$ and $n\in\NN$. Note that $\zt_\om^{(1)}$ may be equal to $0$, but $\zt_\om^{(n)}\geq 1$ for all $n\geq 2$, and so it follows from \eqref{eq: bad interv cont in non-full interv} that 
\begin{align}\label{eq: xi leq zt}
0\leq \log \xi_\om^{(n)}\leq \log \zt_\om^{(n)}
\end{align}
for all $n\geq 2$.
In the interest of having assumptions that are easier to check than \eqref{cond Q1}-\eqref{cond Q3} we introduce the following simpler assumptions which use the collections $\ol\cZ_{\om,F}^{(n)}$ and $\ol\cZ_{\om,U}^{(n)}$\index{$\ol\cZ_{\om,F}^{(n)}$}\index{$\ol\cZ_{\om,U}^{(n)}$} rather than $\cZ_{\om,g}^{(n)}$ and $\cZ_{\om,b}^{(n)}$. We assume the following:
\begin{enumerate}
\item[(\Gls*{Q0hat})]\myglabel{$\widehat{\textrm{Q0}}$}{cond tilde Q0} $\ol\cZ_{\om,F}^{(1)}\neq\emptyset$ for $m$-a.e. $\om\in\Om$, 

\item[(\Gls*{Q1hat})]\myglabel{$\widehat{\textrm{Q1}}$}{cond tilde Q1} 
\begin{align*}
\lim_{n\to\infty}\frac{1}{n}\log\norm{g_{\om}^{(n)}}_\infty
+
\limsup_{n\to\infty}\frac{1}{n}\log\zt_{\om}^{(n)}
<
\lim_{n\to\infty}\frac{1}{n}\log\rho_{\om}^n
=
\int_\Om \log\rho_{\om}\, dm(\om),
\end{align*}

\item[(\Gls*{Q2hat})]\myglabel{$\widehat{\textrm{Q2}}$}{cond tilde Q2} for each $n\in\NN$ we have $\log^+\zt_\om^{(n)}\in L^1(m)$,

\item[(\Gls*{Q3hat})]\myglabel{$\widehat{\textrm{Q3}}$}{cond tilde Q3} for each $n\in\NN$, $\log\hat\dl_{\om,n}\in L^1(m)$, where
\begin{align}\label{eq: def of tilde dl_om,n}
\hat\dl_{\om,n}:=\min_{Z\in\ol\cZ_{\om,F}^{(n)}}\Lm_\om(\ind_Z).
\end{align}
\end{enumerate}
Our assumptions \eqref{cond Q1}-\eqref{cond Q3} are used exclusively in Section~\ref{sec: LY ineq}, and primarily in Lemma \ref{ly ineq}. In the proof of Lemma~\ref{ly ineq}, the good and bad interval collections $\cZ_{\om,g}^{(n)}$ and $\cZ_{\om,b}^{(n)}$ are used only to estimate the variation of a function and can easily be replaced by the collections $\ol\cZ_{\om,F}^{(n)}$ and $\ol\cZ_{\om,U}^{(n)}$. Therefore, we can easily replace the assumptions \eqref{cond Q1}-\eqref{cond Q3} with \eqref{cond tilde Q0}-\eqref{cond tilde Q3} without any changes. In particular, we are still able to construct the number $N_*$ which is defined in \eqref{eq: def of N}. Note that by replacing the $2$ in \eqref{eq: def A partition 2} with the $\hat\al\geq 0$ that appears in \eqref{eq: def A partition 3}, the constant coefficients which appear in the definitions of $A_\om^{(n)}$ and $B_\om^{(n)}$ in \eqref{eq: def of A and B in ly ineq} at the end of Lemma~\ref{ly ineq} may be different, consequently changing the value of $N_*$. This ultimately does not affect our general theory as we only care that such a value $N_*$ satisfying \eqref{eq: def of N} exists.
\begin{remark}
As in Remark~\ref{rem: alternate hypoth}, we again note that checking \eqref{cond Z}, \eqref{cond tilde Q2}, and \eqref{cond tilde Q3} for all $n\in\NN$ could be difficult and that it suffices to instead check these conditions only for $n=N_*$. Thus we may replace \eqref{cond Z}, \eqref{cond tilde Q2}, and \eqref{cond tilde Q3} with the following: 
\begin{enumerate}
\item[(\Gls*{Z'})]\myglabel{Z'}{cond Z'} there exists $\hat\al\geq 0$ such that $\cZ_\om^{(N_*)}\in\ol\sA_\om^{(N_*)}(\hat\al)$ for each $\om\in\Om$,

\item[(\Gls*{Q2hat'})]\myglabel{$\widehat{\textrm{Q2}}$'}{cond tilde Q2'} $\log^+\zt_\om^{(N_*)}\in L^1(m)$,

\item[(\Gls*{Q3hat'})]\myglabel{$\widehat{\textrm{Q3}}$'}{cond tilde Q3'} $\log\hat\dl_{\om,N_*}\in L^1(m)$.
\end{enumerate}
\end{remark}
The following proposition gives that assumption \eqref{cond tilde Q3} is always satisfied, and thus that we really only need to assume \eqref{cond tilde Q1} and \eqref{cond tilde Q2}.
\begin{proposition}\label{prop: tilde Q3 trivial}
Assumption \eqref{cond tilde Q3} is trivially satisfied.
\end{proposition}
\begin{proof}
As the right hand side of \eqref{eq: check tilde Q3} is $\log$-integrable by \eqref{eq: log sup inf g integ} and \eqref{eq: log sup inf L integ}, we must also have $\log\hat\dl_{\om,n}\in L^1(m)$.
\end{proof}
Recall that two elements $W, Z\in\cZ_{\om,*}^{(n)}$ are said to be contiguous if either $W$ and $Z$ are contiguous in the usual sense, i.e. they share a boundary point, or if they are separated by a connected component of $\cup_{j=0}^{n-1} T_\om^{-j}(H_{\sg^j(\om)})$.
The following proposition gives an upper bound for the exponential growth of the number $\zt_\om^{(n)}$ which will be useful in checking our assumption \eqref{cond tilde Q1},  which implies \eqref{cond Q1}.
\begin{proposition} 
The following inequality holds for $\zt_{\om}^{(n)}$, the largest number of contiguous
non-full intervals for $T_\om^n$: 
\begin{align}\label{eq: zt submultish}
\zt_{\om}^{(n)}
\leq n \prod_{j=0}^{n-1} (\zt^{(1)}_{\sg^j(\om)}+2).
\end{align}
Consequently, using \eqref{eq: xi leq zt} and the ergodic theorem, we have that 
\begin{align}\label{eq: zt bet}
\lim_{n\to\infty}\frac{1}{n}\log\xi_\om^{(n)}
\leq
\lim_{n\to\infty}\frac{1}{n}\log\zt_\om^{(n)}
\leq 
\int_\Om\log(\zt_\om^{(1)}+2)\, dm(\om).
\end{align}
\end{proposition}
\begin{proof}
This is  a random version of \cite[Lemma 6.3]{LMD}. We sketch the argument here.
To upper bound $\zt_{\om}^{(n+1)}$, we observe that the largest number of contiguous non-full intervals for $T_{\om}^{n+1}$ is given by 
\begin{equation}\label{eq:nonfullrec}
\zt_{\om}^{(n+1)}\leq \zt_{\om}^{(1)} (\zt^{(n)}_{\sg(\om)}+2)+ 2 \zt^{(n)}_{\sg(\om)}.
\end{equation}
Indeed, the first term on the right hand side accounts for the (worst case) scenario that all non-full branches of $T_{\sg(\om)}^n$ are pulled back inside contiguous non-full intervals for $T_\om$. 
For each non-full interval of $T_\om$, there at most $\zt^{(n)}_{\sg(\om)}+2$ contiguous non-full intervals for $T_\om^{n+1}$, as in addition to the $\zt^{(n)}_{\sg(\om)}$ non-full intervals pulled back from $T_{\sg(\om)}^n$, there may be full branches of $T_{\sg(\om)}^n$ to the left and right of these which are only partially pulled back inside the corresponding branch of  $T_\om$.
The second term in \eqref{eq:nonfullrec} accounts for an extra (at most) $\zt^{(n)}_{\sg(\om)}$ non full branches of $T_{\sg(\om)}^n$ pulled back inside the full branches of $T_\om$ neighboring the cluster of $\zt_{\om}^{(1)}$ non-full branches.

Rearranging \eqref{eq:nonfullrec} yields $\zt_{\om}^{(n+1)}\leq  \zt^{(n)}_{\sg(\om)}(\zt_{\om}^{(1)}+2)+ 2 \zt^{(1)}_{\om}.$
The claim follows directly by induction.
\end{proof}
Let $\pzh_\om$\index{$\pzh_\om$} denote the number of connected components of $H_\om$. The following lemma shows that the conditions 
\begin{lemma}\label{lem: check Q2' with zt}
If assumption \eqref{cond Z'} holds as well as 
\begin{align}\myglabel{CCH}{CCH}
\log \pzh_\om\in L^1(m),
\tag{\Gls*{CCH}}
\end{align}
then $\log\zt_\om^{(1)}\in L^1(m)$. Consequently, we have that \eqref{cond Q2'} and \eqref{cond tilde Q2} hold.
\end{lemma}
\begin{proof}
Since condition \eqref{cond Z'} holds, we see that \eqref{LIP} in conjunction with \eqref{CCH} gives that $\log\#\ol\cZ_{\om,*}^{(1)}\in L^1(m)$, and thus we must have that $\log\zt_\om^{(1)}\in L^1(m)$. 	
In light of \eqref{eq: zt submultish} we see that \eqref{cond Q2'} and \eqref{cond tilde Q2} hold if 
$$
\log(\zt_\om^{(1)}+2)\in L^1(m),
$$
and thus we are done.
\end{proof}
For each $n\in\NN$ define 
\begin{align}\label{eq: def of F_om^n}
F_\om^{(n)}:=\min_{y\in [0,1]}\#\{T^{-n}_\omega (y)\}. 
\end{align}
Since the sequences $\{\om \mapsto \|g_\om^{(n)}\|_\infty\}_{n\in\NN}$ and $\set{\om \mapsto \inf \cL_\om^n \ind_\om}_{n\in\NN}$  are submultiplicative and supermultiplicative, respectively, the subadditive ergodic theorem implies that the assumption that 
\begin{align*}
\lim_{n\to\infty}\frac{1}{n}\log\norm{g_\om^{(n)}}_\infty 
<
\lim_{n\to\infty}\frac{1}{n}\log\inf_{D_{\sg^n(\om),n}}\cL_\om^n\ind_\om
\end{align*}
is equivalent to the assumption that there exist $N \in \NN$ such that
\begin{align}\label{eq: on avg CPN1N2}
\int_\Om\log \norm{g_\om^{(N)}}_\infty\, dm(\om) 
<  
\int_\Om\log \inf_{D_{\sg^{N}(\om),N}} \cL_\om^{N}\ind_\om\, dm(\om).
\end{align}
A useful lower bound for the right hand side is the following:
\begin{align}\label{eq: example calc inf up bound}
\inf_{D_{\sg^{N}(\om),N}} \cL_\om^{N}\ind_\om
\geq
\inf_{X_{\om,N-1}}g_\om^{(N)}F_\om^{(N)}
\geq 
\inf g_{\om,0}^{(N)} F_\om^{(N)}.
\end{align}
The next lemma, which offers a sufficient condition to check assumptions \eqref{cond Q1} and \eqref{cond tilde Q1}, follows from \eqref{eq: zt bet}, \eqref{eq: on avg CPN1N2}, \eqref{eq: example calc inf up bound}, and the calculations from the proof of Lemma~13.18 in \cite{AFGTV20}. 
\begin{lemma}\label{lem: LY example Q1 Lemma}
If there exists $N\in\NN$ such that 
\begin{align*}
\frac{1}{N}\int_\Om \sup S_{N, T}(\vp_{\om,0})-\inf S_{N, T}(\vp_{\om,0})+ N\log(\zt_\om^{(1)}+2)\, dm(\om)
<
\frac{1}{N}\int_\Om \log F_\om^{(N)}\, dm(\om),
\end{align*}	 
then \eqref{cond tilde Q1} (and thus \eqref{cond Q1}) holds.
\end{lemma}

The following definition will be useful in checking our measurability assumptions for examples. 
\begin{definition}
We say that a function $f: \Omega \to \RR_+$ is $m$-continuous function if there is a partition of $\Om \pmod{m}$  into at most countably many Borel sets $\Om_1, \Om_2, \dots$ such that $f$ is constant on each $\Om_j$, say $f|_{\Om_j}=f_j$.
\end{definition}

We now give specific classes of random maps with holes which meet our assumptions. In principle, any of the classes of finitely branched maps discussed in Section~13 of \cite{AFGTV20} (including random non-uniformly expanding maps) will fit our current assumptions given a suitable hole $H$. 
\subsection{Random $\bt$-Transformations With Holes}\label{sec: example bt transform}

For this first example we consider the class of maps described in Section~13.2 in \cite{AFGTV20}. These are $\bt$-transformations for which the last (non-full) branch is not too small so that each branch in the random closed system has a uniform covering time. In particular we assume there is some $\dl>0$ such that for $m$-a.e. $\om\in\Om$ we have 
$$	
\bt_\om\in 
\bigcup_{k=1}^\infty [k+\delta,k+1].
$$
Further suppose that the map $\om\mapsto \bt_\om$ is $m$-continuous. 
We consider the random $\bt$-transformation
$T_\om : [0,1] \to [0,1]$ given by
\[
T_\om(x) = \beta_\omega x \pmod{1}
\]
and the potential 
$$
\vp_{\om,0}=-t\log|T'_\om|=-t\log\bt_\om
$$ 
for $t\geq 0$. 
In addition, we assume that 
\begin{equation}
\label{log9}
\int_\Om \log\lfloor\beta_\omega\rfloor\ dm(\omega)>\log 3
\end{equation}
and
\begin{equation}
\label{blb}
\int_\Om\log\lceil\beta_\omega\rceil\ dm(\omega)<\infty.
\end{equation}
Note that we allow $\bt_\om$ arbitrarily large. It follows from Lemma~13.6 of \cite{AFGTV20} that our assumptions \eqref{T1}-\eqref{T3}, \eqref{LIP}, \eqref{GP}, \eqref{A1}-\eqref{A2}, \eqref{cond M1}, \eqref{cond C1}, and \eqref{cond Z} are satisfied. 

To check the remainder of our assumptions we must now describe the choice of hole $H_\om$. For our holes $H_\om$ we will consider intervals of length at most $1/\bt_\om$ so that $H_\om$ may not intersect more than two monotonicity partition elements. To ensure that \eqref{cond tilde Q0} is satisfied we assume there is a  full-branched element $Z\in\cZ_{\om,F}^{(1)}$ such that $Z\cap H_\om=\emptyset$ for each $\om\in\Om$, and thus, in light of Remark~\ref{rem: check cond D}, we also have that assumption \eqref{cond D} is satisfied with $D_{\om,\infty}=I$ for each $\om\in\Om$. 

Now, we note that since \eqref{blb} implies our assumption \eqref{LIP}, Lemma~\ref{lem: check Q2' with zt} implies that assumption \eqref{cond tilde Q2} is satisfied. Thus, we have only to check the condition \eqref{cond tilde Q1}. 
Depending on $H_\om$ we may have that 
\begin{align*}
\inf\cL_\om\ind_\om=\frac{\floor{\bt_\om}-1}{\bt_\om^t},
\end{align*}
for example if $H_\om$ is the last full branch.
To ensure that \eqref{cond tilde Q1} holds, note that \eqref{eq: on avg CPN1N2} holds with $N=1$, and thus it suffices to have 
\begin{align*}
\int_\Om \log(\floor{\bt_\om}-1)\, dm(\om)
> 
\int_\Om \log(\zt_\om^{(1)}+2)\, dm(\om),
\end{align*}
since 
\begin{align*}
\int_{\Om}\log\inf\cL_\om\ind_\om 
-
\log\norm{g_\om}_\infty\, dm(\om)
\geq 
\int_\Om \log\frac{\floor{\bt_\om}-1}{\bt_\om^t}+\log\bt_\om^t \, dm(\om).
\end{align*}
Depending on the placement of $H_\om$ we may have $\zt_\om^{(1)}=i$ for any $i\in\set{0,1,2,3}$.
Thus, we obtain the following lemma assuming the worst case scenario, i.e. assuming $\zt_\om^{(1)}=3$ for $m$-a.e. $\om\in\Om$. 
\begin{lemma}
Theorems~\ref{main thm: existence}-\ref{main thm: escape rate} hold if 
\begin{align*}
\int_\Om \log(\floor{\bt_\om}-1)\, dm(\om)
> 
\log 5.
\end{align*}
\end{lemma}
On the other hand, if we have that $H_\om$ is equal to the monotonicity partition element which contains $1$, then $\zt_\om^{(1)}=0$ and 
\begin{align*}
\inf\cL_\om\ind_\om=\frac{\floor{\bt_\om}}{\bt_\om^t}.
\end{align*}
Furthermore, the additional hypotheses necessary for Theorem~\ref{main thm: Bowens formula} are satisfied. In particular, the fact that $T$ has large images follows from the fact that these maps have a uniform covering time; see Lemma~13.5 of \cite{AFGTV20}.
Thus, we thus have the following lemma. 
\begin{lemma}
If $H_\om=Z_1$, where $1\in Z_1\in\cZ_\om$, for $m$-a.e. $\om\in\Om$ then Theorems~\ref{main thm: existence}-\ref{main thm: Bowens formula} hold.
\end{lemma}

More generally, we can consider general potentials, non-linear maps, and holes which are unions of finitely many intervals so that condition \eqref{CCH} holds.

\subsection{Random Open Lasota-Yorke Maps}\label{sec: p1 ly map example}
We now present an example of a large class of random Lasota-Yorke maps with holes. The following lemma summarizes the closed setting for this particular class of random maps was treated in Section 13.6 of \cite{AFGTV20}.

\begin{lemma}\label{lem: closed LY example}
Let $\vp_0:\Om \to \BV(I)$ be an $m$-continuous function, and let
$\vp_0:\Om \times I \to \RR$ be given by $\vp_{\om,0}:=-t\log|T_\om'|=\vp_0(\om)$ for $t\geq 0$. Then 
$g_{\om,0} = e^{\vp_{\om,0}}=1/|T_\om'|^t\in \BV(I)$ for \maeom. We further suppose the system satisfies the following: 
\begin{enumerate} 
\item[\mylabel{1}{hyp 1}] $\log\#\cZ_\om\in L^1(m)$,
\item[\mylabel{2}{hyp 2}] there exists $M(n)\in\NN$ such that for any $\om\in\Om$ and any $Z\in\cZ_\om^{(n)}$ we have that $T_\om^{M(n)}(Z)=I$,	
\item[\mylabel{3}{hyp 3}] for each $\om\in\Om$, $Z\in\cZ_\om$, and $x\in Z$ 
\begin{enumerate}
\item $T_\om\rvert_Z\in C^2$, 
\item there exists $K\geq 1$ such that
\begin{align*}
\frac{|T_\om''(x)|}{|T_\om'(x)|}\leq K,
\end{align*}
\end{enumerate}	
\item[\mylabel{4}{hyp 4}] there exist $1< \lm\leq \Lm<\infty$ and $n_0\in\NN$ such that 
\begin{enumerate}
\item[\mylabel{a}{hyp 4a}] $|T_\om'|\leq \Lm$ for $m$-a.e. $\om\in\Om$,
\item[\mylabel{b}{hyp 4b}] $|(T_\om^{n_0})'|\geq \lm^{n_0}$ for $m$-a.e. $\om\in\Om$
\item[\mylabel{c}{hyp 4c}] $\frac{1}{n_0}\int_\Om \log F_\om^{(n_0)} \,dm(\om)>t\log\frac{\Lm}{\lm}$,
\end{enumerate}
\item[\mylabel{5}{hyp 5}] for each $n\in\NN$ there exists 
$$
\ep_n:=\inf_{\om\in\Om}\min_{Z\in\cZ_\om^{(n)}}\diam(Z) >0.
$$
\end{enumerate}
Then Theorems~ 2.19-2.23  of \cite{AFGTV20} hold, and in particular, our assumptions \eqref{T1}-\eqref{T3}, \eqref{LIP}, \eqref{GP}, \eqref{A1}-\eqref{A2}, \eqref{cond M1}, and \eqref{cond C1} hold. 
\end{lemma}

The following lemma gives a large class of random Lasota-Yorke maps with holes for which our results apply. In particular, we allow our hole to be composed of finitely many intervals which may change depending on the fiber $\om$, provided the number of connected components of the hole is $\log$-integrable over $\Om$ \eqref{CCH}. 

\begin{lemma}\label{lem: open LY example}
Let $\vp_{\om,0}=-t\log|T_\om'|$ and suppose the hypotheses of Lemma~\ref{lem: closed LY example} hold. Additionally we suppose that $H\sub\Om\times I$ such that \eqref{CCH} holds as well as the following:
\begin{enumerate}
\item[\mylabel{6}{open 1}] for $m$-a.e. $\om\in\Om$ there exists $Z\in\cZ_\om$ with $Z\cap H_\om=\emptyset$ such that $T_\om(Z)=I$,	
\item[\mylabel{7}{open 2}] $\frac{1}{n_0}\int_\Om \log F_\om^{(n_0)} \,dm(\om)>t\log\frac{\Lm}{\lm}+\int_\Om \log(\zt_\om^{(1)}+2)\, dm(\om)$.
\end{enumerate}
Then the hypotheses of Theorems~\ref{main thm: existence}-\ref{main thm: escape rate} hold. If in addition we have that 
\begin{enumerate}
\item[\mylabel{8}{open 3}] there exists $M:\NN\to\NN$ such that $T_\om^{M(n)}(Z)=I$ for $m$-a.e. $\om\in\Om$  and each $Z\in\cZ_\om^{(n)}$, i.e. there is a uniform covering time,
\item[\mylabel{9}{open 4}] for $m$-a.e. $\om\in\Om$ there exists $Z_1,\dots,Z_{k}\in\cZ_\om$ such that $H_\om=\cup_{j=1}^{k}Z_j$ and $T_\om(Z)=I$ for all $Z\in\cZ_\om$ with 
$Z\cap H_\om=\emptyset$,
\end{enumerate}
then the hypotheses of Theorem~\ref{main thm: Bowens formula} also hold.
\end{lemma}
\begin{proof}
The conclusion of Lemma~\ref{lem: closed LY example} leaves only to check assumptions \eqref{cond D} and  \eqref{cond tilde Q0}-\eqref{cond tilde Q3}. But in light of Proposition~\ref{prop: tilde Q3 trivial} we see that \eqref{cond tilde Q3} holds, and hypothesis \eqref{open 1} implies \eqref{cond D} (by Remark~\ref{rem: check cond D}) and \eqref{cond tilde Q0} hold.

To check our remaining hypotheses on the open system we first show that \eqref{cond Z'} holds. To see this we note that equation (13.20) of \cite{AFGTV20}, together with the fact that hypothesis \eqref{hyp 2} of Lemma~\ref{lem: closed LY example} implies that
$\Lm^{-kn_0t}\leq g_{\om,0}^{(kn_0)}\leq \lm^{-kn_0t}<1$,  gives that for any $\om\in\Om$ and $Z\in\cZ_\om^{(kn_0)}$ we have
\begin{align*}
\var_Z(g_{\om,0}^{(kn_0)})
&\leq 
2\norm{g_{\om,0}^{(kn_0)}}_\infty + \frac{tK}{\Lm-1}\cdot\lt(\frac{\Lm}{\lm^t}\rt)^{kn_0}
\\
&\leq 
2\norm{g_{\om,0}^{(kn_0)}}_\infty + \frac{tK}{\Lm-1}\cdot\lt(\frac{\Lm}{\lm^t}\rt)^{kn_0}\cdot \Lm^{kn_0t}\norm{g_{\om,0}^{(kn_0)}}_\infty
\\
&\leq
\hat\al_k\norm{g_{\om,0}^{(kn_0)}}_\infty,
\end{align*}
where
\begin{align*}
\hat\al_k:=2+\frac{tK}{\Lm-1}\cdot\lt(\frac{\Lm^{2t}}{\lm^t}\rt)^{kn_0}.
\end{align*}
Taking $k_*$ so large that
\begin{align*}
\int_\Om\log Q_{\om}^{(k_*n_0)}\, dm(\om)<0,
\end{align*}
where $Q_\om^{(k_*n_0)}$ is defined as in \eqref{eq: def of Q and K},
and setting $N_*=k_*n_0$, we then see that \eqref{cond Z'} holds, that is we have that $\cZ_\om^{(N_*)}\in\ol\sA_\om^{(N_*)}(\hat\al_{k_*})$ for each $\om\in\Om$. Thus, Lemma~\ref{lem: check Q2' with zt} together with \eqref{CCH} ensures that \eqref{cond tilde Q2} holds.  
Now taking \eqref{open 2} in conjunction with Lemma~\ref{lem: LY example Q1 Lemma} implies assumption \eqref{cond tilde Q1}, and thus the hypotheses of Theorems~\ref{main thm: existence}-\ref{main thm: escape rate} hold.

The hypotheses of Theorem~\ref{main thm: Bowens formula} hold since the assumptions \eqref{open 3} and \eqref{open 4} together imply that $T$ has large images and large images with respect to $H$, and assumptions \eqref{hyp 3} and \eqref{hyp 4}\eqref{hyp 4b} give the bounded distortion condition for $g_{\om,0}$. 

\end{proof}

\begin{remark}
If one wishes to work with general potentials rather than the geometric potentials in Lemmas~\ref{lem: closed LY example} and \ref{lem: open LY example} then one could replace \eqref{hyp 4} of Lemma~\ref{lem: closed LY example} with \eqref{eq: on avg CPN1N2} and \eqref{open 2} of Lemma~\ref{lem: open LY example} with Lemma~\ref{lem: LY example Q1 Lemma}.
\end{remark}

\chapter[Perturbation formulae and quenched extreme value  theory]{Perturbation formulae for quenched random dynamics with applications to open systems and extreme value theory}\label{part 2}	
\normalsize
In this second chapter we first develop a perturbation theory for the quasi-compact linear operators cocycle $\cL_{\om}^n$ and its perturbed version $\cL_{\om, \eps}^n$ By defining $\lambda_{\om, \eps}$ as the leading Lyapunov multiplier of $\cL_{\om, \eps},$ we will  get a first order formula for $\lambda_{\om, \eps}$ in terms of $\lambda_{\om, 0}$ and in the size of the perturbation $\cL_{\om, 0}-\cL_{\om, \eps}.$ Whenever $\cL_{\om, \eps}$ is a transfer operator cocycle for a random map cocycle $T^n_{\om},$ it will be defined by the introduction of small random holes $H_{\om, \eps}.$ The first-order perturbation for the Lyapunov multiplier will therefore been used to obtain a quenched extreme value theory, by using a suitable spectral approach. By pursuing with the perturbative scheme, we will establish the existence of equilibrium states and conditionally invariant measures and we finally prove quenched limit theorems for equilibrium states arising from contracting potentials.
\section{Sequential perturbation theorem}\label{Sec: Gen Perturb Setup}
In this section we briefly depart from the setting of random dynamical systems to prove a general perturbation result for sequential operators acting on sequential Banach spaces. In particular, we will not require any measurability or notion of randomness in this section.  

Suppose that $\Om$ is a set and that the map $\sg:\Om\to\Om$ is invertible.
Furthermore, we suppose that there is a family of (fiberwise) normed vector spaces  $\set{\cB_\om, \norm{\spot}_{\cB_\om}}_{\om\in\Om}$ and dual spaces $\set{\cB_\om^*,\norm{\spot}_{\cB_\om^*}}_{\om\in\Om}$ such that for each $\om\in\Om$ and each $0\leq \ep\leq \ep_0$ there are operators $\cL_{\om,\ep}:\cB_\om\to\cB_{\sg\om}$ such that the following hold.\index{$\cL_{\om,\ep}$}
\begin{enumerate}[align=left,leftmargin=*,labelsep=\parindent]
\item[(\Gls*{P1})]\myglabel{P1}{P1}There exists a
function $C_1:\Om\to\RR_+$ such that for $f\in\cB_\om$  we have
\begin{align*}
\sup_{\ep\geq 0}\norm{\cL_{\om,\ep}(f)}_{\cB_{\sg\om}}\leq C_1(\om)\norm{f}_{\cB_\om}.
\end{align*}
\item[(\Gls*{P2})]\myglabel{P2}{P2} For each $\om\in\Om$ and $\ep\geq 0$ there is a functional $\nu_{\om,\ep}\in\cB_\om^*$, the dual space of $\cB_\om$, $\lm_{\om,\ep}\in\CC\bs\set{0}$, and $\phi_{\om,\ep}\in\cB_\om$ such that
\begin{align*}
\cL_{\om,\ep}(\phi_{\om,\ep})=\lm_{\om,\ep}\phi_{\sg\om,\ep}
\quad\text{ and }\quad
\nu_{\sg\om,\ep}(\cL_{\om,\ep}(f))=\lm_{\om,\ep}\nu_{\om,\ep}(f)
\end{align*}
for all $f\in\cB_\om$. Furthermore we assume that
$$
\nu_{\om,0}(\phi_{\om,\ep})=1.
$$
\item[(\Gls*{P3})]\myglabel{P3}{P3} There is an operator $Q_{\om,\ep}:\cB_\om\to\cB_{\sg\om}$ such that for each $f\in\cB_\om$ we have
\begin{align*}
\lm_{\om,\ep}^{-1}\cL_{\om,\ep}(f)=\nu_{\om,\ep}(f)\cdot\phi_{\sg\om,\ep}+Q_{\om,\ep}(f).
\end{align*}
Furthermore, we have
\begin{align*}
Q_{\om,\ep}(\phi_{\om,\ep})=0
\quad\text{ and }\quad
\nu_{\sg\om,\ep}(Q_{\om,\ep}(f))=0.
\end{align*}
Note that assumptions \eqref{P2} and \eqref{P3} together imply that
\begin{align*}
\nu_{\om,\ep}(\phi_{\om,\ep})=1.
\end{align*}
\item[(\Gls*{P4})]\myglabel{P4}{P4} There exists a function $C_2:\Om\to\RR_+$ such that
\begin{align*}
\sup_{\ep\geq 0}\norm{\phi_{\om,\ep}}_{\cB_{\om}}=C_2(\om)<\infty.
\end{align*}

For each $\om\in\Om$ and $\ep\geq 0$  we define the quantities
\begin{align}\label{def: DL_om}
\Dl_{\om,\ep}:=\nu_{\sg\om,0}\left((\cL_{\om,0}-\cL_{\om,\ep})(\phi_{\om,0})\right)
\end{align}\index{$\Dl_{\om,\ep}$}
and
\begin{align}\label{def: eta_om}
\eta_{\om,\ep}:=\norm{\nu_{\sg\om,0}(\cL_{\om,0}-\cL_{\om,\ep})}_{\cB_\om^*}.
\end{align}\index{$\eta_{\om,\ep}$}
\item[(\Gls*{P5})]\myglabel{P5}{P5} For each $\om\in\Om$ we have
\begin{align*}
\lim_{\ep\to 0}\eta_{\om,\ep}=0.
\end{align*}
\item[(\Gls*{P6})]\myglabel{P6}{P6} For each $\om\in\Om$ such that $\Dl_{\om,\ep}>0$ for every $\ep>0$, we have that there exists a  function $C_3:\Om\to\RR_+$ such that
\begin{align*}
\limsup_{\ep\to 0}\frac{\eta_{\om,\ep}}{\Dl_{\om,\ep}}:=C_3(\om) <\infty.
\end{align*}
Given $\om\in\Om$, if there is $\ep_0>0$ such that for each $\ep\leq \ep_0$ we have that $\Dl_{\om,\ep}=0$ then we also have that $\eta_{\om,\ep}=0$ for each $\ep\leq \ep_0$.
\item[(\Gls*{P7})]\myglabel{P7}{P7}  For each $\om\in\Om$ we have
\begin{align*}
\lim_{\ep \to 0}\nu_{\om,\ep}(\phi_{\om,0})=1.
\end{align*}	
\item[(\Gls*{P8})]\myglabel{P8}{P8} For each $\om\in\Om$ with $\Dl_{\om,\ep}>0$ for all $\ep>0$ we have
\begin{flalign*}
\lim_{n\to\infty}\limsup_{\ep \to 0} \Dl_{\om,\ep}^{-1}\nu_{\sg\om,0}\lt(\lt(\cL_{\om,0}-\cL_{\om,\ep}\rt)\lt(Q_{\sg^{-n}\om,\ep}^n\phi_{\sg^{-n}\om,0}\rt)\rt)=0.
\end{flalign*}
\item[(\Gls*{P9})]\myglabel{P9}{P9} For each $\om\in\Om$ with $\Dl_{\om,\ep}>0$ for all $\ep>0$ we have the limit
\begin{align*}
q_{\om,0}^{(k)}:=\lim_{\ep\to 0} \frac{\nu_{\sg\om,0}\left((\cL_{\om,0}-\cL_{\om,\ep})(\cL_{\sg^{-k}\om,\ep}^k)(\cL_{\sg^{-(k+1)}\om,0}-\cL_{\sg^{-(k+1)}\om,\ep})(\phi_{\sg^{-(k+1)}\om,0})\right)}{\nu_{\sg\om,0}\left((\cL_{\om,0}-\cL_{\om,\ep})(\phi_{\om,0})\right)}
\end{align*}\index{$q_{\om,0}^{(k)}$}
exists for each $k\geq 0$.
\end{enumerate}

Consider the identity
\begin{align}
\lm_{\om,0}-\lm_{\om,\ep}&=\lm_{\om,0}\nu_{\om,0}(\phi_{\om,\ep})-\nu_{\sg\om,0}(\lm_{\om,\ep}\phi_{\sg\om,\ep})\nonumber\\
&=\nu_{\sg\om,0}(\cL_{\om,0}(\phi_{\om,\ep}))-\nu_{\sg\om,0}(\cL_{\om,\ep}(\phi_{\om,\ep}))\nonumber\\
&=\nu_{\sg\om,0}\left((\cL_{\om,0}-\cL_{\om,\ep})(\phi_{\om,\ep})\right).\label{diff eigenvalues identity unif}
\end{align}
It then follows from \eqref{diff eigenvalues identity unif}, together with \eqref{def: eta_om} and assumption \eqref{P4}, that
\begin{align}\label{conv eigenvalues unif}
\absval{\lm_{\om,0}-\lm_{\om,\ep}}\leq C_2(\om)\eta_{\om,\ep}.
\end{align}
In particular, given assumption \eqref{P5}, \eqref{conv eigenvalues unif} implies
\begin{align}\label{limit of eigenvalues}
\lim_{\ep\to 0}\lm_{\om,\ep}=\lm_{\om,0}
\end{align}
for each $\om\in\Om$.
\begin{remark}
Note that \eqref{P6} and \eqref{conv eigenvalues unif} imply that
\begin{align*}
\limsup_{\ep\to 0}\frac{\absval{\lm_{\om,0}-\lm_{\om,\ep}}}{\Dl_{\om,\ep}}\leq C_2(\om)C_3(\om)<\infty.
\end{align*}
\end{remark}
For $n\geq 1$ we define the normalized operator $\~\cL_{\om,\ep}^n:\cB_\om\to\cB_{\sg^n\om}$ by
\begin{align*}
\~\cL_{\om,\ep}^n:=(\lm_{\om,\ep}^n)^{-1}\cL_{\om,\ep}^n
\end{align*}
where
\begin{align*}
\lm_{\om,\ep}^n:=\lm_{\om,\ep}\cdot\dots\cdot\lm_{\sg^{n-1}\om,\ep}.
\end{align*}
In view of assumption \eqref{P3}, induction gives
\begin{align*}
\~\cL_{\om,\ep}^n(f)=\nu_{\om,\ep}(f)\cdot \phi_{\sg^n\om,\ep}+Q_{\om,\ep}^n(f)
\end{align*}
for each $n\geq 1$ and all $f\in\cB_\om$.
Similarly with \eqref{diff eigenvalues identity unif}, we have that
\begin{align*}
\lm_{\om,0}^n-\lm_{\om,\ep}^n=\nu_{\sg^n\om,0}\left(\left(\cL_{\om,0}^n-\cL_{\om,\ep}^n\right)(\phi_{\om,\ep})\right).
\end{align*}
We now arrive at the main result of this section.
We prove a differentiability result for the perturbed quantities $\lm_{\om,\ep}$ as $\ep\to 0$ in the spirit of Keller and Liverani \cite{keller_rare_2009}.

\begin{theorem}\label{thm: GRPT}
Suppose that assumptions \eqref{P1}-\eqref{P8} hold.
If there is some $\ep_0>0$ such that $\Dl_{\om,\ep}=0$ for $\ep\leq \ep_0$ then
\begin{align*}
\lm_{\om,0}=\lm_{\om,\ep},
\end{align*}
or if \eqref{P9} holds then
\begin{align*}
\lim_{\ep\to 0}\frac{\lm_{\om,0}-\lm_{\om,\ep}}{\Dl_{\om,\ep}}=1-\sum_{k=0}^{\infty}(\lm_{\sg^{-(k+1)}\om,0}^{k+1})^{-1}q_{\om,0}^{(k)}.
\end{align*}
\end{theorem}
\begin{proof}

Fix $\om\in\Om$.
First, we note that  assumption \eqref{P6} implies that if there is some $\ep_0$ such that $\Dl_{\om,\ep}=0$ for all $\ep\leq \ep_0$ (which implies, by assumption, that $\eta_{\om,\ep}=0$), then \eqref{conv eigenvalues unif} immediately implies that
\begin{align*}
\lm_{\om,0}=\lm_{\om,\ep}.
\end{align*}
Now we suppose that $\Dl_{\om,\ep}>0$ for all $\ep>0$. 
Using \eqref{P2}, \eqref{P3}, and \eqref{diff eigenvalues identity unif}, for each $n\geq 1$ and all $\om\in\Om$ we have
\begin{align}
&\nu_{\sg^{-n}\om,\ep}(\phi_{\sg^{-n}\om,0})(\lm_{\om,0}-\lm_{\om,\ep})=\nu_{\sg^{-n}\om,\ep}(\phi_{\sg^{-n}\om,0})\nu_{\sg\om,0}\left((\cL_{\om,0}-\cL_{\om,\ep})(\phi_{\om,\ep})\right)
\nonumber\\
&\quad
=\nu_{\sg\om,0}\left((\cL_{\om,0}-\cL_{\om,\ep})(\nu_{\sg^{-n}\om,\ep}(\phi_{\sg^{-n}\om,0})\cdot\phi_{\om,\ep})\right)
\nonumber\\
&\quad
=\nu_{\sg\om,0}\left((\cL_{\om,0}-\cL_{\om,\ep})(\nu_{\sg^{-n}\om,\ep}(\phi_{\sg^{-n}\om,0})\cdot\phi_{\om,\ep}+Q_{\sg^{-n}\om,\ep}^n(\phi_{\sg^{-n}\om,0})-Q_{\sg^{-n}\om,\ep}^n(\phi_{\sg^{-n}\om,0}))\right)
\nonumber\\
&\quad
=\nu_{\sg\om,0}\left((\cL_{\om,0}-\cL_{\om,\ep})(\~\cL_{\sg^{-n}\om,\ep}^n-Q_{\sg^{-n}\om,\ep}^n)(\phi_{\sg^{-n}\om,0})\right)
\nonumber\\
&\quad
=\nu_{\sg\om,0}\left((\cL_{\om,0}-\cL_{\om,\ep})(\~\cL_{\sg^{-n}\om,0}^n-\~\cL_{\sg^{-n}\om,0}^n+\~\cL_{\sg^{-n}\om,\ep}^n-Q_{\sg^{  -n}\om,\ep}^n)(\phi_{\sg^{-n}\om,0})\right)
\nonumber\\
&\quad
=\nu_{\sg\om,0}\left((\cL_{\om,0}-\cL_{\om,\ep})(\phi_{\om,0})\right)
\nonumber\\
&\qquad
-\nu_{\sg\om,0}\left((\cL_{\om,0}-\cL_{\om,\ep})(\~\cL_{\sg^{-n}\om,0}^n-\~\cL_{\sg^{-n}\om,\ep}^n)(\phi_{\sg^{-n}\om,0})\right)
\nonumber\\
&\qquad\qquad
-\nu_{\sg\om,0}\left((\cL_{\om,0}-\cL_{\om,\ep})(Q_{\sg^{-n}\om,\ep}^n(\phi_{\sg^{-n}\om,0}))\right)
\nonumber\\
&\quad
=\Dl_{\om,\ep}-\nu_{\sg\om,0}\left((\cL_{\om,0}-\cL_{\om,\ep})(\~\cL_{\sg^{-n}\om,0}^n-\~\cL_{\sg^{-n}\om,\ep}^n)(\phi_{\sg^{-n}\om,0})\right)
\label{D_2 estimate unif}\\
&\qquad
-\nu_{\sg\om,0}\left((\cL_{\om,0}-\cL_{\om,\ep})(Q_{\sg^{-n}\om,\ep}^n(\phi_{\sg^{-n}\om,0}))\right).
\label{Big O estimate unif}
\end{align}
Since
\begin{align*}
(\~\cL_{\sg^{-n}\om,0}^n-\~\cL_{\sg^{-n}\om,\ep}^n)(\phi_{\sg^{-n}\om,0})
=
\phi_{\om,0} - \~\cL_{\sg^{-n}\om,\ep}^n(\phi_{\sg^{-n}\om,0}),
\end{align*}
using a telescoping argument, the second term of \eqref{D_2 estimate unif}
$$
D_2:=-\nu_{\sg\om,0}\left((\cL_{\om,0}-\cL_{\om,\ep})(\~\cL_{\sg^{-n}\om,0}^n-\~\cL_{\sg^{-n}\om,\ep}^n)(\phi_{\sg^{-n}\om,0})\right)
$$
can be rewritten as
\begin{align}
D_2&=-\sum_{k=0}^{n-1}\nu_{\sg\om,0}\left((\cL_{\om,0}-\cL_{\om,\ep})(\~\cL_{\sg^{-k}\om,\ep}^k)(\~\cL_{\sg^{-(k+1)}\om,0}-\~\cL_{\sg^{-(k+1)}\om,\ep})(\phi_{\sg^{-(k+1)}\om,0})\right)
\nonumber\\
&
=-\sum_{k=0}^{n-1}\nu_{\sg\om,0}\left((\cL_{\om,0}-\cL_{\om,\ep})(\~\cL_{\sg^{-k}\om,\ep}^k)(\~\cL_{\sg^{-(k+1)}\om,0}\right.
\nonumber\\
&\qquad
\left.-\lm_{\sg^{-(k+1)}\om,0}^{-1}\cL_{\sg^{-(k+1)}\om,\ep}+\lm_{\sg^{-(k+1)}\om,0}^{-1}\cL_{\sg^{-(k+1)}\om,\ep}-\~\cL_{\sg^{-(k+1)}\om,\ep})(\phi_{\sg^{-(k+1)}\om,0})\right)
\nonumber\\
&
=-\sum_{k=0}^{n-1}\nu_{\sg\om,0}\left((\cL_{\om,0}-\cL_{\om,\ep})(\~\cL_{\sg^{-k}\om,\ep}^k)(\~\cL_{\sg^{-(k+1)}\om,0}\right.
\label{first summand in D2}	\\
&\qquad
\left.-\lm_{\sg^{-(k+1)}\om,0}^{-1}\cL_{\sg^{-(k+1)}\om,\ep}+\lm_{\sg^{-(k+1)}\om,0}^{-1}\cL_{\sg^{-(k+1)}\om,\ep})(\phi_{\sg^{-(k+1)}\om,0})\right)
\nonumber\\
&\qquad\qquad
+\sum_{k=0}^{n-1}\nu_{\sg\om,0}\left((\cL_{\om,0}-\cL_{\om,\ep})(\~\cL_{\sg^{-k}\om,\ep}^k)(\~\cL_{\sg^{-(k+1)}\om,\ep})(\phi_{\sg^{-(k+1)}\om,0})\right).
\label{second summand in D2}
\end{align}
Reindexing, the second summand of the above calculation, and multiplying by $1$, \eqref{second summand in D2}, can be rewritten as
\begin{align}
\sum_{k=0}^{n-1}&\nu_{\sg\om,0}\left((\cL_{\om,0}-\cL_{\om,\ep})(\~\cL_{\sg^{-(k+1)}\om,\ep}^{k+1})(\phi_{\sg^{-(k+1)}\om,0})\right)
\nonumber\\
&\quad
=\sum_{k=1}^{n}\nu_{\sg\om,0}\left((\cL_{\om,0}-\cL_{\om,\ep})(\~\cL_{\sg^{-k}\om,\ep}^{k})(\phi_{\sg^{-k}\om,0})\right)
\nonumber\\
&\quad
=\sum_{k=1}^{n}\lm_{\sg^{-k}\om,0}^{-1}\lm_{\sg^{-k}\om,0}\nu_{\sg\om,0}\left((\cL_{\om,0}-\cL_{\om,\ep})(\~\cL_{\sg^{-k}\om,\ep}^{k})(\phi_{\sg^{-k}\om,0})\right).
\label{est1 of D2}
\end{align}
And \eqref{first summand in D2} from above can again be broken into two parts and then rewritten as
\begin{align}
&-\sum_{k=0}^{n-1}\lm_{\sg^{-(k+1)}\om,0}^{-1}\nu_{\sg\om,0}\left((\cL_{\om,0}-\cL_{\om,\ep})(\~\cL_{\sg^{-k}\om,\ep}^k)(\cL_{\sg^{-(k+1)}\om,0}-\cL_{\sg^{-(k+1)}\om,\ep})(\phi_{\sg^{-(k+1)}\om,0})\right)
\nonumber\\
&\quad
-\sum_{k=1}^{n}\lm_{\sg^{-k}\om,0}^{-1}\lm_{\sg^{-k}\om,\ep}\nu_{\sg\om,0}\left((\cL_{\om,0}-\cL_{\om,\ep})(\~\cL_{\sg^{-k}\om,\ep}^{k})(\phi_{\sg^{-k}\om,0})\right),
\label{est2 of D2}
\end{align}
where in the second sum we have used the fact that $\cL_{\sg^{-k}\om,\ep}=\lm_{\sg^{-k}\om,\ep}\~\cL_{\sg^{-k}\om,\ep}$.
Altogether using \eqref{est1 of D2} and \eqref{est2 of D2}, $D_2$ we can be written as
\begin{align}
D_2&=
-\sum_{k=0}^{n-1}\lm_{\sg^{-(k+1)}\om,0}^{-1}\nu_{\sg\om,0}\left((\cL_{\om,0}-\cL_{\om,\ep})(\~\cL_{\sg^{-k}\om,\ep}^k)(\cL_{\sg^{-(k+1)}\om,0}-\cL_{\sg^{-(k+1)}\om,\ep})(\phi_{\sg^{-(k+1)}\om,0})\right)
\nonumber\\
&\qquad
-\sum_{k=1}^{n}\lm_{\sg^{-k}\om,0}^{-1}\lm_{\sg^{-k}\om,\ep}\nu_{\sg\om,0}\left((\cL_{\om,0}-\cL_{\om,\ep})(\~\cL_{\sg^{-k}\om,\ep}^{k})(\phi_{\sg^{-k}\om,0})\right)
\nonumber\\
&\qquad\quad
+\sum_{k=1}^{n}\lm_{\sg^{-k}\om,0}^{-1}\lm_{\sg^{-k}\om,0}\nu_{\sg\om,0}\left((\cL_{\om,0}-\cL_{\om,\ep})(\~\cL_{\sg^{-k}\om,\ep}^{k})(\phi_{\sg^{-k}\om,0})\right)
\nonumber\\
&=
-\sum_{k=0}^{n-1}\lm_{\sg^{-(k+1)}\om,0}^{-1}\nu_{\sg\om,0}\left((\cL_{\om,0}-\cL_{\om,\ep})(\~\cL_{\sg^{-k}\om,\ep}^k)(\cL_{\sg^{-(k+1)}\om,0}-\cL_{\sg^{-(k+1)}\om,\ep})(\phi_{\sg^{-(k+1)}\om,0})\right)
\nonumber\\
&\quad
=+\sum_{k=1}^{n}\lm_{\sg^{-k}\om,0}^{-1}\left(\lm_{\sg^{-k}\om,0}-\lm_{\sg^{-k}\om,\ep}\right)\nu_{\sg\om,0}\left((\cL_{\om,0}-\cL_{\om,\ep})(\~\cL_{\sg^{-k}\om,\ep}^{k})(\phi_{\sg^{-k}\om,0})\right).\label{final est D2}
\end{align}
Now for each $k\geq 0$ we let
\begin{align}\label{eq: def of q_ep}
q_{\om,\ep}^{(k)}:=\frac{\nu_{\sg\om,0}\left((\cL_{\om,0}-\cL_{\om,\ep})(\cL_{\sg^{-k}\om,\ep}^k)(\cL_{\sg^{-(k+1)}\om,0}-\cL_{\sg^{-(k+1)}\om,\ep})(\phi_{\sg^{-(k+1)}\om,0})\right)}{\nu_{\sg\om,0}\left((\cL_{\om,0}-\cL_{\om,\ep})(\phi_{\om,0})\right)}.
\end{align}\index{$q_{\om,\ep}^{(k)}$}
Using \eqref{final est D2} we can continue our rephrasing of $\nu_{\sg^{-n}\om,\ep}(\phi_{\sg^{-n}\om,0})(\lm_{\om,0}-\lm_{\om,\ep})$ from \eqref{D_2 estimate unif} to get
\begin{align}
&\nu_{\sg^{-n}\om,\ep}(\phi_{\sg^{-n}\om,0})(\lm_{\om,0}-\lm_{\om,\ep})\nonumber
\\
&\quad=\Dl_{\om,\ep}-\sum_{k=0}^{n-1}\lm_{\sg^{-(k+1)}\om,0}^{-1}\nu_{\sg\om,0}\left((\cL_{\om,0}-\cL_{\om,\ep})(\~\cL_{\sg^{-k}\om,\ep}^k)(\cL_{\sg^{-(k+1)}\om,0}-\cL_{\sg^{-(k+1)}\om,\ep})(\phi_{\sg^{-(k+1)}\om,0})\right)
\nonumber\\
&\qquad
+\sum_{k=1}^{n}\lm_{\sg^{-k}\om,0}^{-1}\left(\lm_{\sg^{-k}\om,0}-\lm_{\sg^{-k}\om,\ep}\right)\nu_{\sg\om,0}\left((\cL_{\om,0}-\cL_{\om,\ep})(\~\cL_{\sg^{-k}\om,\ep}^{k})(\phi_{\sg^{-k}\om,0})\right)
\nonumber\\
&\quad\qquad -\nu_{\sg\om,0}\left((\cL_{\om,0}-\cL_{\om,\ep})(Q_{\sg^{-n}\om,\ep}^n(\phi_{\sg^{-n}\om,0}))\right)
\nonumber\\
&\quad=\Dl_{\om,\ep}\left(1-\sum_{k=0}^{n-1}\lm_{\sg^{-(k+1)}\om,0}^{-1}(\lm_{\sg^{-k}\om,\ep}^k)^{-1}q_{\om,\ep}^{(k)}\right)
\nonumber\\
&\qquad
+\sum_{k=1}^{n}\lm_{\sg^{-k}\om,0}^{-1}\left(\lm_{\sg^{-k}\om,0}-\lm_{\sg^{-k}\om,\ep}\right)\nu_{\sg\om,0}\left((\cL_{\om,0}-\cL_{\om,\ep})(\~\cL_{\sg^{-k}\om,\ep}^{k})(\phi_{\sg^{-k}\om,0})\right)
\nonumber\\
&\quad\qquad 	-\nu_{\sg\om,0}\left((\cL_{\om,0}-\cL_{\om,\ep})(Q_{\sg^{-n}\om,\ep}^n(\phi_{\sg^{-n}\om,0}))\right).\label{main thm long calc}
\end{align}
Dividing the calculation culminating in \eqref{main thm long calc} by $\Dl_{\om,\ep}$ on both sides gives
\begin{align}
\label{maineq}
&\nu_{\sg^{-n}\om,\ep}(\phi_{\sg^{-n}\om,0})\frac{\lm_{\om,0}-\lm_{\om,\ep}}{\Dl_{\om,\ep}}
\\
&\quad=1-\sum_{k=0}^{n-1}\lm_{\sg^{-(k+1)}\om,0}^{-1}(\lm_{\sg^{-k}\om,\ep}^k)^{-1}q_{\om,\ep}^{(k)}
\nonumber\\
&\quad
+\Dl_{\om,\ep}^{-1}\sum_{k=1}^{n}\lm_{\sg^{-k}\om,0}^{-1}\left(\lm_{\sg^{-k}\om,0}-\lm_{\sg^{-k}\om,\ep}\right)\nu_{\sg\om,0}\left((\cL_{\om,0}-\cL_{\om,\ep})(\~\cL_{\sg^{-k}\om,\ep}^{k})(\phi_{\sg^{-k}\om,0})\right)
\label{final calc second summand}\\
&\qquad -	\Dl_{\om,\ep}^{-1}\nu_{\sg\om,0}\left((\cL_{\om,0}-\cL_{\om,\ep})(Q_{\sg^{-n}\om,\ep}^n(\phi_{\sg^{-n}\om,0}))\right).
\label{final calc third summand}
\end{align}
Assumption \eqref{P8} ensures that \eqref{final calc third summand} goes to zero as $\ep\to 0$ and $n\to\infty$. Now, using \eqref{def: eta_om}, \eqref{conv eigenvalues unif}, \eqref{P1}, and \eqref{P4} we bound \eqref{final calc second summand} by
\begin{align}
&\Dl_{\om,\ep}^{-1}\sum_{k=1}^n\absval{\lm_{\sg^{-k}\om,0}}^{-1}\absval{\lm_{\sg^{-k}\om,0}-\lm_{\sg^{-k}\om,\ep}}\eta_{\om,\ep}\norm{\~\cL_{\sg^{-k}\om,\ep}^k(\phi_{\sg^{-k}\om,0})}_{\cB_{\om}}
\nonumber\\
&\qquad\qquad
\leq
\frac{\eta_{\om,\ep}}{\Dl_{\om,\ep}}
\sum_{k=1}^nC_2(\sg^{-k}\om)\absval{\lm_{\sg^{-k}\om,0}}^{-1}\eta_{\sg^{-k}\om,\ep}\absval{\lm_{\sg^{-k}\om,\ep}^k}^{-1}\norm{\cL_{\sg^{-k}\om,\ep}^k(\phi_{\sg^{-k}\om,0})}_{\cB_{\om}}
\nonumber\\
&\qquad\qquad
\leq
\frac{\eta_{\om,\ep}}{\Dl_{\om,\ep}}
\sum_{k=1}^n(C_2(\sg^{-k}\om))^2C_1^k(\sg^{-k}\om)\absval{\lm_{\sg^{-k}\om,0}}^{-1}\eta_{\sg^{-k}\om,\ep}\absval{\lm_{\sg^{-k}\om,\ep}^k}^{-1}.
\label{theorem sum est}
\end{align}
In view of \eqref{P5}, \eqref{P6}, and \eqref{limit of eigenvalues}, for fixed $n$, we may continue from \eqref{theorem sum est} and let $\ep\to 0$ to see that
\begin{align}\label{second summand goes to 0 for final calc}
\lim_{\ep \to 0}
\frac{\eta_{\om,\ep}}{\Dl_{\om,\ep}}
\sum_{k=1}^n(C_2(\sg^{-k}\om))^2C_1^k(\sg^{-k}\om)\absval{\lm_{\sg^{-k}\om,0}}^{-1}\eta_{\sg^{-k}\om,\ep}\absval{\lm_{\sg^{-k}\om,\ep}^k}^{-1}=0.
\end{align}
In light of \eqref{P6}, \eqref{P7}, 
\eqref{second summand goes to 0 for final calc}, and \eqref{maineq}--\eqref{final calc third summand} together with \eqref{P8} and \eqref{P9}, we see that first letting $\ep\to 0$ and then $n\to\infty$ gives
\begin{align*}
\lim_{\ep\to 0}\frac{\lm_{\om,0}-\lm_{\om,\ep}}{\Dl_{\om,\ep}}
= 1-\sum_{k=0}^{\infty}(\lm_{\sg^{-(k+1)}\om,0}^{k+1})^{-1}q_{\om,0}^{(k)}
\end{align*}
as desired.
\end{proof}
In the sequel we will refer to the quantity on the right hand side of the last equation in the proof of the previous theorem by $\ta_{\om,0}$, i.e. we set
\begin{align}\label{eq: def of theta_0}
\ta_{\om,0}:=1-\sum_{k=0}^{\infty}(\lm_{\sg^{-(k+1)}\om,0}^{k+1})^{-1}q_{\om,0}^{(k)}.
\end{align}\index{$\ta_{\om,0}$}

\section{Random open systems}\label{sec:ROS}
We now return to the general random setting of Section \ref{sec:IntroPrelims}. Suppose that $(\mathlist{\bcomma}{\Om, m, \sg, \cJ_0, T, \cB, \cL_0, \nu_0, \phi_0})$ is a closed random dynamical system as in Definition \ref{def CRS}.
That is we have a base dynamical system $(\Om,\sF,m,\sg)$, complete metrisable spaces $\cJ_{\om,0}$ such that the map $\Om\ni \om\mapsto\cJ_{\om,0}$ is a closed random set, and maps $T_\om:\cJ_{\om,0}\to\cJ_{\sg\om,0}$. In addition, we assume that conditions \eqref{M1}, \eqref{M2}, and \eqref{CCM} hold, and that the transfer operator $\cL_{\om,0}$ acting on the family of Banach spaces $\set{\cB_\om,\norm{\spot}_{\cB_\om}}_{\om\in\Om}$ is given by
\begin{align*}
\cL_{\om,0}(f)(x):=\sum_{y\in T_\om^{-1}(x)}f(y)g_{\om,0}(y), \quad f\in\cB_\om, \, x\in\cJ_{\sg\om,0},
\end{align*}
where $g_{\om,0}(x)=e^{\vp_{\om,0}(x)}$ for a suitably chosen random potential $\vp_{\om,0}$.
We will also assume that the fiberwise Banach spaces $\cB_\om\sub L^\infty(\nu_{\om,0})$, where $\nu_0=(\nu_{\om,0})_{\om\in\Om}$ is the random probability measure given by \eqref{CCM}. We denote the norm on $L^\infty(\nu_{\om,0})$ by $\|\cdot\|_{\infty,\om}$.

\, 

Now for each $\ep>0$ we let $H_\ep\sub \cJ_0$ be measurable with respect to the product $\sg$-algebra $\sF\otimes\sB$ on $\cJ_0$ such that 
\begin{enumerate}[align=left,leftmargin=*,labelsep=\parindent]
\item[(\Gls*{A})]\myglabel{A}{A}
$H_\ep'\sub H_\ep$ for each $0<\ep'\leq \ep$.
\end{enumerate}
Then the sets $H_{\om,\ep}\sub \cJ_{\om,0}$ are uniquely determined by the condition that 
\begin{align*}
\set{\om}\times H_{\om,\ep}=H_\ep\cap\lt(\set{\om}\times \cJ_{\om,0}\rt),
\end{align*}
or equivalently that 
\begin{align*}
H_{\om,\ep}=\pi_2(H_\ep\cap(\set{\om}\times \cJ_{\om,0})),
\end{align*}
where $\pi_2:\cJ_0\to \cJ_{\om,0}$ is the projection onto the second coordinate. The sets $H_{\om,\ep}$ are then $\nu_{\om,0}$-measurable, and \eqref{A} implies that 
\begin{enumerate}[align=left,leftmargin=*,labelsep=\parindent]
\item[(\Gls*{A'})]\myglabel{A'}{A'} $H_{\om,\ep'}\sub H_{\om,\ep}$ for each $\ep'\leq \ep$ and each $\om\in\Om$.
\end{enumerate}
For each $\ep>0$ we set 
$$
\Om_{+,\ep}:=\set{\om\in\Om:\mu_{\om,0}(H_{\om,\ep})>0}
$$ 
and then define
\begin{align}
\label{def Om+}
\Om_+:=\bigcap_{\ep>0}\Om_{+,\ep}.
\end{align}\index{$\Om_+$}
\begin{remark}
Note that the set $\Om_+$ is measurable as it is the intersection of a decreasing family of measurable sets and it is not necessarily $\sg$-invariant.
\end{remark}

For each $\om$ and each $\ep> 0$ we define the fibers $\cJ_{\om,\ep}:=\cJ_{\om,0}\bs H_{\om,\ep}$ and 
\begin{align*}
\cJ_\ep:=\cJ_0\bs H_\ep=\union_{\om\in\Om}\{\om\}\times\cJ_{\om,\ep}.
\end{align*}
We define the surviving sets $X_{\om,n,\ep}$, $X_{\om,\infty,\ep}$, $\cX_{n,\ep}$, and $\cX_{\infty,\ep}$ as in Section \ref{sec:IntroPrelims}.\index{$X_{\om,n,\ep}$}\index{$\cX_{n,\ep}$}\index{$X_{\om,\infty,\ep}$}\index{$\cX_{\infty,\ep}$}
\begin{remark}\label{rem check X for ROS}
Note that since \eqref{A} implies that $X_{\om,\infty,\ep}\sub X_{\om,\infty,\ep'}$ for all $\ep'<\ep$, \eqref{cond X} holds, i.e. $X_{\om,\infty,\ep}\neq\emptyset$ for $m$-a.e. $\om\in\Om$,  if there exists $\ep>0$ such that $X_{\om,\infty,\ep}\neq\emptyset$ for $m$-a.e. $\om\in\Om$. Furthermore, since $T_\om(X_{\om,\infty,\ep})\sub X_{\sg\om,\infty,\ep}$, if $X_{\om,\infty,\ep}\neq\emptyset$ then $X_{\sg^N\om,\infty,\ep'}\neq\emptyset$ for each $N\geq 1$ and $\ep'\leq \ep$. As $\cX_{\infty,\ep}$ is forward invariant we have that $X_{\om,\infty,\ep}\neq\emptyset$ not only implies that $\cX_{\infty,\ep}\neq\emptyset$, but also that $\cX_{\infty,\ep}$ is infinite. 
\end{remark}
Now for each $\ep>0$ define the open transfer operator $\cL_{\om,\ep}:\cB_\om\to\cB_{\sg\om}$ by 
\begin{align*}
\cL_{\om,\ep}(f):=\cL_{\om,0}(\ind_{\cJ_{\om,\ep}} f), \qquad f\in\cB_\om.
\end{align*}\index{$\cL_{\om,\ep}$}
Iterates of the perturbed operator $\cL_{\om,\ep}^n:\cB_{\om}\to\cB_{\sg^n\om}$ are given by
\begin{align*}
\cL_{\om,\ep}^n:=\cL_{\sg^{n-1}\om,\ep}\circ\dots\circ\cL_{\om,\ep},
\end{align*}
which, using induction, we may write as
\begin{align*}
\cL_{\om,\ep}^n(f)=\cL_{\om,0}^n\left(f\cdot\hat{X}_{\om,n-1,\ep}\right), \qquad f\in\cB_\om.
\end{align*}
For every $\ep\geq 0$ we let 
$$
\~\cL_{\om,\ep}:=\lm_{\om,\ep}^{-1}\cL_{\om,\ep}.
$$\index{$\~\cL_{\om,\ep}$}

For the remainder of the manuscript we will suppose that $(\mathlist{\bcomma}{\Om, m, \sg, \cJ_0, T, \cB, \cL_0, \nu_0, \phi_0})$ is a closed random dynamical system as Definition \ref{def CRS}, for each $\ep>0$ $H_\ep\sub\cJ_0$ such that condition \eqref{A} holds, and that for each $\ep>0$, $(\mathlist{\bcomma}{\Om, m, \sg, \cJ_0, T, \cB, \cL_0, \nu_0, \phi_0, H_\ep})$ is a random open system as in Definition \ref{def ROS prelim}. In summary, we assume that condition \eqref{A} holds in addition to the assumptions \eqref{M1}, \eqref{M2}, \eqref{CCM}, and \eqref{cond X} from Section \ref{sec:IntroPrelims}.

\subsection{Some of the Terms from Sequential Perturbation Theorem}

In this short section we develop a more thorough understanding of some of the vital terms in the general sequential perturbation theorem of Section \ref{Sec: Gen Perturb Setup} in the setting of random open systems.
We begin by calculating the quantities $\eta_{\om,\ep}$ and $\Dl_{\om,\ep}$ from Section~\ref{Sec: Gen Perturb Setup}. In particular, we have that 
\begin{align}
\Dl_{\om,\ep}&:=\nu_{\sg\om,0}\left((\cL_{\om,0}-\cL_{\om,\ep})(\phi_{\om,0})\right)
\nonumber\\
&=\nu_{\sg\om,0}\left(\cL_{\om,0}(\phi_{\om,0}\cdot\ind_{H_{\om,\ep}})\right)
\nonumber\\
&=\lm_{\om,0}\cdot\nu_{\om,0}(\phi_{\om,0}\cdot\ind_{H_{\om,\ep}})
\nonumber\\
&=\lm_{\om,0}\cdot\mu_{\om,0}(H_{\om,\ep})
\label{eq: Dl = mu H}
\end{align}\index{$\Dl_{\om,\ep}$}
and
\begin{align}
\eta_{\om,\ep}:&=\norm{\nu_{\sg\om,0}\left(\cL_{\om,0}-\cL_{\om,\ep}\right)}_{\cB_\om}
\nonumber
\\
&=\sup_{\norm{\psi}_{\cB_\om}\leq 1}\nu_{\sg\om,0}\left(\cL_{\om,0}(\psi\cdot\ind_{H_{\om,\ep}})\right)
\nonumber
\\
&=\lm_{\om,0}\cdot\sup_{\norm{\psi}_{\cB_\om}\leq 1}\nu_{\om,0}\left(\psi\cdot\ind_{H_{\om,\ep}}\right).
\label{etaineq0}
\end{align}\index{$\eta_{\om,\ep}$}

For $\ep>0$ and $\mu_{\om,0}(H_{\om,\ep})>0$, calculating $q_{\om,\ep}^{(k)}$ in this setting gives
\begin{align*}
&q_{\om,\ep}^{(k)}
:=\frac
{\nu_{\sg\om,0}\left((\cL_{\om,0}-\cL_{\om,\ep})(\cL_{\sg^{-k}\om,\ep}^k)(\cL_{\sg^{-(k+1)}\om,0}-\cL_{\sg^{-(k+1)}\om,\ep})(\phi_{\sg^{-(k+1)}\om,0})\right)}
{\nu_{\sg\om,0}\left((\cL_{\om,0}-\cL_{\om,\ep})(\phi_{\om,0})\right)}
\\
&=\frac{\lm_{\sg^{-(k+1)}\om,0}^{k+1}\cdot\mu_{\sg^{-(k+1)}\om,0}\left(
	T_{\sg^{-(k+1)}\om}^{-(k+1)}(H_{\om,\ep})
	\cap\left(\bigcap_{j=1}^k T_{\sg^{-(k+1)}\om}^{-(k+1)+j} (H_{\sg^{-j}\om,\ep}^c)\right)
	\cap H_{\sg^{-(k+1)}\om,\ep}
	\right)
}
{\mu_{\om,0}(H_{\om,\ep})}
\\
&=\frac{\lm_{\sg^{-(k+1)}\om,0}^{k+1}\cdot\mu_{\sg^{-(k+1)}\om,0}\left(
	T_{\sg^{-(k+1)}\om}^{-(k+1)}(H_{\om,\ep})
	\cap\left(\bigcap_{j=1}^k T_{\sg^{-(k+1)}\om}^{-(k+1)+j} (H_{\sg^{-j}\om,\ep}^c)\right)
	\cap H_{\sg^{-(k+1)}\om,\ep}
	\right)
}
{\mu_{\sg^{-(k+1)}\om,0}\left(
	T_{\sg^{-(k+1)}\om}^{-(k+1)}(H_{\om,\ep})\right)}.		
\end{align*}
For notational convenience we define the quantity $\hat q_{\om,\ep}^{(k)}$ by 
\begin{align}\label{def of hat q}
\hat q_{\om,\ep}^{(k)}:=
\frac{\mu_{\sg^{-(k+1)}\om,0}\left(
	T_{\sg^{-(k+1)}\om}^{-(k+1)}(H_{\om,\ep})
	\cap\left(\bigcap_{j=1}^k T_{\sg^{-(k+1)}\om}^{-(k+1)+j} (H_{\sg^{-j}\om,\ep}^c)\right)
	\cap H_{\sg^{-(k+1)}\om,\ep}
	\right)
}
{\mu_{\sg^{-(k+1)}\om,0}\left(
	T_{\sg^{-(k+1)}\om}^{-(k+1)}(H_{\om,\ep})\right)},
\end{align}
and thus we have that 
\begin{align*}
\hat q_{\om,\ep}^{(k)}=\left(\lm_{\sg^{-(k+1)}\om,0}^{k+1}\right)^{-1} q_{\om,\ep}^{(k)}.
\end{align*}\index{$\hat q_{\om,\ep}^{(k)}$}
In light of \eqref{def of hat q}, one can think of $\hat q_{\om,\ep}^{(k)}$ as the conditional probability (on the fiber $\sg^{-(k+1)}\om$) of a point starting in the hole $H_{\sg^{-(k+1)}\om,\ep}$, leaving and avoiding holes for $k$ steps, and finally landing in the hole $H_{\om,\ep}$ after exactly $k+1$ steps conditioned on the trajectory of the point landing in $H_{\om,\ep}$.
Similarly, for $\om\in\Om_+$, we set 
\begin{align*}
\hat q_{\om,0}^{(k)}:=\left(\lm_{\sg^{-(k+1)}\om,0}^{k+1}\right)^{-1} q_{\om,0}^{(k)}.
\end{align*}\index{$\hat q_{\om,0}^{(k)}$}

\section[Quenched perturbation theorem and escape rate asymptotics]{Quenched perturbation theorem and escape rate asymptotics for random open systems}
\label{sec:goodrandom}
In this section we introduce versions of the assumptions \eqref{P1}--\eqref{P9} tailored to random open systems. Under these assumptions we then prove a derivative result akin to Theorem~\ref{thm: GRPT} for random open systems as well as a similar derivative result for the escape rate.

Suppose $(\mathlist{\bcomma}{\Om, m, \sg, \cJ_0, T, \cB, \cL_0, \nu_0, \phi_0, H_\ep})$ is a random open system. 
We assume the following conditions hold:
\begin{enumerate}[align=left,leftmargin=*,labelsep=\parindent]
\item[(\Gls*{C1})]\myglabel{C1}{C1}There exists a measurable $m$-a.e. finite function $C_1:\Om\to\RR_+$ such that for $f\in\cB_\om$ we have
\begin{align*}
\sup_{\ep\geq 0}\norm{\cL_{\om,\ep}(f)}_{\cB_{\sg\om}}\leq C_1(\om)\norm{f}_{\cB_\om}.
\end{align*}
\item[(\Gls*{C2})]\myglabel{C2}{C2} For each $\ep\geq 0$ there is a random measure  $\set{\nu_{\om,\ep}}_{\om\in\Om}$ supported in $\cJ_{0}$
and measurable functions $\lm_{\ep}:\Om\to(0,\infty)$ 
with $\log\lm_{\om,\ep}\in L^1(m)$ and $\phi_{\ep}:\cJ_0\to \RR$ such that 
\begin{align*}
\cL_{\om,\ep}(\phi_{\om,\ep})=\lm_{\om,\ep}\phi_{\sg\om,\ep}
\quad\text{ and }\quad
\nu_{\sg\om,\ep}(\cL_{\om,\ep}(f))=\lm_{\om,\ep}\nu_{\om,\ep}(f)
\end{align*}
for all $f\in\cB_\om$. Furthermore we assume that for $m$-a.e. $\om\in\Om$
$$
\nu_{\om,0}(\phi_{\om,\ep})=1
\quad\text{ and } \quad
\nu_{\om,0}(\ind)=1.
$$
\item[(\Gls*{C3})]\myglabel{C3}{C3} There is an operator $Q_{\om,\ep}:\cB_\om\to\cB_{\sg\om}$ such that for $m$-a.e. $\om\in\Om$ and each $f\in\cB_\om$ we have
\begin{align*}
\lm_{\om,\ep}^{-1}\cL_{\om,\ep}(f)=\nu_{\om,\ep}(f)\cdot\phi_{\sg\om,\ep}+Q_{\om,\ep}(f).
\end{align*}
Furthermore, for $m$-a.e. $\om\in\Om$ we have
\begin{align*}
Q_{\om,\ep}(\phi_{\om,\ep})=0
\quad\text{ and }\quad
\nu_{\sg\om,\ep}(Q_{\om,\ep}(f))=0.
\end{align*}
\item[(\Gls*{C4})]\myglabel{C4}{C4} 
For each $f\in\cB$ there exist measurable functions $C_f:\Om\to(0,\infty)$ and $\al:\Om\times \NN\to(0,\infty)$ with $\al_\om(N)\to 0$ as $N\to\infty$ such that for $m$-a.e. $\om\in\Om$ and all $N\in\NN$
\begin{align*}
\sup_{\ep\geq 0}\norm{Q_{\om,\ep}^N f_{\om}}_{\infty,\sg^N\om}
&\leq 
C_f(\om)\al_\om(N)\norm{f_{\om}}_{\cB_{\om}}, 
\\
\sup_{\ep\geq 0}\norm{Q_{\sg^{-N}\om,\ep}^Nf_{\sg^{-N}\om}}_{\infty,\om}
&\leq 
C_f(\om)\al_\om(N)\norm{f_{\sg^{-N}\om}}_{\cB_{\sg^{-N}\om}}, 
\end{align*} 
and $\norm{\phi_{\sg^{N}\om,0}}_{\infty,\sg^N\om}\al_\om(N)\to 0$, $\norm{\phi_{\sg^{-N}\om,0}}_{\infty,\sg^{-N}\om}\al_\om(N)\to 0$ as $N\to\infty$.
\item[(\Gls*{C5})]\myglabel{C5}{C5} 
There exists a measurable $m$-a.e. finite function $C_2:\Om\to[1,\infty)$ such that 
\begin{align*}
\sup_{\ep\geq 0}\norm{\phi_{\om,\ep}}_{\infty,\om}\leq C_2(\om) 
\quad\text{ and }\quad
\norm{\phi_{\om,0}}_{\cB_\om}\leq C_2(\om).
\end{align*}

\item[(\Gls*{C6})]\myglabel{C6}{C6} 
For $m$-a.e. $\om\in\Om$ we have
\begin{align*}
\lim_{\ep\to 0}\eta_{\om,\ep}=0. 
\end{align*}

\item[(\Gls*{C7})]\myglabel{C7}{C7} 
There exists a measurable $m$-a.e. finite function $C_3:\Om\to[1,\infty)$ such that  for all $\ep>0$ sufficiently small we have
\begin{align*}
\inf\phi_{\om,0}\geq C_3^{-1}(\om)>0
\qquad\text{ and }\qquad
\essinf_\om\inf\phi_{\om,\ep}\geq 0.
\end{align*}

\item[(\Gls*{C8})]\myglabel{C8}{C8} 
For $m$-a.e. $\om\in\Om_+$ 
we have that the limit $\hat{q}_{\om,0}^{(k)}:=\lim_{\ep\to 0} \hat{q}_{\om,\ep}^{(k)}$ exists for each $k\geq 0$, where $\hat{q}_{\om,\ep}^{(k)}$ is as in \eqref{def of hat q}.
\end{enumerate}

\begin{remark}\label{revision 1}

Note that if there exists a measurable $m$-a.e. finite function $K:\Om\to[1,\infty)$ such that 
\begin{align}\label{B}
\norm{f}_{\infty,\om}\leq K_\om\norm{f}_{\cB_\om} 
\end{align}
for all $f\in\cB_\om$ and each $\om\in\Om$, where $\|\cdot\|_{\infty,\om}$\index{$\|\cdot\|_{\infty,\om}$} denotes the supremum norm with respect to $\nu_{\om,0}$, then it follows from \eqref{etaineq0} that
\begin{align}
\eta_{\om,\ep}
&\leq K_\om\lm_{\om,0}\cdot\nu_{\om,0}\left(H_{\om,\ep}\right).
\label{etaineq}
\end{align}
In particular, assuming \eqref{B} we have that \eqref{C6} holds if 
\begin{align*}
\lim_{\ep\to 0}\nu_{\om,0}(H_{\om,\ep})=0.
\end{align*}
Furthermore, we note that \eqref{B} holds for many Banach spaces including the space of bounded variation functions, the space of H\"older continuous functions, and the space of generalized bounded variation functions, see \cite{kellerBV}. 
\end{remark}

\begin{remark}\label{rem:scaling}
To obtain the scaling required in \eqref{C2} and \eqref{C3}, in particular to obtain the assumption that $\nu_{\om,0}(\phi_{\om,\ep})=1$, suppose that $(\mathlist{\bcomma}{\Om, m, \sg, \cJ_0, T, \cB, \cL_0, \nu_0, \phi_0, H_\ep})$ is a random open system satisfying the following properties:
\begin{enumerate}[align=left,leftmargin=*,labelsep=\parindent]
\item[\mylabel{O1}{O1}] For each $\ep\geq 0$ there is a random measure $\set{\nu_{\om,\ep}'}_{\om\in\Om}$ with $\nu'_{\om,\ep}\in\cB_\om^*$, the dual space of $\cB_\om$, $\lm'_{\om,\ep}\in\CC\bs\set{0}$, and $\phi'_{\om,\ep}\in\cB_\om$ such that
\begin{align*}
\cL_{\om,\ep}(\phi'_{\om,\ep})=\lm'_{\om,\ep}\phi'_{\sg\om,\ep}
\quad\text{ and }\quad
\nu_{\sg\om,\ep}'(\cL_{\om,\ep}(f))=\lm'_{\om,\ep}\nu'_{\om,\ep}(f)
\end{align*}
for all $f\in\cB_\om$. 
\item[\mylabel{O2}{O2}] There is an operator $Q_{\om,\ep}':\cB_\om\to\cB_{\sg\om}$ such that for $m$-a.e. $\om\in\Om$ and each $f\in\cB_\om$ we have
\begin{align*}
(\lm'_{\om,\ep})^{-1}\cL_{\om,\ep}(f)=\nu'_{\om,\ep}(f)\cdot \phi'_{\sg\om,\ep}+Q'_{\om,\ep}(f).
\end{align*}
Furthermore, for $m$-a.e. $\om\in\Om$ we have
\begin{align*}
Q'_{\om,\ep}(\phi'_{\om,\ep})=0
\quad\text{ and }\quad
\nu_{\sg\om,\ep}'(Q'_{\om,\ep}(f))=0.
\end{align*}
\end{enumerate}
Then, for every $\ep>0$ and each $f\in\cB_\om$, by setting 
\begin{align*}
\phi_{\om,0}&:=\phi'_{\om,0},
&
\phi_{\om,\ep}&:=\frac{1}{\nu_{\om,0}(\phi'_{\om,\ep})}\cdot \phi'_{\om,\ep},
\\
\nu_{\om,0}(f)&:=\nu'_{\om,0}(f),
&
\nu_{\om,\ep}(f)&:=\nu_{\om,0}(\phi'_{\om,\ep})\nu'_{\om,\ep}(f),
\\
\lm_{\om,0}&:=\lm'_{\om,0},
&
\lm_{\om,\ep}&:=\frac{\nu_{\sg\om,0}(\phi'_{\sg\om,\ep})}{\nu_{\om,0}(\phi'_{\om,\ep})}\lm'_{\om,\ep},
\\
Q_{\om,0}(f)&:=Q'_{\om,0}(f),
&
Q_{\om,\ep}(f)&:=\frac{\nu_{\om,0}(\phi'_{\om,\ep})}{\nu_{\sg\om,0}(\phi'_{\sg\om,\ep})}Q'_{\om,\ep}(f)
\end{align*}
we see that \eqref{C2} and \eqref{C3} hold, and in particular, $\nu_{\om,0}(\phi_{\om,\ep})=1$. Furthermore, \eqref{O1} and \eqref{O2} together imply that $\nu'_{\om,\ep}(\phi'_{\om,\ep})=\nu_{\om,\ep}(\phi_{\om,\ep})=1$ for $m$-a.e. $\om\in\Om$. Note that the measures $\nu_{\om,\ep}$ (for $\ep>0$) described in \eqref{C2} are not necessarily probability measures and should not be confused with the conformal measures for the open systems.  
\end{remark}

Now given $N\in\NN$ and $\om\in\Om$, for $\psi_{\om}\in\cB_{\om}$ such that $\nu_{\om,0}(\psi_{\om})=1$ we have
\begin{align*}
&\int_{X_{\om,N-1,\ep}}
\psi_{\om}\, d\nu_{\om,0}
=
\int_{\cJ_{\om,0}} \psi_{\om}\cdot
\prod_{j=0}^{N-1}\ind_{\cJ_{\sg^{j-N}\om,\ep}}\circ T_{\om}^j \, d\nu_{\om,0}\\
&\qquad=
\left(\lm_{\om,0}^N\right)^{-1}
\int_{\cJ_{\sg^N\om,0}}\cL_{\om,0}^N\left(\psi_{\om}\cdot\prod_{j=0}^{N-1}\ind_{\cJ_{\sg^{j-N}\om,\ep}}\circ T_{\om}^j \right)\, d\nu_{\sg^N\om,0}\\
&\qquad=
\frac{\lm_{\om,\ep}^N}{\lm_{\om,0}^N}
\int_{\cJ_{\sg^N\om,0}}\~\cL_{\om,\ep}^N\left( \psi_{\om} \right)\, d\nu_{\sg^N\om,0}\\
&\qquad=
\frac{\lm_{\om,\ep}^N}{\lm_{\om,0}^N}\int_{\cJ_{\sg^N\om,0}}
\nu_{\om,\ep}(\psi_{\om})\cdot\phi_{\sg^N\om,\ep}\, d\nu_{\sg^N\om,0}
+
\frac{\lm_{\om,\ep}^N}{\lm_{\om,0}^N}
\int_{\cJ_{\sg^N\om,0}}Q_{\om,\ep}^N\left( \psi_{\om} \right)\, d\nu_{\sg^N\om,0}.	
\end{align*}
Thus if $\psi_{\om}=\ind$ we have
\begin{equation}
\label{NulamN}
\nu_{\om,0}(X_{\om,N-1,\ep})= \frac{\lm_{\om,\ep}^N}{\lm_{\om,0}^N}\nu_{\om, \ep}(\ind)
+\frac{\lm_{\om,\ep}^N}{\lm_{\om,0}^N}\int_{\cJ_{\sg^N\om,0}}Q_{\om,\ep}^N\left( \ind \right)\, d\nu_{\sg^N\om,0},
\end{equation}
and if $\psi_{\om}=\phi_{\om,0}$, then we have
\begin{align}
\label{eq: mu0 of n survivor}
\mu_{\om,0}(X_{\om,N-1,\ep})= \frac{\lm_{\om,\ep}^N}{\lm_{\om,0}^N}
\nu_{\om,\ep}(\phi_{\om,0})
+\frac{\lm_{\om,\ep}^N}{\lm_{\om,0}^N}\int_{\cJ_{\sg^N\om,0}}Q_{\om,\ep}^N\left( \phi_{\om,0} \right)\, d\nu_{\sg^N\om,0}.
\end{align}
\begin{remark}\label{rem lm_0>lm_ep}
Using \eqref{diff eigenvalues identity unif} and \eqref{C7} we see that 
\begin{align*}
\lm_{\om,0}-\lm_{\om,\ep}
&=\nu_{\sg\om,0}\left((\cL_{\om,0}-\cL_{\om,\ep})(\phi_{\om,\ep})\right)
=\nu_{\sg\om,0}\left((\cL_{\om,0}(\ind_{H_{\om,\ep}}\phi_{\om,\ep})\right)
\\
&=\lm_{\om,0}\nu_{\om,0}(\ind_{H_{\om,\ep}}\phi_{\om,\ep})
\geq \lm_{\om,0}\nu_{\om,0}(H_{\om,\ep})\inf\phi_{\om,\ep}\geq 0.
\end{align*}
Thus, we have that $\lm_{\om,0}\geq \lm_{\om,\ep}$. Note that if 
$\nu_{\om,0}(\ind_{H_{\om,\ep}}\phi_{\om,\ep})>0$ then we have that $\lm_{\om,0}>\lm_{\om,\ep}$. Consequently, if
$\nu_{\om,0}(H_{\om,\ep})=0$ then we have that  $\lm_{\om,0}=\lm_{\om,\ep}$.
\end{remark}
Recall from Definition \ref{def: escape rate} that the upper and lower fiberwise escape rate of a random probability measure $\zt$ on $\cJ_{0}$, for each $\ep>0$, is given by the following:
\begin{align*}
\Ul{R}_\ep(\zt_\om):=-\limsup_{N\to\infty}\frac{1}{N}\log \zt_\om(X_{\om,N,\ep})
\quad \text{ and } \quad
\ol{R}_\ep(\zt_\om):=-\liminf_{N\to\infty}\frac{1}{N}\log \zt_\om(X_{\om,N,\ep}).
\end{align*}\index{$\Ul{R}_\ep(\zt_\om)$}\index{$\ol{R}_\ep(\zt_\om)$}
If $\Ul{R}_\ep(\zt_\om)=\ol{R}_\ep(\zt_\om)$ then the fiberwise escape rate exist is denoted the common value by $R_\ep(\zt_\om)$. \index{$R_\ep(\zt_\om)$}\index{fiberwise escape rate}
As an immediate consequence of \eqref{NulamN}, \eqref{eq: mu0 of n survivor}, and assumptions \eqref{C4}, \eqref{C5}, and \eqref{C7} 
for each (fixed) $\ep>0$ we have that 
\begin{align}\label{cons of NulamN}
\lim_{N\to\infty}\frac{1}{N}\log \nu_{\om,0}(X_{\om,N,\ep})=
\lim_{N\to\infty}\frac{1}{N}\log \mu_{\om,0}(X_{\om,N,\ep})=
\lim_{N\to\infty}\frac{1}{N}\log \lm_{\om,\ep}^N-\lim_{N\to\infty}\frac{1}{N}\log \lm_{\om,0}^N.
\end{align}
Since $\log\lm_{\om,\ep}\in L^1(m)$ for all $\ep\geq 0$ by \eqref{C2}, the following proposition follows directly from Birkhoff's Ergodic Theorem.
\begin{proposition}\label{prop: escape rates}
Given a random open system $(\mathlist{\bcomma}{\Om, m, \sg, \cJ_0, T, \cB, \cL_0, \nu_0, \phi_0, H_\ep})$ satisfying conditions \eqref{C1}-\eqref{C7}, for $m$-a.e. $\om\in\Om$ we have that 
\begin{align}\label{eq: lem escape rate}
R_\ep(\nu_{\om,0})=R_\ep(\mu_{\om,0})=\int_\Om \log\lm_{\om,0} \, dm(\om) - \int_\Om\log\lm_{\om,\ep}\, dm(\om).
\end{align}
\end{proposition}

\begin{remark}
We remark that if we were to replace the supremum norm (with respect to $\nu_{\om,0}$) $\|\spot\|_{\infty,\om}$ everywhere in our assumption \eqref{C4} with the norm $\|\spot\|_\infty:=\sup(|\spot|)$, then, using an argument similar to that of Example 7.4 of \cite{AFGTV-IVC}, we would be able to prove a stronger result than is given in Proposition \ref{prop: escape rates}. Namely we could show that for every $0\leq\ep'<\ep$ and $m$-a.e. $\om\in\Om$ we would have 
\begin{align*}
R_\ep(\nu_{\om,\ep'})=R_\ep(\mu_{\om,\ep'})=\int_\Om \log\lm_{\om,\ep'} \, dm(\om) - \int_\Om\log\lm_{\om,\ep}\,dm(\om).
\end{align*}
\end{remark}

We now begin to work toward an application of Theorem \ref{thm: GRPT} to random open systems.
The following implications are immediate: \eqref{C1}$\implies$\eqref{P1}, \eqref{C2}$\implies$\eqref{P2}, \eqref{C3}$\implies$\eqref{P3}, \eqref{C5}$\implies$\eqref{P4}, \eqref{C6}$\implies$\eqref{P5}, and in light of \eqref{eq: Dl = mu H} we also have that  and \eqref{C7}$\implies$\eqref{P6}.
Thus, in order to ensure that Theorem~\ref{thm: GRPT} applies for the random open dynamical setting we need only to check assumptions \eqref{P7} and \eqref{P8}. We now prove the following lemma showing that \eqref{P7} and \eqref{P8} follow from assumptions \eqref{C1}-\eqref{C7}.

Recall that the set $\Om_+$, defined in \eqref{def Om+}, is the set of all fibers $\om$ such that $\mu_{\om,0}(H_{\om,\ep})>0$ for all $\ep>0$.
\begin{lemma}\label{lem: checking P7 and P8}
Given a random open system $(\mathlist{\bcomma}{\Om, m, \sg, \cJ_0, T, \cB, \cL_0, \nu_0, \phi_0, H_\ep})$ satisfying conditions \eqref{C1}-\eqref{C7}, for $m$-a.e. $\om\in\Om_+$  we have that 
\begin{align}\label{eq: lem check P7}
\lim_{\ep \to 0}\nu_{\om,\ep}(\phi_{\om,0})=1
\end{align}
and 
\begin{align}\label{eq: lem check P8}
\lim_{N\to\infty}\limsup_{\ep \to 0}\Dl_{\om,\ep}^{-1} \nu_{\sg\om,0}\lt((\cL_{\om,0}-\cL_{\om,\ep})(Q_{\sg^{-N}\om,\ep}^N\phi_{\sg^{-N}\om,0})\rt)=0.
\end{align}
\end{lemma}
\begin{proof}
First, using \eqref{eq: Dl = mu H}, we note that if $\mu_{\om,0}(H_{\om,\ep})>0$ then so is $\Dl_{\om,\ep}$.
To prove \eqref{eq: lem check P7}, we note that for fixed $N\in\NN$ we have 
\begin{align}\label{eq: lim ratio times measure is 1}
\lim_{\ep \to 0}\frac{\lm_{\om,0}^N}{\lm_{\om,\ep}^N}\mu_{\om,0}(X_{\om,N-1,\ep})=1
\end{align}
since $\lm_{\om,0}^N/\lm_{\om,\ep}^N\to 1$ (by \eqref{conv eigenvalues unif}, \eqref{etaineq}, \eqref{C5}, and \eqref{C6}) and non-atomicity of $\nu_{\om,0}$ \eqref{CCM} together with \eqref{C3} imply that $\mu_{\om,0}(X_{\om,N-1,\ep})\to 1$ as $\ep\to 0$. 
Using  \eqref{eq: mu0 of n survivor} we can write 
\begin{align}\label{eq: solve for nu eps}
\nu_{\om,\ep}(\phi_{\om,0})
= \frac{\lm_{\om,0}^N}{\lm_{\om,\ep}^N}\mu_{\om,0}(X_{\om,N-1,\ep})
- 
\nu_{\sg^N\om,0}(Q_{\om,\ep}^N(\phi_{\om,0})),
\end{align}
and thus using \eqref{eq: lim ratio times measure is 1} and \eqref{C4}, 
for each $\om\in\Om$ and each $N\in\NN$ we can write 
\begin{align*}
\lim_{\ep\to 0}\absval{1-\nu_{\om,\ep}(\phi_{\om,0})}
&\leq
\lim_{\ep\to 0}
\absval{ 1-\frac{\lm_{\om,0}^N}{\lm_{\om,\ep}^N}\mu_{\om,0}(X_{\om,N-1,\ep})}
+
\norm{Q_{\om,\ep}^N(\phi_{\om,0})}_{\infty,\sg^N\om}
\\
&\leq
C_{\phi_0}( \om)\al_\om(N)\norm{\phi_{\om,0}}_{\cB_\om}.
\end{align*}
As this holds for each $N\in\NN$ and as the right-hand side of the previous equation goes to zero as $N\to\infty$, we must in fact have that  
\begin{align*}
\lim_{\ep\to 0}\absval{1-\nu_{\om,\ep}(\phi_{\om,0})}=0,
\end{align*}
which yields the first claim.

Now, for the second claim, using \eqref{eq: Dl = mu H}, we note that \eqref{C7} implies 
\begin{align*}
\Dl_{\om,\ep}^{-1} \nu_{\sg\om,0}\lt((\cL_{\om,0}-\cL_{\om,\ep})(Q_{\sg^{-N}\om,\ep}^N\phi_{\sg^{-N}\om,0})\rt)
&=
\frac{\nu_{\om,0}\lt(\ind_{H_{\om,\ep}}\cdot Q_{\sg^{-N}\om,\ep}^N(\phi_{\sg^{-N}\om,0}) \rt) }{\mu_{\om,0}(H_{\om,\ep})}
\\
&\leq
\frac{\nu_{\om,0}(H_{\om,\ep})}{\mu_{\om,0}(H_{\om,\ep})}\norm{Q_{\sg^{-N}\om,\ep}^N(\phi_{\sg^{-N}\om,0})}_{\infty,\om}
\\
&\leq C_3(\om)\norm{Q_{\sg^{-N}\om,\ep}^N(\phi_{\sg^{-N}\om,0})}_{\infty,\om}.
\end{align*}
Thus, letting $\ep\to 0$ first and then $N\to\infty$, the second claim follows from \eqref{C4}.
\end{proof}

Now recall from \eqref{eq: def of theta_0} that if $\mu_{\om,0}(H_{\om,\ep})>0$ for each $\ep>0$, then $\ta_{\om,0}$ is given by 
\begin{align*}
\ta_{\om,0}:=
1-\sum_{k=0}^{\infty}(\lm_{\sg^{-(k+1)}\om,0}^{k+1})^{-1}q_{\om,0}^{(k)}
=
1-\sum_{k=0}^{\infty}\hat q_{\om,0}^{(k)}.
\end{align*}\index{$\ta_{\om,0}$}
In light of Lemma~\ref{lem: checking P7 and P8}, we see that \eqref{C1}-\eqref{C8} imply \eqref{P1}-\eqref{P9}, and thus we have the following Theorem and first main result of this section.

\begin{theorem}\label{thm: dynamics perturb thm}
Suppose that \eqref{C1}-\eqref{C7} hold for a random open system $(\mathlist{\bcomma}{\Om, m, \sg, \cJ_0, T, \cB, \cL_0, \nu_0, \phi_0, H_\ep})$. 
For $m$-a.e. $\om\in\Om$ if there is some $\ep_0>0$ such that $\mu_{\om,0}(H_{\om,\ep})=0$ for $\ep\leq \ep_0$ then
\begin{align*}
\lm_{\om,0}=\lm_{\om,\ep},
\end{align*}
or if \eqref{C8} holds, then 
\begin{align*}
\lim_{\ep\to 0}
\frac{\lm_{\om,0}-\lm_{\om,\ep}}{\lm_{\om,0}\mu_{\om,0}(H_{\om,\ep})}
=\ta_{\om,0}.
\end{align*}
Furthermore, the map $\Om_+\ni\om\mapsto\ta_{\om,0}$ is measurable. 
\end{theorem}
\begin{proof}
All statements follow directly from Theorem \ref{thm: GRPT}, except measurability of $\theta_{\omega,0}$, which follows from its construction as a limit of measurable objects.
\end{proof}
\begin{remark}
As Theorem \ref{thm: GRPT} does not require any measurability, one could restate
Theorem~\ref{thm: dynamics perturb thm} to hold for sequential open systems satisfying the sequential versions of hypotheses \eqref{C1}-\eqref{C8}.
\end{remark}
The following lemma will be useful for bounding $\ta_{\om,0}$.

\begin{lemma}
\label{qbound2}
Suppose $m(\Om\bs\Om_+)=0$. 
Then for each $\ep>0$ and $m$-a.e. $\om\in\Om$ we have
\begin{align*}
\sum_{k=0}^\infty (\lm_{\sg^{-(k+1)}\om,0}^{k+1})^{-1}q_{\om,\ep}^{(k)}=\sum_{k=0}^\infty \hat q^{(k)}_{\om,\ep}= 1.	
\end{align*}
\end{lemma}
\begin{proof}
Recall from Section~\ref{sec: open systems} that the hole $H_\ep\sub\cJ_0$ is given by  
$$
H_\ep:=\bigcup_{\om\in \Om}\set{\om}\times H_{\om,\ep},
$$
and the skew map $T:\cJ_0\to\cJ_0$ is given by
$$
T(\om, x)=(\sg\om, T_{\om}(x)).
$$
Let $\tau_{H_\ep}(\om,x)$ denote the first return time of the point $(\om, x)\in H_\ep$ into $H_\ep.$  For each $k\geq 0$ let $B_k:=\set{(\om,x)\in H_\ep: \tau_{H_\ep}(\om,x)=k+1}$ be the set of points $(\om,x)$ that remain outside of the hole $H_\ep$ for exactly $k$ iterates, and for each $\om\in\Om$ we set $B_{k,\om}:=\pi_2(B_k\cap\set{\om}\times\cJ_{\om,0})$. Then $B_{k,\om}$ is precisely the set 
$$
B_{k,\om}=\set{x\in H_{\om,\ep}: T_\om^{k+1}(x)\in H_{\sg^{k+1}\om,\ep} \text{ and } T_\om^j(x)\not\in H_{\sg^j\om,\ep} \text{ for all } 1\leq j\leq k}.
$$
For each $k\geq 0$ we can disintegrate $\mu_0$ as
\begin{align*}
\mu_0\left(B_k\right)
&=\int_\Om\mu_{\om,0}(B_{k,\om})\ dm(\om),
\end{align*}
Note that for $\om\in\Om_+$ and each $k\geq 0$ we have
\begin{align}\label{qsum1}
&\mu_{\sg^{-(k+1)}\om,0}(B_{k,\sg^{-(k+1)}\om})
=
\mu_{\sg^{-(k+1)}\om,0}\left(
T_{\sg^{-(k+1)}\om}^{-(k+1)}(H_{\om,\ep})
\cap\left(\bigcap_{j=1}^k T_{\sg^{-(k+1)}\om}^{-(k+1)+j} (H_{\sg^{-j}\om,\ep}^c)\right)
\cap H_{\sg^{-(k+1)}\om,\ep}
\right).
\end{align}
Recall from \eqref{def of hat q} that $\hat q_{\om,\ep}^{(k)}$ is given by
\begin{align*}
\hat q_{\om,\ep}^{(k)}:=
\frac{\mu_{\sg^{-(k+1)}\om,0}\left(
T_{\sg^{-(k+1)}\om}^{-(k+1)}(H_{\om,\ep})
\cap\left(\bigcap_{j=1}^k T_{\sg^{-(k+1)}\om}^{-(k+1)+j} (H_{\sg^{-j}\om,\ep}^c)\right)
\cap H_{\sg^{-(k+1)}\om,\ep}
\right)
}
{\mu_{\sg^{-(k+1)}\om,0}\left(
T_{\sg^{-(k+1)}\om}^{-(k+1)}(H_{\om,\ep})\right)},
\end{align*}
which is well defined since $\om\in\Om_+$ by assumption. Note that since $\sum_{k=0}^\infty\mu_{\sg^{-(k+1)}\om,0}(B_{k,\sg^{-(k+1)}\om})\leq \mu_{\om,0}(H_{\om,\ep})$ we must have that $\sum_{k=0}^{\infty}\hat{q}^{(k)}_{\om,\ep}\leq1$. 
Now, as the measure $\mu_0$ is $T$-invariant, the right-hand side of \eqref{qsum1} is equal to 
$\hat{q}^{(k)}_{\om,\ep}\mu_{\om,0}(H_{\om,\ep})$, and therefore 
\begin{align}\label{B_k measure}
\int_\Om \hat{q}^{(k)}_{\om,\ep}\ \mu_{\om,0}(H_{\om,\ep})\ dm(\om)
=\int_\Om \mu_{\om,0}(B_{k,\om})\ dm(\om)
=\mu_0(B_k).
\end{align}
By the Poincar\'e recurrence theorem we have
$$
\sum_{k=0}^{\infty}\mu_0(B_k)=\mu_0(H_\ep)=\int_\Om\mu_{\om,0}(H_{\om,\ep})\, dm(\om).
$$
By interchanging the sum with the integral above (possible by Tonelli's Theorem) and using \eqref{B_k measure}, we have
$$
\int_\Om \left(\sum_{k=0}^{\infty}\hat{q}^{(k)}_{\om,\ep}\ \mu_{\om,0}(H_{\om,\ep})\right)\ dm(\om)
=
\int_\Om \mu_{\om,0}(H_{\om,\ep})\ dm(\om),
$$
which implies
$$
\int_\Om \lt(\mu_{\om,0}(H_{\om,\ep})\left(\sum_{k=0}^{\infty}\hat{q}^{(k)}_{\om,\ep}-1\right)\rt) dm(\om)=0.
$$
Since we already have that $\sum_{k=0}^{\infty}\hat{q}^{(k)}_{\om,\ep}\leq1$ for $m$-a.e. $\om\in\Om$, we must in fact have $\sum_{k=0}^{\infty}\hat{q}^{(k)}_{\om,\ep}=1$, which completes the proof.
\end{proof}
By definition, we have that $\hat q_{\om,\ep}^{(k)}\in[0,1]$ and thus \eqref{C8} implies that $\hat q_{\om,0}^{(k)}\in[0,1]$ for each $k\geq 0$ as well. In light of the previous lemma, dominated convergence implies that  
\begin{align}\label{theta in [0,1]}
\ta_{\om,0}\in[0,1].
\end{align}

\subsection{Escape Rate Asymptotics}

If we have some additional $\om$-control on the size of the holes we obtain the following corollary of Theorem~\ref{thm: dynamics perturb thm}, which provides a formula for the escape rate asymptotics for small random holes.
The scaling of the holes (\ref{EVT style cond}) takes a similar form to 
the scaling we will shortly use in the next section for our quenched extreme value theory.

\begin{corollary}\label{esc rat cor}
Suppose that \eqref{C1}--\eqref{C8} hold for a random open system $(\mathlist{\bcomma}{\Om, m, \sg, \cJ_0, T, \cB, \cL_0, \nu_0, \phi_0, H_\ep})$ and there exists $A\in L^1(m)$ such that for $m$-a.e. $\om\in\Om$ and all $\ep>0$ sufficiently small we have 
\begin{align}\label{DCT cond 1*}
\frac{\log\lm_{\om,0}-\log\lm_{\om,\ep}}{\mu_{\om,0}(H_{\om,\ep})}\leq A(\om).
\end{align}
Further suppose that there is some $\kappa(\ep)$ with $\kp(\ep)\to\infty$ as $\ep\to 0$, $t\in L^\infty(m)$ with $t_\om>0$, and $\xi_{\om,\ep}\in L^\infty(m)$ with $\xi_{\om,\ep}\to 0$ as $\ep\to 0$ for $m$-a.e. $\om\in\Om$ such that
\begin{align}\label{EVT style cond}
\mu_{\om,0}(H_{\om,\ep})=\frac{t_\om+\xi_{\om,\ep}}{\kp(\ep)}
\end{align}
and $|\xi_{\om,\ep}|<Ct_\om$ for all $\ep>0$ sufficiently small for some $C>0$.
Then for $m$-a.e. $\om\in\Om$ we have
\begin{align*}
\lim_{\ep\to 0}\frac{R_\ep(\mu_{\om,0})}{\mu_{\om,0}(H_{\om,\ep})}
=
\frac{\int_\Om \ta_{\om,0}t_\om\, dm(\om)}{t_\om}.
\end{align*}
In particular, if $t_\om$ is constant $m$-a.e. then 
\begin{align*}
\lim_{\ep\to 0}\frac{R_\ep(\mu_{\om,0})}{\mu_{\om,0}(H_{\om,\ep})}
=
\int_\Om \ta_{\om,0}\, dm(\om).
\end{align*}
\end{corollary}
\begin{proof}
First, we note that \eqref{EVT style cond} implies that $\mu_{\om,0}(H_{\om,\ep})>0$ for $m$-a.e. $\om\in\Om$ and all $\ep>0$ sufficiently small, i.e. $m(\Om\bs\Om_+)=0$. 
Using Theorem \ref{thm: dynamics perturb thm} we have that
\begin{align}\label{fiber esc rat deriv}
\lim_{\ep\to 0}\frac{\log\lm_{\om,0}-\log\lm_{\om,\ep}}{\mu_{\om,0}(H_{\om,\ep})}=\ta_{\om,0}
\end{align}
since
\begin{align*}
\frac{\log\lm_{\om,0}-\log\lm_{\om,\ep}}{\mu_{\om,0}(H_{\om,\ep})}
&=
\frac{\log\lm_{\om,0}-\log\lm_{\om,\ep}}{\lm_{\om,0}-\lm_{\om,\ep}}
\cdot
\frac{\lm_{\om,0}-\lm_{\om,\ep}}{\mu_{\om,0}(H_{\om,\ep})}
\lra
\frac{1}{\lm_{\om,0}}\cdot\lm_{\om,0}\ta_{\om,0}=\ta_{\om,0}
\end{align*}
as $\ep\to 0$. 
In light of \eqref{cons of NulamN} and (\ref{eq: lem escape rate}), we use \eqref{EVT style cond} to get that 
\begin{align*}
\lim_{\ep\to 0}\frac{R_\ep(\mu_{\om,0})}{\mu_{\om,0}(H_{\om,\ep})}
&=
\lim_{\ep\to 0}\lim_{N\to\infty}\frac{1}{N}\frac{\log\lm_{\om,0}^N-\log\lm_{\om,\ep}^N}{\mu_{\om,0}(H_{\om,\ep})}
\\
&=
\lim_{\ep\to 0}\lim_{N\to\infty}\frac{1}{N}\sum_{j=0}^{N-1}\frac{\log\lm_{\sg^j\om,0}-\log\lm_{\sg^j\om,\ep}}{\mu_{\sg^j\om,0}(H_{\sg^j\om,\ep})}\cdot\frac{\mu_{\sg^j\om,0}(H_{\sg^j\om,\ep})}{\mu_{\om,0}(H_{\om,\ep})}
\\
&=
\lim_{\ep\to 0}\lim_{N\to\infty}\frac{1}{N}\sum_{j=0}^{N-1}\frac{\log\lm_{\sg^j\om,0}-\log\lm_{\sg^j\om,\ep}}{\mu_{\sg^j\om,0}(H_{\sg^j\om,\ep})}\cdot\frac{t_{\sg^j\om}+\xi_{\sg^j\om,\ep}}{t_{\om}+\xi_{\om,\ep}}
\\
&=
\lim_{\ep\to 0}\frac{1}{t_\om+\xi_{\om,\ep}}\lim_{N\to\infty}\frac{1}{N}\sum_{j=0}^{N-1}\frac{\log\lm_{\sg^j\om,0}-\log\lm_{\sg^j\om,\ep}}{\mu_{\sg^j\om,0}(H_{\sg^j\om,\ep})}\cdot (t_{\sg^j\om}+\xi_{\sg^j\om,\ep})
\\
&=
\lim_{\ep\to 0}\frac{1}{t_\om+\xi_{\om,\ep}}\int_\Om\frac{\log\lm_{\om,0}-\log\lm_{\om,\ep}}{\mu_{\om,0}(H_{\om,\ep})}\cdot (t_{\om}+\xi_{\om,\ep})\, dm(\om),
\end{align*}
where the last line follows from Birkhoff (which is applicable thanks to \eqref{DCT cond 1*} and the fact that $t_\om,\xi_{\om,\ep}\in L^\infty(m)$).
As $\xi_{\om,\ep}\to 0$ by assumption, \eqref{DCT cond 1*} allows us to apply Dominated Convergence which, in view of \eqref{fiber esc rat deriv}, implies
\begin{align*}
\lim_{\ep\to 0}\frac{R_\ep(\mu_{\om,0})}{\mu_{\om,0}(H_{\om,\ep})}
=
\frac{1}{t_\om}\lim_{\ep\to 0}\int_\Om \frac{\log\lm_{\om,0}-\log\lm_{\om,\ep}}{\mu_{\om,0}(H_{\om,\ep})}\cdot (t_{\om}+\xi_{\om,\ep})\, dm(\om)
=
\frac{\int_\Om\ta_{\om,0}t_\om\, dm(\om)}{t_\om},
\end{align*}
completing the proof.
\end{proof}
In the next section we will present easily checkable assumptions which will imply the hypotheses of Corollary~\ref{esc rat cor}.

\section{Quenched extreme value law}\label{EEVV}

\subsection{Gumbel's law in the quenched regime}
Suppose that $h_\om:\cJ_{\om,0}\to\RR$ is a continuous function for each $\om\in\Om$.
In our quenched random extreme value theory, $h_\omega$ is an observation function that is allowed to depend on $\omega$.
For each $\om\in\Om$ let $\ol{z}_\om$ be the essential supremum of $h_\om$ with respect to $\nu_{\om,0}$, that is
\begin{align*}
\ol{z}_\om:=\sup\set{z\in\RR\; :\nu_{\om,0}(\set{x\in\cJ_{\om,0}:h_\om(x)\geq z})>0}.
\end{align*}\index{$\ol z_\om$}
Similarly we define $\Ul z_\om$\index{$\Ul z_\om$} to be the essential infimum of $h_\om$ with respect to $\nu_{\om,0}$. 
Suppose $\Ul{z}_\om<\ol{z}_\om$ for $m$-a.e. $\om\in\Om$, 
and for each $z\in[\Ul{z}_\om,\ol{z}_\om]$ we define the set
\begin{align*}
V_{\om,z_\omega}:=\set{x\in\cJ_{\om,0}:h_\om(x)-z_\omega>0},
\end{align*}\index{$V_{\om,z_\omega}$}
which represents points $x$ in our phase space where the observation $h_\omega$ exceeds a random threshold $z_\omega$ at base configuration $\omega$.
Our theory allows for random thresholds $z_\omega$ so that we may consider both ``anomalous'' and ``absolute'' exceedances.
For example in a real-world application, $h_\omega(x)$ may represent the surface ocean temperature at a spatial location $x$ for an ocean system configuration $\omega$.
Random temperature exceedances above $z_\omega$ allow one to describe extreme value statistics for temperature anomalies, e.g.\ those above a climatological seasonal mean (on average the surface ocean is warmer in summer and colder in winter). 
On the other hand, non-random absolute temperature exceedances above $z$ are more relevant for marine life.
We suppose that
\begin{align*}
\nu_{\om,0}(V_{\om,\ol{z}_\om})=0.
\end{align*}

As is standard in extreme value theory, to develop an exponential law we will consider an increasing sequence of thresholds $z_{\omega,0}<z_{\omega,1}<\cdots$.
For each $N\geq 0$ we take  $z_{\om,N}\in[\Ul{z}_\om,\ol{z}_\om]$\index{$z_{\om,N}$} and the choice will be made explicit in a moment. 
For each $k,N\in\NN$ define $G_{\om,N}^{(k)}:\cJ_{\om,0}\to\RR$ by
\begin{align*}
G_{\om,N}^{(k)}(x):=h_{\sg^N\om}(T_\om^N(x))-z_{\sg^N\om,k}.
\end{align*}
Thus $G_{\om,N}^{(k)}>0$\index{$G_{\om,N}^{(k)}$} if and only if $h_{\sg^N\om}(T_\om^N(x))>z_{\sg^N\om,k}$.
Our extreme value law concerns the large $N$ limit of the likelihood of continued threshold non-exceedances:
\begin{equation}
\label{evtexpression}
\nu_{\om,0}\left(\set{x\in\cJ_{\om,0}:\max\left(G_{\om,0}^{(N)}(x), \dots,G_{\om,N-1}^{(N)}(x)\right)\leq 0}\right).
\end{equation}
We may easily transform (\ref{evtexpression}) into the language of random open systems:
one immediately has that $G_{\om,N}^{(k)}>0$ is equivalent to $T_\om^N(x)\in V_{\sg^N\om,z_{\sg^N\om,k}}$.
Therefore, for each $N\geq 0$ we have
\begin{equation}
\label{evtexpression2}
(\ref{evtexpression})=
\nu_{\om,0}\left(\set{x\in\cJ_{\om,0}:T_\om^j(x)\notin V_{\sg^j\om,z_{\sg^j\om,N}} \text{ for }j=0,\dots, N-1}\right)\\
=\nu_{\om,0}\left(X_{\om,N-1,\ep_N}\right),
\end{equation}
where to obtain the second equality we 
identify the sets $V_{\om,z_{\om,N}}$ with holes $H_{\om,\ep_N}\subset \cJ_{\omega,0}$\index{$H_{\om,\ep_N}$} for each $N\in\NN$
\footnote{
In Section~\ref{EEVV} we consider a decreasing sequence of holes, whereas in previous sections we considered a decreasing family of holes $H_{\om,\ep}$ parameterised by $\ep>0$. For the sake of notational continuity, in this section, and in the sequel, we denote a decreasing sequence of holes by $H_{\om,\ep_N}$. 
Note that $\ep_N$ here is just a parameter and should not be thought of as a real number, but instead as an index and that $H_{\om,\ep_N}$ serves as an alternative notation to $H_{\om,N}$. Furthermore, note that the measure of the hole depends on the fiber $\om$ with $\mu_{\om,0}(H_{\om,\ep_N})\to 0$ as $N\to\infty$. }.
Now using (\ref{NulamN}) and (\ref{evtexpression2}) we may convert  (\ref{evtexpression}) into the spectral expression:
\begin{equation}
\label{evtexpression3}
(\ref{evtexpression})=\frac{\lm_{\om,\ep_N}^N}{\lm_{\om,0}^N}
\lt(\nu_{\om,\ep_N}(\ind)+\nu_{\sg^N\om,0}\lt(Q_{\om,\ep_N}^N( \ind)\rt)\rt).
\end{equation}
Similarly, using \eqref{eq: mu0 of n survivor}, we can write 
\begin{align}
&\mu_{\om,0}\left(\set{x\in\cJ_{\om,0}:T_\om^j(x)\notin V_{\sg^j\om,z_{\sg^j\om,N}} \text{ for }j=0,\dots, N-1}\right)
=
\mu_{\om,0}\left(X_{\om,N-1,\ep_N}\right)
\nonumber\\
&\qquad=\frac{\lm_{\om,\ep_N}^N}{\lm_{\om,0}^N}
\lt(\nu_{\om,\ep_N}(\phi_{\om,0})+\nu_{\sg^N\om,0}\lt(Q_{\om,\ep_N}^N( \phi_{\om,0})\rt)\rt).
\label{evtexpression2mu}
\end{align}
Before we state our main result concerning the $N\to\infty$ limit, in addition to \eqref{C2}, \eqref{C3}, and \eqref{C8}, 
we make the following uniform adjustments to some of the assumptions in Section \ref{sec:goodrandom} as well as an assumption 
on the choice of the sequence of thresholds, condition \eqref{xibound}. 
At the end of this section, we will compare it with the H\"usler condition, which is the usual prescription for non stationary processes, as we anticipated in the Introduction.

\begin{enumerate}[align=left,leftmargin=*,labelsep=\parindent]
\item[(\Gls*{S})]\myglabel{S}{xibound}
For any fixed random scaling function $t\in L^\infty(m)$ with $t>0$, we may find sequences of functions $z_{N},\xi_N\in L^\infty(m)$  and a constant $W<\infty$ satisfying 
$$\mu_{\omega,0}(\{h_\omega(x)-z_{\omega,N}>0\})=(t_\omega+\xi_{\omega,N})/N, \mbox{ for a.e.\ $\omega$ and each $N\ge 1$}
$$\index{$t_\om$}\index{$\xi_{\om,N}$}
where:\\
(i)
$\lim_{N\to\infty} \xi_{\omega,N}=0$ for a.e.\ $\omega$ and \\
(ii) $|\xi_{\omega,N}|\le W$ for a.e.\ $\omega$ and all $N\ge 1$.
\end{enumerate}

\begin{enumerate}[align=left,leftmargin=*,labelsep=\parindent]
\item[(\Gls*{C1'})]\myglabel{C1'}{C1'}There exists $C_1\geq 1$ such that for $m$-e.a. $\om\in\Om$ we have 
\begin{align*}
C_1^{-1}\leq \cL_{\om,0}\ind\leq C_1.
\end{align*}

\item[(\Gls*{C4'})]\myglabel{C4'}{C4'} 
For each $f\in\cB$ and each $N\in\NN$ there exists $C_f>0$ and $\al(N)>0$ (independent of $\om$) with $\al:=\sum_{N=1}^\infty\al(N)<\infty$ such that for $m$-a.e. $\om\in\Om$, all $N\in\NN$
\begin{align*}
\sup_{\ep\geq 0}\norm{Q_{\om,\ep}^N f_{\om}}_{\infty,\sg^N\om}\leq C_f\al(N)\norm{f_{\om}}_{\cB_{\om}}.
\end{align*}
\item[(\Gls*{C5'})]\myglabel{C5'}{C5'} There exists $C_2\geq 1$ such that
\begin{align*}
\sup_{\ep\geq 0}\norm{\phi_{\om,\ep}}_{\infty,\om}\leq C_2
\quad\text{ and }\quad
\norm{\phi_{\om,0}}_{\cB_\om}\leq C_2
\end{align*}
for $m$-a.e. $\om\in\Om$.

\item[(\Gls*{C7'})]\myglabel{C7'}{C7'} 
There exists $C_3\geq 1$ such that for all $\ep>0$ sufficiently small we have
\begin{align*}
\essinf_\om\inf\phi_{\om,0}\geq C_3^{-1} >0
\qquad\text{ and }\qquad 
\essinf_\om\inf\phi_{\om,\ep}\geq 0.
\end{align*}

\end{enumerate}
\begin{remark}\label{zero hole}
Note that 
\eqref{xibound} implies that $\mu_{\om,0}(H_{\om,\ep_N})>0$ for each $N\in\NN$ since $t_\om>0$, and in particular we have $m(\Om\bs\Om_+)=0$. 
\end{remark}
\begin{remark}
\label{rem61}
Note that since $\lm_{\om,0}=\nu_{\sg\om,0}(\cL_{\om,0}\ind)$, \eqref{C1'} implies that 
\begin{align*}
C_1^{-1}\leq \lm_{\om,0}\leq C_1
\end{align*}
for $m$-a.e. $\om$.
\end{remark}
\begin{remark}\label{rem checking esc cor cond}
Note that conditions \eqref{xibound}, \eqref{C5'}, and \eqref{C7'} together imply \eqref{C6}, thus Theorem~\ref{thm: dynamics perturb thm} applies. 
Furthermore, these same conditions along with Remark~\ref{rem61} imply that there exists $A\geq 1$ such that 
\begin{align*}
\frac{\log\lm_{\om,0}-\log\lm_{\om,\ep}}{\mu_{\om,0}(H_{\om,\ep})}\leq A
\end{align*}
for $m$-a.e. $\om\in\Om$ and all $\ep>0$ sufficiently small, and thus \eqref{DCT cond 1*} and \eqref{EVT style cond} hold, meaning that  Corollary~\ref{esc rat cor} applies as well. 
\end{remark}

The following lemma shows that $\nu_{\om,\ep_N}(\phi_{\om,0})$ converges to $1$ uniformly in $\om$ under our assumptions \eqref{C1'}, \eqref{C4'}, \eqref{C5'}, \eqref{C7'}, and \eqref{xibound}.
\begin{lemma}\label{lemC9'}
For each $N\in\NN$ there exists $C_{\ep_N}\geq 1$ with $C_{\ep_N}\to 1$ as $N\to\infty$ such that
\begin{align}\label{C9'ineq}
C_{\ep_N}^{-1}\leq \nu_{\om,\ep_N}(\ind),\, \nu_{\om, \ep_N}(\phi_{\om,0}) \leq C_{\ep_N}
\end{align}
for $m$-a.e.\ $\omega\in\Omega$.
\end{lemma}
\begin{proof}
Note that 
\begin{equation}
\label{opdiff}
\nu_{\sg\om,0}\lt((\cL_{\om,0}-\cL_{\om,\ep_N})(\~\cL_{\sg^{-k}\om,\ep_N}^k \phi_{\sg^{-k}\om,0})	\rt)
=
\lm_{\om,0}\nu_{\om,0}\lt(\ind_{H_{\om,\ep_N}}\~\cL_{\sg^{-k}\om,\ep_N}^k\phi_{\sg^{-k}\om,0}\rt).
\end{equation}
Further, by (\ref{C7'}) we have
\begin{equation}
\label{EVTeta2}
\nu_{\om,0}(H_{\om,\ep_N})
\leq
C_3
\mu_{\om,0}(H_{\om,\ep_N})
=
\frac{C_3
	(t_\om+\xi_{\om,N})}{N}
\end{equation}
for $m$-a.e. $\om\in\Om$.
Following the same derivation of \eqref{diff eigenvalues identity unif}, and using \eqref{C2} gives that 
\begin{align}
\lm_{\om,0}-\lm_{\om,\ep_N}
&=\nu_{\sg\om,0}\left((\cL_{\om,0}-\cL_{\om,\ep_N})(\phi_{\om,\ep_N})\right)
=\lm_{\om,0}\nu_{\om,0}(\ind_{H_{\om,\ep_N}}\phi_{\om,\ep_N}).
\label{diff eigenvalues identity unif unif}
\end{align}
Thus it follows from Remark~\ref{rem lm_0>lm_ep}, \eqref{diff eigenvalues identity unif unif}, \eqref{C5'}, and \eqref{EVTeta2} that
\begin{equation}
\label{EVTlamdiff2}
0\leq \lm_{\om,0}-\lm_{\om,\ep_N}
\leq \frac{C_2C_3\lm_{\om,0}(t_\om+\xi_{\om,N})}{N}.
\end{equation}    
Using Remark~\ref{rem61} and \eqref{EVTlamdiff2}, for all $N$ sufficiently large, we see that 
\begin{align}\label{LmUniErr}
1\leq\frac{\lm_{\om,0}}{\lm_{\om,\ep_N}}
=
\frac{\lm_{\om,0}}{\lm_{\om,0}-\lm_{\om,0}\nu_{\om,0}(\ind_{H_{\om,\ep_N}}\phi_{\om,\ep_N})}
\leq 
\frac{1}{1-C_2\nu_{\om,0}(H_{\om,\ep_N})}
\leq
\frac{1}{1-E_N},
\end{align}
where
\begin{align*}
E_N:=\frac{C_2C_3(|t|_\infty +W)}{N}\to 0
\end{align*}
as $N\to\infty$.

\begin{claim}\label{claim Lemma 6.1.a}
For every $n$, $\ep_N$, and $m$-a.e. $\om\in\Om$ we have
\begin{align*}
\absval{1-\nu_{\om,\ep_N}(\phi_{\om,0})}
\leq
n\left(\frac{1}{1-E_N}\right)^n\cdot
\left(C_1E_N+\frac{C_1C_2E_N}{1-E_N}\right)
+
C_2C_{\phi_0}\al(n).
\end{align*}
\end{claim}
\begin{subproof}
Using \eqref{C2} and a telescoping argument, we can write
\begin{align}
\absval{1-\nu_{\om,\ep_N}(\phi_{\om,0})}
&=\absval{\nu_{\sg^n\om,0}(\phi_{\sg^n\om,0})-\nu_{\om,\ep_N}(\phi_{\om,0})}
\nonumber\\
&
=\absval{\nu_{\sg^n\om,0}(\phi_{\sg^n\om,0})-\nu_{\sg^n\om,0}\big(\nu_{\om,\ep_N}(\phi_{\om,0})\cdot \phi_{\sg^n\om,\ep_N}\big)}
\nonumber\\	
&
=\absval{\nu_{\sg^n\om,0}\left(\phi_{\sg^n\om,0}-\~\cL_{\om,\ep_N}^n(\phi_{\om,0})+Q_{\om,\ep_N}^n(\phi_{\om,0})\right)}
\nonumber\\
&
\leq\sum_{k=0}^{n-1}\absval{\nu_{\sg^n\om,0}\left(\~\cL_{\sg^{n-k}\om,\ep_N}^k(\phi_{\sg^{n-k}\om,0})-\~\cL_{\sg^{n-(k+1)}\om,\ep_N}^{k+1}(\phi_{\sg^{n-(k+1)}\om,0})\right)}
\nonumber\\ 	
&\qquad
+\absval{\nu_{\sg^n\om,0}\left(Q_{\om,\ep_N}^n(\phi_{\om,0})\right)}
\nonumber\\
& 
=\sum_{k=0}^{n-1}\absval{\nu_{\sg^n\om,0}\left(\left(\~\cL_{\sg^{n-k}\om,\ep_N}^{k}\right)\left(\~\cL_{\sg^{n-(k+1)}\om,0}-\~\cL_{\sg^{n-(k+1)}\om,\ep_N}\right)(\phi_{\sg^{n-(k+1)}\om,0})\right)}
\label{lemma a sum}\\
&\qquad
+\absval{\nu_{\sg^n\om,0}\left(Q_{\om,\ep_N}^n(\phi_{\om,0})\right)}.
\label{lemma a error term}
\end{align}
First we note that we can estimate \eqref{lemma a error term} as
\begin{align}
\absval{\nu_{\sg^n\om,0}\left(Q_{\om,\ep_N}^n(\phi_{\om,0})\right)}\leq \norm{Q_{\om,\ep_N}^n(\phi_{\om,0})}_{\infty,\sg^n\om}
\leq 
C_2C_{\phi_0}\al(n)
.
\label{unifeq0}
\end{align}	
Now recall that
\begin{align*}
\hat{X}_{\sg^{-k}\om,k,\ep_N}:=\ind_{\bigcap_{j=0}^k T_{\sg^{-k}\om}^{-j}(\cJ_{\sg^{j-k}\om,\ep_N})}.
\end{align*}
Using \eqref{C2}, we can write \eqref{lemma a sum} as 
\begin{align}
&\sum_{k=0}^{n-1}\absval{\nu_{\sg^n\om,0}\left(\left(\~\cL_{\sg^{n-k}\om,\ep_N}^{k}\right)\left(\~\cL_{\sg^{n-(k+1)}\om,0}-\~\cL_{\sg^{n-(k+1)}\om,\ep_N}\right)(\phi_{\sg^{n-(k+1)}\om,0})\right)}
\nonumber\\
&\quad
=\sum_{k=0}^{n-1}\absval{\nu_{\sg^n\om,0}\left(\left(\lm_{\sg^{n-k}\om,\ep_N}^k\right)^{-1}
	\cL_{\sg^{n-k}\om,0}^k\left(\hat{X}_{\sg^{n-k}\om,k,\ep_N}\cdot\left(\~\cL_{\sg^{n-(k+1)}\om,0}-\~\cL_{\sg^{n-(k+1)}\om,\ep_N}\right)(\phi_{\sg^{n-(k+1)}\om,0})\right)\right)}
\nonumber\\	
&\quad
=\sum_{k=0}^{n-1}\frac{\lm_{\sg^{n-k}\om,0}^k}{\lm_{\sg^{n-k}\om,\ep_N}^k}\absval{\nu_{\sg^{n-k}\om,0}\left(\hat{X}_{\sg^{n-k}\om,k,\ep_N}\cdot\left(\~\cL_{\sg^{n-(k+1)}\om,0}-\~\cL_{\sg^{n-(k+1)}\om,\ep_N}\right)(\phi_{\sg^{n-(k+1)}\om,0})\right)}.
\label{unifeq1.0}
\end{align}	
Now, since $\sfrac{\lm_{\om,0}}{\lm_{\om,\ep_N}}\geq1$ (by Remark~\ref{rem lm_0>lm_ep}) for $m$ a.e. $\om$ and $\ep\geq 0$, and thus $\sfrac{\lm_{\sg^{n-k}\om,0}^k}{\lm_{\sg^{n-k}\om,\ep_N}^k}\geq1$ for each $k\geq 1$, we can write
\begin{align}
\eqref{unifeq1.0}&\leq
\frac{\lm_{\om,0}^n}{\lm_{\om,\ep_N}^n}\sum_{k=0}^{n-1}\absval{\nu_{\sg^{n-k}\om,0}\left(\hat{X}_{\sg^{n-k}\om,k,\ep_N}\cdot\left(\~\cL_{\sg^{n-(k+1)}\om,0}-\~\cL_{\sg^{n-(k+1)}\om,\ep_N}\right)(\phi_{\sg^{n-(k+1)}\om,0})\right)}
\nonumber\\
&
\leq 
\left(\frac{1}{1-E_N}\right)^n\cdot
\sum_{k=0}^{n-1}\absval{\nu_{\sg^{n-k}\om,0}\left(\hat{X}_{\sg^{n-k}\om,k,\ep_N}\cdot\left(\~\cL_{\sg^{n-(k+1)}\om,0}-\~\cL_{\sg^{n-(k+1)}\om,\ep_N}\right)(\phi_{\sg^{n-(k+1)}\om,0})\right)},
\label{unifeq1} 		
\end{align}
Using 
\begin{align*}
\~\cL_{\om,0}-\~\cL_{\om,\ep}&=\~\cL_{\om,0}-\lm_{\om,0}^{-1}\cL_{\om,\ep}+\lm_{\om,0}^{-1}\cL_{\om,\ep}-\~\cL_{\om,\ep}\nonumber\\
&=\lm_{\om,0}^{-1}\left(\cL_{\om,0}-\cL_{\om,\ep}\right)+(\lm_{\om,0}^{-1}-\lm_{\om,\ep}^{-1})\cL_{\om,\ep}\nonumber\\
&=\lm_{\om,0}^{-1}\left(\left(\cL_{\om,0}-\cL_{\om,\ep}\right)+\lm_{\om,0}\lm_{\om,\ep}(\lm_{\om,0}^{-1}-\lm_{\om,\ep}^{-1})\~\cL_{\om,\ep}\right)\nonumber\\
&=\lm_{\om,0}^{-1}\left(\left(\cL_{\om,0}-\cL_{\om,\ep}\right)+(\lm_{\om,\ep}-\lm_{\om,0})\~\cL_{\om,\ep}\right),
\end{align*}
the fact that
\begin{align*}
\cL_{\om,\ep_N}\left((f_{\sg\om}\circ T_\om)\cdot h_{\om}\right)=f_{\sg\om}\cdot\cL_{\om,\ep_N}(h_{\om}),
\end{align*}
for all $\ep_N\geq 0$ and $\om\in\Om$, and Remark~\ref{rem61}, we may estimate the sum in \eqref{unifeq1} by 
\begin{align}
&\sum_{k=0}^{n-1}\absval{\nu_{\sg^{n-k}\om,0}\left(\hat{X}_{\sg^{n-k}\om,k,\ep_N}\cdot\left(\~\cL_{\sg^{n-(k+1)}\om,0}-\~\cL_{\sg^{n-(k+1)}\om,\ep_N}\right)(\phi_{\sg^{n-(k+1)}\om,0})\right)}
\nonumber\\
&\quad
\leq\sum_{k=0}^{n-1}\lm_{\sg^{n-(k+1)}\om,0}^{-1}\absval{\nu_{\sg^{n-k}\om,0}\left(\hat{X}_{\sg^{n-k}\om,k,\ep_N}\cdot\left(\cL_{\sg^{n-(k+1)}\om,0}-\cL_{\sg^{n-(k+1)}\om,\ep_N}\right)(\phi_{\sg^{n-(k+1)}\om,0})\right)}
\nonumber\\
&\qquad
+\sum_{k=0}^{n-1}\lm_{\sg^{n-(k+1)}\om,0}^{-1}\absval{(\lm_{\sg^{n-(k+1)}\om,\ep_N}-\lm_{\sg^{n-(k+1)}\om,0})\cdot \nu_{\sg^{n-k}\om,0}\left(\hat{X}_{\sg^{n-k}\om,k,\ep_N}\cdot\~\cL_{\sg^{n-(k+1)}\om,\ep_N}(\phi_{\sg^{n-(k+1)}\om,0})\right)} 
\nonumber\\
&\quad
= \sum_{k=0}^{n-1}\lm_{\sg^{n-(k+1)}\om,0}^{-1}
\absval{\nu_{\sg^{n-k}\om,0}\left(\left(\cL_{\sg^{n-(k+1)}\om,0}-\cL_{\sg^{n-(k+1)}\om,\ep_N}\right)\left(\left(\hat{X}_{\sg^{n-k}\om,k,\ep_N}\circ T_{\sg^{n-(k+1)}\om}\right)\cdot\phi_{\sg^{n-(k+1)}\om,0}\right)\right)}
\label{unifeq2.1}\\
&\qquad
+\sum_{k=0}^{n-1}\bigg(\lm_{\sg^{n-(k+1)}\om,0}^{-1}\absval{1-\frac{\lm_{\sg^{n-(k+1)}\om,0}}{\lm_{\sg^{n-(k+1)}\om,\ep_N}}}
\nonumber\\
&\quad\qquad
\cdot\absval{ \nu_{\sg^{n-k}\om,0}\left(\cL_{\sg^{n-(k+1)}\om,\ep_N}\left(\left(\hat{X}_{\sg^{n-k}\om,k,\ep_N}\circ T_{\sg^{n-(k+1)}\om}\right)\cdot\phi_{\sg^{n-(k+1)}\om,0}\right)\right)}\bigg).
\label{unifeq2.2}
\end{align}
Using \eqref{diff eigenvalues identity unif unif}, \eqref{EVTlamdiff2}, and Remark~\ref{rem61} we can estimate 
\eqref{unifeq2.1} to get 
\begin{align}
&\sum_{k=0}^{n-1}\lm_{\sg^{n-(k+1)}\om,0}^{-1}\absval{\nu_{\sg^{n-k}\om,0}\left(\left(\cL_{\sg^{n-(k+1)}\om,0}-\cL_{\sg^{n-(k+1)}\om,\ep_N}\right)\left(\left(\hat{X}_{\sg^{n-k}\om,k,\ep_N}\circ T_{\sg^{n-(k+1)}\om}\right)\cdot\phi_{\sg^{n-(k+1)}\om,0}\right)\right)}
\nonumber\\
&\qquad
=\sum_{k=0}^{n-1}\absval{\nu_{\sg^{n-(k+1)}\om,0}\left(\ind_{H_{\sg^{n-(k+1)}\om,\ep_N}}\cdot\left(\left(\hat{X}_{\sg^{n-k}\om,k,\ep_N}\circ T_{\sg^{n-(k+1)}\om}\right)\cdot\phi_{\sg^{n-(k+1)}\om,0}\right)\right)}
\nonumber\\
&\qquad
\leq \sum_{k=0}^{n-1}\nu_{\sg^{n-(k+1)}\om,0}\left(\ind_{H_{\sg^{n-(k+1)}\om,\ep_N}}\right)\norm{\phi_{\sg^{n-(k+1)}\om,0}}_{\infty,\sg^{n-(k+1)}\om}
\leq
\frac{C_2C_3(|t|_\infty +W)\cdot n}{N}\leq nE_N.
\label{unifeq3}
\end{align}
Since $\cL_{\om,\ep}(f)=\cL_{\om,0}(\hat X_{\om,0,\ep}f)$ we can rewrite the second product in the sum in \eqref{unifeq2.2} so that we have 
\begin{align}
&\nu_{\sg^{n-k}\om,0}\left(\cL_{\sg^{n-(k+1)}\om,\ep_N}\left(\left(\hat{X}_{\sg^{n-k}\om,k,\ep_N}\circ T_{\sg^{n-(k+1)}\om}\right)\cdot\phi_{\sg^{n-(k+1)}\om,0}\right)\right)
\nonumber\\
&\qquad=
\nu_{\sg^{n-k}\om,0}\left(\cL_{\sg^{n-(k+1)}\om,0}\left(\hat{X}_{\sg^{n-(k+1)}\om,0,\ep_N}\left(\hat{X}_{\sg^{n-k}\om,k,\ep_N}\circ T_{\sg^{n-(k+1)}\om}\right)\cdot\phi_{\sg^{n-(k+1)}\om,0}\right)\right)
\nonumber\\
&\qquad
\leq \norm{\cL_{\sg^{n-(k+1)}\om,0}\ind}_{\infty,\sg^{n-k}\om}\norm{\phi_{\sg^{n-(k+1)}\om,0}}_{\infty,\sg^{n-(k+1)}\om}
\leq C_1C_2.
\label{unifeq4}
\end{align}
Inserting \eqref{unifeq4} into \eqref{unifeq2.2} 
and using \eqref{LmUniErr} yields 
\begin{align}
&
\sum_{k=0}^{n-1}\absval{1-\frac{\lm_{\sg^{n-(k+1)}\om,0}}{\lm_{\sg^{n-(k+1)}\om,\ep_N}}}\cdot\absval{ \nu_{\sg^{n-k}\om,0}\left(\cL_{\sg^{n-(k+1)}\om,\ep_N}\left(\left(\hat{X}_{\sg^{n-k}\om,k,\ep_N}\circ T_{\sg^{n-(k+1)}\om}\right)\cdot\phi_{\sg^{n-(k+1)}\om,0}\right)\right)}
\nonumber\\
&\qquad
\leq 
\sum_{k=0}^{n-1}\frac{C_1C_2E_N}{1-E_N} 
= n\cdot \frac{C_1C_2E_N}{1-E_N}. 
\label{unifeq5}
\end{align}
Thus, collecting the estimates \eqref{unifeq2.1}-\eqref{unifeq5} together with \eqref{unifeq0} and inserting into \eqref{unifeq1} yields
\begin{align}
\absval{1-\nu_{\om,\ep_N}(\phi_{\om,0})}
\leq
\left(\frac{1}{1-E_N}\right)^n\cdot
\left(nE_N+\frac{C_1C_2E_N\cdot n}{1-E_N}\right)
+
C_2C_{\phi_0}\al(n),
\label{nuepsphi0control}
\end{align}
which finishes the proof of the claim.
\end{subproof}

To finish the proof of Lemma~\ref{lemC9'}, we note that \eqref{nuepsphi0control} holds for $m$-a.e.\ $\omega$, every $N$ sufficiently large, and each $n\ge 1$.
Given a $\delta>0$, choose and fix $n$ so that $C_2C_{\phi_0}\alpha(n)<\delta/2$.
Because $\lim_{N\to\infty}E_N=0$, 
we may choose $N$ large enough so that first summand in (\ref{nuepsphi0control}) is also smaller than $\delta/2$.
Thus, $\lim_{N\to \infty}\nu_{\omega,\epsilon_N}(\phi_{\omega,0})=1$, uniformly in $\omega$.
This proves \eqref{C9'ineq} for $\nu_{\om,\ep_N}(\phi_{\omega,0})$;  we immediately obtain the other inequality using (\ref{C5'}) and (\ref{C7'}), and thus the proof of Lemma~\ref{lemC9'} is complete. 

\end{proof}

We now obtain a formula for the explicit form of Gumbel law for the extreme value distribution.
\begin{theorem}
\label{evtthm}
Given a random open system $(\mathlist{\bcomma}{\Om, m, \sg, \cJ_0, T, \cB, \cL_0, \nu_0, \phi_0, H_\ep})$ satisfying \eqref{C1'}, \eqref{C2}, \eqref{C3}, \eqref{C4'}, \eqref{C5'}, \eqref{C7'}, \eqref{C8},  and \eqref{xibound}, 
for almost every $\omega\in\Omega$ one has
\begin{equation}
\label{evtthmeqn}
\lim_{N\to\infty}\nu_{\om,0}\left(X_{\om,N-1,\ep_N}\right)
=
\lim_{N\to\infty}\mu_{\om,0}\left(X_{\om,N-1,\ep_N}\right)
=
\lim_{N\to\infty}\frac{\lm_{\om,\ep_N}^N}{\lm_{\om,0}^N}
=
\exp\left(-\int_\Om t_\om\ta_{\om,0}\, dm(\om)\right).
\end{equation}
\end{theorem}

\begin{proof}
\quad

\textbf{Step 1: Estimating $\lambda_{\omega,\epsilon_N}/\lambda_{\omega,0}$.}
To work towards constructing an estimate for $\lm^N_{\om,\ep_N}/\lm^N_{\om,0}$, we first estimate $\lm_{\om,\ep_N}/\lm_{\om,0}$.
For brevity, in this step we drop the $N$ subscripts on $\ep$.
Following the proof of Theorem~\ref{thm: GRPT} up to equation (\ref{maineq}), which uses assumptions \eqref{C2} and \eqref{C3}, we have 
\begin{eqnarray}
\lefteqn{\nu_{\sg^{-n}\om,\ep}(\phi_{\sg^{-n}\om,0})\frac{\lm_{\om,0}-\lm_{\om,\ep}}{\Dl_{\om,\ep}}}
\nonumber\\
\label{def:ta_ep,n}		
\qquad&=&\underbrace{1-\sum_{k=0}^{n-1} \lm_{\sg^{-(k+1)}\om,0}^{-1}(\lm^k_{\sg^{-k}\om,\ep})^{-1}q^{(k)}_{\om,\ep}}_{=:\ta_{\om,\ep,n}}
\\
\label{def:ta'_ep,n}		
\qquad&+&\Dl^{-1}_{\om,\ep}\underbrace{\sum_{k=1}^n \lm^{-1}_{\sg^{-k}\om,0}(\lm_{\sg^{-k}\om,0}-\lm_{\sg^{-k}\om,\ep})\nu_{\sg\om,0}((\cL_{\om,0}-\cL_{\om,\ep})(\~\cL_{\sg^{-k}\om,\ep}^k)(\phi_{\sg^{-k}\om,0}))}_{=:\ta'_{\om,\ep,n}}
\\
\label{def:ta''_ep,n}	
\qquad&+&\Dl^{-1}_{\om,\ep}\underbrace{\nu_{\sg\om,0}((\cL_{\om,0}-\cL_{\om,\ep})(Q^n_{\sg^{-n}\om,\ep}(\phi_{\sg^{-n}\om,0})))}_{=:\ta''_{\om,\ep,n}}.
\end{eqnarray}
By first rearranging to solve for $\lm_{\om,\ep}$ we have
\begin{eqnarray*}
\lm_{\om,\ep}&=&\lm_{\om,0}-\frac{\ta_{\om,\ep,n}\Dl_{\om,\ep}+\ta'_{\om,\ep,n}+\ta''_{\om,\ep,n}}{\nu_{\sg^{-n}\om,\ep}(\phi_{\sg^{-n}\om,0})},
\end{eqnarray*}
and thus
\begin{align}
\frac{\lm_{\om,\ep}}{\lm_{\om,0}}=1-
\underbrace{
	\frac{\ta_{\om,\ep,n}\Dl_{\om,\ep}}{\lm_{\om,0}\nu_{\sg^{-n}\om,\ep}(\phi_{\sg^{-n}\om,0})}}_{=:Y^{(1)}_{\om,\ep,n}}
-\underbrace{\frac{\ta'_{\om,\ep,n}}{\lm_{\om,0}\nu_{\sg^{-n}\om,\ep}(\phi_{\sg^{-n}\om,0})}}_{Y^{(2)}_{\om,\ep,n}}
-\underbrace{\frac{\ta''_{\om,\ep,n}}{\lm_{\om,0}\nu_{\sg^{-n}\om,\ep}(\phi_{\sg^{-n}\om,0})}}_{=:Y^{(3)}_{\om,\ep,n}}.
\label{Y^ieqn}
\end{align}
Setting $Y_{\om,\ep,n}:=Y^{(1)}_{\om,\ep,n}+Y^{(2)}_{\om,\ep,n}+Y^{(3)}_{\om,\ep,n}$ applying Taylor to $\log(1-\cdot)$, we obtain
\begin{equation}
\label{Yeqn}
\frac{\lm_{\om,\ep}}{\lm_{\om,0}}
=\exp\left(-Y_{\om,\ep,n}-\frac{Y_{\om,\ep,n}^2}{2(1-y)^2}\right),
\end{equation}
where $0\le y\le Y_{\om,\ep,n}$.
Setting $\ep=\ep_N$ in (\ref{Yeqn}) we obtain
\begin{eqnarray}
\label{82}
\frac{\lm^N_{\om,\ep_N}}{\lm_{\om,0}^N}
&=&
\exp\left(
-\sum_{i=0}^{N-1}Y_{\sg^i\om,\ep_N,n} -\sum_{i=0}^{N-1}\frac{Y_{\sg^i\om,\ep_N,n}^2}{2(1-y)^2}
\right)\\
\label{82a}	&=&	\exp\left(
-\frac{1}{N}\sum_{i=0}^{N-1}\left(g^{(1)}_{N,n}(\sg^i\om)+g^{(2)}_{N,n}(\sg^i\om)+g^{(3)}_{N,n}(\sg^i\om)\right) -\sum_{i=0}^{N-1}\frac{Y_{\sg^i\om,\ep_N,n}^2}{2(1-y)^2}
\right),
\end{eqnarray}
where $g^{(j)}_{N,n}(\omega)=NY^{(j)}_{\omega,\epsilon_N,n}$ for $j=1,2,3$ and $0\le y\le Y_{\omega,\epsilon_N,n}$.

\textbf{Step 2: A sequential ergodic lemma.}
In preparation for estimating the products along orbits contained in $\lm^N_{\om,\ep_N}/\lm^N_{\om,0}$, we state 
the following result of Maker \cite{Maker}:
\begin{lemma}
\label{ptwiselemma2}
For $N\ge 0$, let $g_N\in L^1(m)$.
Suppose that as $N\to\infty$, $g_N\to g$ $m$-almost everywhere for some $g\in L^1(m)$ and that there exists $g^*\in L^1(m)$ such that $|g_N|\leq g^*$ for each $N\in\NN$.
Then for $m$-a.e. $\om\in\Om$, $\lim_{N\to\infty}\frac{1}{N}\sum_{i=0}^{N-1} g_N(\sg^i\om)$ exists and equals $\mathbb{E}(g)=\int_\Om g\ dm$.
\end{lemma}

\textbf{Step 3: Estimating $g^{(1)}$.}

In this step we construct estimates of $g^{(1)}$ that are required to apply Lemma \ref{ptwiselemma2}.
In preparation for the first use of Lemma \ref{ptwiselemma2} we recall that 
$$
g_{N,n}^{(1)}(\omega):=\frac{N\theta_{\omega,\epsilon_N,n}\Delta_{\omega,\epsilon_N}}{\lambda_{\omega,0}\nu_{\sigma^{-n}\omega,\epsilon_N}(\phi_{\sigma^{-n}\omega,0})}\\
=\frac{N\theta_{\omega,\epsilon_N,n}\mu_{\omega,0}(H_{\omega,\epsilon_N})}{\nu_{\sigma^{-n}\omega,\epsilon_N}(\phi_{\sigma^{-n}\omega,0})}
=\frac{\theta_{\omega,\epsilon_N,n}(t_\omega+\xi_{\omega,N})}{\nu_{\sigma^{-n}\omega,\epsilon_N}(\phi_{\sigma^{-n}\omega,0})},$$
where  $\lim_{N\to\infty}\xi_{\omega,N}=0$ for a.e.\ $\omega$ and $|\xi_{\om,N}|\le W$ by (\ref{xibound}).
We also set $g_n^{(1)}(\omega):=\theta_{\omega,0,n}t_\omega$, where 
\begin{align}\label{theta_0,n def}
\ta_{\om,0,n}:=1-\sum_{k=0}^{n-1}\left(\lm_{\sg^{-(k+1)}\om,0}^{k+1}\right)^{-1}q_{\om,0}^{(k)}
=1-\sum_{k=0}^{n-1}\hat q_{\om,0}^{(k)}.
\end{align}\index{$\ta_{\om,0,n}$}

By Lemma \ref{qbound2} and \eqref{C8} we see that $0\le \hat q_{\om,0}^{(k)},\hat q_{\om,\ep_N}^{(k)}\le 1$ for each $k$, $N$, and $m$-a.e.\ $\omega$. 
From \eqref{def:ta_ep,n} we have that 
\begin{align}
\ta_{\om,\ep_N,n}
:=1-\sum_{k=0}^{n-1}\lm_{\sg^{-(k+1)}\om,0}^{-1}(\lm^k_{\sg^{-k}\om,\ep_N})^{-1}q_{\om,\ep_N}^{(k)}
=1-\sum_{k=0}^{n-1}\frac{\lm_{\sg^{-k}\om,0}^k}{\lm_{\sg^{-k}\om,\ep_N}^k}\hat q_{\om,\ep_N}^{(k)},
\end{align}\index{$\ta_{\om,\ep_N,n}$}
and thus that $\ta_{\om,0,n}\in[0,1]$ and $\ta_{\om,\ep_N,n}\leq 1$ for each $n$ and $m$-a.e.\ $\omega$.
Using \eqref{LmUniErr} and the fact that $\sfrac{\lm_{\om,0}}{\lm_{\om,\ep_N}}\geq 1$ (by Remark~\ref{rem lm_0>lm_ep}), we have that 
\begin{align*}
\ta_{\om,\ep_N,n}
\geq 1-\frac{\lm_{\sg^{-(n-1)}\om,0}^{n-1}}{\lm_{\sg^{-(n-1)}\om,\ep_N}^{n-1}}\sum_{k=0}^{n-1}\hat q_{\om,\ep_N}^{(k)}
\geq
1-\left(1+\frac{C_1E_N}{C_1-E_N}\right)^n.
\end{align*}
Using Lemma~\ref{lemC9'} and \eqref{xibound} we see that $g_{N,n}^{(1)},g_n^{(1)}\in L^1(m)$ for each $N,n$.
Moreover using these same facts	we have that $g_{N,n}^{(1)}$, and thus $|g_{N,n}^{(1)}-g_n^{(1)}|$, is dominated in $L^1(m)$.

Referring to the first term in the $Y$-sum in (\ref{82}), we may now apply Lemma \ref{ptwiselemma2} to conclude that 
$$\sum_{i=0}^{N-1}\frac{\theta_{\sigma^i\omega,\epsilon_N,n}\Delta_{\sigma^i\omega,\epsilon_N}}{\lambda_{\sigma^i\omega,0}\nu_{\sigma^{-n+i}\omega,\epsilon_N}(\phi_{\sigma^{-n+i}\omega,0})}=\frac{1}{N}\sum_{i=0}^{N-1}g_{N,n}^{(1)}(\sigma^i\omega)\to \int_\Omega g_n^{(1)}(\omega)\ dm(\omega)=\int_\Omega \theta_{\omega,0,n}t_\omega\ dm(\omega),$$
as $N\to\infty$ for each $n$ and $m$-a.e.\ $\omega$.

\textbf{Step 4: Estimating $g^{(2)}$.} 
We now perform a similar analysis to the previous step to control the terms in the sum (\ref{82}) corresponding to $\theta'$.
Again in preparation for applying Lemma \ref{ptwiselemma2}, recall that
$$g_{N,n}^{(2)}(\omega)=\frac{N\theta'_{\omega,\epsilon_N,n}}{\lambda_{\omega,0}\nu_{\sigma^{-n}\omega,\epsilon_N}(\phi_{\sigma^{-n}\omega,0})}$$
and set $g_n^{(2)}(\omega)\equiv 0$ for each $n$.
Using \eqref{def:ta'_ep,n}, (\ref{opdiff}), (\ref{EVTlamdiff2}), and (\ref{C7'}), for sufficiently large $N$ we have  
\begin{eqnarray*}
\ta_{\om,\ep_N,n}'
&:=&\sum_{k=1}^n \lm^{-1}_{\sg^{-k}\om,0}(\lm_{\sg^{-k}\om,0}-\lm_{\sg^{-k}\om,\ep})\nu_{\sg\om,0}((\cL_{\om,0}-\cL_{\om,\ep})(\~\cL_{\sg^{-k}\om,\ep}^k)(\phi_{\sg^{-k}\om,0}))
\\
&\leq& 
\frac{C_2C_3\lambda_{\omega,0}}{N}\sum_{k=1}^n\lambda_{\sigma^{-k}\omega,0}^{-1}\left(\lambda_{\sigma^{-k}\omega,0}(t_{\sg^{-k}\om}+\xi_{\sg^{-k}\om,N})\right)
\nu_{\om,0}\lt(\ind_{H_{\om,\ep_N}}\~\cL_{\sg^{-k}\om,\ep_N}^k\phi_{\sg^{-k}\om,0}\rt)
\\
&\leq&
\frac{C_2C_3\lambda_{\omega,0}}{N}\sum_{k=1}^n(t_{\sg^{-k}\om}+\xi_{\sg^{-k}\om,N})
\nu_{\om,0}(H_{\om,\ep_N})\norm{\~\cL_{\sg^{-k}\om,\ep_N}^k\phi_{\sg^{-k}\om,0}}_{\infty,\om}
\\
&\leq& 
\frac{C_2C_3\lambda_{\omega,0}}{N}\cdot \frac{C_3(t_\om+\xi_{\om,N})}{N}
\sum_{k=1}^n(t_{\sg^{-k}\om}+\xi_{\sg^{-k}\om,N})
\norm{\~\cL_{\sg^{-k}\om,\ep_N}^k\phi_{\sg^{-k}\om,0}}_{\infty,\om}.
\end{eqnarray*}
Using Lemma~\ref{lemC9'}, (\ref{C5'}), and (\ref{C4'}) we note that
\begin{eqnarray}
\nonumber		\norm{\~\cL_{\sg^{-k}\om,\ep_N}^k\phi_{\sg^{-k}\om,0}}_{\infty,\om}
&=&
\norm{\nu_{\sg^{-k}\om,\ep_N}(\phi_{\sg^{-k}\om,0})\phi_{\om,\ep_N}+Q_{\sg^{-k}\om,\ep_N}^k\phi_{\sg^{-k}\om,0}}_{\infty,\om}
\\
\nonumber		&\leq &
\norm{C_{\ep_N}C_2+C_{\phi_0}C_2\al(k)}_{\infty,\om}
\\
\label{Lpowerinfbound}		&\leq&
C_2C_{\ep_N}+C_2C_{\phi_0}\al.
\end{eqnarray}
Finally, we have for $N$ sufficiently large and by (\ref{xibound}) 
\begin{eqnarray*}
g_{N,n}^{(2)}(\omega)&\le& \frac{NC_2C_3\lambda_{\omega,0}(C_2C_{\ep_N}+C_2C_{\phi_0}\al)}{N\lambda_{\omega,0}\nu_{\sg^{-n}\om,\ep_{N}}(\phi_{\sg^{-n}\om,0})}\cdot \frac{C_3(t_\om+\xi_{\om,N})}{N}
\sum_{k=1}^n(t_{\sg^{-k}\om}+\xi_{\sg^{-k}\om,N}).
\end{eqnarray*}

Integrability and domination of $g_{N,n}^{(2)}$ follows from the the fact that $t\in L^\infty(m)$ and $|\xi_{\om,N}|\le W,$ for almost all $\om.$ 
Moreover by  (\ref{xibound}), for each $n$ we have that $g_{N,n}^{(2)}\to 0$ almost everywhere as $N\to\infty$.
Referring to $Y_{\om,\ep_N,n}^{(2)}$ in \eqref{Y^ieqn}, we may now apply Lemma \ref{ptwiselemma2} to conclude that 
$$
\sum_{i=0}^{N-1}\frac{\theta'_{\sigma^i\omega,\epsilon_N,n}}{\lambda_{\sigma^i\omega,0}\nu_{\sigma^{-n+i}\omega,\epsilon_N}(\phi_{\sigma^{-n+i}\omega,0})}=\frac{1}{N}\sum_{i=0}^{N-1}g_{N,n}^{(2)}(\sigma^i\omega)\to \int_\Omega g_n^{(2)}(\omega)\ dm(\omega)=0,
$$
as $N\to\infty$ for each $n$ and a.e.\ $\omega$.

\textbf{Step 5: Estimating $g^{(3)}$.}
We repeat a similar analysis to control the terms in the sum (\ref{82}) corresponding to $\theta''$.
Again in preparation for applying Lemma \ref{ptwiselemma2}, recall that
$$g_{N,n}^{(3)}(\omega)=\frac{N\theta''_{\omega,\epsilon_N,n}}{\lambda_{\omega,0}\nu_{\sigma^{-n}\omega,\epsilon_N}(\phi_{\sigma^{-n}\omega,0})}.$$
We begin developing an upper bound for $|g_{N,n}^{(3)}|$.
Using \eqref{def:ta''_ep,n}, (\ref{C2}), (\ref{C4'}), and  (\ref{C7'}) we have for sufficiently large $N$ that  
\begin{align*}
|\ta_{\om,\ep_N,n}''|&:=
\nu_{\sg\om,0}(\cL_{\om,0}(\ind_{H_{\om,\ep_N}}Q^n_{\sg^{-n}\om,\ep}(\phi_{\sg^{-n}\om,0})))
\\
&=
\lambda_{\omega,0}\nu_{\om,0}(\ind_{H_{\om,\ep_N}}Q_{\sg^{-n}\om,\ep_N}^n\phi_{\sg^{-n}\om,0})		\\
&\leq 
\lambda_{\omega,0}\nu_{\om,0}(H_{\om,\ep_N})\norm{Q_{\sg^{-n}\om,\ep_N}^n\phi_{\sg^{-n}\om,0}}_{\infty,\om}
\\
&\leq
\lambda_{\omega,0}\nu_{\om,0}(H_{\om,\ep_N})C_{\phi_0}\norm{\phi_{\sg^{-n}\om,0}}_{\cB_{\sg^{-n}\om}}\al(n)
\\
&\leq
\lambda_{\omega,0}C_{\phi_0}C_2C_3\al(n)\mu_{\om,0}(H_{\om,\ep_N})
\\
&=
\frac{\lambda_{\omega,0}C_{\phi_0}C_2C_3\al(n) (t_\om+\xi_{\om,N})}{N}.
\end{align*}
Therefore, using Lemma~\ref{lemC9'}
$$|g_{N,n}^{(3)}(\omega)|\le \frac{N\lm_{\om,0}C_{\phi_0}C_2C_3\al(n) (t_\om+\xi_{\om,N})}{N\lm_{\om,0}\nu_{\sg^{-n}\om,\ep_{N}}(\phi_{\sg^{-n}\om,0}))}\le C_{\epsilon_N}C_{\phi_0}C_2C_3\al(n) (t_\om+\xi_{\om,N})=:\tilde{g}^{(3)}_{N,n}(\omega).
$$
We set $\tilde{g}^{(3)}_n(\omega)=C_{\phi_0}C_2C_3\alpha(n)t_\omega$. 
Integrability of $\tilde{g}_{N,n}^{(3)}$ and $\tilde{g}_{n}^{(3)}$ follows from (\ref{xibound}) and the fact $t\in L^\infty(m)$ and $|\xi_{\om,N}|\le W,$ for almost all $\om.$ 
Similarly, (recalling that $C_{\ep_N}\to 1$ as $N\to\infty$ by Lemma~\ref{lemC9'}) for each $n$, $\tilde{g}_{N,n}^{(3)}\to \tilde{g}_{n}^{(3)}$  almost everywhere as $N\to\infty$.

Referring to $Y_{\om,\ep_N,n}^{(3)}$ in \eqref{Y^ieqn}, we may now apply Lemma \ref{ptwiselemma2} to $\tilde{g}_{N,n}^{(3)}$ and $\tilde{g}_n^{(3)}$ to conclude that 
$$\frac{1}{N}\sum_{i=0}^{N-1}g_{N,n}^{(3)}(\sigma^i\omega)\le\frac{1}{N}\sum_{i=0}^{N-1}\tilde{g}_{N,n}^{(3)}(\sigma^i\omega)
\to \int_\Omega \tilde{g}_n^{(3)}(\omega)\ dm(\omega)
=C_{\phi_0}C_2C_3\alpha(n)\int_\Omega t_\omega\ dm(\omega)$$
as $N\to\infty$ for each $n$ and a.e.\ $\omega$.

\textbf{Step 6: Finishing up.}\label{step6}
Recall from (\ref{evtexpression3}) that
\begin{equation*}
(\ref{evtexpression})=\frac{\lm_{\om,\ep_N}^N}{\lm_{\om,0}^N}
\lt(\nu_{\om,\ep_N}(\ind)+\nu_{\sg^N\om,0}\lt(Q_{\om,\ep_N}^N( \ind)\rt)\rt).
\end{equation*}
Using (\ref{C4'}) and Lemma~\ref{lemC9'} we see that $\lim_{N\to\infty}\lt(\nu_{\om,\ep_N}(\ind)+\nu_{\sg^N\om,0}\lt(Q_{\om,\ep_N}^N( \ind)\rt)\rt)=1$ for $m$-a.e.\ $\omega\in\Omega$.
By (\ref{82a}) and Steps 3, 4, and 5 we see that for any $n$ 
\begin{eqnarray}
\label{ratiobound}	\lefteqn{\lim_{N\to\infty}\exp\left(	-\frac{1}{N}\sum_{i=0}^{N-1}\left(g^{(1)}_{N,n}(\sg^i\om)+g^{(2)}_{N,n}(\sg^i\om)+g^{(3)}_{N,n}(\sg^i\om)\right)\right)\exp\left( -\sum_{i=0}^{N-1}\frac{Y_{\sg^i\om,\ep_N,n}^2}{2(1-y)^2}
	\right)}\\
\nonumber	&&\le \lim_{N\to\infty} \frac{\lambda_{\omega,\epsilon_N}^N}{\lambda_{\omega,\epsilon_N}^N}\\
\nonumber	&&\le
\lim_{N\to\infty}\exp\left(
-\frac{1}{N}\sum_{i=0}^{N-1}\left(g^{(1)}_{N,n}(\sg^i\om)+g^{(2)}_{N,n}(\sg^i\om)-\tilde{g}^{(3)}_{N,n}(\sg^i\om)\right)\right)\exp\left( -\sum_{i=0}^{N-1}\frac{Y_{\sg^i\om,\ep_N,n}^2}{2(1-y)^2}
\right),
\end{eqnarray}
where $0\le y\le Y_{\sg^i\om,\ep_N,n}$.
We now treat the Taylor remainder terms.
From Steps 3, 4, and 5 for all $N$ sufficiently large we have and for almost all $\om$: 
\begin{eqnarray}
\label{g1bound}|g^{(1)}_{N,n}|&\le& C_{\epsilon_N}(|t|_\infty+W), 
\\
\label{g2bound}|g^{(2)}_{N,n}|&\le& {C_{\epsilon_N}C_2C_3(C_2C_{\ep_N}+C_2C_{\phi_0}\al)}\cdot \frac{C_3n(|t|_\infty+W)^2}{N},
\\
\label{g3bound}|g^{(3)}_{N,n}|&\le& \~ g^{(3)}_{N,n}:= C_{\phi_0}C_2C_3\al(|t|_\infty+W).
\end{eqnarray}
Further,
\begin{equation}
\label{remainderbound}\sum_{i=0}^{N-1}\frac{Y_{\sg^i\om,\ep_N,n}^2}{2(1-y)^2}\le \sum_{i=0}^{N-1}\frac{Y_{\sg^i\om,\ep_N,n}^2}{2(1-Y_{\sg^i\om,\ep_N,n})^2}=\sum_{i=0}^{N-1}\frac{(G_{\sg^i\om,\ep_N,n}/N)^2}{2(1-(G_{\sg^i\om,\ep_N,n}/N))^2},
\end{equation}
where $G_{\omega,\ep_N,n}:=g^{(1)}_{N,n}+g^{(2)}_{N,n}+g^{(3)}_{N,n}$.
Using the bounds (\ref{g1bound})--(\ref{g3bound}) we see that (\ref{remainderbound}) approaches $0$ for almost all $\omega$ for each $n$ as $N\to\infty$.
Therefore, combining the expressions developed in Steps 3, 4, and 5 for the $N\to\infty$ limits with (\ref{ratiobound}) we have
\begin{eqnarray*}
\exp\lt(-
\int_{\Omega}\theta_{\omega,0,n}t_\omega\ dm(\omega)
\rt)
\le 
\lim_{N\to\infty} \frac{\lambda_{\omega,\epsilon_N}^N}{\lambda_{\omega,0}^N}
\le 
\exp\lt(-
\int_{\Omega}\theta_{\omega,0,n}t_\omega\ dm(\omega)+C_2C_3C_{\phi_0}\alpha(n)\int_\Omega t_\omega\ dm(\omega)
\rt).
\end{eqnarray*}
Recalling the definitions of $\theta_{\omega,0,n}$ \eqref{theta_0,n def} and $\theta_{\omega,0}$ \eqref{eq: def of theta_0} and the fact that $0\le\theta_{\omega,0,n}\le 1$ (by \eqref{theta in [0,1]}), we may use dominated convergence to take the $n\to\infty$ limit to obtain
$$
\exp\lt(-
\int_{\Omega}\theta_{\omega,0}t_\omega\ dm(\omega) \rt)
\le \lim_{N\to\infty} \frac{\lambda_{\omega,\epsilon_N}^N}{\lambda_{\omega,0}^N}
\le \exp\lt(-
\int_{\Omega}\theta_{\omega,0}t_\omega\ dm(\omega)\rt),$$
thus completing the proof that 
\begin{equation*}
\lim_{N\to\infty}\nu_{\om,0}\left(X_{\om,N-1,\ep_N}\right)
=
\lim_{N\to\infty}\frac{\lm_{\om,\ep_N}^N}{\lm_{\om,0}^N}
=
\exp\left(-\int_\Om t_\om\ta_{\om,0}\, dm(\om)\right).
\end{equation*}
To see that $\lim_{N\to\infty}\mu_{\om,0}\left(X_{\om,N-1,\ep_N}\right)$ is also equal to this value, we simply recall that \eqref{evtexpression2mu} gives that
\begin{align*}
\mu_{\om,0}(X_{\om,N-1,\ep_N})
=
\frac{\lm_{\om,\ep_N}^N}{\lm_{\om,0}^N}
\lt(\nu_{\om,\ep_N}(\phi_{\om,0})+\nu_{\sg^N\om,0}\lt(Q_{\om,\ep_N}^N( \phi_{\om,0})\rt)\rt).
\end{align*}
Now since (\ref{C4'}) and Lemma~\ref{lemC9'} together give that $$\lim_{N\to\infty}\lt(\nu_{\om,\ep_N}(\phi_{\om,0})+\nu_{\sg^N\om,0}\lt(Q_{\om,\ep_N}^N( \phi_{\om,0})\rt)\rt)=1$$ for $m$-a.e.\ $\omega\in\Omega$, we must in fact have that 
\begin{equation*}
\lim_{N\to\infty}\mu_{\om,0}\left(X_{\om,N-1,\ep_N}\right)
=
\lim_{N\to\infty}\frac{\lm_{\om,\ep_N}^N}{\lm_{\om,0}^N}
=
\exp\left(-\int_\Om t_\om\ta_{\om,0}\, dm(\om)\right),
\end{equation*}
which completes the proof of Theorem~\ref{evtthm}.

\end{proof}

\subsection{The relationship between condition \eqref{os} and the H\"usler condition}
We now return to the discussion initiated in the Introduction to compare our assumption (\ref{os}) for the thresholds $z_{\om, N}$ with the H\"usler type condition (\ref{BLH}). 
We show that in the more general situation considered in our paper with random boundary level $t_{\omega}$, the limit (\ref{BLH}) will follow from  the simpler assumption (\ref{os}),  provided we replace $t$ in (\ref{BLH}) with the expectation of $t_{\omega}.$ 

Recall our  assumption for the choice of the thresholds (\ref{xibound}) is $\mu_{\omega,0}(h_{\omega}(x)> z_{\omega, N})=(t_{\omega}+\xi_{\omega, N})/N$,
where $\xi_{\om, N}$ goes to zero almost surely when $N\rightarrow \infty$ and $\xi_{\omega,N}\le W$ for a.e. $\omega$.  It is immediate to see by dominated convergence that: 
\begin{equation}\label{dc}
\lim_{N\rightarrow \infty}\int_\Om |N\mu_{\omega,0}(h_{\omega}(x)> z_{\omega, N})-t_{\omega}|\ dm(\omega)=0.
\end{equation}
Applying our non-standard ergodic  Lemma \ref{ptwiselemma2} with $g_N(\om):=N\mu_{\om,0}(h_{\omega}(x)> z_{\omega, N})$ and $g(\om):=t_\om$, 
one may transform the sum in  (\ref{BLH}) as follows: 
$$
\lim_{N\rightarrow \infty}\frac{1}{N} \sum_{i=0}^{N-1}N\mu_{\omega,0}(h_{\sigma^j\omega}(T^j_{\omega}(x))> z_{\sigma^j\omega, N})
=\int_\Om t_{\omega}\ dm(\omega),
$$
which is the condition (\ref{BLH}) with $t$ replaced by the expectation of $t_{\omega}.$ 
\subsection{Hitting time statistics}
\label{sec:hts}
It is well known that in the deterministic setting there is a close relationship between extreme value theory and the statistics of first hitting time, see for instance \cite{FFT10, book}. We now show how our Theorem \ref{evtthm}, with a slight modification,  can be interpreted in terms of a suitable definition of (quenched) first hitting time distribution.
Let us consider as in the previous sections, a sequence of small random holes $\mathcal{H}_{\om,N}:=\{H_{\sigma^j\om, \epsilon_N}\}_{j\ge 0},$ and define
the first random hitting time as
$$
\tau_{\om, \mathcal{H}_{\om,N}}(x)=\inf\{k\ge 1, T^k_{\om}(x)\in H_{\sigma^k\om, \epsilon_N}\}.
$$\index{$\tau_{\om, \mathcal{H}_{\om,N}}$}

We recall that the usual statistics of hitting times is written in the form
$
\mu_{\om,0}\left(\tau_{\om, \mathcal{H}_{\om,N}}>t\right),
$
for nonnegative values of $t.$ 
Since the sets $H_{\sigma^j\om, \epsilon_N}$ have measure tending to zero when $N\rightarrow \infty,$ and therefore the first  hitting times could eventually grow to infinity, one needs a rescaling in order to get a meaningful limit distribution. This is achieved in the next Proposition. 
In our current setting, condition (\ref{xibound}) reads: 	
$
\mu_{\om,0}(H_{\om, \epsilon_N})=\frac{t_{\om}+\xi_{\om, N}}{N}, 
$
with 
$\lim_{N\to\infty} \xi_{\omega,N}=0$ for a.e.\ $\omega$ and 
$|\xi_{\omega,N}|\le W$ for a.e.\ $\omega$ and all $N\ge 1$.

\begin{proposition}\label{hve}
If our random open system $(\mathlist{\bcomma}{\Om, m, \sg, \cJ_0, T, \cB, \cL_0, \nu_0, \phi_0, H_\ep})$ satisfies the assumptions of Theorem \ref{evtthm} with  the sequence $H_{\om,\ep_N}$ verifying condition (\ref{xibound}),
then the first random hitting time satisfies the limit, for $\om$ $m$-a.e.
\begin{equation}
\label{eqhit}
\lim_{N\to\infty}\mu_{\om,0}\left(\tau_{\om, \mathcal{H}_{\om,N}} \mu_{\om,0}(H_{\om, \epsilon_N})>t_{\om}\right) = \exp\left(-{\int_\Om t_{\om} \theta_{\om,0}dm}\right).
\end{equation}
\end{proposition}
\begin{proof}
For $N\ge 1$ the event, 
\begin{equation*}
\{\tau_{\om,  \mathcal{H}_{\om,N}}>N\}=\{x\in I; T_{\om}(x)\in H_{\sigma\om, \epsilon_N}^c, \dots, T^N_{\om}(x)\in H_{\sigma^N\om, \epsilon_N}^c\}
\end{equation*}
is also equal to
$$
T^{-1}_{\om}\left(x\in I; x\in H_{\sigma\om, \epsilon_N}^c, T_{\sigma\om}(x)\in H_{\sigma^2\om, \epsilon_N}^c,\dots, T_{\sigma\om}^{N-1}(x)\in H_{\sigma^N\om, \epsilon_N}^c\right).
$$
Then, by  equivariance of $\mu_0$ we obtain the link between the statistics of hitting time and extreme value theory:
\begin{eqnarray}\label{li}
\mu_{\om,0}\left(\tau_{\om, \mathcal{H}_{\om,N}}>N\right)
\nonumber&=&
\mu_{\om,0}\left(T^{-1}_{\om}\left(x\in I: x\in H_{\sigma\om, \epsilon_N}^c, T_{\sigma\om}(x)\in H_{\sigma^2\om, \epsilon_N}^c,\dots, T_{\sigma\om}^{N-1}(x)\in H_{\sigma^N\om, \epsilon_N}^c\right)\right)\\
\nonumber&=&
\mu_{\sigma\om,0}\left(x\in H_{\sigma\om, \epsilon_N}^c, T_{\sigma\om}(x)\in H_{\sigma^2\om, \epsilon_N}^c,\dots, T_{\sigma\om}^{N-1}(x)\in H_{\sigma^N\om, \epsilon_N}^c\right)\\
\label{X}&=&\mu_{\sigma\omega,0}(X_{\sigma\om,N-1,\ep_N}).
\end{eqnarray}

In order to rescale the eventually growing first random  hitting times, we invoke the condition  \eqref{xibound};  by substituting $N=(t_{\om}+\xi_{\om, N})/\mu_{\om,0}(H_{\om, \epsilon_N})$ in the LHS of the inequality above
we have
\begin{equation}\label{ev}
\mu_{\om,0}\left(\tau_{\om, \mathcal{H}_{\om,N} }>N\right)=\mu_{\om,0}\left(\tau_{\om, \mathcal{H}_{\om,N}} \mu_{\om,0}(H_{\om, \epsilon_N})>t_{\om}+\xi_{\om, N}\right).
\end{equation}

Our final preparation before applying Theorem \ref{evtthm} is to show that 
\begin{equation}\label{ee}
|\mu_{\om,0}\left(\tau_{\om, \mathcal{H}_{\om,N}} \mu_{\om,0}(H_{\om, \epsilon_N})>t_{\om}+\xi_{\om, N}\right)-\mu_{\om,0}\left(\tau_{\om, \mathcal{H}_{\om,N}} \mu_{\om,0}(H_{\om, \epsilon_N})>t_{\om}\right)|\rightarrow 0, \ N\rightarrow \infty.
\end{equation}
Since by a standard trick, see for instance eq. 5.3.6 in \cite{book},
$$
\{\tau_{\om, \mathcal{H}_{\om,N}} \mu_{\om,0}(H_{\om, \epsilon_N})>t_{\om}\}\backslash \{\tau_{\om, \mathcal{H}_{\om,N}} \mu_{\om,0}(H_{\om, \epsilon_N})>t_{\om}+\xi_{\om, N}\}\subset \bigcup_{j=\left\lceil{\frac{t_{\om}}{\mu_{\omega,0}(H_{\om, \epsilon_N})}}\right\rceil}^{\left\lceil{\frac{t_{\om}+\xi_{\omega,n}}{\mu_{\omega,0}(H_{\om, \epsilon_N})}}\right\rceil}
T_{\om}^{-j}(H_{\sigma^j\om, \epsilon_N})
$$
we have by equivariance
$$
|\mu_{\om,0}\left(\tau_{\om, \mathcal{H}_{\om,N}} \mu_{\om,0}(H_{\om, \epsilon_N})>t_{\om}+\xi_{\om, N}\right)-\mu_{\om,0}\left(\tau_{\om, \mathcal{H}_{\om,N}} \mu_{\om,0}(H_{\om, \epsilon_N})>t_{\om}\right)|\le \sum_{j=\left\lceil{\frac{t_{\om}}{\mu_{\omega,0}(H_{\om, \epsilon_N})}}\right\rceil}^{\left\lceil{\frac{t_{\om}+\xi_{\omega,n}}{\mu_{\omega,0}(H_{\om, \epsilon_N})}}\right\rceil}
\mu_{\sigma^j\om,0}({H}_{\sigma^j\om,N}).
$$
For $N$ large enough:
$$
\sum_{j=\left\lceil{\frac{t_{\om}}{\mu_{\omega,0}(H_{\om, \epsilon_N})}}\right\rceil}^{\left\lceil{\frac{t_{\om}+\xi_{\omega,n}}{\mu_{\omega,0}(H_{\om, \epsilon_N})}}\right\rceil}
\mu_{\sigma^j\om,0}({H}_{\sigma^j\om,N})
\le \left\lceil{\frac{|\xi_{\om, N}|}{\mu_{\omega,0}(H_{\om,\epsilon_N})}}\right\rceil\frac{|t|_\infty+W}
{N}
\le \left\lceil{\frac{|\xi_{\om, N}|\ N}{t_{\om}-|\xi_{\om, N}|}}\right\rceil\frac{|t|_\infty+W}{N},
$$
which goes to zero by \eqref{xibound}.
Recalling $t_\om>0$ for a.e.\ $\om$, the final expression above goes to zero as $N\rightarrow \infty$  for $\om$ $m$-a.e. 

Using \eqref{X}--\eqref{ee} and noting that $\lim_{N\to\infty}\mu_{\sigma\om,0}(X_{\sigma\om,N-1,\ep_N})$ is nonrandom, applying Theorem \ref{evtthm} yields
\begin{equation}\label{hts}
\lim_{N\to\infty}\mu_{\om,0}\left(\tau_{\om, \mathcal{H}_{\om,N}} \mu_{\om,0}(H_{\om, \epsilon_N})>t_{\om}\right) = \exp\left(-\int_\Om t_{\om} \theta_{\om,0}\ dm\right).
\end{equation}
\end{proof}

\section[thermodynamic formalism for random open interval maps via perturbation]{Quenched thermodynamic formalism for random open interval maps via perturbation}\label{sec: existence}
In this section we present an explicit class of random piecewise-monotonic interval maps for which our Theorem~\ref{thm: dynamics perturb thm}, Corollary~\ref{esc rat cor},  and Theorem~\ref{evtthm}  apply. 
Using a perturbative approach, we introduce a family of small random holes parameterised by $\epsilon>0$ into a random closed dynamical system, and for every small $\ep$ we prove (i) the existence of a unique random conformal measure $\set{\nu_{\om,\ep}}_{\om\in\Om}$ with fiberwise support in $X_{\om,\infty,\ep}$ and (ii) a unique random absolutely continuous invariant measure $\set{\mu_{\om,\ep}}_{\om\in\Om}$ which satisfies an exponential decay of correlations and is the unique relative equilibrium state for the random open system $(\mathlist{\bcomma}{\Om, m, \sg, \cJ_0, T, \cB, \cL_0, \nu_0, \phi_0, H_\ep})$. In addition, we prove the existence of a random absolutely continuous (with respect to $\nu_{\om,0}$) conditionally invariant probability measure $\set{\vrho_{\om,\ep}}_{\om\in\Om}$ with fiberwise support in $[0,1]\bs H_{\om,\ep}$.

We now suppose that the spaces $\cJ_{\om,0}=[0,1]$ for each $\om\in\Om$ and the maps $T_\om:[0,1]\to[0,1]$ are surjective, finitely-branched, piecewise monotone, nonsingular (with respect to Lebesgue), and that there exists $C\geq 1$ such that 
\begin{align}
\tag{\Gls*{E1}}\myglabel{E1}{E1}
\esssup_\om |T_\om'|\leq C
\qquad\text{ and }\qquad
\esssup_\om D(T_\om)
\leq C,
\end{align}
where $D(T_\om):=\sup_{y\in[0,1]}\# T_\om^{-1}(y)$.
We let $\cZ_{\om,0}$ denote the (finite) monotonicity partition of $T_\om$ and for each $n\geq 2$ we let $\cZ_{\om,0}^{(n)}$ denote the partition of monotonicity of $T_\om^n$.
\begin{enumerate}[align=left,leftmargin=*,labelsep=\parindent]
\item[(\Gls*{MC})]\myglabel{MC}{M}
The map $\sg:\Om\to\Om$ is a homeomorphism, the skew-product map $T:\Om\times [0,1]\to\Om\times [0,1]$ is measurable, and $\omega\mapsto T_\omega$ has countable range. 
\end{enumerate}
\begin{remark}\label{M2 holds}
Under assumption (\ref{M}), the family of transfer operator cocycles $\{\mathcal{L}_{\omega,\epsilon}\}_{\epsilon\ge 0}$ satisfies the conditions of Theorem 17 \cite{FLQ2} ($m$-continuity and $\sigma$ a homeomorphism). Note that condition \eqref{M}
implies that $T$ satisfies \eqref{M1} and the cocycle generated by $\cL_0$ satisfies condition \eqref{M2}.
\end{remark}

Recall that the \emph{variation} of $f:[0,1]\to\mathbb{R}_+$ on $Z\subset [0,1]$ be
\begin{align*}
\var_{Z}(f)=\underset{x_0<\dots <x_k, \, x_j\in Z}{\sup}\sum_{j=0}^{k-1}\absval{f(x_{j+1})-f(x_j)}, 
\end{align*}
and  $\var(f):=\var_{[0,1]}(f)$. We let $\BV=\BV([0,1])$ denote the set of functions on $[0,1]$ that have bounded variation. 
Given a non-atomic and fully supported measure $\nu$ (i.e.\ for any nondegenerate interval $J\sub [0,1]$ we have $\nu(J)>0$) we let $\BV_\nu\sub L^\infty(\nu)$\index{$\BV_\nu$} be the set of (equivalence classes of) functions of \emph{bounded variation} on $[0,1]$, with norm given by
\begin{align*}
\|f\|_{\BV_\nu}:=\underset{\tilde{f}=f \ \nu\text{ a.e.}}{\inf} \var(\tilde{f})+\nu(|f|).    
\end{align*} 
If we require to emphasise that elements of $\BV_\nu$ are equivalence classes, we denote these by $[f]_\nu$ (resp. $[f]_1$).
Note that if $f\in\BV$ is a function of bounded variation, then it is always possible to choose a representative of minimal variation from the equivalence class $[f]_\nu$.
We define $\BV_{1}\sub L^\infty(\Leb)$\index{$\BV_1$} and $\|\spot\|_{\BV_1}$ similarly, with the measure $\nu$ replaced with Lebesgue measure.  We denote the supremum norm on $L^\infty(\Leb)$ by $\|\spot\|_{\infty,1}$. 
It follows from Rychlik \cite{rychlik_bounded_1983} that $\BV_\nu$ and $\BV_1$ are Banach spaces. 
The following proposition gives the equivalence of the norms $\|\spot\|_{\BV_\nu}$ and $\|\spot\|_{\BV_1}$.
\begin{proposition}\label{prop norm equiv}
Given a fully supported and non-atomic measure $\nu$ on $[0,1]$ and $f\in\BV$ we have that 
\begin{equation*}
(1/2) \|f\|_{\BV_1}\le \|f\|_{\BV_\nu}\le 2\|f\|_{\BV_1}.
\end{equation*}
\end{proposition} 

\begin{proof}
We first show that for $f\in\BV$ we have 
\begin{align}\label{eq class cap BV}
[f]_\nu\cap \BV = [f]_1\cap\BV.
\end{align}
To see this let $\~f\in[f]_\nu\cap \BV$. As $\nu$ is a fully supported and non-atomic measure, we must have that the set $\{x: f(x)\neq\~f(x)\}$ is countable. Thus $\~f\in[f]_1\cap \BV$. As $\Leb$ is also fully supported and non-atomic the same reasoning implies that the reverse inclusion also holds, proving \eqref{eq class cap BV}. As a direct consequence of \eqref{eq class cap BV} we have that 
\begin{align}\label{equiv cl var eq}
\underset{\tilde{f}=f \ \nu\text{ a.e.}}{\inf} \var(\tilde{f})
=
\underset{\tilde{f}=f \ \Leb\text{ a.e.}}{\inf} \var(\tilde{f}).
\end{align}
Since $f$ is continuous everywhere except on a set of at most countably many points, letting $\cC$ denote the set of intervals of continuity for $f$,  we have 
\begin{align}\label{nu Leb ineq 1}
\Lebessinf_{[0,1]} f 
&=\inf_{J \in\cC } \Lebessinf_J f
=\inf_{J \in\cC } \inf_J f
=\inf_{J \in\cC } \nuessinf_J f
= \nuessinf_{[0,1]} f. 
\end{align}
Using similar reasoning we must also have 
\begin{align}\label{nu Leb ineq 2}
\Lebesssup f = \nuesssup f.
\end{align}
Combining \eqref{equiv cl var eq} and \eqref{nu Leb ineq 1} we have 
\begin{align*}
\|f\|_{\BV_1}=\inf_{\tilde{f}=f\ \Leb\ a.e.} \var(\tilde{f})+\Leb(|f|) &\leq 
2\inf_{\tilde{f}=f\ \Leb\ a.e.} \var(\tilde{f})+\Lebessinf |f|
\\
&= 2\inf_{\tilde{f}=f\ \nu\ a.e.} \var(\tilde{f})+\nuessinf |f|
\\
&\leq 2\inf_{\tilde{f}=f\ \nu\ a.e.} \var(\tilde{f})+\nu(|f|)
\leq 2\|f\|_{\BV_\nu}.
\end{align*}
Similarly, using \eqref{equiv cl var eq} and \eqref{nu Leb ineq 2} we have $\|f\|_{\BV_\nu}\leq 2 \|f\|_{\BV_1}$, and thus the proof is complete. 

\end{proof}
Proposition \ref{prop norm equiv} will be used later to provide (non-random) equivalence of $\|\spot\|_{\BV_{\nu_{\omega,0}}}$ and $\|\spot\|_{\BV_1}$ for each $\omega\in\Omega$.
We set $J_\om:=|T_\om'|$\index{$J_\om$} and define the random Perron--Frobenius operator, acting on functions in $BV$
$$
P_\om(f)(x):=\sum_{y\in T_\om^{-1}(x)}\frac{f(y)}{J_\om(y)}.
$$
The operator $P$ satisfies the well-known property that 
\begin{align}\label{PF prop}
\int_{[0,1]}P_\om(f)\,d\Leb=\int_{[0,1]}f\,d\Leb
\end{align}\index{$P_\om$}
for $m$-a.e. $\om\in\Om$ and all $f\in\BV$.
Recall from Section~\ref{sec:ROS} that $g_0=\set{g_{\om,0}}_{\om\in\Om}$ and that
$$\cL_{\om,0}(f)(x):=\sum_{y\in T_\om^{-1}(x)}g_{\om,0}(y)f(y),\quad f\in\BV.$$
We assume that the weight function $g_{\om,0}$ lies in $\BV$ for each $\om\in\Om$ and satisfies
\begin{equation}
\tag{\Gls*{E2}}\myglabel{E2}{E2}
\esssup_\omega \| g_{\omega,0}\|_{\infty,1}<\infty,
\end{equation}
and 
\begin{equation}
\tag{\Gls*{E3}}\myglabel{E3}{E3}
\essinf_\omega \inf g_{\omega,0}>0.
\end{equation}
Note that \eqref{E1} and \eqref{E2} together imply 
\begin{align}\label{fin sup L1}
\esssup_\om\norm{\cL_{\om,0}\ind}_{\infty,1}\leq \esssup_\om D(T_\om)\norm{g_{\om,0}}_{\infty,1}<\infty
\end{align}
and 
\begin{align}\label{fin gJ}
\esssup_\om\norm{g_{\om,0}J_\om}_{\infty,1}<\infty.
\end{align}
We also assume a uniform covering\index{covering} condition\footnote{We could replace the covering condition with the assumption of a strongly contracting potential. See \cite{AFGTV-IVC} for details.}
:
\begin{enumerate}[align=left,leftmargin=*,labelsep=\parindent]
\item[(\Gls*{E4})]\myglabel{E4}{E4}
For every subinterval $J\subset [0,1]$ there is a $k=k(J)$ such that for a.e.\ $\omega$ one has $T^k_\omega(J)=[0,1]$.
\end{enumerate}
Concerning the open system we assume that the holes $H_{\om,\ep}\sub [0,1]$ are chosen so that assumption
\eqref{A} 
holds. We also assume for each $\om\in\Om$ and each $\ep>0$ that $H_{\om,\ep}$ is composed of a finite union of intervals such that 
\begin{enumerate}[align=left,leftmargin=*,labelsep=\parindent]
\item[(\Gls*{E5})]\myglabel{E5}{E5}
There is a uniform-in-$\epsilon$ and uniform-in-$\omega$ upper bound on the number of connected components of $H_{\om,\ep}$,
\end{enumerate}
and
\begin{align}
\tag{\Gls*{E6}}\myglabel{E6}{E6}
\lim_{\ep\to 0}\esssup_\om \Leb(H_{\om,\ep})=0,
\end{align}
and 

\begin{enumerate}[align=left,leftmargin=*,labelsep=\parindent]
\item[(\Gls*{EX})]\myglabel{EX}{EX}
There exists an $\ep>0$ and an open neighborhood $\~H_{\om,\ep}\bus H_{\om,\ep}$ such that  
$T_\om(U_\om)\bus \~H_{\sg\om,\ep}^c$, where $U_\om:=\cup_{Z\in\cZ_{\om,0}}\ol{Z}_{\ep}$ and $\ol{Z_{\ep}}$ denotes the closure of $Z_\ep\in A_{\om,\ep}:=\{Z\cap \~H_{\om,\ep}^c:Z\in\cZ_{\om,0}\}$ with $m(\set{\om\in\Om: \#A_{\om,\ep}\geq 2})>0$.
\end{enumerate}
\begin{remark}\label{rem check cond X}
Assumption \eqref{EX} is satisfied for any random open system such that each map contains at least two intervals of monotonicity, the holes are contained in the interior of exactly one interval of monotonicity, and the image of the complement of the hole is the full interval, i.e. $T_\om(H_{\om,\ep}^c)=[0,1]$. In particular, \eqref{EX} is satisfied if there exists a full branch outside of the hole. 
\end{remark}
Recall that condition \eqref{cond X} states 
\begin{enumerate}
\item[\mylabel{X}{cond X restatement}]
For $m$-a.e. $\om\in\Om$ we have $X_{\om,\infty,\ep}\neq\emptyset$.
\end{enumerate}
\begin{remark}\label{rem check X}
Note that since \eqref{A} (i.e. $H_{\ep'}\sub H_{\ep}$ for $\ep'<\ep$) implies that $X_{\om,\infty,\ep'}\bus X_{\om,\infty,\ep}$ for all $\ep'<\ep$, \eqref{cond X} holds if there exists $\ep>0$ such that $X_{\om,\infty,\ep}\neq\emptyset$ for $m$-a.e. $\om\in\Om$. Furthermore, since $T_\om(X_{\om,\infty,\ep})\sub X_{\sg\om,\infty,\ep}$, if $X_{\om,\infty,\ep}\neq\emptyset$ then $X_{\sg^N\om,\infty,\ep'}\neq\emptyset$ for each $N\geq 1$ and $\ep'\leq \ep$. As $\cX_{\infty,\ep}$ is forward invariant we have that $X_{\om,\infty,\ep}\neq\emptyset$ not only implies that $\cX_{\infty,\ep}\neq\emptyset$, but also that $\cX_{\infty,\ep}$ is infinite. 
\end{remark}

The following proposition ensures that condition \eqref{cond X} holds.
\begin{proposition}\label{prop check cond X}

The assumption \eqref{EX} implies \eqref{cond X}.
\end{proposition}
\begin{proof}
In light of Remark \ref{rem check X}, to show that \eqref{cond X} is satisfied, it suffices to show that there is some $\ep>0$ such that  $X_{\om,\infty,\ep}\neq\emptyset$ for $m$-a.e. $\om\in\Om$.

Let $\~T_{\om,Z}$ denote the continuous extension of $T_\om$ onto $\ol{Z}_{\ep}$ for each $Z_\ep\in A_{\om,\ep}$, and let $\~X_{\om,\infty,\ep}$ denote the survivor set for the open system consisting of the maps $\~T_\om$ and holes $\~H_{\om,\ep}$. By Proposition~\ref{prop surv nonemp} (taking $V_\om=[0,1]\bs\~H_{\om,\ep}$ and $U_{\om,j}=\ol{Z}_{\ep}$ for each $1\leq j\leq \#A_{\om,\ep}$) we see that $\~X_{\om,\infty,\ep}$ is uncountable. 
Let 
$$
\cD_\om:=\bigcup_{j\geq 0}\~T_\om^{-j}(\cup_{Z_\ep\in A_{\sg^j\om,\ep}}\ol{Z}_{\ep}\bs Z_{\ep}).
$$
Since the survivor set for the original (unmodified) open system 
$X_{\om,\infty,\ep}\sub \~X_{\om,\infty,\ep}\bs\cD_\om$,
and since $\cD_\om$ is at most countable, we must in fact have that $X_{\om,\infty,\ep}\neq\emptyset$, thus satisfying \eqref{cond X}.
\end{proof}

Because we are considering small holes shrinking to zero measure, it is natural to  assume that
\begin{align}
\tag{\Gls*{E7}}\myglabel{E7}{E7}
T_\om(\cJ_{\om,\ep})=[0,1]
\end{align}
for $m$-a.e. $\om\in\Om$ and all $\ep>0$ sufficiently small.
Further we suppose that there exists $n'\geq1$ and $\epsilon_0>0$ such that
\footnote{Note that the $9$ appearing in \eqref{E8} is not optimal. See Section \ref{sec: examples} and \cite{AFGTV20} for how this assumption may be improved. }
\begin{equation}
\tag{\Gls*{E8}}\myglabel{E8}{E8}
9\cdot\esssup_\om\|g_{\omega,0}^{(n')}\|_{\infty,1}
<
\essinf_\om\inf_{0\leq\ep\le \ep_0}\inf\cL_{\om,\ep}^{n'}\ind,
\end{equation}
where 
$$
\cL_{\om,\ep}(f)(x):=\cL_{\om,0}(\ind_{H_{\om,\ep}^c}f)(x)
=
\sum_{y\in T_{\om}^{-1}(x)}g_{\om,\ep}(y)f(y)
,\quad f\in\BV
$$
and $g_{\om,\ep}:=g_{\om,0}\ind_{H_{\om,\ep}^c}$\index{$g_{\om,\ep}$} as in Section~\ref{sec: open systems}.

Note that \eqref{E7} and \eqref{E3} together imply that $\cL_{\om,\ep}\ind(x)>0$   
for all $x\in [0,1]$:
\begin{align}\label{uniflbLeps*}
\essinf_\om\inf_{\ep\le \ep_0}\inf\cL_{\om,\ep}\ind
>\essinf_\om\inf g_{\om,0}>0,
\end{align}
and since $\norm{g_{\om,\ep}}_{\infty,1}\leq \norm{g_{\om,0}}_{\infty,1}$ for all $\ep>0$, \eqref{E8} is equivalent to the following 
\begin{align}\label{strong cont pot}
9\cdot\sup_{0\leq \ep\leq \ep_0}\esssup_\om\|g_{\omega,\ep}^{(n')}\|_{\infty,1}
<
\essinf_\om\inf_{0\leq\ep\le \ep_0}\inf\cL_{\om,\ep}^{n'}\ind.
\end{align}

\begin{remark}
Note that the assumption \eqref{E7} is equivalent to there existing
$N'\geq 1$ such that for $m$-a.e. $\om\in\Om$ and all $\ep\geq0$ sufficiently small 
$$
T_\om^{N'}(X_{\om,N'-1,\ep})=[0,1].
$$
Indeed, since the surviving sets are forward invariant \eqref{surv set forw inv}, we have that 
$T_\om^{N'-1}(X_{\om,N'-1,\ep})\sub X_{\sg^{N'-1}\om,0,\ep}=\cJ_{\sg^{N'-1}\om,\ep}$, and thus, 
$$
[0,1]=T_\om^{N'}(X_{\om,N'-1,\ep})\bus T_{\sg^{N'-1}\om}(\cJ_{\sg^{N'-1}\om,\ep}).
$$
Furthermore, we note that without assumption \eqref{E7}, it is easy to construct examples for which $X_{\om,\infty,\ep}=\emptyset$ for $m$-a.e. $\om\in\Om$. 
\end{remark}

For each $n\in\NN$ and $\om\in\Om$ we let $\sA_{\om,0}^{(n)}$ be the collection of all finite partitions of $[0,1]$ such that
\begin{align}\label{eq: def A partition}
\var_{A_i}(g_{\om,0}^{(n)})\leq 2\norm{g_{\om,0}^{(n)}}_{\infty,1}
\end{align}
for each $\cA=\set{A_i}\in\sA_{\om,0}^{(n)}$.
Given $\cA\in\sA_{\om,0}^{(n)}$, let $\widehat\cZ_{\om,\ep}^{(n)}(\cA)$ be the coarsest partition amongst all those finer than $\cA$ and $\cZ_{\om,0}^{(n)}$ such that all elements of $\widehat\cZ_{\om,\ep}^{(n)}(\cA)$ are either disjoint from $X_{\om,n-1,\ep}$ or contained in $X_{\om,n-1,\ep}$.
\begin{remark}
Note that if $\var_Z(g_{\om,0})\leq 2\norm{g_{\om,0}^{(n)}}_{\infty,1}$ for each $Z\in\cZ_{\om,0}^{(n)}$ then we can take the partition $\cA=\cZ_{\om,0}^{(n)}$. Furthermore, the $2$ above can be replaced by some $\hat\al\geq 0$ (depending on $g_{\om,0}$) following the techniques of Section~\ref{sec: examples} and \cite{AFGTV20}. 
\end{remark}
Define the subcollection 
\begin{align}\label{def of Z_*}
\cZ_{\om,*,\ep}^{(n)}:=\{Z\in \widehat\cZ_{\om,\ep}^{(n)}(\cA): Z\sub X_{\om,n-1,\ep} \}.    
\end{align}\index{$\cZ_{\om,*,\ep}^{(n)}$}
Recalling that $g_{\om,\ep}^{(n)}:=g_{\om,0}^{(n)}\ind_{X_{\om,n-1,\ep}}$, \eqref{eq: def A partition} implies that 
\begin{align}\label{eq: def A partition for g_ep}
\var_{Z}(g_{\om,\ep}^{(n)})\leq 2\norm{g_{\om,0}^{(n)}}_{\infty,1}
\end{align}
for each $Z\in \cZ_{\om,*,\ep}^{(n)}$.    
We assume the following covering condition for the open system 
\begin{enumerate}[align=left,leftmargin=*,labelsep=\parindent]
\item[(\Gls*{E9})]\myglabel{E9}{E9}
There exists $k_o(n')\in\NN$ such that for $m$-a.e. $\om\in\Om$, all $\ep>0$ sufficiently small, and all $Z\in\cZ_{\om,*,\ep}^{(n')}$ we have $T_\om^{k_o(n')}(Z)=[0,1]$, where $n'$ is the number coming from \eqref{E8}.
\end{enumerate}
\begin{remark}
Note that the uniform open covering time assumption \eqref{E9} clearly holds if \eqref{E4} holds and if there are only finitely many maps $T_\om$. In Remark~\ref{Alt E9 Remark} we present an alternative assumption to \eqref{E9}.
\end{remark}

The following lemma extends several results in \cite{DFGTV18A} from the specific weight $g_{\omega,0}=1/|T'_\omega|$ to general weights satisfying the conditions just outlined.

\begin{lemma}
\label{DFGTV18Alemma}
Assume that a family of random piecewise-monotonic interval maps $\{T_{\omega}\}$ satisfies (\ref{E2}), (\ref{E3}), and (\ref{E4}), as well as (\ref{E8}) and \eqref{M} for $\epsilon=0$.
Then 
\eqref{C1'} and	the $\epsilon=0$ parts of \eqref{C2}, \eqref{C3}, \eqref{C4'}, \eqref{C5'}, and  (\ref{C7'}) as well as \eqref{CCM} hold. 
Further, $\nu_{\om,0}$ is fully supported, condition \eqref{C4'} holds with $C_f=K$, for some $K<\infty$, and with $\alpha(N)=\gamma^N$ for some $\gamma<1$.
\end{lemma}
\begin{proof}
See Appendix \ref{appB}.
\end{proof}

In what follows we consider transfer operators acting on the Banach spaces $\mathcal{B}_\omega=\BV_{\nu_{\om,0}}$ for a.e. $\om\in\Om$.
The norm we will use is $\|\cdot\|_{\mathcal{B}_\omega}=\|\cdot\|_{\BV_{\nu_{\om,0}}}:=\var(\cdot)+\nu_{\omega,0}(|\cdot|)$.
As $\nu_{\om,0}$ is fully supported and non-atomic (Lemma \ref{DFGTV18Alemma}), Proposition \ref{prop norm equiv} implies that 
\begin{equation}
\label{normequiv}
(1/2) \|f\|_{\BV_1}\le \|f\|_{\mathcal{B}_\omega}\le 2\|f\|_{\BV_1}
\end{equation}
for $m$-a.e. $\om\in\Om$ and $f\in\BV$.
Furthermore, applying \eqref{normequiv} twice, we see that 
\begin{equation}
\label{normequiv2}
(1/4) \|f\|_{\cB_\om}\le \|f\|_{\mathcal{B}_{\sg^n\omega}}\le 4\|f\|_{\cB_\om}
\end{equation}
for $m$-a.e. $\om\in\Om$ and all $n\in\ZZ$.
It follows from the proof of Proposition \ref{prop norm equiv} that 
\begin{equation}
\label{sup norm equal}
\|f\|_{\infty,\om} = \|f\|_{\infty,1}
\end{equation}
for all $f\in\BV$ and $m$-a.e. $\om\in\Om$, where $\|\spot\|_{\infty,\om}$ denotes the supremum norm with respect to $\nu_{\om,0}$.
From \eqref{normequiv} we see that \eqref{B} is clearly satisfied, and thus we note that Remark \ref{revision 1} implies that \eqref{C6} holds if $\lim_{\ep\to 0}\nu_{\om,0}(H_{\om,\ep})=0$.

From Lemma~\ref{DFGTV18Alemma} we have that $\lm_{\om,0}:=\nu_{\sg\om,0}(\cL_{\om,0}\ind)$ and thus we may update \eqref{uniflbLeps*} to get 
\begin{align}\label{uniflbLeps}
\essinf_\om\lm_{\om,0}^{n'}
\geq \essinf_\om\inf_{\ep\le \ep_0}\inf\cL_{\om,\ep}^{n'}\ind
\geq \essinf_\om\inf g_{\om,0}^{(n')}>0.
\end{align}
Note that since the conditions \eqref{cond X} and \eqref{CCM} have been verified and we have assumed \eqref{A} and \eqref{M}, 
we see that $(\Om, m, \sg, [0,1], T, \BV, \cL_0, \phi_0, H_\ep)$ forms a random open system as defined in Section~\ref{sec: open systems} for all $\ep>0$ sufficiently small. 
We now use hyperbolicity of the $\epsilon=0$ transfer operator cocycle to guarantee that we have hyperbolic cocycles for small $\epsilon>0$, which will yield \eqref{C2}, \eqref{C3}, \eqref{C4'}, \eqref{C5'}, \eqref{C6}, and \eqref{C7'} for small positive $\epsilon$.
\begin{lemma}
\label{harrylemma2}
Assume that the conditions (\ref{E1})--(\ref{E9}) hold 
for the random open system 
$(\mathlist{\bcomma}{\Om, m, \sg, [0,1], T, \BV_1, \cL_0, \nu_0, \phi_0, H_\ep})$.
Then for sufficiently small $\epsilon>0$, conditions \eqref{C2}, \eqref{C3}, \eqref{C4'}, \eqref{C5'}, \eqref{C6}, and \eqref{C7'} hold.
Furthermore, the functionals $\nu_{\om,\ep}\in\BV_1^*$ can be identified with non-atomic Borel measures.
\end{lemma}
\begin{proof}
For each $\om$ and $\ep>0$ we define $\hat\cL_{\om,\ep}:=\lm_{\om,0}^{-1}\cL_{\om,\ep}$;\index{$\hat\cL_{\om,\ep}$}  note that $\hat\cL_{\om,0}=\~\cL_{\om,0}$.\index{$\hat\cL_{\om,0}$}
Our strategy is to apply Theorem 4.4 \cite{C19} (Theorem \ref{HarryThm4.4}), to conclude that for small $\epsilon$ the cocycles $\{\hat\cL_{\omega,\epsilon}\}$ are uniformly hyperbolic when considered as cocycles on the Banach space  $(\BV_1,\|\cdot\|_{\BV_1})$. 
Because of (\ref{normequiv}), we will conclude the existence of a uniformly hyperbolic splitting in $\|\cdot\|_{\cB_\om}$ for a.e. $\om$. For the convenience of the reader, in Appendix \ref{appA} we have compiled the necessary definitions, assumptions, and results of \cite{C19} which we will use here.

First, we note that Theorem 4.4 \cite{C19} (Theorem \ref{HarryThm4.4}) assumes that the Banach space on which the transfer operator cocycle acts is separable.
A careful check of the proof of Theorem 4.4 \cite{C19} (Theorem \ref{HarryThm4.4}) shows that it holds  for the Banach space $(\BV_1,\|\cdot\|_{\BV_1})$ under the alternative condition \eqref{M}
(see Appendix \ref{appA}).
To apply Theorem 4.4 \cite{C19} (Theorem \ref{HarryThm4.4}) we require, in our notation, that:

\begin{enumerate}
\item\label{harry assum 1 hat} $\hat\cL_{\omega,0}$ is a hyperbolic transfer operator cocycle on $\BV_1$ with norm $\|\cdot\|_{\BV_1}$ and a one-dimensional leading Oseledets space (see Definition \ref{harry def 3.1} \cite[Definition 3.1]{C19}), and slow and fast growth rates $0<\gamma<\Gamma$, respectively. We will construct $\gamma$ and $\Gamma$ shortly.
\item\label{harry assum 2 hat} The family of cocycles $\{\hat\cL_{\omega,\epsilon}\}_{0\le \epsilon\le\epsilon_0}$ satisfy a uniform Lasota--Yorke inequality 
\begin{equation*}
\|\hat\cL^k_{\omega,\epsilon}f\|_{\BV_1}\le A\alpha^k\|f\|_{\BV_1}+B^k\|f\|_{1}
\end{equation*}
for a.e.\ $\omega$ and $0\le\epsilon\le\epsilon_0$, where $\alpha\le \gamma<\Gamma\le B$.
\item\label{harry assum 3 hat} $\lim_{\epsilon\to 0}\esssup_{\omega}\trinorm{\hat\cL_{\omega,0}-\hat\cL_{\omega,\epsilon}}= 0$, where $\vertiii{\spot}$ is the $\BV-L^1(\Leb)$ triple norm.
\end{enumerate}
By Lemma \ref{DFGTV18Alemma}
we obtain a unique measurable family of equivariant functions $\{\phi_{\omega,0}\}$ satisfying (\ref{C7'}) and (\ref{C5'}) for $\epsilon=0$.
We have the equivariant splitting $\spn\{\phi_{\omega,0}\}\oplus V_\om$, where $V_\om=\{f\in \BV_1: \nu_{\omega,0}(f)=0\}$.
We claim that this splitting is hyperbolic in the sense of Definition \ref{harry def 3.1} \cite[Definition 3.1]{C19};  this will yield item \eqref{harry assum 1 hat} above.
To show this, we verify conditions \eqref{H1}--\eqref{H3} in \cite{C19} (See Appendix \ref{appA}).
In our setting, Condition \eqref{H1} \cite{C19} requires the norm of the projection onto the top space spanned by $\phi_{\omega,0}$, along the annihilator of $\nu_{\omega,0}$, to be uniformly bounded in $\omega$.
This is true because this projection acting on $f\in \BV_1$ is $\nu_{\omega,0}(f)\phi_{\omega,0}$ and therefore 
$$
\|\nu_{\omega,0}(f)\phi_{\omega,0}\|_{\BV_1}\le\esssup_\omega\|\phi_{\omega,0}\|_{\BV_1}\cdot \nu_{\omega,0}(f)\le 2\esssup_\omega\|\phi_{\omega,0}\|_{\BV_1}\cdot\|f\|_{\BV_1},
$$
using (\ref{C5'}) and equivalence of $\|\cdot\|_{\BV_1}$ and $\|\cdot\|_{\cB_\om}$ \eqref{normequiv}.
Next, we define 
$$
\alpha^{n'}:=\frac{9\esssup_\om \|g_{\om,0}^{( n')}\|_{\infty,1}}{\essinf_\om\inf_{\ep\geq 0} \inf\mathcal{L}^{ n'}_{\om,\ep}\ind}
<1,
$$ 
which is possible by (\ref{E8}).
Condition \eqref{H2} requires $\|\hat\cL^n_{\omega,0}\phi_{\omega,0}\|_{\BV_1}\ge C\Gamma^n\|\phi_{\omega,0}\|_{\BV_1}$ for some $C>0$, $\Gamma>0$, all $n$ and a.e.\ $\omega$.
By (\ref{C7'}) one has 
$$
\|\hat\cL^n_{\omega,0}\phi_{\omega,0}\|_{\BV_1}=\|\phi_{\sigma^n\omega,0}\|_{\BV_1}\ge \essinf_\omega \inf\phi_{\sigma^n\omega,0}>0,
$$
and thus we obtain \eqref{H2} with $C=\Gm=1$.
Condition \eqref{H3} requires $\|\hat\cL^n_{\omega,0}|_{V_\om}\|_{\BV_1}\le  K\gamma^n$ for some $K<\infty$, $\al\leq \gamma<1$, all $n$ and a.e.\ $\omega$.
This is provided by the  $\epsilon=0$ part of (\ref{C4'})---specifically the stronger exponential version guaranteed by Lemma \ref{DFGTV18Alemma}---and the equivalence of $\|\cdot\|_{\BV_1}$ and $\|\cdot\|_{\cB_{\sg^n\om}}$.

For item \eqref{harry assum 2 hat} we begin with the Lasota--Yorke inequality for $\var(\cdot)$ and $\nu_{\omega,0}(|\cdot|)$ provided by the final line of the proof of Lemma \ref{closed ly ineq App} (equation (\ref{ORLYLY})).
Dividing through by $\lm_{\om,0}^{n'}$ we obtain 
\begin{align}
\var(\hat\cL_{\omega,\epsilon}^{n'} f)
&\le 
\frac{9\|g_{\omega,0}^{({n'})}\|_{\infty,1}}{\lm_{\om,0}^{n'}}\var(f)
+\frac{8\|g_{\omega,0}^{(n')}\|_{\infty,1}}{\lm_{\om,0}^{n'}\min_{Z\in \mathcal{Z}_{\omega,*,\ep}^{(n')}(\mathcal{A})}\nu_{\omega,0}(Z)}\nu_{\omega,0}(|f|)
\nonumber\\
&\le 
\alpha^{n'}\var(f)+\frac{\al^{n'}}{\min_{Z\in \mathcal{Z}_{\omega,*,\ep}^{(n')}(\mathcal{A})}\nu_{\omega,0}(Z)}\nu_{\omega,0}(|f|).
\label{LYineqorly hat}
\end{align}
Note that $\essinf_\om\min_{Z\in\mathcal{Z}_{\omega,*,\ep}^{(n')}(\mathcal{A})}\nu_{\omega,0}(Z)>0$, since the uniform open covering assumption \eqref{E9} together with \eqref{E3}, \eqref{uniflbLeps}, 
and equivariance of the backward adjoint cocycle together imply that 
for $Z\in\cZ_{\om,*,\ep}^{(n')}$ we have
\begin{align}\label{LY LB calc}
\nu_{\om,0}(Z)
=
\nu_{\sg^{k_o(n')}\om,0}\left(\left(\lm_{\om,0}^{k_o(n')}\right)^{-1}\cL_{\om,0}^{k_o(n')}\ind_Z\right)
\geq
\frac{\inf g_{\om,0}^{k_o(n')}}{\lm_{\om,0}^{k_o(n')}}>0.
\end{align}
As the holes $H_{\om,\ep}$ are composed of finite unions of disjoint intervals, assumption \eqref{E6} implies that the radii of each of these intervals must go to zero as $\ep\to 0$. Thus, using \eqref{E6} together with Remark \ref{revision 1} and the fact that $\nu_{\om,0}$ is fully supported and non-atomic,
we see that \eqref{C6} must hold.

We construct a uniform Lasota--Yorke inequality for all $n$ in the usual way by using blocks of length $jn'$;  we write this as
\begin{equation}
\label{LYfull hat}
\var(\hat\cL_{\omega,\epsilon}^n f)\le A_1\alpha^n\var(f)+A_2^n\nu_{\omega,0}(|f|)
\end{equation}
for some $A_2>\alpha$.
We now wish to convert this to an inequality
\begin{equation}
\label{targetinequality hat}
\var(\hat\cL_{\omega,\epsilon}^n f)\le A_1'\alpha^n\var(f)+(A_2')^n\|f\|_{1}.
\end{equation}
In light of \eqref{E2}, \eqref{uniflbLeps}, \eqref{E1}, and \eqref{PF prop}, we see that there is a constant $B$ so that for $m$-a.e.\ $\omega\in\Om$ we have
\begin{align*}
\norm{\hat\cL_{\om,\ep}^{n'}f}_1
&=
\left(\lm_{\om,0}^{n'}\right)^{-1}\int_{[0,1]} \left|\sum_{y\in T_\om^{-n'}x} g_{\om,0}^{(n')}(y) \hat X_{\om,n'-1,\ep}(y)f(y)\right| d\Leb(x)
\\
&\leq
\frac{\norm{g_{\om,0}^{(n')} J_\om^{(n')}}_{\infty,1}}{\lm_{\om,0}^{n'}} \int_{[0,1]} \left|\sum_{y\in T_\om^{-n'}(x)} \frac{f(y)}{J_\om^{(n')}(y)} \right| d\Leb(x)
\\
&=
\frac{\norm{g_{\om,0}^{(n')} J_\om^{(n')}}_{\infty,1}}{\lm_{\om,0}^{n'}} \int_{[0,1]} \left| P_\om^{n'}(f) \right| d\Leb(x)
\\
&\leq
\frac{\norm{g_{\om,0}^{(n')} J_\om^{(n')}}_{\infty,1}}{\lm_{\om,0}^{n'}}\norm{f}_1
\leq B^{n'}\norm{f}_1.
\end{align*}
Using the non-atomicicty of $\nu_{\om,0}$ from \eqref{CCM} (shown in Lemma \ref{DFGTV18Alemma}) and the fact that $\var(|f|)\le \var(f)$, we may apply Lemma 5.2 \cite{BFGTM14} to $\nu_{\omega,0}$ to obtain that for each $\zeta>0$, there is a $B_\zeta<\infty$ such that $\nu_{\omega,0}(|f|)\le \zeta\var(f)+B_\zeta\|f\|_1$.
Now using (\ref{LYfull hat}) we see that
\begin{align*}
\|\hat{\mathcal{L}}_{\omega,\epsilon}^nf\|_{\BV_1}
&=\var(\hat{\mathcal{L}}_{\omega,\epsilon}^nf)+\|\hat{\mathcal{L}}_{\omega,\epsilon}^n(f)\|_1
\\
&\le A_1\alpha^n\var(f)+A_2^n\nu_{\omega,0}(|f|)+B^n\|f\|_1
\\
&\le(A_1\alpha^n+A_2^n\zeta)\var(f)+(B^n+A_2^nB_\zeta)\|f\|_1
\\
&\le(A_1\alpha^n+A_2^n\zeta)\|f\|_{\BV_1}+(B^n+A_2^n(B_\zeta-\zeta)-A_1\alpha^n))\|f\|_1.
\end{align*}
Selecting $\zeta$ sufficiently small and $n''$ sufficiently large so that $C(\alpha^{n''}+\zeta)<1$ we again (by proceeding in blocks of $n''$) arrive at a uniform Lasota--Yorke inequality of the form (\ref{targetinequality hat}) for all $n\ge 0$.

For item \eqref{harry assum 3 hat} we note that 
\begin{align*}
\vertiii{\hat\cL_{\omega,0}-\hat\cL_{\omega,\epsilon}}
&:=\sup_{\|f\|_{\BV_1}=1}\|(\hat\cL_{\omega,0}-\hat\cL_{\omega,\epsilon})f\|_1
=\sup_{\|f\|_{\BV_1}=1}\|\hat\cL_{\omega,0}(f\ind_{H_{\om,\ep}})\|_1
\le 
\|\hat\cL_{\omega,0}(\ind_{H_{\om,\ep}})\|_1
\\
&=
\lm_{\om,0}^{-1}\int_{[0,1]} \left|\sum_{y\in T_\om^{-1}(x)} g_{\om,0}(y) \ind_{H_{\om,\ep}}(y)\right| d\Leb(x)
\\
&\leq
\frac{\norm{g_{\om,0} J_\om}_{\infty,1}}{\lm_{\om,0}} \int_{[0,1]} \left|\sum_{y\in T_\om^{-1}(x)} \frac{\ind_{H_{\om,\ep}}(y)}{J_\om(y)} \right| d\Leb(x)
\\
&=
\frac{\norm{g_{\om,0} J_\om}_{\infty,1}}{\lm_{\om,0}}\int_{[0,1]} \left| P_\om(\ind_{H_{\om,\ep}}) \right| d\Leb(x)
\\
&\le 
\esssup_\omega \frac{\norm{g_{\om,0} J_\om}_{\infty,1}}{\lm_{\om,0}}\cdot\esssup_\omega\mathrm{Leb}(H_{\omega,\epsilon}).
\end{align*}
Because \eqref{E2}, \eqref{E1} and \eqref{uniflbLeps} imply $\esssup_\omega\sfrac{\norm{g_{\om,0} J_\om}_{\infty,1}}{\lm_{\om,0}}<\infty$, and since \eqref{E6} implies that $\lim_{\epsilon\to 0}\esssup_\omega\mathrm{Leb}(H_{\omega,\epsilon})=0$, we obtain item \eqref{harry assum 3 hat}.

We may now  apply Theorem 4.4 \cite{C19} (Theorem \ref{HarryThm4.4}) to conclude that given $\delta>0$ there is an $\epsilon_0>0$ such that for all $\epsilon\le \epsilon_0$ the cocycle generated by $\hat\cL_\epsilon$ is hyperbolic, with 
\begin{enumerate}[i]
\item[\mylabel{i}{h1}] the existence of an equivariant family $\hat\phi_{\om,\ep}\in\BV$ with $\esssup_\omega\|\hat\phi_{\om,\ep}-\phi_{\om,0}\|_{\BV_1}<\delta$,
\item[\mylabel{ii}{h2}] existence of corresponding Lyapunov multipliers $\hat\lambda_{\om,\ep}$ satisfying $|\hat\lambda_{\om,\ep}-1|<\delta$,
\item[\mylabel{iii}{h3}] operators $\hat Q_{\om,\ep}$ satisfying $\|(\hat Q_{\om,\ep})^n\|_{\BV_1}\le K'(\gamma+\delta)^n$, where $\gamma$ is the decay rate for $Q_{\om,0}$ from the proof of Lemma \ref{DFGTV18Alemma}.
\end{enumerate}
To obtain an $\hat\cL_{\om,\ep}^*$-equivariant family of linear functionals $\hat\nu_{\omega,\epsilon}\in\BV_1^*$  we apply Corollary 2.5 \cite{dragicevic_spectral_2018}.
Using the one-dimensionality of the leading Oseledets space for the forward cocycle, this result shows
that the leading Oseledets space for the backward adjoint cocycle  is also one-dimensional.
This leading Oseledets space is spanned by some $\hat\nu_{\omega,\epsilon}\in \BV_1^*$, satisfying $\hat\nu_{\sigma\om,\ep}(\hat\cL_{\om,\ep}(f))=\hat\vta_{\om,\ep}\hat\nu_{\om,\ep}(f)$, for Lyapunov multipliers $\hat\vta_{\om,\ep}$.
By  Lemma 2.6 \cite{dragicevic_spectral_2018}, we may scale the $\hat\nu_{\om,\ep}$ so that $\hat\nu_{\om,\ep}(\hat\phi_{\om,\ep})=1$ for a.e.\ $\omega$.
We show that in fact $\hat\vta_{\om,\ep}=\hat\lambda_{\om,\ep}$ $m$-a.e.
Indeed, 
$$
1=\hat\nu_{\sg\om,\ep}(\hat\phi_{\sg\om,\ep})=\hat\nu_{\sg\om,\ep}(\hat\cL_{\om,\ep}\hat\phi_{\om,\ep}/\hat\lambda_{\om,\ep})=(\hat\vta_{\om,\ep}/\hat\lambda_{\om,\ep})\hat\nu_{\om,\ep}(\hat\phi_{\om,\ep})=\hat\vta_{\om,\ep}/\hat\lambda_{\om,\ep}.
$$
Note that $\cL_{\om,\ep} = \lm_{\om,0}\hat\cL_{\om,\ep}$, and we now define $\phi_{\om,\ep}, \lambda_{\om,\ep}, Q_{\om,\ep}$, and $\nu_{\om,\ep}$ by the following:
\begin{align*}
\phi_{\om,\ep}&:=\frac{1}{\nu_{\om,0}(\hat\phi_{\om,\ep})}\cdot \hat\phi_{\om,\ep},
&
\nu_{\om,\ep}(f)&:=\nu_{\om,0}(\hat\phi_{\om,\ep})\hat\nu_{\om,\ep}(f),
\\
\lm_{\om,\ep}&:=\lm_{\om,0}\frac{\nu_{\sg\om,0}(\hat\phi_{\sg\om,\ep})}{\nu_{\om,0}(\hat\phi_{\om,\ep})}\hat\lm_{\om,\ep},
&
Q_{\om,\ep}(f)&:=\frac{\nu_{\om,0}(\hat\phi_{\om,\ep})}{\nu_{\sg\om,0}(\hat\phi_{\sg\om,\ep})}\hat Q_{\om,\ep}(f)
\end{align*}
Clearly all of the properties of \eqref{C2} and \eqref{C3} are now satisfied except for the log-integrability of $\lambda_{\om,\ep}$ in \eqref{C2}.
To demonstrate this last point, we note that by uniform hyperbolicity of the perturbed cocycles, $\lambda'_{\om,\ep}$ are uniformly bounded below and are therefore log-integrable.
Since 
\begin{align*}
|1-\nu_{\omega,0}(\hat\phi_{\omega,\epsilon})|=|\nu_{\omega,0}(\phi_{\omega,0})-\nu_{\omega,0}(\hat\phi_{\omega,\epsilon})|\le \|\hat\phi_{\om,\ep}-\phi_{\om,0}\|_{\BV_1}<\delta,
\end{align*}
the $\lambda_{\omega,\epsilon}$ are uniformly small perturbations of the $\hat\lambda_{\omega,\epsilon}$, and since $\lm_{\om,0}$ is $\log$-integrable by \eqref{E2} and \eqref{uniflbLeps}, we must therefore have that the log integrability condition on $\lambda_{\om,\ep}$ in \eqref{C2} is satisfied. 
Point \eqref{h1} above, combined with the $\ep=0$ part of \eqref{C5'} (resp. \eqref{C7'}) and the uniform estimate for $|1-\nu_{\om,0}(\hat\phi_{\om,\ep})|$, immediately yields the $\ep>0$ part of \eqref{C5'} (resp. \eqref{C7'}).
Point \eqref{h3} above combined with the same estimates also ensures that the norm of $\|Q_{\om,\ep}^n\|_{\infty,1}$ decays exponentially fast, uniformly in $\om$ and $\ep$, satisfying the stronger exponential version of \eqref{C4'}. In fact point \eqref{h3} implies the stronger statement that $\|Q_{\om,\ep}^n\|_{\BV_1}$ decays exponentially fast, uniformly in $\om$ and $\ep$.

Finally we show that $\nu_{\om,\ep}:C^0([0,1])\to\mathbb{C}$ is a positive linear functional with $\nu_{\om,\ep}(f)\in \mathbb{R}$ if $f$ is real.
From this fact it will follow by Riesz-Markov (e.g.\ Theorem A.3.11 \cite{viana_oliveira}) that $\nu_{\om,\ep}$ can be identified with a real finite Borel measure on $[0,1]$.
By linearity we may consider the two cases: (i) $f=\phi_{\om,\ep}>0$ (the generator of the leading Oseledets space) and (ii) $f\in F_{\omega,\ep}$, where $F_{\omega,\ep}$ is the Oseledets space complementary to $\mathrm{span\{\phi_{\om,\ep}\}}$.
In case (i)  $\nu_{\om,\ep}(\phi_{\om,\ep})=1>0$.
In case (ii), Lemma 2.6 \cite{dragicevic_spectral_2018} implies $\nu_{\om,\ep}(f)=0$. 
In summary we see that $\nu_{\om,\ep}$ is positive.
\end{proof}

\begin{remark}
As we have just shown that assumptions \eqref{C1'}, \eqref{C2}, \eqref{C3}, \eqref{C4'}, \eqref{C5'}, \eqref{C6}, \eqref{C7'} (Lemmas~\ref{DFGTV18Alemma} and \ref{harrylemma2}), we see that Proposition~\ref{prop: escape rates} holds as well as Theorem~\ref{thm: dynamics perturb thm} for the random open system $(\Om, m, \sg, [0,1], T, \BV, \cL_0, \nu_0, \phi_0, H_\ep)$ under the additional assumption of \eqref{C8}.
In light of Remark~\ref{rem checking esc cor cond}, if we assume \eqref{xibound} in addition to \eqref{C8} then both Corollary~\ref{esc rat cor} and Theorem~\ref{evtthm} apply. 
\end{remark}

The following theorem is the main result of this section and elaborates on the dynamical significance of the perturbed objects produced in Lemma~\ref{harrylemma2}.
\begin{theorem}\label{EXISTENCE THEOREM}
Suppose $(\Om, m, \sg, [0,1], T, \cB, \cL_0, \nu_0, \phi_0, H_\ep)$ is a random open system and that the assumptions of Lemma~\ref{harrylemma2} hold. 
Then there exists $\ep_0>0$ sufficiently small such that for every $0\leq\ep<\ep_0$ we have the following:
\begin{enumerate}
\item\label{item 1}
There exists a unique random probability measure $\zt_\ep=\set{\zt_{\om,\ep}}_{\om\in\Om}$ on $[0,1]$ such that, for $\ep>0$, $\zt_{\om,\ep}$ is supported in $X_{\om,\infty,\ep}$ and 
\begin{align*}
\zt_{\sg\om,\ep}(\cL_{\om,\ep} f)=\rho_{\om,\ep}\zt_{\om,\ep}(f),
\end{align*}
for $m$-a.e. $\om\in\Om$ and each $f\in\BV$, where 
\begin{align*}
\rho_{\om,\ep}:=\zt_{\sg\om,\ep}(\cL_{\om,\ep}\ind).
\end{align*} 	
Furthermore, for $\ep=0$ we have $\zt_{\om,0}=\nu_{\om,0}$ and $\rho_{\om,0} = \lm_{\om,0}$ and for $\ep>0$ we have that $C_1^{-1}\leq \rho_{\om,\ep}\leq C_1$ for $m$-a.e. $\om\in\Om$. 

\

\item\label{item 2}
There exists a measurable function $\psi_\ep:\Om\times[0,1]\to(0,\infty)$ 
such that $\zt_{\om,\ep}(\psi_{\om,\ep})=1$ and  
\begin{align*}
\cL_{\om,\ep} \psi_{\om,\ep}=\rho_{\om,\ep} \psi_{\sg\om,\ep}
\end{align*}
for $m$-a.e. $\om\in\Om$.
Moreover, $\psi_\ep$ is unique modulo $\zt_\ep$, and there exists $C\geq 1$ such that $C^{-1}\leq \psi_{\om,\ep}\leq C$ for $m$-a.e. $\om\in\Om$. Furthermore, for $m$-a.e. $\om\in\Om$ we have that $\rho_{\om,\ep}\to\rho_{\om,0}=\lm_{\om,0}$ and 
$\psi_{\om,\ep}\to \psi_{\om,0}=\phi_{\om,0}$ (in $\cB_\om$) as $\ep\to 0$, where $\phi_{\om,0}$ and $\lm_{\om,0}$ are defined in Lemma~\ref{DFGTV18Alemma}.

\

\item\label{item 3} 
The random measure $\mu_\ep=\set{\mu_{\om,\ep}:=\psi_{\om,\ep}\zt_{\om,\ep}}_{\om\in\Om}$ is a $T$-invariant and ergodic random probability measure whose fiberwise support, for $\ep>0$, is contained in $X_{\om,\infty,\ep}$. 
Furthermore, $\mu_\ep$ is the unique relative equilibrium state, i.e. 
\begin{align*}
\int_\Om\log\rho_{\om,\ep}\ dm(\om)&=:\cEP_\ep(\log g_0)
=
h_{\mu_\ep}(T)+\int_{\cJ_0} \log g_0 \,d{\mu_\ep}
=
\sup_{\eta_\ep\in\cP_{T,m}^{H_\ep}(\cJ_0)} \lt(h_{\eta_\ep}(T)+\int_{\cJ_0} \log g_0 \,d\eta_\ep\rt),
\end{align*}
where $h_{\eta_\ep}(T)$ is the entropy of the measure $\eta_\ep$, $\cEP_\ep(\log g_0)$ is the expected pressure of the weight function $g_0=\set{g_{\om,0}}_{\om\in\Om}$, 
and $\cP_{T,m}^{H_\ep}(\cJ_0)$ denotes the collection of $T$-invariant random probability measures $\eta_\ep$ on $\cJ_0$ whose disintegration $\{\eta_{\om,\ep}\}_{\om\in\Om}$ satisfies $\eta_{\om,\ep}(H_{\om,\ep})=0$ for $m$-a.e. $\om\in\Om$. Furthermore, $\lim_{\ep\to 0}\cEP_\ep(\log g_0)=\cEP_0(\log g_0):=\int_\Om\log \lm_{\om,0}\,dm.$

\

\item\label{item 4}
For $\ep>0$, let $\vrho_\ep=\set{\varrho_{\om,\ep}}_{\om\in\Om}$ be the random probability measure with fiberwise support in $[0,1]\bs H_{\om,\ep}$ whose disintegrations are given by
\begin{align*}
\varrho_{\om,\ep}(f):=\frac{\nu_{\om,0}\lt(\ind_{H_{\om,\ep}^c} \psi_{\om,\ep} f\rt)}{\nu_{\om,0}\lt(\ind_{H_{\om,\ep}^c} \psi_{\om,\ep}\rt)}
\end{align*} 
for all $f\in\BV$.  $\vrho_{\om,\ep}$ is the unique random conditionally invariant probability measure that is absolutely continuous (with respect to $\set{\nu_{\om,0}}_{\om\in\Om}$) with density of bounded variation.

\

\item\label{item 5}
For each $f\in\BV$ there exists $D>0$ and $\kp_\ep\in(0,1)$ such that for $m$-a.e. $\om\in\Om$ and all $n\in\NN$ we have
\begin{align*}
\norm{\lt(\rho_{\om,\ep}^n\rt)^{-1}\cL_{\om,\ep}^n f - \zt_{\om,\ep}(f)\psi_{\sg^n\om,\ep}}_{\cB_{\sg^n\om}}\leq D\norm{f}_{\cB_\om}\kp_\ep^n.
\end{align*}
Furthermore, for all $A\in\sB$ and $f\in\BV$ we have 
\begin{align*}
\absval{
\nu_{\om,0}\lt(T_\om^{-n}(A)\,\rvert\, X_{\om,n,\ep}\rt)
-
\varrho_{\sg^n\om,\ep}(A)	
}
\leq 
D\kp_\ep^n,
\end{align*}
and 
\begin{align*}
\absval{
\frac{\varrho_{\om,\ep}\lt(f\rvert_{X_{\om,n,\ep}}\rt)}
{\varrho_{\om,\ep}(X_{\om,n,\ep})}
-
\mu_{\om,\ep}(f)	
}
\leq 
D\norm{f}_{\cB_\om}\kp_\ep^n	.		
\end{align*}
In addition, we have $\lim_{\ep\to 0}\kp_\ep=\kp_0$, where $\kp_0$ is defined in Lemma~\ref{DFGTV18Alemma}.

\

\item\label{ITEM 6}
There exists $C>0$ such that for every $f,h\in \BV$,
every $n\in\NN$ sufficiently large, and for $m$-a.e. $\om\in\Om$ we have 
\begin{align*}
\absval{
\mu_{\om,\ep}
\lt(\lt(f\circ T_{\om}^n\rt)h \rt)
-
\mu_{\sg^{n}\om,\ep}(f)\mu_{\om,\ep}(h)
}
\leq C
\norm{f}_{\infty,\om}\norm{h}_{\cB_\om}\kp_\ep^n.
\end{align*} 

\end{enumerate}

\end{theorem}
\begin{proof}

First we note that the claims of items \eqref{item 1} -- \eqref{item 3} and \eqref{item 5} -- \eqref{ITEM 6} above for $\ep=0$ follow immediately from Lemma~\ref{DFGTV18Alemma}. Now we are left to prove each of the claims for $\ep>0$.

Claims \eqref{item 1} -- \eqref{item 3} follow from Lemma~\ref{harrylemma2} with the scaling:
\begin{align*}
\psi_{\om,\ep}&:=\nu_{\om,\ep}(\phi_{\om,\ep}) \phi_{\om,\ep},
&
\zt_{\om,\ep}(f)&:=\frac{\nu_{\om,\ep}(f)}{\nu_{\om,\ep}(\ind)},
&
\rho_{\om,\ep}&:=\frac{\nu_{\om,\ep}(\ind)}{\nu_{\sg\om,\ep}(\ind)}\lm_{\om,\ep}.
\end{align*}
The uniform boundedness on $\rho_{\om,\ep}$ and $\psi_{\om,\ep}$ follows from the uniform boundedness on $\lm_{\om,\ep}$ and $\phi_{\om,\ep}$ coming from Lemma~\ref{harrylemma2}.
The fact that $\mu_\ep$ is the unique relative equilibrium state follows similarly to the proof of Theorem 2.23 in \cite{AFGTV20} (see also Remark 2.24, Lemma 12.2 and Lemma 12.3).
The claim that $\supp(\zt_{\om,\ep})\sub X_{\om,\infty,\ep}$ follows similarly to Lemma \ref{lem: Lm is conf meas}. Noting that $\|f\|_{\infty,\om}=\|f\|_{\infty,1}$ for $f\in\BV$ by the proof of Proposition \ref{prop norm equiv}, we now proof Claims \eqref{item 4} -- \eqref{ITEM 6}.

Claim \eqref{item 4} follows from Lemma \ref{LMD lem 1.1.1}  and the uniqueness of the density $\psi_{\om,\ep}\in\BV$.

The first item of Claim \eqref{item 5} follows from Lemma \ref{harrylemma2} and the remaining items are proved similarly to Corollary \ref{cor: exp conv of eta}.

Claim \eqref{ITEM 6} 
follows from Claim \eqref{item 5} and is proven in Appendix \ref{appDec}.
\end{proof}
\begin{remark}\label{rem seq exist}
If one considers a two-sided (bi-infinite) sequential analogue of random open systems, then because Theorem 4.4 \cite{C19} (Theorem \ref{HarryThm4.4}) also applies to two-sided sequential systems,  one could prove similar results to items \eqref{item 1}, \eqref{item 2}, \eqref{item 4}, \eqref{item 5}, and \eqref{ITEM 6} of Theorem \ref{EXISTENCE THEOREM}.
\end{remark}

\section{Limit theorems}\label{sec: limit theorems}
In this section we prove  a few  limit theorems   for the {\em closed} systems  discussed in Section~\ref{sec: existence} $(\Om, m, \sg, [0,1], T, \BV, \cL_0, \nu_0, \phi_0)$. We will in fact show that such systems   are  {\em admissible} in the sense of \cite{dragicevic_spectral_2018}. This will allow us to adapt  to our setting the spectral approach {\em \`a la Nagaev-Guivarc'h} developed in \cite{dragicevic_spectral_2018} and get a quenched  central limit theorem, a quenched large deviation theorem and a quenched local central limit theorem. We will also present an alternative approach based on martingale techniques \cite{DFGTV18A,atnipASIP,DHasip}, which produces an  almost sure invariance principle (ASIP) for the random measure $\mu_0=\set{\mu_{\om,0}}_{\om\in\Om}$. Moreover, the ASIP implies that  $\mu_{0}$ satisfies the central limit theorem as well as the law of the iterated logarithm. The martingale approach will also give an upper bound for any (large) deviation from the expected value and a Borel-Cantelli dynamical lemma. At the moment we could not extend the previous limit theorems to the {\em open} systems investigated in Section~\ref{sec: existence}. There are  a few  reasons for that which concern the Banach space  $\mathcal{B}_{\om,\ep}$ associated to those systems and  defined by the norm: $||\cdot||_{\mathcal{B}_{\om,\ep}}=\text{var}(\cdot)+\zeta_{\om,\ep}(|\cdot|).$ First of all,  we do not know if the  random cocycle $\mathcal{R}_{\eps}=(\Omega, m, \sigma, \mathcal{B}_{\om,\ep}, \tcL_{\om, \epsilon})$ is quasi-compact which is an essential requirement for admissibility. 
Second, the results of Theorem \ref{EXISTENCE THEOREM} are not particularly compatible with the Banach spaces  $\mathcal{B}_{\om,\ep}$, $\ep>0$, since the inequalities of items \eqref{item 5} and \eqref{ITEM 6} are in terms of the norms $\|\spot\|_{\infty,\om}$ and $\|\spot\|_{\cB_\om}$ which are defined modulo $\nu_{\om,0}$ and the $\mathcal{B}_{\om,\ep}$ norm is defined via $\zt_{\om,\ep}$, a measure which is supported on a $\nu_{\om,0}$-null set.

\subsection {The Nagaev-Guivarc'h approach}\label{NG approach} The paper \cite{dragicevic_spectral_2018} developed a general scheme to adapt  the  Nagaev-Guivarc'h approach to random quenched dynamical systems,  allowing one to prove limit theorems by exploring the connection between a {\em twisted} operator cocycle and the distribution of the Birkhoff sums. The results in \cite{dragicevic_spectral_2018} were confined to the geometric potential $|\det(DT_{\om})|^{-1}$ and the associated conformal measure, Lebesgue measure. We now show how  to extend those results to the systems verifying the assumptions stated in Section \ref{sec: existence} and the results of Theorem \ref{EXISTENCE THEOREM}, whenever $\epsilon=0,$ that is we will consider random closed systems  for a larger class of potentials. 

The starting point is to replace the linear operator $\cL_{\om,0}$ associated to the geometric potential and the (conformal) Lebesgue measure introduced in \cite{dragicevic_spectral_2018}, with our operator $\cL_{\om, 0}$ and the associated conformal measures $\nu_{\om, 0}.$ In particular, if we work with the normalized operator $\tcL_{\om, 0} := \lm_{\om,0}^{-1}\cL_{\om,0}$ the results in \cite{dragicevic_spectral_2018} are reproducible almost verbatim with a few precautions which we are going to explain. 
We now let $\mathcal{B}_{\om}$ denote the Banach space of complex-valued bounded variation functions with norm defined by $||\cdot||_{\mathcal{B}_{\om}}=\text{var}(\cdot)+\nu_{\om,0}(|\cdot|),$ where the variation is defined using equivalence classes mod-$\nu_{\om, 0}.$ 
In order to apply the theory in \cite{dragicevic_spectral_2018} we must show that our random cocycle is {\em admissible}. This  reduces to check two sets of properties which were listed in  \cite{dragicevic_spectral_2018} respectively as conditions (V1) to (V9) and conditions  (C0) up to (C4).  The first set of conditions reproduces the classical properties of the total variation of a function and its relationship with the $L^1(\Leb)$ norm. We should emphasize that in our case the variation is defined using equivalence classes mod-$\nu_{\om, 0}.$ 
It is easy to  check that properties (V1, V2, V3, V5, V8, V9) hold  with respect to this variation. In particular (V3) asserts that for any $f\in \mathcal{B}_{\om}$ we have $||f||_{L^{\infty}(\nu_{\om, 0})}\le ||f||_{\mathcal{B}_{\om}};$ we will refer to it in the following just as the (V3) property.
{\em Notation}: recall that given an element $f\in \mathcal{B}_{\om}$, the $L^{\infty}$ norm of $f$ with respect to the measure $\nu_{\om, 0}$ is denoted by $||f||_{\infty,\om}$ in the rest of this section.

Property (V7) is not used in the current paper; (V6) is a general density embedding result proved in  Hofbauer and Keller (Lemma 5 \cite{HHKK}).
We elaborate on property (V4). To obtain (V4) in our situation one needs to prove that  the unit ball of $\mathcal{B}_{\om}$ is compactly injected  into $L^{1}(\nu_{\om, 0})$. As we will see, this is used  to get the quasi-compactness of the random cocycle. The result follows easily by adapting Proposition 2.3.4 in \cite{gora} to our conformal measure $\nu_{\om, 0},$ which is fully supported and non-atomic. 
We now rename the other set of  properties (C0)--(C4) in \cite{dragicevic_spectral_2018} as  ($\mathcal{C}$0) to ($\mathcal{C}$4) to distinguish them from our (C) properties stated earlier in Sections \ref{sec:goodrandom} and \ref{EEVV}.  \begin{itemize}
\item 
Assumption ($\mathcal{C}$0) coincides with our condition \eqref{M}.
\item 
Condition ($\mathcal{C}$1)  requires us to prove in our case that \begin{equation}\label{C1CMP}
||\tilde{\mathcal{L}}_{\om,0}f||_{\mathcal{B}_{\sigma\om}}\le K ||f||_{\mathcal{B}_{\om}},
\end{equation}
for every $f\in\mathcal{B}_{\om}$ and for $m$-a.e. $\omega$, with $\om$-independent $K$.
\item 
Condition ($\mathcal{C}$2) asks that there exists $N\in \mathbb{N}$ and measurable $\tilde{\alpha}^N, \tilde{\beta}^N:\Omega\rightarrow (0, \infty),$ with $\int_{\Omega}\log\tilde{\alpha}^N(\om)d m(\om)<0,$ such that for every $f\in \mathcal{B}_{\om}$ and $m$-a.e. $\om\in \Omega,$
\begin{equation}\label{C2CMP}
||\tilde{\mathcal{L}}^N_{\om,0}f||_{\mathcal{B}_{\sigma^N\om}}\le \tilde{\alpha}^N(\om)||f||_{\mathcal{B}_{\om}}+\tilde{\beta}^N(\om)||f||_{L^1(\nu_{\om, 0})}.
\end{equation}
\item Condition ($\mathcal{C}$3)  is  the content of the first display equation in item  \ref{item 5} in the statement of Theorem \ref{EXISTENCE THEOREM}. 
\item Condition ($\mathcal{C}$4) is only used to obtain Lemma 2.11 in \cite{dragicevic_spectral_2018}. There are three results in that Lemma which we now compare with our situation. The third result is the decay of correlations stated in item \ref{ITEM 6} of Theorem \ref{EXISTENCE THEOREM}. The second result  is the almost-sure strictly positive lower bound for the density $\phi_{\om, 0}$ stated in item \ref{item 2} of Theorem \ref{EXISTENCE THEOREM}. The first result requires that $\esssup_{\om\in \Omega} ||\phi_{\om,0}||_{\mathcal{B}_{\om}}<\infty.$ This follows by checking that the proof of Proposition 1 in \cite{DFGTV18A} works in our current setting with the obvious modifications;  we note that Proposition 1 \cite{DFGTV18A} only assumes conditions  ($\mathcal{C}$1) and ($\mathcal{C}$3).
\end{itemize}
We are thus left with showing conditions ($\mathcal{C}$1) and ($\mathcal{C}$2) in our setting. We will get both at the same time as a consequence of the following argument, which consists in adapting to our current situation  the final part 
of Lemma \ref{closed ly ineq App}. Our starting point will be the inequality (\ref{closed var ineq over partition}). 
Since we are working with the normalized operator cocycle $\tilde{\mathcal{L}}_{\omega,0}^N$, we have to divide the expression in  (\ref{closed var ineq over partition}) by $\lm^{n'}_{\om, 0};$ moreover  we have to replace the measure $\nu_{\om, 0}$  with the equivalent conformal probability measure $\nu_{\om, 0}.$
Therefore (\ref{closed var ineq over partition}) now becomes
\begin{equation}\label{zero}
(\lm_{\om, 0}^n)^{-1}\var\lt(\ind_{T_\om^n(Z)} \lt((f g_{\om, 0}^{(n)})\circ T_{\om,Z}^{-n}\rt)\rt)
\le 9\frac{\norm{g_{\om, 0}^{(n)}}_{\infty,\om}}{\lm^n_{\om, 0 }}\text{var}_Z(f)+
\frac{8\norm{g_{\om, 0}^{(n)}}_{\infty,\om}}
{\lm^n_{\om, 0}\nu_{\om, 0}(Z)}\nu_{\om, 0}(|f|_Z),
\end{equation}
where $Z$ is an element of $\mathcal{Z}^{(n)}_{\om, 0}.$ Since each  element  of the partition  $\mathcal{Z}^{(n)}_{\om, 0}$ has nonempty interior and the measure $\nu_{\om, 0}$  charges open intervals, if we set
$$
\overline{\nu}^{(n)}_{\om,0}:=\min_{Z\in\cZ_{\om,0}^{(n)}}
\nu_{\om,0}(Z)>0,
$$
and we take the sum over the $Z\in \mathcal{Z}^{(n)}_{\om, 0}$ we finally get
\begin{equation}\label{p}
\text{var}(\tilde{\mathcal{L}}^n_{\om,0}f)\le 9\frac{\norm{g_{\om,0}^{(n)}}_{\infty,\om}}{\lm^{n}_{\om, 0}}\text{var}(f)+8\frac{\norm{g_{\om,0}^{(n)}}_{\infty,\om}\nu_{\om,0}(|f|)}{\lm^n_{\om, 0}\overline{\nu}_{\om,0}^{(n)}}.
\end{equation}Notice that  condition ($\mathcal{C}$2) requires that  $\tilde{\alpha}^N<1$, which in our case becomes
\begin{equation}\label{fgr}
9\frac{\esssup_{\om}\norm{g_{\om,0}^{(n)}}_{\infty,\om}}{\essinf_{\om}\lm^{n}_{\om, 0}}<1, 
\end{equation}
or equivalently 
$$
9\frac{\esssup_{\om}\norm{g_{\om,0}^{(n)}}_{\infty,\om}}{\essinf_{\om}\inf\mathcal{L}^n_{\om,0}\ind }<1,
$$
which is guaranteed by (E8).
We now move on to check condition ($\mathcal{C}$1).
From (\ref{p}), setting $n=1$ we  see that a sufficient condition for ($\mathcal{C}$1) is
\begin{equation}\label{dd}
\frac{\esssup_{\om}\norm{g_{\om, 0}^{(1)}}_{\infty,\om}}{\essinf_{\om}\lm^1_{\om, 0}\overline{\nu}_{\om,0}}<\infty.
\end{equation}
or equivalently
\begin{equation}\label{d}
\frac{\esssup_{\om}\norm{g_{\om,0}^{(1)}}_{\infty,\om}}{\essinf_{\om}\inf\mathcal{L}_{\omega,0}\ind\overline{\nu}_{\om,0}}<\infty
\end{equation}
Condition (\ref{d}) is in principle checkable ($n=1$) and we could {\em assume} it as a part of the admissibility condition for our random cocycle $\mathcal{R}=(\Omega, m, \sigma, \mathcal{B}_{\om}, \tcL_{\om,0}).$  

The admissibility conditions stated in \cite{dragicevic_spectral_2018}, in particular ($\mathcal{C}$2) and  ($\mathcal{C}$3), were sufficient to prove the quasi-compactness of the random cocycle introduced in  \cite{dragicevic_spectral_2018},  but they rely on another assumption, which was a part of the Banach space construction, namely that  $\BV_1$ was compactly injected into $L^1(\text{Leb}).$ We saw above that the same result holds for our Banach space $\mathcal{B}_{\om}$ and our measure $\nu_{\om, 0}.$ In order to prove quasi-compactness following Lemma 2.1 in \cite{dragicevic_spectral_2018}, we use condition  
($\mathcal{C}$2) with the almost-sure bound  $\tilde{\alpha}^N(\om)<1.$
We must additionally prove that the top Lyapunov exponent\index{Lyapunov exponent} $\Lambda(\mathcal{R})$\index{$\Lambda(\mathcal{R})$} defined by the limit 
$\Lambda(\mathcal{R})=\lim_{n\rightarrow \infty}\frac1n \log||\tcL_{\om,0}^n||_{\mathcal{B}_{\om}}$ is not less than zero.
This will follow by applying item \ref{item 5} and the bound on the density $\phi_{\om,0}$ in item \ref{item 2}, both in Theorem \ref{EXISTENCE THEOREM}.
Since the density $\phi_{\om, 0}$ is for $m$-a.e. $\om\in \Omega$ bounded from below by the constant $C^{-1},$ we have that $\limsup_{n\rightarrow \infty}\frac{1}{n}\log \|\phi_{\sigma^n\om,0}\|_{\mathcal{B}_{\sg^n\om}}\ge\limsup_{n\rightarrow \infty}\frac{1}{n}\log C^{-1}=0$.
Now we continue as in the proof of Theorem 3.2 in \cite{dragicevic_spectral_2018}. Let $N$ be the integer for which 
$\tilde{\alpha}^N(\om)<1$  $\om$ a.s. We then consider the cocycle $\mathcal{R}_N$  generated by the map $\om\rightarrow \cL_{\om,0}^N.$ It is easy to verify that
$\Lambda(\mathcal{R}_N)=N\Lambda(\mathcal{R})$ and the index of compactness
$\kappa$ (see section 2.1 in \cite{dragicevic_spectral_2018} for the definition) 
of the two cocycles also satisfies
$\kappa(\mathcal{R}_N)=N\kappa(\mathcal{R}).$ Because $\int \log \tilde{\alpha}^N(\om)\ 
dm < 0 \le \Lambda(\mathcal{R}_N)$, Lemma 2.1 in \cite{dragicevic_spectral_2018} guarantees that $\kappa(\mathcal{R}_N) \le \int \log \tilde{\alpha}^N(\om)\ 
dm$, 
proving
quasicompactness of $\mathcal{R}.$
\begin{remark}
    We just checked that assumptions (V1)--(V9) and ($\cC$0)--($\cC$4) of \cite{dragicevic_spectral_2018} continue to hold for fiberwise Banach spaces. The next step is to invoke the implicit function theorem of Sections 3.3 and 3.4 in \cite{dragicevic_spectral_2018}, which instead uses the same Banach space. We believe that such a theorem could be adapted to the setting of fiberwise norms and we plan to develop it in the near future. In the meanwhile we adapt to that theorem by taking $\|\cdot\|_{\cB_\om}=\|\cdot\|_{\BV_1}$ for each $\oio$ to be the same equivalent norm on each fiber, where $\|\cdot\|_{\BV_1}=\var(\cdot)+\|\cdot\|_{L^1(\Leb)}$ as in Section \ref{sec: existence}. Notice however that in Section \ref{sec: martingale} we will adopt the martingale approach, and in this case there is no obstruction to using fiberwise norms. 
\end{remark}
As we said at the beginning of this section, we could now follow almost verbatim the proofs in \cite{dragicevic_spectral_2018} by using  our operators $\cL_{\om,0}$ and by replacing the Lebesgue measure with  the conformal measures $\nu_{\om, 0}.$ We briefly sketch the main steps of the approach; we first define the  {\em observable} $v$ as a measurable map $v:\Omega\times [0,1] \rightarrow \mathbb{R}$ with the additional properties 
that $||v(\omega, x)||_{L^1(\nu_0)}<\infty$ and 
$\text{ess}\sup_{\om}||v_{\om}||_{\infty,\om}<\infty,$ where we set  $v_{\om}(x)=v(\om, x).$ 
Moreover we assume $v_\omega$ is fibrewise centered: $\int v_{\om}(x)d\mu_{\om,0}(x)=0,$ for $m$-a.e. $\om\in \Omega.$ We then define the twisted (normalized) transfer operator cocycle as:
$$
\tcL_{\om,0}^{\theta}(f)=\tcL_{\om,0}(e^{\theta v_{\om}(\cdot)}f), \ f\in\mathcal{B}_{\om}. 
$$
The link with the limit theorems is provided by the following equality which follows easily by the duality property of the operator:
$$
\int \tcL_{\om,0}^{\theta,n}(f)d\nu_{\sigma^n\om, 0}=\int e^{\theta S_n v_{\om}(\cdot)}f d\nu_{\om, 0},
$$
where $\tcL_{\om,0}^{\theta,n}=\tcL_{\sigma^{n-1}\om,0}^{\theta}\circ \cdots \circ \tcL_{\om,0}^{\theta}.$ The adaptation of Theorem 3.12 in \cite{dragicevic_spectral_2018} to our case, shows that the twisted operator is quasi-compact (in $\mathcal{B}_{\om}$) for $\theta$ close to zero. Moreover, by denoting with 
$$
\Lambda(\theta)=\lim_{n\rightarrow \infty}\frac1n\log ||\tcL_{\om,0}^{\theta,n}||_{\mathcal{B}_{\om}}
$$ 
the top Lyapunov exponent of the cocycle, the map $\theta\rightarrow \Lambda(\theta)$ is of class $C^2$ and strictly convex in a neighborhood of zero. We have also the analog of Lemma 4.3 in \cite{dragicevic_spectral_2018}, linking the asymptotic behavior of characteristic functions associated to Birkhoff sums with $\Lambda(\theta):$
$$
\lim_{n\rightarrow \infty}\frac1n\log\left|\int e^{\theta S_nv_{\om}(x)}d\mu_{\om, 0}\right|=\Lambda(\theta).
$$
We are now ready to collect our results on a few limit theorems; we first define the variance 
$$
\Sigma^2=\int v_{\om}(x)^2\ d\mu_{\om, 0}\ dm+2\sum_{n=1}^{\infty}\int v_{\om}(x)v(T_\om^n(x))\ d\mu_{\om, 0}\ dm.
$$
We also define the {\em aperiodicity condition} by asking that for $m$-a.e. $\om\in \Omega$ and for every compact interval $J\subset \mathbb{R}\setminus \{0\},$ there exists $C(\om)>0$ and $\rho\in(0,1)$ such that $||\tcL_{\om,0}^{it,n}||_{\mathcal{B}_{\om}}\le C(\omega)  \rho^n,$ for $t\in J$ and $n\ge 0.$
\begin{theorem}
Suppose that our random  cocycle $\mathcal{R}=(\Omega, m, \sigma, \mathcal{B}_{\om}, \tcL_{\om,0})$ is admissible and take the centered observable $v$ verifying  $\text{ess}\sup_{\om}||v_{\om}||_{\infty,\om}<\infty.$ Then:
\begin{itemize}
\item ({\bf Large deviations}). There exists $\varkappa_0>0$ and a non-random function $c:(-\varkappa_0, \varkappa_0)\rightarrow\mathbb{R}$ which is nonnegative, continuous, strictly convex, vanishing only at $0$ and such that
$$
\lim_{n\rightarrow \infty} \frac1n\log \mu_{\om, 0}(S_nv_{\om}(\cdot)>n\varkappa)=-c(\varkappa), \ \text{for}\ 0<\varkappa<\varkappa_0, \ \text{and}\ m-\text{a.e.} \ \om\in \Omega.
$$
\item ({\bf Central limit theorem}). Assume that $\Sigma^2>0.$ Then for every bounded and continuous function $\phi:\mathbb{R}\rightarrow \mathbb{R}$ and $m$-\text{a.e.}  $\om\in \Omega$, we have
$$
\lim_{n\rightarrow \infty}\int \phi\left(\frac{S_nv_{\om}(x)}{\sqrt{n}}\right)d\mu_{\om, 0}=\int \phi\ d\mathcal{N}(0, \Sigma^2).
$$
\item ({\bf Local central limit theorem}). Suppose the aperiodicity condition holds. Then for $m$-a.e. $\om\in \Omega$  and every bounded interval $J\subset \mathbb{R},$ we have
$$
\lim_{n\rightarrow \infty}\sup_{s\in \mathbb{R}}\left|\Sigma\sqrt{n}\mu_{\om,0}(s+S_nv_{\om}(\cdot)\in J)-\frac{1}{\sqrt{2\pi}}e^{-\frac{s^2}{2n\Sigma^2}}|J|\right|=0.
$$

\end{itemize}
\end{theorem}
\subsection{The martingale approach}\label{sec: martingale}
The  deviation result quoted above allows to control, asymptotically, deviations of order $\varkappa,$ for $\varkappa$ in a sufficiently small bounded interval around $0.$ We now show how to extend that result to   any $\varkappa$ by getting an exponential bound on the deviation of the distribution function instead of an asymptotic expansion; in particular our bound shows that deviation probability stays small for finite $n$,which is a typical concentration property.  
We now derive this result since it will provide us with the martingale that is used to obtain the ASIP. 
Recall that the equivariant measure $\mu_{\om,0}$ is equivalent to $\nu_{\om, 0}.$ We consider again the fibrewise centered observable $v$ from the previous section, and 
we wish to estimate
$$\mu_{\om,0}\left(\left|\frac1n\sum_{k=0}^{n-1}v_{\sigma^k\om}\circ
T^k_{\om}\right|>\varkappa\right).
$$
We will use the following result (Azuma, \cite{AZ}):
Let $\{M_i\}_{i\in \mathbb{N}}$ be a sequence of martingale differences.
If there is $a>0$ such that $||M_i||_{\infty}<a$ for all $i,$ then we have
for all $b\in \mathbb{R}:$
$$
\mu_{\om,0}\left(\sum_{i=1}^nM_i\ge nb\right)\le e^{-n\frac{b^2}{2a^2}}.
$$

If we denote  $\mathcal{F}^k_{\om}:=(T^k_{\om})^{-1}(\mathcal{F}),$
we can easily prove that for a measurable map
$\phi:[0,1]\rightarrow \mathbb{{R}},$ we have (the expectations $\mathbb{E}_{\om}$
will be
taken with respect to $\mu_{\om,0}$):
\begin{equation}\label{one}
\mathbb{E}_{\om}(\phi\circ T^l_{\o}|\mathcal{F}^n_{\om})=
\left(\frac{\tcL_{\sigma^l\om,0}^{n-l}(\lm_{\sigma^l\om,0}\phi)}{\lm_{\sigma^n\om,0}}\right)
\circ T_{\om}^n
\end{equation}
We now set:
$$M_n:= v_{\sigma^n\om}+G_n-G_{n+1}\circ T_{\sigma^n\om},$$ with $G_0=0$ and
\begin{equation}\label{two}
G_{n+1}=\frac{\tcL_{\sigma^n\om,0}(v_{\sigma^n\om}\lm_{\sigma^n\om,0}+G_n \lm_{\sigma^n\om,0})} {\lm_{\sigma^{k+1}\om,0}}
\end{equation}
It is easy to check that
$$
\mathbb{E}_{\om}(M_n\circ T_{\om}^n| \mathcal{F}^{n+1}_{\om} )=0,
$$
which means that the sequence $(M_n\circ T_{\om}^n)$ is a reversed martingale difference with respect to the filtration $\mathcal{F}^{n}_{\om}.$
By iterating (\ref{two}) we get
\begin{equation}\label{three}
G_n=\frac{1}{\lm_{\sigma^n\om,0}}\sum_{j=0}^{n-1}
\tcL^{(n-j)}_{\sigma^j\om,0}
(v_{\sigma^j\om}
\lm_{\sigma^j\om,0}),
\end{equation}
and by a telescopic trick we have
$$\sum_{k=0}^{n-1}M_k\circ T_{\om}^k=\sum_{k=0}^{n-1}v_{\sigma^k\om}T_{\om}^n-G_n\circ T_{\om}^n.$$
Suppose for the moment we could bound $G_n$ uniformly in $n$, but not necessarily in $\om$, by $||G_n||_{\infty,\om}\le C_1(\om).$ Since by assumption there exists a constant $C_2$ such that $\text{ess}\sup_{\om}||v_{\om}||_{\infty,\om}\le C_2,$  we have $||M_n||_{\infty,\om}\le C_2+2C_1({\om}),$ and by Azuma:
$$
\mu_{\om,0}\left(\left|\frac1n\sum_{k=0}^{n-1}M_k\circ
T^k_{\om}\right|>\frac{\varkappa}{2}\right)\le 2\exp\left\{-\frac{\varkappa^2}{8(C_2+2C_1({\om}))^2}n\right\}.
$$
Therefore
\begin{align*}
\mu_{\om,0}\left(\left|\frac1n\sum_{k=0}^{n-1}v_{\sigma^k\om}\circ
T^k_{\om}\right|>\varkappa\right)
&\le 
\mu_{\om,0}\left(\left|\frac1n\sum_{k=0}^{n-1}M_k\circ
T^k_{\om}\right|+\frac1n C_1(\om)>\varkappa\right)
\\
&\le
\mu_{\om,0}\left(\left|\frac1n\sum_{k=0}^{n-1}M_k\circ
T^k_{\om}\right|>\frac{\varkappa}{2}\right)
\\
&\le 
2\exp\left\{-\frac{\varkappa^2}{8(C_2+2C_1({\om}))^2}n\right\},
\end{align*}
provided $n>n_0,$ where $n_0$ verifies $\frac1n_0 C_1(\om)\le\frac{\varkappa}{2}.$\\
In order to estimate $C_1,$ we proceed in the following manner. We have
\begin{equation}\label{four}
||G_n||_{\infty}\le \frac{1}{\lm_{\sigma^n\om,0}}\sum_{j=0}^{n-1}
||\tcL^{(n-j)}_{\sigma^j\om,0}
(v_{\sigma^j\om}
\lm_{\sigma^j\om,0})||_{\infty,\om}.
\end{equation}
The multiplier $\lm_{\om,0}$ 
is bounded from above  and  from below $\om$ a.s. respectively by, say, $U$ and $1/U$ by  conditions $(\mathcal{C}1)$ and item \ref{item 1} Theorem \ref{EXISTENCE THEOREM}, respectively. Then we use item  \ref{item 5} of Theorem \ref{EXISTENCE THEOREM} and property (V3) to bound the term into the sum. In conclusion we get:

$$
||G_n||_{\infty,\om}\le U^2DC_2 \sum_{j=0}^{\infty} \kappa^{n-j}:=C_1.
$$
We summarize  this result in the following
\begin{theorem}({\bf Large deviations bound})\\
Suppose that our random  cocycle $\mathcal{R}=(\Omega, m, \sigma, \mathcal{B}_{\om}, \tcL_{\om,0})$ is admissible and take the fibrewise centered observable $v$ satisfying  $\text{ess}\sup_{\om}||v_{\om}||_{\infty,\om}=C_2<\infty.$ Then there exists $C_1>0$ such that for $m$-a.e. $\om\in \Omega$
and  for any $\varkappa>0:$
$$
\mu_{\om,0}\left(\left|\frac1n\sum_{k=0}^{n-1}v_{\sigma^k\om}\circ
T^k_{\om}\right|>\varkappa\right)\le2\exp\left\{-\frac{\varkappa^2}{8(C_2+2C_1)^2}n\right\},
$$
provided $n>n_0,$ where $n_0$ is any integer number satisfying $\frac{1}{n_0} C_1\le\frac{\varkappa}{2}.$  
\end{theorem}

We point out that whenever our random cocycle  $\mathcal{R}=(\Omega, m, \sigma, \mathcal{B}_{\om}, \tcL_{\om,0})$ is admissible, we satisfy all the assumptions in the paper \cite{DFGTV18A}, which allows us to get the almost sure invariance principle. We  should simply replace the (conformal) Lebesgue measure  with our random conformal measure $\nu_{\om, 0},$ and use our operators $\tcL_{\om,0},$ as we already did in Section \ref{NG approach} for the Nagaev-Guivarc'h approach. It is worthwhile to observe that \cite{DFGTV18A} relies on the construction of a suitable martingale, which in our case is precisely  the reversed martingale difference $M_n\circ T^n_{\om}$ obtained above in the proof of the large deviation bound. We denote with $\Sigma^2_n$ the variance $\Sigma^2_n=\mathbb{E}_{\om}\left(\sum_{k=0}^{\infty}v_{\sigma^k\om}\circ T_{\om}^k\right)^2,$ where $v_{\om}$ is a centered observable, and with $\Sigma^2$ the quantity introduced above, in the statement of the central limit theorem. We have the following:
\begin{theorem}({\bf Almost sure invariance principle})

Suppose that our random  cocycle $\mathcal{R}=(\Omega, m, \sigma, \mathcal{B}_{\om}, \tcL_{\om,0})$ is admissible and consider  a centered   observable $v_{\om}.$ Then $\lim_{n\rightarrow \infty}\frac1n\Sigma^2_n=\Sigma^2.$ Moreover one of the following cases hold:

(i) $\Sigma=0,$ and this is equivalent to the existence of $\psi\in L^2(\mu_0)$
$$
v=\psi-\psi\circ T,
$$
where $T$ is the induced skew product map and $\mu$ its invariant measure, see Definition \ref{def: random prob measures}.

(ii) $\Sigma^2>0$ and in this case for $m$ a.e. $\om\in \Omega,$ $\forall \varrho>\frac54$, by enlarging the probability space $(X, \mathcal{F}, \mu_{\om,0})$ if necessary, it is possible to find a sequence $(B_k)_k$ of independent centered Gaussian random variables such that
$$
\sup_{1\le k \le n} \left|\sum_{i=1}^k(v_{\sigma^i\om}\circ T_{\om}^i)-\sum_{i=1}^k B_i\right|=O(n^{1/4}\log^{\varrho}(n)),\ \mu_{\om,0} \ \text{a.s.}.
$$

\end{theorem}

We show in Section \ref{hts} that the distribution of the first hitting random time in a decreasing sequence of holes follows an exponential law, see Proposition \ref{hve}. 
We now prove a recurrence result by giving a quantitative estimate of the {\em asymptotic number of entry times in a descending sequence of balls}. This is known as the {\em shrinking target problem}. 
\begin{proposition}\label{shri}
Suppose that our random cocycle $\mathcal{R}=(\Omega, m, \sigma, \mathcal{B}_{\om}, \tcL_{\om,0})$ is admissible.
For each $\om$ let $B_{\om,\xi_{k,\om}}(p)$ be a  descending sequence of balls centered at the point $p$ and with radii $\xi_{k+1,\om}<\xi_{k,\om},$  such that 
$$
E_{\om, n}:=\sum_{k=0}^{n-1}\mu_{\sigma^k_\om,0}(B_{\sg^k\om,\xi_{k,\sigma^k\om}}(p))\rightarrow \infty 
\quad \mbox{ for $m$-a.e.\ $\omega$}.
$$
Then for $m$-a.e. $\om$ and $\mu_{\om,0}$-almost all $x,$ $T^k_{\om}(x)\in B_{\sg^k\om,\xi_{k,\sigma^k\om}}(p)$ for infinitely many $k$ and
$$
\frac{1}{E_{\om, n}}\#\{0\leq k< n: T^k_{\om}(x)\in B_{\sg^k\om,\xi_{k,\sigma^k\om}}(p)\}\rightarrow 1.
$$
\end{proposition}
The proposition is a simple consequence of the following {\em Borel-Cantelli} like property, whenever we put the observable $v$ in the theorem below as $v_{k,\sigma^k\om}=\ind_{B_{\sg^k\om,\xi_{k,\sigma^k\om}}(p)}.$
\begin{theorem}\label{BC}
Suppose that our random cocycle $\mathcal{R}=(\Omega, m, \sigma, \mathcal{B}_{\om}, \tcL_{\om,0})$ is admissible. Take a sequence of nonnegative  observables $v_n,$  satisfying  $\sup_n\text{ess}\sup_{\om}||v_{n,\om}||_{\mathcal{B}_{\om}}\le M$.
Suppose that for $m$ a.e. $\om\in \Omega$
$$
E_{\om, n}:=\sum_{k=0}^{n-1} \int v_{k,\sigma^k\om}(x)d\mu_{\sigma^k_{\om},0}(x)\rightarrow \infty, \ n\rightarrow \infty.
$$
Then for $m$ a.e. $\om\in \Omega$
$$
\lim_{n\rightarrow\infty}\frac{1}{E_{\om, n}}\sum_{k=0}^{n-1}v_{k,\sigma^k\om}(T^k_{\om}x)=1,
$$
for $\mu_{\om,0}$-almost all $x$.
\end{theorem}
\begin{proof}
Let us write 
$$
S_{\om}:=\int \left(\sum_{m<k\le n}\left(v_{k,\sigma^k\om}(T^k_{\om}y)-\int v_{k,\sigma^k\om}d\mu_{\sigma^k_{\om},0}\right)\right)^2d\mu_{\om,0}(y).
$$
If we could prove that $S_{\om}\le \text{constant}\sum_{m<k\le n} \int v_{k,\sigma^k\om}d\mu_{\sigma^k\om,0},$
then the result will follow by applying Sprindzuk theorem \cite{Spri}\footnote{We recall here the Sprindzuk theorem. Let $(\Omega, \mathcal{B}, \mu_0)$ be a probability space and let $(f_k)_k$ be a sequence of nonnegative and measurable functions on $\Omega.$ Moreover, let $(g_k)_k$ and $(h_k)_k$ be bounded sequences of real numbers such that $0\le g_k\le h_k.$ Assume that there exists $C>0$ such that
$$
\int \left(\sum_{m<k\le n}(f_k(x)-g_k)\right)^2d\mu_0(x)\le C\sum_{m<k\le n}h_k
$$
for $m<n.$ Then, for every $\eps>0$ we have
$$
\sum_{1\le k\le n}f_k(x)=\sum_{1\le k\le n}g_k+O(\Theta^{1/2}(n)\log^{3/2+\eps}\Theta(n)),
$$
for $\mu$ a.e. $x\in \Omega,$ where $\Theta(n)=\sum_{1\le k\le n}h_k.$
}, where we identify $\int v_{k,\sigma^k\om}d\mu_{\sigma^k\om,0}$ with the functions $g_k$ and $h_k$ in the previous footnote. So we have
\begin{align*}
S_{\om}&=\int \left[\sum_{m<k\le n}v_{k,\sigma^k\om}(T^k_{\om}y)\right]^2d\mu_{\om,0}(y)-\left[\sum_{m<k\le n}\int v_{k,\sigma^k\om}d\mu_{\sigma^k_{\om},0}\right]^2
\\
&=
\sum_{m<k\le n}\int v_{k,\sigma^k\om}(T^k_{\om}y)^2d\mu_{\om,0}+2\sum_{m<i<j\le n}\int v_{i,\sigma^i\om}(T^i_{\om}y)v_{j,\sigma^j\om}(T^j_{\om}y)d\mu_{\om,0}
\\
&\qquad-
\sum_{m<k\le n}\left[\int v_{k,\sigma^k\om}d\mu_{\sigma^k\om,0}\right]^2-2\sum_{m<i<j\le n}\int v_{i,\sigma^i\om}d\mu_{\sigma^i\om,0}\int v_{j,\sigma^j\om}d\mu_{\sigma^j\om,0}.
\end{align*}
We now discard the third negative piece and bound the first by H\"older inequality as
$$
\sum_{m<k\le n}\int v_{k,\sigma^k\om}(T^k_{\om}y)^2d\mu_{\om,0}\le \sum_{m<k\le n}M \int v_{k,\sigma^k\om}d\mu_{\sigma^k\om,0}.
$$
Then for the remaining two pieces we use the decay of correlations given in Theorem 2.21 of our paper \cite{AFGTV20}, where the observables are taken in $\mathcal{B}_{\om}$ and in $L^1(\mu_{\om,0}).$  We can apply it to our case since the two observables coincide with $v_{\om}$, and the measures $\mu_{\om,0}$ and $\nu_{\om,0}$ are equivalent. We point out that this result improves the decay bound given in item \eqref{ITEM 6} of Theorem \ref{EXISTENCE THEOREM} where the presence of holes for $\epsilon>0$ forced us to use $L^{\infty}(\nu_{\om,\eps})$ instead of $L^1(\nu_{\om,\eps}).$ 
Hence:
$$
\left|\int v_{i,\sigma^i\om}(T^i_{\om}y)v_{j,\sigma^j\om}(T^j_{\om}y)d\mu_{\om,0}-\int v_{i,\sigma^i\om}d\mu_{\sigma^i\om,0}\int v_{j,\sigma^j\om}d\mu_{\sigma^j\om,0}\right|\le
CM \kappa^{j-i}\int v_{j,\sigma^j\om}d\mu_{\sigma^j\om,0}.
$$
In conclusion, we get
\begin{align*}
S_{\om}&\le \sum_{m<k\le n}M \int v_{k,\sigma^k\om}d\mu_{\sigma^k\om,0}+C M\sum_{m<i<j\le n}\kappa^{j-i}\int v_{j,\sigma^j\om}d\mu_{\sigma^j\om,0}
\\
&\le
M\left(\frac{C}{1-\kappa}+1\right) \sum_{m<k\le n} \int v_{k,\sigma^k\om}d\mu_{\sigma^k\om,0},
\end{align*}
which satisfies Sprindzuk.
\end{proof}

\section{Examples}
\label{sec:examples}
In this section we present explicit examples of random dynamical systems to illustrate the quenched extremal index formula in Theorem \ref{evtthm}.
In all cases, 
Theorem \ref{EXISTENCE THEOREM} can be applied to guarantee the existence of all objects therein for the perturbed random open system (in brief, the thermodynamic formalism for the closed random system implies a thermodynamic formalism for the perturbed random open system).
In Examples \ref{example1}--\ref{example3} we derive explicit expressions for the quenched extreme value laws. 
Our examples are piecewise-monotonic interval map cocycles 
whose transfer operators act on $\mathcal{B}_\omega=\BV:=\BV([0,1])$ for a.e.\ $\omega$.
The norm we will use is $\|\cdot\|_{\mathcal{B}_\omega}=\|\cdot\|_{\BV_{\nu_{\om,0}}}:=\var(\cdot)+\nu_{\omega,0}(|\cdot|)$.
In Section \ref{sec: existence} we noted that (\ref{B}) is automatically satisfied.
Lemma \ref{DFGTV18Alemma} shows that (\ref{CCM}) holds under assumptions (\ref{E2}), (\ref{E3}), (\ref{E4}), and (\ref{E8}).
To obtain a random open system, in addition to (\ref{B}) and (\ref{CCM}) we need (\ref{M}) (implying \eqref{M1} and  (\ref{M2})), (\ref{A}), and (\ref{cond X}).
Therefore, in each example we verify conditions (\ref{A}), (\ref{cond X}) (Section~\ref{sec: open systems})  and (\ref{M}) and (\ref{E1})--(\ref{E9}) (Section~\ref{sec: existence});  and where required (\ref{xibound}). 
\begin{example}\label{example1}
\textbf{Random maps and random holes centered on a non-random fixed point (random observations with non-random maximum)}

We show that nontrivial quenched extremal indices occur even for holes centered on a non-random fixed point $x_0$.
To define a random (open) map, let $(\Om, m)$ be a complete probability space,
and $\sg:(\Om, m)\to (\Om, m)$ be an ergodic, invertible, probability-preserving transformation. For example, one could consider an irrational rotation on the circle. 
Let $\Om=\cup_{j=1}^\infty \Om_j$ be a countable partition of $\Om$ into measurable sets on which $\omega\mapsto T_\omega$ is constant. 
This ensures that \eqref{M} is satisfied.
For each $\om \in \Om$ let $\cJ_{\om,0} = [0,1]$ and let
$T_\omega:\mathcal{J}_{\omega,0}\to \mathcal{J}_{\sigma\omega,0}$ be random, with all maps fixing $x_0\in [0,1]$.
The observation functions $h_\omega:\mathcal{J}_{\omega,0}\to\mathbb{R}$ have a unique maximum at $x_0$ for a.e.\ $\omega$.

To make all of the objects and calculations as simple as possible and illustrate some of the underlying mechanisms for nontrivial extremal indices we use the following specific family of maps $\{T_\omega\}$:
\begin{equation}
\label{eg1}
T_\omega(x)=\left\{
	\begin{array}{ll}
	1-2x/(1-1/s_\omega), & \hbox{$0\le x\le (1-1/s_\omega)/2$;} \\
	s_\omega x-(s_\omega-1)/2, & \hbox{$(1-1/s_\omega)/2\le x\le (1+1/s_\omega)/2$;} \\ 1-(2x-(1+1/s_\omega))/(1-1/s_\omega)& \hbox{$(1+1/s_\omega)/2\le x\le 1$,}
	\end{array}
	\right.
	\end{equation}
	where $1<s\le s_\omega\le S<\infty$ and $s\rvert_{\Om_j}$ is constant for each $j\geq 1$;  thus (\ref{M1}) holds. 
	These maps have three full linear branches, and therefore Lebesgue measure is preserved by all maps $T_\omega$;  i.e.\ $\mu_{\omega,0}=\mathrm{Leb}$ for a.e. $\omega$.
	The central branch has slope $s_\omega$ and passes through the fixed point $x_0=1/2$, which lies at the center of the central branch;  see Figure \ref{fig:map}.
	\begin{figure}[hbt]
	\centering
	\includegraphics[width=6cm]{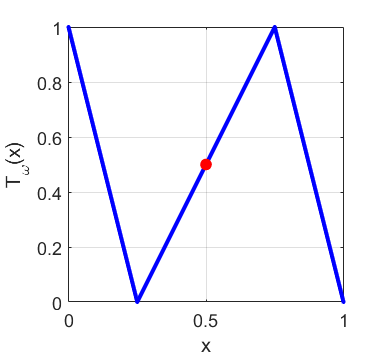}
	\caption{Graph of map $T_\omega$, with $s_\omega=2$.}\label{fig:map}
	\end{figure}
	A specific random driving could be $\sigma:S^1\to S^1$ given by $\sigma(\omega)=\omega+\alpha$ for some $\alpha\notin\mathbb{Q}$ and $s_\omega=s^0+s^1\cdot\omega$ for $1< s^0<\infty$ and $0<s^1<\infty$, but only the ergodicity of $\sigma$ will be important for us.
	Setting $g_{\omega,0}=1/|T'_\omega|$ it is immediate that (\ref{E1}), (\ref{E2}), (\ref{E3}), and (\ref{E4}) hold. 
	
	We select a measurable family of  observation functions $h:\Omega\times [0,1]\to\mathbb{R}$.
	For a.e.\ $\omega\in\Omega$, $h_\omega$ is $C^1$, has a unique maximum at $x_0=1/2$,  and is a locally even function about $x_0$.
	For small $\epsilon_N$ we will then have that $H_{\omega,\epsilon_N}=\{x\in [0,1]:h_\omega(x)>z_{\omega,N}\}$ is a small neighbourhood of $x_0$, satisfying (\ref{E5}) and (\ref{A}). 
	In light of Remark~\ref{rem check cond X} and Proposition \ref{prop check cond X} we see that \eqref{cond X} holds.
	The corresponding cocycle of open operators $\mathcal{L}_{\om,\ep}$ satisfy (\ref{M}).
	Given an essentially bounded function $t$, the $z_{\omega,N}$ are chosen to satisfy (\ref{xibound}). 
	Because the $h_\omega$ are $C^1$ with unique maxima and $\mu_{\omega,0}$ is Lebesgue for all $\omega$ we can make such a choice of $z_{\omega,N}$ with\footnote{In the situation where $z_{\omega,N}$ is constant a.e.\ for each $N\ge 1$, in general we cannot satisfy (\ref{xibound}) for a given fixed scaling function $t$.  In the case where $t$ is also constant a.e.\ if $h_\omega$ is random but sufficiently close to a non-random observation $h$ for each $\omega$, we can satisfy (\ref{xibound}) by incorporating the error term $\xi_{\omega,N}$. This situation corresponds to a decreasing family of holes whose length fluctuates slightly in $\omega$ with the fluctuation decreasing to zero sufficiently fast.  A similar situation can occur if the scaling function is a general measurable function.  These situations motivate the use of the error term $\xi_{\omega,N}$ in condition (\ref{xibound}).} $\xi_{\omega,N}\equiv 0$. 
	Furthermore, as $\mu_{\om,0}$ is Lebesgue we have that \eqref{xibound} implies that assumption \eqref{E6} is automatically satisfied.
	With the above choices,  (\ref{E7}) is clearly satisfied for $n'=1$.
	To show condition (\ref{E8}) with $n'=1$ we note that 
	$\|g_{\omega,0}\|_\infty=(1-1/s_\omega)/2<(1-1/S)/2$, while
	$\inf\mathcal{L}_{\omega,\ep}\ind\ge 2\cdot (1-1/s_\omega)/2\ge 1-1/s$.
	Therefore we require $(1-1/S)/2<1-1/s$, which by elementary rearrangement is always true for $0<s<S<\infty$.
	Because there is a full branch outside the branch with the hole, (\ref{E9}) is satisfied with $n'=1$ and $k_o(n')=1$.
	
	At this point we have checked all of the hypotheses of Theorem \ref{EXISTENCE THEOREM} and we obtain that for sufficiently small holes as defined above, the corresponding random open dynamical system has a quenched thermodynamic formalism and quenched decay of correlations. 
	We note that this result does not require (\ref{xibound});  the holes $H_{\epsilon_N}$ do not need to scale in any particular way with $N$, they simply need to be sufficiently small.
	In order to next apply Theorem \ref{evtthm} 
	we do require (\ref{xibound}) and additionally \eqref{C8}.

	To finish the example, we verify (\ref{C8}).
	We claim that for each fixed $k>0$, one has $\hat{q}_{\omega,\epsilon_N}^{(k)}=0$ for sufficiently sufficiently large $N$ (sufficiently small $\epsilon_N$).
	If this were not the case, there must be a positive $\mu_{\omega,0}$-measure set of points that (i) lie in $H_{\sigma^{-k}\omega,\epsilon_N}$, (ii)  land outside the sequence of holes $H_{\sigma^{-(k-1)}\omega,\epsilon_N},\ldots,H_{\sigma^{-1}\omega,\epsilon_N}$ for the next $k-1$ iterations, and then (iii) land in $H_{\omega,\epsilon_N}$ on the $k^{\mathrm{th}}$ iteration.
	In this example, because all $H_{\omega,\epsilon_N}$ are neighbourhoods of $x_0$ of diameter smaller than $(|t|_\infty+\gm_N)/N$, and the maps are locally expanding about $x_0$, for fixed $k$ one can find a large enough $N$ so that it is impossible to leave a small neighbourhood of $x_0$ and return after $k$ iterations.
	For $k=0$, by definition \begin{equation}
	\label{q0}
	\hat{q}_{\omega,\epsilon_N}^{(0)}=\frac{\mu_{\sigma^{-1}\omega,0}(T^{-1}_{\sigma^{-1}\omega} H_{\omega,\epsilon_N}\cap H_{\sigma^{-1}\omega,\epsilon_N})}{\mu_{\sigma^{-1}\omega,0}(T^{-1}_{\sigma^{-1}\omega}H_{\omega,\epsilon_N})}=\frac{\mu_{\sigma^{-1}\omega,0}(T^{-1}_{\sigma^{-1}\omega} H_{\omega,\epsilon_N}\cap H_{\sigma^{-1}\omega,\epsilon_N})}{\mu_{\omega,0}(H_{\omega,\epsilon_N})}.
	\end{equation}
	There are two cases to consider:
	
	\textit{Case 1:} $H_{\sigma^{-1}\omega,\epsilon_N}$ is larger than the local preimage of $H_{\omega,\epsilon_N}$;  that is, 
    $T^{-1}_{\sigma^{-1}\omega} H_{\omega,\epsilon_N}\cap H_{\sigma^{-1}\omega,\epsilon_N}\subset  H_{\sigma^{-1}\omega,\epsilon_N}$. 
    Because of the linearity of the branch containing $x_0$, one has $\hat{q}_{\omega,\epsilon_N}^{(0)}=1/T'_{\sigma^{-1}\omega}(x_0)$.
	
	\textit{Case 2:} $H_{\sigma^{-1}\omega,\epsilon_N}$ is smaller than the local preimage of $H_{\omega,\epsilon_N}$;  that is, 
    $T^{-1}_{\sigma^{-1}\omega} H_{\omega,\epsilon_N}\cap H_{\sigma^{-1}\omega,\epsilon_N}=  H_{\sigma^{-1}\omega,\epsilon_N}$.
	By (\ref{q0}) and (\ref{xibound}) we then have $$\frac{t_{\sigma^{-1}\omega}-\gm_N}{t_\omega+\gm_N}\le \hat{q}_{\omega,\epsilon_N}^{(0)}\le \frac{t_{\sigma^{-1}\omega}+\gm_N}{t_\omega-\gm_N}$$ and thus for such an $\omega$, $\lim_{N\to\infty}\hat{q}_{\omega,\epsilon_N}^{(0)}=t_{\sigma^{-1}\omega}/t_{\omega}$.
	Thus, combining the two cases,
	$$\hat{q}^{(0)}_{\omega,0}:=\lim_{N\to\infty}\hat{q}^{(0)}_{\omega,\epsilon_N}=\min\left\{\frac{t_{\sigma^{-1}\omega}}{t_\omega},\frac{1}{|T'_{\sigma^{-1}\omega}(x_0)|}\right\}$$ exists for a.e.\ $\omega$, verifying $\eqref{C8}$.
	
	Recalling that $\theta_{\omega,0}=1-\sum_{k=0}^\infty \hat{q}^{(k)}_{\omega,0}=1-\hat{q}^{(0)}_{\omega,0}$, we may now apply Theorem \ref{evtthm} to obtain the quenched extreme value law:
	\begin{eqnarray}
	\nonumber
	\lim_{N\to\infty}\nu_{\om,0}\left(X_{\om,N-1,\ep_N}\right)
	&	=&
	\exp\left(-\int_\Om t_\om\ta_{\om,0}\, dm(\om)\right)\\
	\label{evtlaweg1}&=&\exp\left(-\int_\Om t_\omega\left(1- \min\left\{\frac{t_{\sigma^{-1}\omega}}{t_\omega},\frac{1}{|T'_{\sigma^{-1}\omega}(x_0)|}\right\}\right)\, dm(\om)\right)
	\end{eqnarray}
	
	This formula is a generalization of the formula in Remark 8 \cite{keller_rare_2012}, where we may create nontrivial laws from either the random dynamics $T_\omega$, or the random scalings $t_\omega$, or both.
	The following two special cases consider these mechanisms separately.
	\begin{enumerate}
	\item \textit{Random maps, non-random scaling ($t_\omega$ takes a constant value $t>0$):}  Since $|T_\omega'|>1$, in this case (\ref{evtlaweg1}) becomes
	\begin{equation}
	\label{fixedscale}
	\lim_{N\to\infty}\nu_{\om,0}\left(X_{\om,N-1,\ep_N}\right)=\exp\left(-t\left[1-\int_\Omega \frac{1}{|T_\omega'(x_0)|}\ dm(\omega)\right]\right),
	\end{equation}
	and we see that we can interpret $\theta=1-\int_\Omega \frac{1}{|T_\omega'(x_0)|}\ dm(\omega)$ as an extremal index.
	
	\item \textit{Fixed map, random scaling:}
	Suppose $T_\omega\equiv T$;  then we may replace $T_\omega'(x_0)$ with $T'(x_0)$ in (\ref{evtlaweg1}),  
	and we see we the extremal index  depends on the choice of random scalings $t_\omega$ alone;  of course the thresholds $z_{\omega,N}$ depend on the chosen $t_\omega$.
	\end{enumerate}

	Similar results could be obtained with the $T_\omega$ possessing nonlinear branches.
	The arguments above can also be extended to the case where $x_0$ is a periodic point of prime period $p$ for all maps;   the formula (\ref{evtlaweg1}) now includes $(T_\omega^p)'(x_0)$ and $t_\omega,   t_{\sigma^{1}\omega},\ldots,t_{\sigma^{-(p-1)}\omega}$.
	If the scaling $t_\omega=t$ is non-random, one would simply replace $T'_\omega(x_0)$ in (\ref{fixedscale}) with $(T_\omega^p)'(x_0)$.
	
	We recall that for deterministic $T$ (including non-uniformly hyperbolic maps), 
	the extremal index enjoys a dichotomy, in the sense that it is equal to $1$ when a single hole $H_{\ep_N}$ shrinks to an aperiodic point, and it is strictly smaller than $1$ when the hole shrinks to a periodic point.
	In the latter case the extremal index can be expressed in terms of the period; see \cite{book} for a general account of the above facts. 
	Example \ref{example1} shows that in a simple random setting there are many more ways to obtain nontrivial exponential limit laws, e.g.\ by random scaling, or by the existence of periodic orbits for only a positive measure set of $\omega$.
	\end{example} 
	\begin{example}\label{example2}
	\textbf{Random $\beta$-maps, random holes containing a non-random fixed point (random observations with non-random maximum), general geometric potential}

	We show how a nontrivial extremal index can arise from random $\beta$-maps where statistics are generated by an equilibrium state of a general geometric potential.
	Consider the ``no short branches'' random $\beta$-map example of Section 13.2 \cite{AFGTV20}, where $\beta_\omega\in \{2\}\cup\bigcup_{2\le k\le K}[k+\delta,k+1]$ for a.e.\ $\omega$, and some $\delta>0$ and $K<\infty$.
	The measurability of $\omega\mapsto\beta_\omega$ yields (\ref{M1}).
	We use the weight $g_\omega=1/|T'_\omega|^r$, $r\geq 0$.
	To obtain (\ref{M}) the base dynamics is driven by an ergodic homeomorphism $\sigma$ on a Borel subset $\Omega$ of a separable, complete metric space, assuming that $\omega\mapsto \mathcal{
	L}_{\omega,0}$ has countable range, as in Example~\ref{example1}.
	Our random observation function $h:\Omega\times [0,1]\to \mathbb{R}$ is measurable and for a.e.\ $\omega$, $h_\omega$ is $C^1$ with a unique maximum at $x=0$.
	We select a measurable scaling function $t_\omega$;  either or both of $h_\omega$ and $t_\omega$ could be $\omega$-independent.
	By assigning thresholds $z_{\omega,N}$ (which could also be $\omega$-independent) we obtain a decreasing family of holes $H_{\om,\ep_N}$, which are decreasing intervals with left endpoint at $x=0$.
	
	Clearly (\ref{E2}) and (\ref{E3}) hold and Lemma 13.5 \cite{AFGTV20} provides (\ref{E4}).
	Conditions (\ref{E5}) and (\ref{E6}), and (\ref{A}) are immediate, as is (\ref{E7}) with $n'=1$.	
	Condition (\ref{EX}) holds because there is at least one full branch outside the branch with the hole.
	Regarding (\ref{E8}), 
	$\|g_{\omega,\epsilon}\|_\infty\le \beta_\omega^{-r}$ and arguing as in the previous example, 
	since $\inf\mathcal{L}_{\omega,\epsilon}\ind\ge \lfloor\beta_\omega\rfloor/\beta_\omega^r$, we see (\ref{E8}) holds with $n'=1$ because $\beta_\omega\ge 2$.
	Because there is at least one full branch outside the branch with the hole, (\ref{E9}) holds with $n'=1$ and $k_o(n')=1$.
	We have now checked all of the hypotheses of Theorem \ref{EXISTENCE THEOREM}.
	By this theorem, for sufficiently small holes $H_\epsilon$, the corresponding random open dynamical system has a quenched thermodynamic formalism and quenched decay of correlations. 
	As in the previous example, we emphasise that this result does not require (\ref{xibound}).
	
	Because $\phi_{\omega,0}$ is uniformly bounded above and below, as long as $z_{\omega,N}$ is random, we may adjust $z_{\omega,N}$ to obtain (\ref{xibound}).
	We now verify (\ref{C8}) for our holes, which are of the form $H_{\omega,\epsilon_N}=[0,r_{\omega,N}]$, 
	in preparation to apply Theorem \ref{evtthm}. 
	The same arguments from Case 1 and Case 2 of the previous example apply.
	Case 2 is unchanged.
	In Case 1 we have
	$$\hat{q}^{(0)}_{\omega,\epsilon_N}=\frac{\mu_{\sigma^{-1}\omega,0}(T^{-1}_{\sigma^{-1}\omega} H_{\omega,\epsilon_N}\cap H_{\sigma^{-1}\omega,\epsilon_N})}{\mu_{\sigma^{-1}\omega,0}(T^{-1}_{\sigma^{-1}\omega}H_{\omega,\epsilon_N})}.$$
	Because $T^{-1}_{\sigma^{-1}\omega} H_{\omega,\epsilon_N}$ is a finite union of left-closed intervals, using the fact that $\phi_{\sigma^{-1}\omega,0}\in \BV$, and therefore has left- and right-hand limits everywhere, we may redefine $\phi_{\sigma^{-1}\omega,0}(y)$ at the finite collection of points $y\in T^{-1}_{\sigma^{-1}\omega}(0)$ so that 
	$$\lim_{N\to\infty}\frac{\mu_{\sigma^{-1}\omega,0}(T^{-1}_{\sigma^{-1}\omega} H_{\omega,\epsilon_N}\cap H_{\sigma^{-1}\omega,\epsilon_N})}{\mu_{\sigma^{-1}\omega,0}(T^{-1}_{\sigma^{-1}\omega}H_{\omega,\epsilon_N})}=\frac{\phi_{\sigma^{-1}\omega,0}(0)}{\sum_{y\in T^{-1}_{\sigma^{-1}\omega}(0)} \phi_{\sigma^{-1}\omega,0}(y)},$$
	recalling that in Case 1, 
	$T^{-1}_{\sigma^{-1}\omega} H_{\omega,\epsilon_N}\subset H_{\sigma^{-1}\omega,\epsilon_N}$.
	Taking the minimum of Cases 1 and 2 as in Example \ref{example1}, we see that
	$\hat{q}_{\omega,0}^{(0)}=\lim_{N\to\infty} \hat{q}_{\omega,\epsilon_N}^{(0)}$ exists, which verifies \eqref{C8}.
	Finally,
	\begin{equation}
	\label{evtlaweg2}
	\lim_{N\to\infty}\nu_{\om,0}\left(X_{\om,N-1,\ep_N}\right)
	=\exp\left(-\int_\Om t_\omega\left(1- \min\left\{\frac{t_{\sigma^{-1}\omega}}{t_\omega},\frac{\phi_{\sigma^{-1}\omega,0}(0)}{\sum_{y\in T^{-1}_{\sigma^{-1}\omega}(0)} \phi_{\sigma^{-1}\omega,0}(y)}\right\}\right)\, dm(\om)\right)
\end{equation}
\end{example}

\begin{example}\label{example3}
\textbf{A fixed map with random holes containing a fixed point (random observations with a non-random maximum)}

Let us now consider more closely the case of a fixed map and holes moving randomly around a fixed point $z$, a situation previously considered in \cite{BahsounVaienti13} in the annealed framework. 

We take a piecewise uniformly expanding map $T$ of the unit interval $I$ of class $C^2$, and such that $T$ is surjective on the domains of local injectivity. Moreover $T$  preserves a mixing measure $\mu$ equivalent to the Lebesgue measure $\text{Leb},$ with strictly positive density $\rho$.  
We moreover assume that $T$ and $\rho$ are continuous at a fixed point $x_0$. 
The observation functions $h_\omega$ have a common maximum at $x_0$, leading to holes $H_{\omega,\epsilon_N}$ that are closed intervals  containing the point $x_0$ for any $N$.
In Example \ref{example1} the holes were centered on $x_0$ but could vary dramatically in diameter between $\omega$-fibers.
In this example the holes need not be centered on $x_0$ but must become more identical as they shrink.
Specifically, we assume
\begin{equation}\label{rf}
\sup_k\frac{\text{Leb}(H_{\omega,\epsilon_N}\Delta H_{\sigma^{-k}\omega,\epsilon_N})}{\text{Leb}(H_{\omega,\epsilon_N})}\rightarrow 0, \ N\rightarrow \infty,
\end{equation} 
where the use of Lebesgue is for simplicity.

Since the sample measures $\mu_{\omega,0}$ coincide with $\mu,$ 
we may easily verify  condition (\ref{xibound}) by choosing conveniently the size of the hole $H_{\omega,\epsilon_N}$.  
Moreover, by choosing the holes to be contained entirely within exactly one interval of monotonicity, we see that \eqref{EX} holds, and thus \eqref{cond X} holds via Remark \ref{rem check cond X} and Proposition \ref{prop check cond X}. 
Since the weights $g_\omega$ are nonrandom and equal to $1/|T'|,$ conditions (\ref{E1}) to (\ref{E7}) clearly hold. 
Since $T$ is continuous in $x_0,$, the holes will belong to one branch of $T.$  Therefore if $D(T)$ will denote the number of branches of $T$ and $\lambda_m:=\min_{I} |T'|,$ $\lambda_M:=\max_{I} |T'|,$ it will be enough to have $D(T)-1>\frac{\lambda_M}{\lambda_m}$ in order to satisfy (\ref{E9}) with $n'=1.$ Moreover, still keeping $n'=1$ and since we have finitely many branches, condition (\ref{E9a}) in Remark \ref{Alt E9 Remark} is satisfied with $k_o(n')=1$.
We may now apply Theorem \ref{EXISTENCE THEOREM} to obtain a quenched thermodynamic formalism for our fixed map with sufficiently small random holes.

Whenever the point $x_0$ is not periodic, one obtains that all the $\hat{q}^{(k)}_{\om,\epsilon_N}$ are zero by repeating the argument given in Example \ref{example1} for $k\ge 0$ and using the continuity of $T$ at $x_0$. 
We now take $x_0$ as a periodic point of minimal period $p$ and we assume that $T^p$ and $\rho$ are continuous at $x_0$.
We now begin to evaluate 
\begin{equation}\label{fed}
\hat{q}_{\omega,\epsilon_N}^{(p-1)}=\frac{\mu(T^{-p} H_{\omega,\epsilon_N}\cap H_{\sigma^{-p}\omega,\epsilon_N})}{\mu(H_{\omega,\epsilon_N})}. 
\end{equation}
Since $T^p$ is continuous and expanding in the neighborhood of  $x_0$, by taking $N$ large enough, the set $T^{-p} H_{\omega,\epsilon_N}\cap H_{\sigma^{-p}\omega,\epsilon_N}$ 
has only one connected component.
Denote the local branch of $T$ through $x_0$ by $T_{x_0}$.
Therefore by a local change of variable  we have for the upper bound of the numerator 
\begin{eqnarray}
\label{eg3eq1}
\mu(T^{-p} H_{\omega,\epsilon_N}\cap H_{\sigma^{-p}\omega,\epsilon_N})&=& \int_{T^p(T^{-p} H_{\omega,\epsilon_N}\cap H_{\sigma^{-p}\omega,\epsilon_N})}\rho(T^{-p}_{x_0}y)|DT^p(T^{-p}_{x_0}y)|^{-1}d\text{Leb}(y)\\
\nonumber &\le& \sup_{H_{\omega,\epsilon_N}}|DT^p|^{-1}\sup_{H_{\omega,\epsilon_N}}\rho\ \text{Leb}(H_{\omega,\epsilon_N}).
\end{eqnarray}
For the lower bound of the numerator, since $T^p$ is locally expanding, $T^p(H_{\sigma^{-p}\omega,\epsilon_N})\supset H_{\sigma^{-p}\omega,\epsilon_N}$, and by (\ref{eg3eq1})
\begin{align*}
&\mu(T^{-p} H_{\omega,\epsilon_N}\cap H_{\sigma^{-p}\omega,\epsilon_N})
\ge\int_{H_{\omega,\epsilon_N}\cap H_{\sigma^{-p}\omega,\epsilon_N}}\rho(T^{-p}_{x_0}y)|DT^p(T^{-p}_{x_0}y))|^{-1}d\text{Leb}(y)
\\
& \quad
\ge\inf_{H_{\omega,\epsilon_N}}|DT^p|^{-1}\inf_{H_{\omega,\epsilon_N}}\rho\ \text{Leb}(H_{\omega,\epsilon_N})-\int_{H_{\omega,\epsilon_N}\setminus H_{\sigma^{-p}\omega,\epsilon_N}}\rho(T^{-p}_{x_0}y)|DT^p(T^{-p}_{x_0}y)|^{-1}d\text{Leb}(y)    
\end{align*}

We can bound the second negative term as
$$
\int_{H_{\omega,\epsilon_N}\setminus H_{\sigma^{-p}\omega,\epsilon_N}}\rho(T^{-p}_{x_0}y)|DT^p(T^{-p}_{x_0}y)|^{-1}(y)d\text{Leb}(y)\le \sup_{H_{\omega,\epsilon_N}}\rho \  \sup_{H_{\omega,\epsilon_N}}|DT^p|^{-1}\ \text{Leb}(H_{\omega,\epsilon_N}\Delta H_{\sigma^{-p}\omega,\epsilon_N}).
$$

We are now in a position to verify the existence of the limit $\lim_{\epsilon_N\to 0} \hat{q}_{\omega,\epsilon_N}^{(p-1)}$.
Using the above bounds we have 
$$ \frac{\inf_{H_{\omega,\epsilon_N}}|DT^p|^{-1}\inf_{H_{\omega,\epsilon_N}}\rho\ \text{Leb}(H_{\omega,\epsilon_N})}{\sup_{H_{\omega,\epsilon_N}}\rho\text{Leb}(H_{\omega,\epsilon_N})}-\frac{\  \sup_{H_{\omega,\epsilon_N}}|DT^p|^{-1}\ \text{Leb}(H_{\omega,\epsilon_N}\Delta H_{\sigma^{-p}\omega,\epsilon_N})}{\text{Leb}(H_{\omega,\epsilon_N})}
$$
$$
\le \hat{q}_{\omega,\epsilon_N}^{(p-1)}\le \frac{\sup_{H_{\omega,\epsilon_N}}|DT^p|^{-1}\sup_{H_{\omega,\epsilon_N}}\rho\ \text{Leb}(H_{\omega,\epsilon_N})}{\inf_{H_{\omega,\epsilon_N}}\rho\text{Leb}(H_{\omega,\epsilon_N})}.$$

Since the map $T^p$ and the density $\rho$  are continuous at ${x_0},$  and by using the assumption (\ref{rf}), we will finally obtain the $\omega$-independent value
$$
\theta_{\omega, 0}=1-\frac{1}{|DT^p({x_0})|}.
$$
Theorem \ref{evtthm}  may now be applied to obtain the  quenched extreme value law. 
\end{example}

\begin{example}\label{example4}
\textbf{Random maps  with random holes}

We saw in Examples \ref{example1} and \ref{example3} that an extremal index less than one could be obtained for observables reaching their maximum around a point which was periodic for all the random maps, or, for a fixed map, when the holes shrink around the periodic point.
We now produce an example where periodicity is not responsible for getting an extremal index less than one. 
This example is the quenched version of the annealed cases investigated  in sections 4.1.2 and  4.2.1 of the paper \cite{Caby_et_al_2020}.

Let $\Omega=\{0,\dots,l-1\}^{\mathbb{Z}},$ with $\sigma$ the bilateral shift map, and $m$ an invariant ergodic measure.
To each letter $j=0,\dots,l-1$ we associate a point $v(j)$ in the unit circle  $S^1$ and we consider the well-known observable in the extreme value theory literature:
$$
h_{\omega}(x)=-\log|x-v(\omega_0)|, \ x\in S
$$
where $\omega_0$ denotes the $0$-th  coordinate of $\omega\in \Omega.$

In this case the hole $H_{\omega, \epsilon_N}$ will be the ball $B(v(\omega_0), e^{-z_N(\om)}),$ of center $v(\omega_0)$ and radius $e^{-z_N(\om)};$.
For each $\omega\in \Omega$ we associate a map $T_{\omega_0},$ where $T_0,\dots,T_{l-1}$ are maps  of the circle which we will take as $\beta$-maps of the form\footnote{The reason for this choice is that,  in order to compute the quantities  $\hat{q}_{\omega,0}^{(k)}$, we have to follow the itinerary of the points $v(i)$ under the composition of the maps $T_i$ and compare with their predecessors.  
	As it will be clear in the computation below for $k=0$ and $1,$ 
	the task will be relatively easy
	and generalizable to any $k > 0$, if all the maps are at least $C^1$ which in particular means that the image of the point $T_i(v_j), i,j=0,\dots,l-1$ is not a discontinuous point of the $T_i, i=0,\dots, l-1,$. 
	It will also be true that all the maps $T_i, 0=1,\dots, l-1$ are differentiable with bounded derivative in small neighbourhoods of any $v_i, i=0,\dots, l-1$.
	If these conditions are relaxed, it could be that the  limit defining the  $\hat{q}_{\omega,0}^{(k)}$ for some $k$ is not immediately computable, we defer to section 3.3 and Proposition 3.4 in \cite{AFV15} for a detailed discussion of the computation of the extremal index in presence of discontinuities. Another advantage of our choice is that, as in Example \ref{example1}, all the sample measure $\mu_{\om,0}$ are equivalent to Lebesgue, $\text{Leb}.$} $T_i(x)=\beta_ix+r_i\pmod 1$, with $\beta_i\in \mathbb{N}$, $\bt_i\geq 3$, and $0\le r_i<1$. Since the range of $\om\mapsto \mathcal{L}_{\omega,0}$ is finite and the shift is a homeomorphism on $\Omega$ with respect to the usual metrics for $\Omega$,  assumption (\ref{M}) is verified.
Since the potential is equal to $1/|T_\om'|,$ conditions (\ref{E1}) to (\ref{E7}) clearly hold.
Condition (\ref{E8}) holds as in Example \ref{example1}, which uses similar piecewise linear expanding maps. Condition (\ref{E9}) is a consequence of  the fact that we have finitely many maps each of which is full branches; it will be therefore enough to invoke (\ref{E9a}) with $n'=k_o(n')=1.$
As we have chosen $\bt_i\geq 3$ we have that \eqref{EX} holds and thus \eqref{cond X} follows via Proposition \ref{prop check cond X}.

At this point we may apply Theorem \ref{EXISTENCE THEOREM} to obtain a quenched thermodynamic formalism for sufficiently small holes.

As a more concrete example, we will consider an alphabet of four letters $\mathcal{A}:=\{0,1,2,3\}$  and we set the associations
$$
i\rightarrow v_i:=v(i), \  i=0,1,2,3,
$$
where the points $v_i\in (0,1)$, are chosen on the unit interval  according to the following prescriptions:
\begin{equation}\label{grammar}
T_1(v_1)=T_2(v_2)=v_0; \ T_0(v_0)=v_3; \ T_3(v_i)\neq v_3, i=0,1,2,3.
\end{equation}
Since the sample measures coincide with the Lebesgue measure, condition (\ref{xibound}) reduces to
$2e^{-z_N(\om)}=\frac{t_{\om}+\xi_{\om, N}}{N}$ and we can solve for $z_N(\om)$ by setting $\xi_{\om,N}\equiv0$. 
The prescription (\ref{grammar}) clearly avoids any sort of periodicity when we take the first iteration of the random maps needed to compute $\hat{q}_{\omega,0}^{(0)};$ nevertheless we will show that  $\hat{q}_{\omega,0}^{(0)}>0$ and this is sufficient to conclude that $\theta_{\om,0}$ is smaller than $1$. 
Of course we need to prove that the limits defining  all the other 
$\hat{q}_{\omega,0}^{(k)}$ for any $k>0$ exist.   We will show it for $k=1$ and the same arguments could be generalized to any $k>1.$

Using $T$-invariance of Lebesgue, we have (for simplicity we write $z_N$ instead of $z_N(\om)$):
$$
\hat{q}_{\omega,0}^{(0)}= \lim_{N\rightarrow \infty}\frac{\text{Leb}\left(T^{-1}_{\sigma^{-1}\om}
	B(v(\omega_0), e^{-z_N})\cap B(v((\sigma^{-1}\omega)_0), e^{-z_N})\right)}{\text{Leb}(B(v(\omega_0), e^{-z_N}))},
$$
provided the limit  exists, which we are going to establish.
Consider first a point $\omega$ in the cylinder $[\omega_{-1}=0, \omega_0=3]$; the quantity we have to compute is therefore
\begin{equation}\label{ratio}
\hat{q}_{\omega,0}^{(0)}=\frac{\text{Leb}\left(T^{-1}_{0}
	B(v_3, e^{-z_N})\cap B(v_0, e^{-z_N})\right)}{\text{Leb}(B(v_3, e^{-z_N}))}
\end{equation}
If $N$ is large enough, by the prescription (\ref{grammar}), $T_0(v_0)= v_3,$ we see that the local preimage of $B(v_3, e^{-z_N})$ under $T_0^{-1}$ will be strictly included into $B(v_0, e^{-z_N})$ and its length will be contracted by a factor $\beta_0^{-1}.$ Therefore the ratio (\ref{ratio}) will be simply $\beta_0^{-1}.$ The same happens with the cylinders $[\omega_{-1}=1, \omega_0=0]$ and $[\omega_{-1}=2, \omega_0=0],$  producing  respectively the quantities $\beta_1^{-1}$ and $\beta_2^{-1}.$ For all the other cylinders the ratio will be zero since the sets entering the numerator of (\ref{ratio}) will be disjoint for $N$ large enough. In conclusion we have
$$
\int_{\Omega}t_{\om}\hat{q}_{\omega,0}^{(0)}dm=\int_{[\omega_{-1}=0, \omega_0=3]}t_{\om}\beta_0^{-1}dm+\int_{[\omega_{-1}=1, \omega_0=0]}t_{\om}\beta_1^{-1}dm+\int_{[\omega_{-1}=2, \omega_0=0]}t_{\om}\beta_2^{-1}dm.
$$
If we now take $t_{\om}=t$, $\om$-independent, and we choose $m$ as the Bernoulli measure with equal weights $1/4$, the preceding expression assumes the simpler form
$$
\int_{\Omega}t\hat{q}_{\omega,0}^{(0)}dm=\frac{t}{16}(\beta_0^{-1}
+\beta_1^{-1}+\beta_2^{-1}).
$$
To compute $\hat{q}_{\omega,0}^{(1)}$ we have to split the integral over cylinders of length three. As a concrete example let us consider the cylinder $[\om_{-2}=i, \om_{-1}=j, \om_0=k],$ where $i,j,k\in \mathcal{A}.$ By using the preceding notations, we have to control the set
$$
T^{-1}_{i}T^{-1}_{j}
B(v_k, e^{-z_N})\cap T^{-1}_{i}B^c(v_j, e^{-z_N})\cap B(v_i, e^{-z_N})=
$$
$$
T^{-1}_{i}\left(T^{-1}_{j}
B(v_k, e^{-z_N})\cap B^c(v_j, e^{-z_N})\right)\cap B(v_i, e^{-z_N}).
$$
Let us first consider the intersection $T^{-1}_{j}
B(v_k, e^{-z_N})\cap B^c(v_j, e^{-z_N}).$ Call $u$ one of the preimages of $T^{-1}_{j}(v_k)$ and $T^{-1}_{j,u}$ the inverse branch giving $T^{-1}_{j,u}(v_k)=u$. If $u=v_j$ then the intersection $T^{-1}_{j,u}
B(v_k, e^{-z_N})\cap B^c(v_j, e^{-z_N})$ is  empty. Otherwise, 
by taking $N$ large enough, the set  $T^{-1}_{j,u}
B(v_k, e^{-z_N})$ will be completely included in $B^c(v_j, e^{-z_N})$ and moreover we have, by the linearity of the maps, $T^{-1}_{j,u}
B(v_k, e^{-z_N})=B(u, \beta_j^{-1}e^{-z_N}).$ We are therefore left with the computation of $T^{-1}_{i}B(u, \beta_j^{-1}e^{-z_N})\cap B(v_i, e^{-z_N}).$ If $u\in \mathcal{A}$ we proceed as in the computation of $\hat{q}_{\omega,0}^{(0)}$ by using the prescriptions (\ref{grammar}); otherwise such an intersection will be empty for $N$ large enough. For instance, if, in the example we are considering, we take $k=0, j=2, i=0$ and we suppose that among the $\beta_2$ preimages of $v_0$ there is, besides $v_0,$ $v_3$ too, namely $T_2(v_3)=v_0,$ and no other element of $\mathcal{A},$  \footnote{This condition does not intervene to compute $\hat{q}_{\omega,0}^{(0)}$ and shows that it is strictly less than one.} then we get the contribution for $\hat{q}_{\omega,0}^{(1)}:$
$$
\int_{[\om_{-2}=0, \om_{-1}=2, \om_0=0]}t_{\om}\beta_2^{-1}\beta_0^{-1}dm.
$$
From the above it follows that $\theta_{\om,0}$ exists for a.e.\ $\omega$, and that we may apply Theorem \ref{evtthm}.\\

It is not difficult to give an example where all the $\hat{q}_{\omega,0}^{(k)}$ can be explicitly  computed. Let us take our  beta maps  $T_i(x)=\beta_ix+r_i$-$\text{mod}\ 1$  in the particular case where all the $r_i$ are  equal to the irrational number $r.$ Then  take a sequence of random balls  $B(v((\sigma^{k}\om)_0), e^{-z_N}), k\ge 0$ with the centers $v((\sigma^{k}\om)_0), k\ge 0$ which are rational numbers. From what we discussed above, it follows that  a necessary condition to get a $\hat{q}_{\om,0}^{(k)}\neq0$ is that the center $v((\sigma^{-(k+1)}\om)_0)$ will be sent to the center $v(\om_0).$  Let $z$ one of this rational centers; the iterate $T^n_{\om}(z)$ has the form $T^n_{\om}(z)=\beta_{\om_{n-1}}\cdots \beta_{\om_0}z+k_nr,$-mod$1,$ where $k_n$ is an integer number.
Therefore such an iterate will be never a rational number which 
shows that all the $\hat{q}_{\om,0}^{(k)}=0$ for any $k\ge 0$ and $\om,$ and therefore $\theta_{\om,0}=1.$

\end{example}

\begin{subappendices}
\section{A version of \cite{DFGTV18A} with general weights}
\label{appB}

In this appendix we outline how to extend relevant results from \cite{DFGTV18A} from the Perron--Frobenius weight $g_{\omega,0}=1/|T_\omega'|$ to the  general class of weights $g_{\omega,0}$ in Section \ref{sec: existence}.
To begin, we note that there is a unique measurable equivariant family of functions $\{\phi_{\omega,0}\}_{\omega\in\Omega}$ guaranteed by Theorem 2.19 \cite{AFGTV20} (called $q_\omega$ there).
We wish to obtain uniform control on the essential infimum and essential supremum of $\phi_{\omega,0}$ for a suitable class of maps.

In \cite{DFGTV18A} we work with the space $\BV_1=\{h\in L^1(\Leb) :\var(h)<\infty\}$ and use the norm $\|\cdot\|_{\BV_1}=\var(\cdot)+\|\cdot\|_1$.
Here we have a measurable family of random conformal probability measures $\nu_{\omega,0}$ (guaranteed by Theorem 2.19 \cite{AFGTV20}) and we work with the random spaces $\cB_\om=\BV_{\nu_{\om,0}}=\{h\in L^1(\nu_{\omega,0}):\var(h)<\infty\}$ and the random norms $\|\cdot\|_{\cB_\om}=\var(\cdot)+\|\cdot\|_{L^1(\nu_{\omega,0})}$.
We also work with the normalised transfer operator $\tilde{\mathcal{L}}_{\omega,0}(f):=\mathcal{L}_{\omega,0}(f)/\nu_{\omega,0}(\mathcal{L}_{\omega,0}\ind)$.
All of the variation axioms (V1)--(V8) in \cite{DFGTV18A} hold with the obvious replacements.

\begin{proof}[Proof of Lemma \ref{DFGTV18Alemma}]
\mbox{ }

\textbf{C1':}
Since $\essinf_\om\inf\cL_{\om,0}\ind\geq \essinf_\om\inf g_{\om,0}$, \eqref{E3} and \eqref{fin sup L1} together imply that \eqref{C1'} holds.
\vspace{.25cm}

\textbf{C7' ($\epsilon=0$):} We show $\essinf_\omega\inf\phi_{\om,0}>0$; this will give us \eqref{C7'}.
The statement $\essinf_\omega\inf\phi_{\om,0}>0$ is a generalized version of Lemmas 1 and 5 \cite{DFGTV18A} and we follow the strategy in \cite{DFGTV18A}.  The result follows from Lemma \ref{lemma1} below, which in turn depends on (\ref{LY20}) and Lemma \ref{lemmaA1}.
\vspace{.25cm}

\textbf{C4' ($\epsilon=0$):}  It will be sufficient to show that there is a $K<\infty$ and $0<\gamma<1$ such that for all $f\in \cB_\om$ with $\nu_{\om,0}(f)=0$ and a.e.\ $\omega$, one has
\begin{equation}
\label{DECgeneral}
\|\tilde{\mathcal{L}}_{\omega,0}^n f\|_{\cB_{\sigma^n\omega}}\le K\gamma^n\|f\|_{\cB_\om}\mbox{ for all $n\ge 0$}.
\end{equation}
This is a generalized version of Lemma 2 \cite{DFGTV18A}, which has an identical proof, making the replacements outlined in the proof of Lemma \ref{lemmaA1} below, and using Lemmas \ref{lemmaA1}--\ref{lemmaA4} and (\ref{LY20}).
We also use the non-random equivalence (\ref{normequiv}) of the $\cB_\om$ norm to the usual $\BV_1$ norm.
\vspace{.25cm}

\textbf{C2 ($\epsilon=0$), C3 ($\epsilon=0$), C5' ($\epsilon=0$):}
We wish to show that 
there is a unique measurable, nonnegative family $\phi_{\om,0}$ with the property that $\phi_{\om,0}\in \cB_\om$, $\int \phi_{\om,0}\ d\nu_{\omega,0}=1$, $\tilde{\mathcal{L}}_{\omega,0} \phi_{\om,0}=\phi_{\sigma\omega,0}$ for a.e. $\omega$, $\esssup_\omega \|\phi_{\om,0}\|_{\cB_\om}<\infty$, and
\begin{equation}
\label{unifeqnphi2}
\esssup_\omega \|\phi_{\om,0}\|_{\cB_\om}.
\end{equation} 
We note that again the norm equivalence (\ref{normequiv}).
This is a generalized version of Proposition 1 \cite{DFGTV18A}.
To obtain this generalization, in the proof of Proposition 1 \cite{DFGTV18A}, one modifies the space $Y$ to become 
$$
Y=\{v:\Omega\times X\to\mathbb{R}: v\mbox{ measurable, } v_\omega:=v(\omega,\cdot)\in \cB_\om\mbox{ and }\esssup_\omega\|v_\omega\|_{\cB_\om}<\infty\}.
$$
All of the arguments go through as per \cite{DFGTV18A} with the appropriate substitutions.
Our modified proof of Proposition 1 \cite{DFGTV18A} will also use the modified Lemmas \ref{lemmaA1}--\ref{lemmaA4}, and inequality (\ref{LY20}) below.

\textbf{CCM:} Finally, we note that \eqref{CCM} follows from \eqref{C2} ($\ep=0$) together with non-atomicity of $\nu_{\om,0}$. But non-atomicity of $\nu_{\om,0}$ follows from the random covering assumption as in the proof of Proposition 3.1 \cite{AFGTV20}

\textbf{Full Support:}
Following the proof of Claim 3.1.1 of \cite{AFGTV20}, we are able to show that $\nu_{\om,0}$ is fully supported on $[0,1]$, i.e. $\nu_{\om,0}(J)>0$ for any non-degenerate interval $J\sub [0,1]$.
\end{proof}

We note that by (\ref{LYfull hat}) 
with $\epsilon=0$, we have our uniform Lasota--Yorke equality.
\begin{equation}
\label{LY20}
\var(\tilde{\mathcal{L}}^{n}_{\omega,0}\psi)\le A \alpha^n\var(\psi)+B\nu_{\om,0}(|\psi|),
\end{equation}
for all $n\ge 1$ and a.e.\ $\omega$.
This immediately provides a suitable general version of \eqref{H2} \cite{DFGTV18A}, which is that there is a $C<\infty$ such that
\begin{equation}
\label{H2'}
\|\tilde{\mathcal{L}}_{\omega,0}\psi\|_{\cB_{\sigma\omega}}\le C\|\psi\|_{\cB_{\omega}}\mbox{ for a.e.\ $\omega$.}
\end{equation}
Define random cones $\mathcal{C}_{a,\omega}=\{\psi\in \cB_\om: \psi\ge 0, \var\psi\le a\int \psi\ d\nu_{\om,0}\}$.

\begin{lemma}[General weight version of Lemma A.1 \cite{DFGTV18A}]
\label{lemmaA1} For sufficiently large $a>0$ we have that $\tilde{\mathcal{L}}_{\omega,0}^{RN}\mathcal{C}_{a,\omega}\subset \mathcal{C}_{a/2,\sigma^{RN}\omega}$ for sufficiently large $R$ and a.e.\ $\omega$.
\end{lemma}
\begin{proof}Identical to \cite{DFGTV18A}, substituting $\mathcal{C}_{a,\omega}$ for $\mathcal{C}_{a}$, $\tilde{\mathcal{L}}_{\omega,0}$ for $\mathcal{L}_{\omega,0}$, and $\nu_{\om,0}$ for Lebesgue.
\end{proof}

\begin{lemma}[General weight version of Lemma 1 \cite{DFGTV18A}]
\label{lemma1}
If one has uniform covering in the sense of (11) \cite{DFGTV18A}, then there is an $N$ such that for each $a>0$ and sufficiently large $n$, there exists $c>0$ such that
\begin{equation}
\label{l1eqn}
\essinf_\omega \tilde{\mathcal{L}}_\omega^{Nn}h\ge (c/2)\nu_{\om,0}(|h|)\mbox{ for every $h\in \mathcal{C}_{\omega,a}$ and a.e.\ $\omega$}.
\end{equation}
\end{lemma}
\begin{proof}
Making all of the obvious substitutions, as per Lemma \ref{lemmaA1} and its proof, we subdivide the unit interval into an equipartition according to $\nu_{\om,0}$ mass.
This is possible because $\nu_{\om,0}$ is non-atomic (Proposition 3.1 \cite{AFGTV20}).
We conclude, as in the proof of Lemma 1 \cite{DFGTV18A}, that there is an interval $J$ of $\nu_{\om,0}$-measure $1/n$ such that for each $f\in\mathcal{C}_{\omega,a},$ one has $\inf_J f\ge (1/2)\nu_{\om,0}(f)$.
Then using uniform covering and the facts that $\essinf_\omega\inf g_{\omega,0}>0$, $\esssup_\omega g_{\omega,0}<\infty$, and $\esssup_\omega D(T_\omega)<\infty$, we obtain $\essinf_\omega\inf \tilde{\mathcal{L}}_{\omega,0}^{k}f\ge\alpha^*_0>0$, where $k$ is the uniform covering time for the interval $J$.
The rest of the proof follows as in \cite{DFGTV18A}.
\end{proof}

\begin{lemma}[General weight version of Lemma A.2 \cite{DFGTV18A}]
\label{lemmaA2}
Assume that $\psi,\psi'\in \mathcal{C}_{a,\omega}$ and  $\int \psi\ d\nu_{\om,0}=\int \psi'\ d\nu_{\om,0}=1$. Then $\|\psi-\psi'\|_{\cB_\om}\le 2(1+a)\Theta_{a,\omega}(\psi,\psi')$.
\end{lemma}
\begin{proof}
Identical to \cite{DFGTV18A}, substituting $\nu_{\om,0}$ for Lebesgue.
The randomness of the Hilbert metric $\Theta_{a,\omega}$ only appears because the functions lie in $\mathcal{C}_{a,\omega}$.
\end{proof}

\begin{lemma}[General weight version of Lemma A.3 \cite{DFGTV18A}]
\label{lemmaA3}
For any $a\ge 2\var(\ind_X)$, we have that $\tilde{\mathcal{L}}_{\omega,0}^{RN}$ is a contraction on $\mathcal{C}_{\omega,a}$ for any sufficiently large $R$ and a.e.\ $\omega\in\Omega$.
\end{lemma}
\begin{proof}
The proof in \cite{DFGTV18A} may be followed, making the substitutions as in Lemma \ref{lemmaA1} and its proof.
The first inequality reads 
$$
\esssup_\omega \tilde{\mathcal{L}}^{RN}_{\omega,0} f\le \nu_{\sigma^{RN}\omega}(|\tilde{\mathcal{L}}^{RN}_{\omega,0} f|)+C_{var}(\tilde{\mathcal{L}}^{RN}_{\omega,0} f)\le (1+C_{var}a/2)\nu_{\om,0}(|f|)=1+C_{var}a/2,
$$
where we have used axiom (V3) \cite{DFGTV18A} and the weak contracting property of $\tilde{\mathcal{L}}^{RN}_{\omega,0}$ in the $\nu_{\om,0}$ norm.
The rest of the proof follows as in \cite{DFGTV18A}.
\end{proof}

Let $\cB_{\omega,0}=\{\psi\in \cB_\om: \int \psi\ d\nu_{\om,0}=0\}$.

\begin{lemma}[General weight version of Lemma A.4 \cite{DFGTV18A}]
\label{lemmaA4}

\end{lemma}
\begin{proof}
Identical to \cite{DFGTV18A}, substituting as per Lemma \ref{lemmaA1} and Lemma \ref{lemmaA2} and their proofs above, and using (\ref{H2'}).
\end{proof}

\begin{lemma}[General weight version of Lemma 5 \cite{DFGTV18A}]
\label{lemma5}
$\essinf_\omega \phi_{\om,0}\ge c/2$ for a.e.\ $\omega$.
\end{lemma}
\begin{proof}
Identical to \cite{DFGTV18A} with the appropriate substitutions.
\end{proof}

\section{A summary of checks that relevant results from \cite{C19} can be applied to $\BV_1$}
\label{appA}

The stability result Theorem 4.4 \cite{C19} assumes that the underlying Banach space is separable, however this separability assumption is only used to obtain measurability of various objects (and in fact Theorem 3.4 \cite{C19} may be applied to sequential dynamics).
We use these results for the non-separable space $\BV_1$ in the proof of Lemma \ref{harrylemma2}, and we therefore need to check that all relevant results in \cite{C19} hold for $\BV_1$, under the $m$-continuity assumption on $\omega\mapsto\mathcal{L}_{\omega,0}$;  the latter will provide the required measurability.
All section, theorem, proposition, and lemma numbers below refer to numbering in \cite{C19}.

There are no issues of measurability in Section 3, including Theorem 3.4, until Section 4, so we begin our justifications from Section 4.
\vspace{.5cm}

We now fix some notation from \cite{C19} which will be used only in this section. 
For each $\oio$, let$(X_\om,\|\cdot\|_\om)$ be a Banach space and $(X_\om,\|\cdot\|_\om,|\cdot|_\om)$ a normal Saks space\footnote{A Saks space is a Banach space $(E,\|\cdot\|)$ equipped with a second locally convex topology $\tau$ that is coarser than the norm topology such that the closed unit ball $B_E$ in the norm topology is closed and bounded in the $\tau$-topology. In our setting we take the $\tau$-topology to be the one generated by the second norm $|\cdot|$.}. For a thorough treatment of Saks spaces see the seminal work of Cooper \cite{Cooper}, and for an introduction into Saks spaces for dynamical systems see \cite{C19}. 
Let $\XX=\bigsqcup_{\oio}\{\om\}\times X_\om$, and let $\pi:\XX\to\Om$ denote the projection onto $\Om$. Let $\LL(X_\om,X_{\sg\om})$ denote the collection of bounded linear operators from $X_\om$ to $X_{\sg\om}$ and let $\End(\XX,\sg)$ denote the set of bounded linear endomorphisms of $\XX$ covering $\sg$, i.e. 
\begin{align*}
    \End(\XX,\sg)=\lt\{L:\XX\to\XX: \pi\circ L=\sg\circ\pi \text{ and } f\mapsto\tau_{\sg\om}(L(\om,f))\in\LL(X_\om,X_{\sg\om})\rt\},
\end{align*}
where $\tau_\om:\pi^{-1}(\om)\to X_\om$ is given by $\tau_\om(\om,f)=f$.
We now introduce the notion of a hyperbolic splitting from \cite{C19}. 
We denote the norm on $\LL((X_\om,|\cdot|_\om),(X_{\sg\om}, |\cdot|_{\sg\om}))$ by $|\cdot|$, the norm on $\LL(X_\om,X_{\sg\om})$ by $\|\cdot\|$, and the norm on $\LL((X_\om,\|\cdot\|_\om),(X_{\sg\om},|\cdot|_{\sg\om}))$ by $\trinorm{\cdot}$.
\begin{definition}[Definition 3.1 of \cite{C19}]\label{harry def 3.1}
    Suppose that $L\in\End(\XX,\sg)$, $d\in\NN$, $0\leq \mu<\lm$, $(E_\om)_{\oio}\in\prod_{\oio}\cG_d(X_\om)$ and $(F_\om)_{\oio}\in\prod_{\oio}\cG^d(X_\om)$. We say that $(E_\om)_{\oio}$ and $(F_\om)_{\oio}$ form  a $(\mu,\lm,d)$-\textit{hyperbolic splitting} for $L$, and that $L$ has a \textit{hyperbolic splitting index} $d$, if there exists constants $C_\lm,C_\mu,\Ta>0$ such that:
    \begin{enumerate}
        \item[(\Gls*{H1})]\myglabel{H1}{H1} For every $\oio$ we have $E_\om\bigoplus F_\om=X_\om$ and 
        \begin{align*}
            \max\{\|\Pi_{F_\om\|E_\om}\|, \|\Pi_{E_\om\|F_\om}\|\}\leq \Ta.
        \end{align*}
        \item[(\Gls*{H2})]\myglabel{H2}{H2} For each $\oio$ we have $L_\om F_\om=E_{\sg\om}$. Moreover, for every $n\in\NN$ and $f\in E_\om$ we have 
        \begin{align*}
            \|L_\om^n f\|\geq C_\lm \lm^n\|f\|.
        \end{align*}
        \item[(\Gls*{H3})]\myglabel{H3}{H3} For each $\oio$ we have $L_\om F_\om \sub F_{\sg\om}$ and for every $n\in\NN$ we have 
        \begin{align*}
            \|L_\om^n\rvert_{F_\om}\|\leq C_\mu \mu^n.
        \end{align*}
    \end{enumerate}
    We call $(E_\om)_{\oio}$ and $(F_\om)_{\oio}$ the \textit{equivariant fast} and \textit{slow} spaces for $L$, respectively. 
\end{definition}
Let $\LL\YY(C_1,C_2,r,R)$ denote the collection of operators $L\in\End(\XX,\sg)$ which satisfy the following uniform Lasota-Yorke inequality: for each $\oio$, $f\in X_\om$, and $n\in\NN$ we have 
\begin{align*}
    \|L_\om^n f\|\leq C_1 r^n\|f\|+C_2R^n|f|.
\end{align*}
We let $\End_S(\XX,\sg)$ denote the set of all Saks space equicontinuous endomorphisms meaning that $\sup_{\oio}\|L_\om\|<\infty$ and for each $\eta>0$ there exists $C_\eta>0$ such that for every $\oio$ and $f\in X_\om$
\begin{align*}
    |L_\om f|\leq \eta\|f\|+C_\eta|f|.
\end{align*}
Finally, if $L\in\End_S(\XX,\sg)$, then for each $\ep>0$ set 
\begin{align*}
    \cO_\ep(L)=\lt\{S\in \End(\XX,\sg): \sup_{\oio}\trinorm{L_\om-S_\om}<\ep\rt\}.
\end{align*}
\,
\\
For the convenience of the reader, we restate Theorem 4.4 of \cite{C19} here. 
\begin{theorem}\label{HarryThm4.4}
    Suppose that $(X,\|\cdot\|,|\cdot|)$ is a Saks space, with $(X,\|\cdot\|)$ a Banach space, that $\cQ=(\Om,\cF,m,\sg,X,Q)$ is a separable strongly measurable random linear system with ergodic invertible base and a uniform hyperbolic Oseledets splitting of dimension $d\in\NN$, and that $Q\in \LL\YY(C_1,C_2,r,R)\cap \End_S(\XX,\sg)$ for some $C_1,C_2,R>0$ and $r\in[0,e^{\mu_Q})$. There exists $\ep_0>0$ such that if $\cP=(\Om,\cF,m,\sg,X,Q)$ is a separable strongly measurable random linear system with $P\in\LL\YY(C_1,C_2,r,R)\cap \cO_{\ep_0}(\cQ)$, then $\cP$ also has an Oseledets splitting of dimension $d$. In addition, letting $\Lm_{Q,i}$ ($1\leq i\leq k_Q$) be the Lyapunov exponents of $\cQ$, 
    there exists $c_0<2^{-1}\min_{1\leq i\leq k_Q}\{\Lm_{i,Q}-\Lm_{i+1,Q}\}$ such that each $I_i=(\Lm_{i,Q}-c_0), \max\{\Lm_{i,Q}, \log(\dl_{1i}R)\}+c_0)$ ($1\leq i\leq k_Q$) separates the Lyapunov spectrum of $\cP$, and the corresponding projections satisfy 
    \begin{align*}
        \forall i\in\{1,\dots,k_Q\},\, m\text{-a.e. }\oio \quad \rank(\Pi_{I_i, P}(\om)=d_{i,Q},
    \end{align*}
    and 
    \begin{align*}
        \sup\lt\{\esssup_{\oio}\|\Pi_{I_i, P}(\om)\|: P\in\LL\YY(C_1,C_2,r,R)\cap \cO_{\ep_0}(\cQ), 1\leq i\leq k_Q\rt\}<\infty.
    \end{align*}
    Moreover, for every $\nu>0$ there exists $\ep_\nu\in(0,\ep_0)$ so that if $P\in$, then 
    \begin{align*}
        \sup_{1\leq i\leq d}|\gm_{i,Q}-\gm_{i,P}|\leq \nu,
    \end{align*}
    \begin{align*}
        \sup_{1\leq i\leq k_\cQ}\esssup_{\oio}\trinorm{\Pi_{I_i, Q}(\om),\Pi_{I_i, P}(\om)}\leq \nu,
    \end{align*}
    and 
    \begin{align*}
        \esssup_{\oio} d_H(F_{k_Q, Q}(\om), F_{k_P, P}(\om))\leq \nu.
    \end{align*}
\end{theorem}
We now describe the necessary changes in order to apply Theorem 4.4 to our setting. Theorem, Proposition, Lemma, and equation numbers below refer to numbers in \cite{C19}.

\textit{Theorem 4.4 (Theorem \ref{HarryThm4.4}):}
This is the main stability theorem.
We would substitute ``separable strongly measurable random dynamical system'' with ``$m$-continuous random dynamical system''. This theorem relies on Propositions 4.6 and 4.7, Lemma 4.11, Proposition 4.12, and Lemma 4.13.
\vspace{.5cm}

\emph{Proposition 4.6:} There is no measurability involved.
\vspace{.5cm}

\emph{Proposition 4.7:} Uses Lemmas 4.8 and 4.10. We note that the bounds (93)--(95) in the proof are simpler in our application of this result as our top space is one-dimensional.
\vspace{.5cm}

\emph{Lemma 4.8:} This may be replaced by Theorem 17 \cite{FLQ2}, which treats the $m$-continuous setting.  This removes any use of Proposition B.1 and Lemma B.7.
\vspace{.5cm}

\emph{Lemma 4.9:} This uses Lemma 4.8 and the fact that compositions of $m$-continuous maps are $m$-continuous. The latter replaces the use of Lemma A.5 \cite{GTQ14}, which is used in several results.  This replacement will not be mentioned further.
\vspace{.5cm}

\emph{Lemma 4.10:} We will assume that $\omega\mapsto\Pi_\omega$ is $m$-continuous in the statement of the lemma.  At the start of the proof we would now instead have $\omega\mapsto \Pi_\omega(X)$ is $m$-continuous by the definition of $m$-continuity (see e.g.\ (4) in \cite{FLQ2}); this removes the use of Lemma B.2 in the proof. 
Lian's thesis is quoted regarding measurable maps/bases connected with a measurable space $\Pi_\omega(X)$. 
In our application, $\Pi_\omega$ has rank 1 and therefore stating that there is a (in our case $m$-continuous) map $e:\Omega\to \Pi_\omega(X)$ is trivial.  One proceeds similarly for the dual basis.
\vspace{.5cm}

\emph{Lemma 4.11:} This concerns measurability and integrability of $\omega\mapsto \det(\mathcal{L}_\omega^n|E_{i,\omega})$ and $\omega\mapsto \log\|\mathcal{L}_\omega^n|E_{i,\omega}\|$ where $E_{i,\omega}$ is an Oseledets space.
$m$-continuity of $\omega\mapsto E_{i,\omega}$ is provided by Theorem 17 \cite{FLQ2}, removing the need for Lemma B.2.
The $m$-continuity of $\omega\mapsto \log\|\mathcal{L}_\omega^n|E_{i,\omega}\|$ follows from Lemma 7 \cite{FLQ2}.
Because in our application setting we only require one-dimensional $E_{i,\omega}$, the determinants are given by norms and there is nothing more to do concerning determinants.
This removes the need for Proposition B.8.
Lemma B.16 \cite{GTQ14} may be replaced with Lemma 7 \cite{FLQ2} to cover the $m$-continuous setting.
Proposition B.6 is not required in the $P$-continuity setting.
\vspace{.5cm}

\emph{Proposition 4.12:} There is no measurability involved.
\vspace{.5cm}

\emph{Lemma 4.13:} There is no measurability involved.

\section{A $\var$--$\nu_{\omega,0}(|\cdot|)$ Lasota--Yorke inequality}\label{appC}

Recall from Section~\ref{sec: existence} that $\cZ_{\om,0}^{(n)}$ denotes the partition of monotonicity of $T_\om^n$ and that $\sA_{\om,0}^{(n)}$ is the collection of all finite partitions of $[0,1]$ such that
\begin{align}\label{eq: def A partition App}
\var_{A_i}(g_{\om,0}^{(n)})\leq 2\norm{g_{\om,0}^{(n)}}_{\infty}
\end{align}
for each $\cA=\set{A_i}\in\sA_{\om,0}^{(n)}$.
Given $\cA\in\sA_{\om,0}^{(n)}$, we set $\cZ_{\om,*,\ep}^{(n)}:=\set{Z\in \widehat\cZ_{\om,\ep}^{(n)}(\cA): Z\sub X_{\om,n-1,\ep} }$ where $\widehat\cZ_{\om,\ep}^{(n)}(\cA)$ is the coarsest partition amongst all those finer than $\cA$ and $\cZ_{\om,0}^{(n)}$ such that all elements of $\widehat\cZ_{\om,\ep}^{(n)}(\cA)$ are either disjoint from $X_{\om,n-1,\ep}$ or contained in $X_{\om,n-1,\ep}$. Then \eqref{eq: def A partition App} implies that 
\begin{align}\label{eq: def A partition for g_ep App}
\var_{Z}(g_{\om,\ep}^{(n)})\leq 2\norm{g_{\om,0}^{(n)}}_{\infty}
\end{align}
for each $Z\in \cZ_{\om,*,\ep}^{(n)}$.
We now prove a Lasota--Yorke inequality inspired by Lemma \ref{ly ineq}.
\begin{lemma}\label{closed ly ineq App} 
For any $f\in\BV_{\nu_{\om,0}}$ we have 
\begin{align*}
\var(\cL_{\om,\ep}^n(f))\leq 9\norm{g_{\om,\ep}^{(n)}}_{\infty}\var(f)+
\frac{8\norm{g_{\om,\ep}^{(n)}}_{\infty}}{\min_{Z\in\cZ_{\om,*,\ep}^{(n)}(A)}\nu_{\om,0}(Z)}\nu_{\om,0}(|f|).
\end{align*}
\end{lemma}
\begin{proof}
Since $\cL_{\om,\ep}^n(f)=\cL_{\om,0}^n(f\cdot\hat X_{\om,n-1,\ep})$, if $Z\in\widehat\cZ_{\om,\ep}^{(n)}(\cA)\bs\cZ_{\om,*,\ep}^{(n)}$, then  $Z\cap X_{\om,n-1,\ep}=\emptyset$, and thus, we have $\cL_{\om,\ep}^n(f\ind_Z)=0$ for each $f\in\BV_{\nu_{\om,0}}$. Thus, considering only the intervals $Z$ in $\cZ_{\om,*,\ep}^{(n)}$, we are able to write 
\begin{align}\label{eq: ly ineq 1}
\cL_{\om,\ep}^nf=\sum_{Z\in\cZ_{\om,*,\ep}^{(n)}}(\ind_Z f g_{\om,\ep}^{(n)})\circ T_{\om,Z}^{-n}
\end{align} 
where 
$$	
T_{\om,Z}^{-n}:T_\om^n(I_{\om,\ep})\to Z
$$ 
is the inverse branch which takes $T_\om^n(x)$ to $x$ for each $x\in Z$. Now, since 
$$
\ind_Z\circ T_{\om,Z}^{-n}=\ind_{T_\om^n(Z)},
$$
we can rewrite \eqref{eq: ly ineq 1} as 
\begin{align}\label{eq: closed ly ineq 2}
\cL_{\om,\ep}^nf=\sum_{Z\in\cZ_{\om,*,\ep}^{(n)}}\ind_{T_\om^n(Z)} \lt((f g_{\om,\ep}^{(n)})\circ T_{\om,Z}^{-n}\rt).
\end{align}
So,
\begin{align}\label{closed var tr op sum}
\var(\cL_{\om,\ep}^nf)\leq \sum_{Z\in\cZ_{\om,*,\ep}^{(n)}}\var\lt(\ind_{T_\om^n(Z)} \lt((f g_{\om,\ep}^{(n)})\circ T_{\om,Z}^{-n}\rt)\rt).
\end{align}
Now for each $Z\in\cZ_{\om,*,\ep}^{(n)}$, using the fact that $\|g_{\om,\ep}^{(n)}\|_\infty\leq \|g_{\om,0}^{(n)}\|_\infty$ and \eqref{eq: def A partition for g_ep App} we have 
\begin{align}
&\var\lt(\ind_{T_\om^n(Z)} \lt((f g_{\om,\ep}^{(n)})\circ T_{\om,Z}^{-n}\rt)\rt)
\leq \var_Z(f g_{\om,\ep}^{(n)})+2\sup_Z\absval{f g_{\om,\ep}^{(n)}}
\nonumber\\
&\qquad\qquad\leq 3\var_Z(f g_{\om,\ep}^{(n)})+2\inf_Z\absval{f g_{\om,\ep}^{(n)}}
\nonumber\\
&\qquad\qquad\leq 3\norm{g_{\om,\ep}^{(n)}}_{\infty}\var_Z(f)+3\sup_Z|f|\var_Z(g_{\om,\ep}^{(n)})+2\norm{g_{\om,\ep}^{(n)}}_{\infty}\inf_Z|f|
\nonumber\\
&\qquad\qquad\leq 
3\norm{g_{\om,0}^{(n)}}_{\infty}\var_Z(f)+6\norm{g_{\om,0}^{(n)}}_{\infty}\sup_Z|f|+2\norm{g_{\om,0}^{(n)}}_{\infty}\inf_Z|f|
\nonumber\\
&\qquad\qquad\leq 
9\norm{g_{\om,0}^{(n)}}_{\infty}\var_Z(f)+8\norm{g_{\om,0}^{(n)}}_{\infty}\inf_Z|f|
\nonumber\\
&\qquad\qquad\leq
9\norm{g_{\om,0}^{(n)}}_{\infty}\var_Z(f)+8\norm{g_{\om,0}^{(n)}}_{\infty}\frac{\nu_{\om,0}(|f\rvert_Z|)}{\nu_{\om,0}(Z)}.
\label{closed var ineq over partition}
\end{align}
Using \eqref{closed var ineq over partition}, we may further estimate \eqref{closed var tr op sum} as
\begin{eqnarray}
\nonumber	\var(\cL_{\om,\ep}^nf)
&\leq& 
\sum_{Z\in\cZ_{\om,*,\ep}^{(n)}} \lt(9\norm{g_{\om,0}^{(n)}}_{\infty}\var_Z(f)+8\norm{g_{\om,0}^{(n)}}_{\infty}\frac{\nu_{\om,0}(|f\rvert_Z|)}{\nu_{\om,0}(Z)}\rt)
\\
\label{ORLYLY}
&\leq& 
9\norm{g_{\om,0}^{(n)}}_{\infty}\var(f)+
\frac{8\norm{g_{\om,0}^{(n)}}_{\infty}}{\min_{Z\in\cZ_{\om,*,\ep}^{(n)}(A)}\nu_{\om,0}(Z)}\nu_{\om,0}(|f|),
\end{eqnarray}
and thus we are done.
\end{proof}
\begin{remark}\label{Alt E9 Remark}
Note that we could have used $\Leb$ or any probability measure in \eqref{closed var ineq over partition} rather than $\nu_{\om,0}$. Furthermore, Lemma~\ref{closed ly ineq App} could be applied to Section~\ref{sec: existence} with a measure other than $\nu_{\om,0}$ if the appropriate changes are made to the assumption \eqref{E9} so that a uniform-in-$\om$ lower bound similar to \eqref{LY LB calc} may be calculated. 

In particular if we replace \eqref{E9} with the following:
\begin{enumerate}[align=left,leftmargin=*,labelsep=\parindent]
\item[(\Gls*{E9a})]\myglabel{E9a}{E9a}
There exists $k_o(n')\in\NN$ and $\dl>0$ such that for $m$-a.e. $\om\in\Om$, all $\ep>0$ sufficiently small, we have $\Leb(Z)>\dl$ for all $Z\in\cZ_{\om,*,\ep}^{(n')}(\cA)$,
\item[(\Gls*{E9b})]\myglabel{E9b}{E9b}
There exists $c>0$ such that $\essinf_\om |T_\om'|>c$,
\end{enumerate}
then the claims of Section~\ref{sec: existence} hold with $\nu_{\om,0}$ in \eqref{closed var ineq over partition} replaced with $\Leb$.

Indeed, to obtain a replacement for \eqref{LY LB calc} one could use \eqref{E9a} and \eqref{E9b} to get 
\begin{align*}
\Leb(Z)
&=
\Leb(\~\cL_{\om,0}^{k_o(n')}\ind_Z)
\geq 
\frac{\inf g_{\om,0}^{(k_o(n'))}\inf J_\om^{(k_o(n'))}}{\lm_{\om,0}^{k_o(n')}}\Leb(P_\om^{k_o(n')}\ind_Z)
\\
&\geq \essinf_\om \frac{\inf g_{\om,0}^{(k_o(n'))}\inf J_\om^{(k_o(n'))}}{\lm_{\om,0}^{k_o(n')}}\Leb(Z)
>0
\end{align*}
for all $Z\in\cZ_{\om,*,\ep}^{(n')}(\cA)$, where we have also used \eqref{E3} for the final inequality.
As \eqref{E9} is only used to prove \eqref{LY LB calc}, the remainder of Section~\ref{sec: existence} can be carried out with the appropriate notational changes. In fact, the proof of Lemma~\ref{harrylemma2} can be simplified by replacing $\nu_{\om,0}$ with $\Leb$ as Lemma 5.2 of \cite{BFGTM14} would no longer be needed. 

\end{remark}

\begin{remark}
Note that the $2$ appearing in \eqref{eq: def A partition App}, and thus the $9$ and $8$ appearing in \eqref{ORLYLY}, are not optimal. See \cite{AFGTV20} and Section \ref{sec: examples} for how these estimates can be improved. 
\end{remark}

\section{Proof of Claim \protect{\eqref{ITEM 6}} of Theorem \protect{\ref{EXISTENCE THEOREM}}}\label{appDec}
\begin{proof}[Proof of Claim \eqref{ITEM 6} of Theorem \ref{EXISTENCE THEOREM}:]
Set $\cB_\om=\BV_{\nu_{\om,0}}$.
Define the fully normalized operator $\sL_{\om,\ep}:\cB_\om\to\cB_{\sg\om}$ given by
$$
\sL_{\om,\ep}(f):=\frac{1}{\rho_{\om,\ep}\psi_{\sg\om,\ep}}\cL_{\om,\ep}(f\cdot \psi_{\om,\ep}).
$$
Then we have that $\sL_{\om,\ep}\ind = \ind$ and $\mu_{\sg\om,\ep}(\sL_{\om,\ep}f)=\mu_{\om,\ep}(f)$.
Following the proof of Theorem \ref{thm: exp convergence of tr op}, we can prove the following similar statement to Claim \eqref{item 5} of Theorem \ref{EXISTENCE THEOREM}: 	For each $h\in\cB_\om$, $m$-a.e. $\om\in\Om$ and all $n\in\NN$ we have
\begin{align}\label{exp conv fn trop}
\norm{\sL_{\om,\ep}^n h - \mu_{\om,\ep}(h)\ind}_{\cB_{\sg^n\om}}
=\norm{\sL_{\om,\ep}^n \hat h}_{\cB_{\sg^n\om}}
\leq D\norm{h}_{\cB_\om}\kp_\ep^n,
\end{align}
where $\hat h := h - \mu_{\om,\ep}(h)$ and $D$, $\kp_\ep$ are as in Claim \eqref{item 5} of Theorem \ref{EXISTENCE THEOREM}.
Using standard arguments (see Theorem 11.1 \cite{AFGTV20}) we have that
\begin{align*}
\absval{
	\mu_{\om,\ep}
	\lt(\lt(f\circ T_{\om}^n\rt)h \rt)
	-
	\mu_{\sg^{n}\om,\ep}(f)\mu_{\om,\ep}(h)
}
= \mu_{\sg^n\om,\ep}\lt(\lt|f\sL_{\om,\ep}^n\hat h\rt|\rt).
\end{align*}
Note that at this stage we are unable to apply \eqref{exp conv fn trop} as the $\|\cdot\|_{\cB_\om}$ norm and the measure $\mu_{\om,\ep}$ are incompatible.
Now from the third statement of Claim \eqref{item 5} of Theorem \ref{EXISTENCE THEOREM} we have that 
\begin{align*}
\lt|\mu_{\sg^n\om,\ep}\lt(\lt|f\sL_{\om,\ep}^n\hat h\rt|\rt)
- 
\frac{\vrho_{\sg^n\om,\ep}\lt(\lt|f \sL_{\om,\ep}^n\hat h\rt|\hat X_{\sg^n\om,n,\ep}\rt)}
{\vrho_{\sg^n\om,\ep}\lt( X_{\sg^n\om,n,\ep}\rt)}
\rt|
\leq D\norm{f\sL_{\om,\ep}^n\hat h}_{\cB_{\sg^n\om}}\kp_\ep^n, 
\end{align*}
and thus we must have that 
\begin{align}
\mu_{\sg^n\om,\ep}\lt(\lt|f\sL_{\om,\ep}^n\hat h\rt|\rt)
\leq 
D\norm{f\sL_{\om,\ep}^n\hat h}_{\cB_{\sg^n\om}}\kp_\ep^n +
\frac{\vrho_{\sg^n\om,\ep}\lt(\lt|f\sL_{\om,\ep}^n\hat h\rt|\hat X_{\sg^n\om,n,\ep}\rt)}
{\vrho_{\sg^n\om,\ep}(X_{\sg^n\om,n,\ep})}. \label{dec of corr proof final ineq1}
\end{align}
Using \eqref{exp conv fn trop} and \eqref{normequiv2}, we have that 
\begin{align}
\frac{\vrho_{\sg^n\om,\ep}\lt(\lt|f\sL_{\om,\ep}^n\hat h\rt|\hat X_{\sg^n\om,n,\ep}\rt)}
{\vrho_{\sg^n\om,\ep}(X_{\sg^n\om,n,\ep})}
&\leq \norm{f\sL_{\om,\ep}^n\hat h}_{\sg^n\om,\infty}
\leq \norm{f}_{\sg^n\om,\infty}\norm{\sL_{\om,\ep}^n\hat h}_{\cB_{\sg^n\om}}
\leq D\norm{f}_{\infty,\om}\|h\|_{\cB_\om}\kp_\ep^n.
\label{dec of corr proof final ineq2}
\end{align}
Combining \eqref{dec of corr proof final ineq1} and \eqref{dec of corr proof final ineq2} and using \eqref{exp conv fn trop} again we see that 
\begin{align*}
&\absval{
	\mu_{\om,\ep}
	\lt(\lt(f\circ T_{\om}^n\rt)h \rt)
	-
	\mu_{\sg^{n}\om,\ep}(f)\mu_{\om,\ep}(h)
}
\leq  \mu_{\sg^n\om,\ep}\lt(\lt|f\sL_{\om,\ep}^n\hat h\rt|\rt)
\\
&\qquad\qquad\leq
D\norm{f\sL_{\om,\ep}^n\hat h}_{\cB_{\sg^n\om}}\kp_\ep^n +
\frac{\vrho_{\sg^n\om,\ep}\lt(\lt|f\sL_{\om,\ep}^n\hat h\rt|\hat X_{\sg^n\om,n,\ep}\rt)}
{\vrho_{\sg^n\om,\ep}(X_{\sg^n\om,n,\ep})}
\\
&\qquad\qquad\leq 
D\norm{f\sL_{\om,\ep}^n\hat h}_{\cB_{\sg^n\om}}\kp_\ep^n +
D\norm{f}_{\infty,\om}\|h\|_{\cB_\om}\kp_\ep^n
\\
&\qquad\qquad\leq 
D^2\norm{f}_{\cB_\om}\norm{h}_{\cB_\om}\kp_\ep^{2n} +
D\norm{f}_{\infty,\om}\|h\|_{\cB_\om}\kp_\ep^n
\\
&\qquad\qquad\leq 
\~D\norm{f}_{\infty,\om}\|h\|_{\cB_\om}\kp_\ep^n
\end{align*}
for all $n$ sufficiently large,
and thus the proof of Claim \eqref{ITEM 6} of Theorem \ref{EXISTENCE THEOREM} is complete.

\end{proof}
\end{subappendices}

\chapter{Open Questions and Future Directions}\label{part 3}
In this final chapter we present several open questions and directions that have yet to be explored. 
\section[Questions from Chapter 1]{Questions from Chapter \ref{part 1}}
\begin{question}
    Can the hypotheses \eqref{cond Q1} on the maximum number of contiguous ``bad'' intervals be weakened or removed for ``large'' holes?  
\end{question}
Chapter \ref{part 2} answered this question in the affirmative for ``small'' holes, meaning holes that are taken to be sufficiently small such that a perturbation theory applies. However, it is not clear at all whether this assumption can be removed or weakened in the case of ``large'' holes. Furthermore, the assumptions on the contiguous ``bad'' intervals does not translate well to higher dimensional systems, so it seems necessary to develop a better assumption in order to develop the theory of ``large hole'' open dynamical systems for higher dimensions.
\,
\begin{question}
    Can the hypotheses of large images and large images with respect to $H$ be removed from Theorem \ref{thm: Bowen's Formula}?
\end{question}
This question is answered in the affirmative for a large class of  expanding random interval maps with holes in an upcoming work of the first author and Nathan Dalaklis. 
\,
\begin{question}
    Are finer statistical limit laws applicable to open dynamical systems?
\end{question}
In \cite{CM99} (see also \cite[Theorem 8.34]{CMM}) the authors prove a ``conditioned'' Central Limit Theorem for ACCIMs (referred to there as quasi-stationary distributions) in the context of finite state Markov chains, where the relevant measure is conditioned on the surviving sets $X_n$. 
Within the language Chapter \ref{part 1}, Collet and Mart\'inez prove that for sufficiently ``nice'' observables $h$ and ACCIMs $\eta$
\begin{align*}
\lim_{n\to\infty}\frac{\eta\lt(X_n\cap\lt\{x\in [0,1]:\frac{1}{\sqrt n}\sum_{j=0}^{n-1}h \circ T^j(x) \le a\rt\}\rt)}{\eta(X_n)}
=
\frac{1}{\sqrt{2\pi \sigma ^2(f)}} \int ^a _{-\infty} \exp\left(-\frac{z^2}{2\sigma ^2(h)}\right)\ dz.
\end{align*}
It begs the question whether a similar result holds for our random open systems and whether stronger statistical limit laws can be proven in either the deterministic or random settings.  
\begin{question}
    Can our approach to thermodynamic formalism for random open systems be extended to random higher dimensional piecewise expanding or hyperbolic systems?
\end{question}
The approach of Chapter \ref{part 1} heavily relies on  the geometry of the interval. In particular, there appears to be no straightforward generalization of our assumption \eqref{cond Q1} to higher dimensions. If one is willing to consider small holes, the perturbative approach of Chapter \ref{part 2}, in particular Section \ref{sec: existence}, offers a promising approach. Even still, one must then find a suitable Banach space in which to work, as $\BV$ is not necessarily the most useful space for higher dimensions. The quasi-H\"older, or Keller spaces \cite{kellerBV}, could be a suitable option as it has been used in both higher dimensions \cite{saussol00} as well as open systems \cite{ferguson_escape_2012, pollicott_open_2017}.
\\
\section[Questions from Chapter 2]{Questions from Chapter \ref{part 2}}
\begin{question}
    Can the perturbative approach to thermodynamic formalism for random open systems taken in Section \ref{sec: existence} be used on other (closed) random dynamical systems?
\end{question}
    The perturbative approach to open dynamical systems has been investigated previously in the deterministic setting in \cite{LMD,ferguson_escape_2012,FFT015,pollicott_open_2017,BDT18}. Can the perturbative approach be used to prove thermodynamic formalism results, like Theorem \ref{main thm: existence}, in other random settings, e.g. \cite{mayer_distance_2011}? 
\begin{question}
      In Theorem \ref{urp main thm B} (Theorem \ref{evtthm}) we introduced the spectral approach to extreme value theory and return time statistics---pioneered by Keller \cite{keller_rare_2012}---to the quenched random setting. Can this perturbative spectral approach be adapted to investigate higher return time statistics?
\end{question}
In \cite{AFGTV-CP}, we have adapted this approach to prove that the sequence of random variables that count the number of returns to a sequence of random shrinking targets converges to a compound Poisson distribution. In particular, it is shown that under the assumption of the scaling \eqref{xibound}, the distributions
\begin{align*}
    \mu_{\om,0}\left(\sum_{j=0}^{n-1}\ind_{H_{\sg^j\om,n}}(T_\om^jx)=k\right)
    \qquad \text{ for } k=0,1,\dots
\end{align*}
converge almost surely to a compound Poisson random variable.
One can then ask what further results can be obtained with this spectral approach. Can the techniques Chapter \ref{part 2} and \cite{AFGTV-CP} be adapted to prove stable laws? 
\begin{question}
    Can asymptotic formulas for the change in the Hausdorff dimension of the perturbed random surviving sets be found in the style of Theorem \ref{thm: dynamics perturb thm}?
\end{question}
    In \cite{pollicott_open_2017} Pollicott and Urba\'nski find asymptotic results \cite[Theorem 1.0.9]{pollicott_open_2017}, with a similar flavor to Theorem \ref{thm: dynamics perturb thm}, for the Hausdorff dimension of the (deterministic) perturbed surviving set $X_{\infty,\ep}$ as $\ep\to 0$ in the setting of conformal graph directed Markov systems. Can such a result be proven in the random setting for the interval maps of Section \ref{sec: existence} in either the presence or absence of the large images property? In particular, what can be said of the existence and/or value of the limiting quantity
    \begin{align*}
        \lim_{\ep\to 0}\frac{1-\HD(X_{\om,\infty,\ep})}{\mu_{\om,0(H_{\om,\ep})}}?
    \end{align*}

\clearpage
\thispagestyle{plain}
\chapter*{Acknowledgments}
The authors thank Harry Crimmins and Yushi Nakano for helpful discussions.
JA, GF, and CG-T thank the Centro de Giorgi in Pisa and CIRM in Luminy for their support and hospitality.
JA is supported by the ARC Discovery projects DP180101223 and DP210100357, and thanks the School of Mathematics and Physics at the University of Queensland  for their hospitality.
GF, CG-T, and SV are partially supported by the ARC Discovery Project DP180101223.
The research of SV was supported by the project {\em Dynamics and Information Research Institute} within the agreement between UniCredit Bank R\&D group for financial support through the “Dynamics and Information Theory Institute” at the Scuola Normale Superiore di Pisa and by the Laboratoire International Associé LIA LYSM, of the French CNRS and  INdAM (Italy).
The authors thank an anonymous referee for their helpful comments and bibliographic suggestions.

\bibliographystyle{plain}
\bibliography{combined}

\clearpage 

\printglossary[title=Glossary of Assumptions]
\clearpage
\printindex
\clearpage

\end{document}